\documentclass[11pt,a4paper]{article}
\usepackage{geometry}
\usepackage[english]{babel}
\usepackage{amsmath,amssymb,graphicx,textcomp,enumerate,xcolor,latexsym,amsthm,breqn}
\usepackage{mathrsfs,dsfont}

\usepackage[hidelinks]{hyperref}
\usepackage[alphabetic]{amsrefs}

\allowdisplaybreaks

\newcommand{\assign}{:=}
\newcommand{\cdummy}{\cdot}
\newcommand{\comma}{{,}}
\newcommand{\mathD}{\mathrm{D}}
\newcommand{\mathd}{\mathrm{d}}
\newcommand{\nobracket}{}

\newcommand{\nosymbol}{}

\newcommand{\tmem}[1]{{\em #1\/}}

\newcommand{\tmop}[1]{\ensuremath{\operatorname{#1}}}
\newcommand{\tmscript}[1]{\text{\scriptsize{$#1$}}}

\newcommand{\tmtextit}[1]{{\itshape{#1}}}
\newcommand{\um}{-}

\newcommand{\para}{\,\mathord{\prec}\,}
\newcommand{\lpara}{\,\mathord{\succ}\,}
\newcommand{\mpara}{\,\mathord{\prec\!\!\!\prec}\,}
\newcommand{\reso}{\,\mathord{\circ}\,}

\newcommand{\rbe}{\mathrm{rbe}}
\newcommand{\kpz}{\mathrm{kpz}}
\newcommand{\rhe}{\mathrm{rhe}}
\newcommand{\expp}{\mathrm{exp}}

\newcommand{\arr}[1]{\overleftarrow{#1}}

\newtheorem{theorem}{Theorem}[section]
\newtheorem{lemma}[theorem]{Lemma}
\newtheorem{corollary}[theorem]{Corollary}
\newtheorem{remark}[theorem]{Remark}
\newtheorem{proposition}[theorem]{Proposition}
\newtheorem*{assumption}{Assumption}

\theoremstyle{definition}
\newtheorem{definition}[theorem]{Definition}

\newenvironment{enumeratenumeric}{\begin{enumerate}[1.] }{\end{enumerate}}
\newcommand{\zzone}{\text{\resizebox{.7em}{!}{\includegraphics{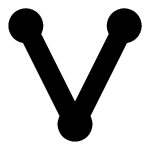}}}}
\newcommand{\zztwo}{\text{\resizebox{.7em}{!}{\includegraphics{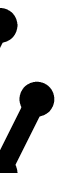}}}}
\newcommand{\zzthree}{\text{\resizebox{.7em}{!}{\includegraphics{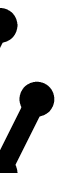}}}}
\newcommand{\zzfour}{\text{\resizebox{1em}{!}{\includegraphics{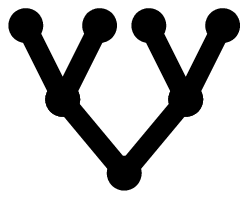}}}}
\newcommand{\zzfive}{\text{\resizebox{.7em}{!}{\includegraphics{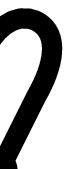}}}}
\newcommand{\zzsix}{\text{\resizebox{.7em}{!}{\includegraphics{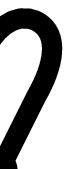}}}}
\newcommand{\zzseven}{\text{\resizebox{.7em}{!}{\includegraphics{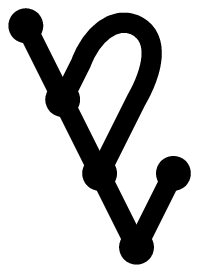}}}}
\newcommand{\zzeight}{\text{\resizebox{.7em}{!}{\includegraphics{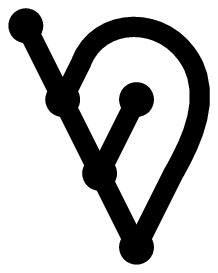}}}}
\newcommand{\zzthreereso}{\text{\resizebox{.7em}{!}{\includegraphics{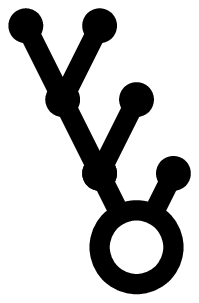}}}}
\newcommand{\zzfivereso}{\text{\resizebox{.7em}{!}{\includegraphics{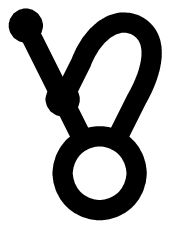}}}}
\newcommand{\zzsixreso}{\text{\resizebox{.7em}{!}{\includegraphics{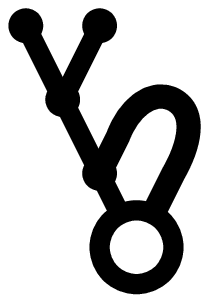}}}}
\newcommand{\zzsevenreso}{\text{\resizebox{.7em}{!}{\includegraphics{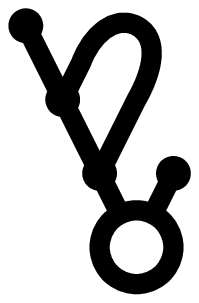}}}}
\newcommand{\zzeightreso}{\text{\resizebox{.7em}{!}{\includegraphics{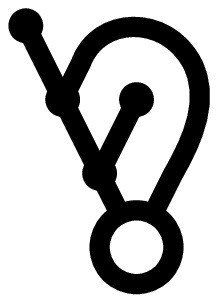}}}}
\newcommand{\zznine}{\text{\resizebox{1em}{!}{\includegraphics{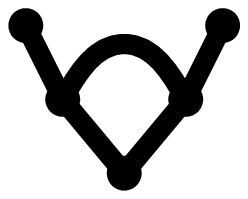}}}}
\newcommand{\CC}{\mathscr{C}}
\newcommand{\CD}{\mathscr{D}}
\newcommand{\CA}{\mathscr{A}}
\newcommand{\CB}{\mathscr{B}}
\newcommand{\CS}{\mathscr{S}}
\newcommand{\CF}{\mathscr{F}}
\newcommand{\DD}{\mathscr{D}}
\newcommand{\LL}{\mathscr{L}}

\newcommand{\CM}{\mathcal{M}}
\newcommand{\CX}{\mathcal{X}}
\newcommand{\F}{\mathcal{F}}
\newcommand{\PC}{\mathcal{P}}

\newcommand{\dd}{\mathd}

\newcommand{\R}{\mathbb{R}}
\newcommand{\X}{\mathbb{X}}
\newcommand{\T}{\mathbb{T}}
\newcommand{\Z}{\mathbb{Z}}
\newcommand{\E}{\mathbb{E}}
\renewcommand{\P}{\mathbb{P}}
\newcommand{\N}{\mathbb{N}}
\newcommand{\Y}{\mathbb{Y}}
\newcommand{\C}{\mathbb{C}}
\newcommand{\1}{\mathds{1}}

\newcommand{\Q}{\mathbb{Q}}

\newcommand{\prog}{\mathfrak{pm}}
\newcommand{\FM}{\mathfrak{M}}

\newcommand{\Ykpz}{\mathcal{Y}_{\mathrm{kpz}}}
\newcommand{\Xrbe}{\mathcal{X}_{\mathrm{rbe}}}

\title{\textbf{KPZ reloaded}}

\author{
  Massimiliano Gubinelli \\
  CEREMADE \& UMR 7534 CNRS \\
  Universit{\'e} Paris--Dauphine \\
  and Institut Universitaire de France \\
  \texttt{gubinelli@ceremade.dauphine.fr}
  \and
  Nicolas Perkowski \\
  Humboldt--Universit\"at zu Berlin \\
  Institut f\"ur Mathematik \\
  \texttt{perkowsk@math.hu-berlin.de}
}

\begin{document}

\maketitle

\begin{abstract}
  We analyze the one-dimensional periodic Kardar--Parisi--Zhang equation in the language of paracontrolled distributions, giving an alternative viewpoint on the seminal results of Hairer. 
  
  Apart from a basic existence and uniqueness result for paracontrolled solutions to the KPZ equation we perform a thorough study of some related problems. We rigorously prove the links between KPZ equation, stochastic Burgers equation, and (linear) stochastic heat equation and also the existence of solutions starting from quite irregular initial conditions. We also build a partial link between energy solutions as introduced by Gon\c{c}alves and Jara and paracontrolled solutions.
  
  Interpreting the KPZ equation as the value function of an optimal control problem, we give a pathwise proof for the global existence of solutions and thus for the strict positivity of solutions to the stochastic heat equation.
  
  Moreover, we study Sasamoto-Spohn type discretizations of the stochastic Burgers equation and show that their limit solves the continuous Burgers equation  possibly with an additional linear transport term. As an application, we give a proof of the invariance of the white noise for the stochastic Burgers equation which does not rely on the Cole--Hopf transform.
\end{abstract}

\tableofcontents

\section{Introduction}

The Kardar--Parisi--Zhang
(KPZ) equation is the stochastic partial differential equation (SPDE)
\begin{equation}\label{eq:kpz}
  \LL h (t, x) = (\mathD h (t, x))^2 + \xi (t, x), \hspace{2em} x \in \R\comma \hspace{1em} t \geqslant 0,
\end{equation}
where $h : \R_+ \times \R \rightarrow \R$ is a real
valued function on the real line,
 $\LL = \partial_t - \Delta$ denotes the heat operator, $\mathD = \partial / \partial x$ and $\partial_t = \partial / \partial t$
are the spatial respectively temporal derivatives and  $\xi$ is a
space-time white noise: the centered Gaussian space-time random
distribution with covariance
\[
   \E [\xi (t, x) \xi (s, y)] = \delta (t - s) \delta (x - y), \hspace{2em} t, s \geqslant 0 \comma \hspace{1em} x, y \in \R.
\]
The KPZ equation was introduced by Kardar, Parisi, and Zhang~\cite{Kardar1986} as an SPDE model describing the large scale fluctuations of a growing interface represented by a \emph{height field} $h$. Based on non-rigorous perturbative renormalization group arguments they predicted that upon a particular ``1-2-3" rescaling and centering, the height field $h$ (or at least its finite dimensional distributions) must converge to a scale invariant random field $H_{\text{kpz}}$ (the \emph{KPZ fixed point}) obtained as
\begin{equation}
\label{eq:kpz-fixpoint}
H_{\text{kpz}}(t,x) = \lim_{\lambda\to \infty} \lambda h(\lambda^3 t,\lambda^2 x) - c(\lambda) t .
\end{equation}
According to the general renormalization group (RG) understanding of dynamic critical phenomena a large class of similar interface growth mechanisms must feature the same large scale statistical behavior, namely their height fields $\tilde h(t,x)$ converge, upon rescaling (and maybe centering and a suitable Galilean transformation to get rid of a uniform spatial drift) to the same KPZ fixed point $H_{\text{kpz}}$. 

Proving any kind of rigorous results about the convergence stated in~\eqref{eq:kpz-fixpoint} (called sometimes the \emph{strong KPZ universality conjecture}) is a wide open problem for any continuous SPDE modeling interface growth and in particular for the KPZ equation; see however~\cite{SasamotoSpohn2010,AmirCorwin2011} for recent breakthroughs. On the other side there has been a tremendous amount of progress in understanding the large scale behavior of certain discrete probabilistic models of growth phenomena belonging to the same universality class, mainly thanks to a special feature of these models called \emph{stochastic integrability}. For further details see the excellent recent surveys~\cite{Corwin2012, Quastel2014, QuastelSpohn2015}.

A weaker form of universality comes from looking at the KPZ equation as a \emph{mesoscopic model} of a special class of growth phenomena. Indeed by the same theoretical physics RG picture it is expected that the KPZ equation is itself a universal description of the fluctuations of weakly asymmetric growth models, a prediction that is commonly referred to as the \emph{weak KPZ universality conjecture}. In this case the microscopic model possesses a parameter which drives the system out of equilibrium and controls the speed of growth. It is then possible to keep the nonlinear contributions of the same size as the diffusion and noise effects by tuning the parameter and at the same time rescaling the height field  diffusively. It is then expected that the random field so obtained satisfies the KPZ equation~\eqref{eq:kpz}. 

For a long time, the main difficulty in addressing these beautiful problems and obtaining any kind of invariance principle for space--time random fields describing random growth has been the elusiveness of the KPZ equation itself.  The main difficulty in making sense of the equation stems from the fact that for any fixed
time $t \geqslant 0$ the dependence of the solution $h (t, x)$ on the space
variable $x \in \T$ cannot be better than Brownian. That is, if we
measure spatial regularity in the scale of H{\"o}lder spaces $\CC^{\gamma}
(\R)$, we expect that $h (t, \cdot) \in \CC^{\gamma} (\R)$
for any $\gamma < 1 / 2$ but not better, in particular certainly  the quadratic term
$(\mathD h (t, x))^2$ is not well defined.

The first rigorous results about KPZ is due to Bertini and Giacomin~\cite{BertiniGiacomin1997}, who show that the height function of the weakly asymmetric simple exclusion process (WASEP) converges under appropriate rescaling to a continuous random field $h$ which they characterize using the Cole--Hopf transform. Namely they show that the random field $w = e^h$ is the solution of a particular It\^o SPDE : the stochastic heat equation (SHE)
\begin{equation}
\label{eq:she-intro}
   \dd w(t,x) = \Delta w(t,x) \dd t + w(t,x) \dd B_t(x),
\end{equation}
where $B_t(x) = \int_0^t \xi(s,x) \dd s$ is a cylindrical Brownian motion in $L^2(\R)$. The regularity of the solution of this equation does not allow to determine the intrinsic dynamics of $h$, but the convergence result shows that any candidate notion of solution to the KPZ equation should have the property that $e^h$ solves the SHE~\eqref{eq:she-intro}.

The key tool in Bertini and Giacomin's proof is G\"artner's microscopic Cole--Hopf transform~\cite{Gartner1988} which allows them to prove the convergence of a transformed version of the WASEP to the stochastic heat equation. But the microscopic Cole--Hopf transform only works for this specific model (or at least only for a few special models, see the recent works~\cite{Dembo2013, CorwinTsai2015} for extensions of the Bertini-Giacomin result) and cannot be used to prove universality. 

Another notion of solution for the KPZ equation has been introduced by Gon\c{c}alves and Jara~\cite{GoncalvesJara2014} who proved tightness for a quite general class of rescaled weakly asymmetric exclusion processes and provided a description of the dynamics of the limiting random fields by showing that they must be \emph{energy solutions} of the KPZ equation. In particular, they showed that (after a subtraction) the quadratic term in~\eqref{eq:kpz} would make
sense only as a space-time distribution and not better. This notion of
\tmtextit{energy solution} has been subsequently interpreted by Gubinelli and Jara in the
language of controlled paths in~{\cite{GubinelliJara2013}}. The
key observation is that the fast time decorrelation of the noise provides a
mechanism to make sense of the quadratic term. Unfortunately, energy solutions
are very weak and allow only for few a priori estimates and currently there is
no hint that they are unique.
In this respect  Gon\c{c}alves and Jara's result falls short of proving weak universality for systems without microscopic Cole--Hopf transformation. 

A related problem is studied in the the recent work of Alberts, Khanin and Quastel~\cite{AlbertsKhaninQuastel2014} where a universality result is rigorously shown for discrete random directed polymers with weak noise, which converge under rescaling to the continuum random directed polymer whose partition function is given by the stochastic heat equation.

This was the state of the art until 2011, when Hairer~{\cite{hairer_solving_2011}} established a
well-posedness theory for the periodic KPZ equation using three key ingredients: a partial series
expansion of the solution, an explicit control of various stochastic terms,
and a fixed point argument involving rough path theory. As a final outcome he
was able to show very explicitly that solutions $h$ to equation~\eqref{eq:kpz} are
limits for $\varepsilon \rightarrow 0$ of the approximate solutions
$h_{\varepsilon}$ of the equation
\begin{equation}
  \partial_t h_{\varepsilon} (t, x) = \Delta h_{\varepsilon} (t, x) + (\mathD
  h_{\varepsilon} (t, x))^2 - C_{\varepsilon} + \xi_{\varepsilon} (t, x),
  \hspace{2em} x \in \T, \hspace{1em} t \geqslant 0, \label{eq:kpz-eps}
\end{equation}
where $\T=\R/ (2 \pi\Z)$ and $\Delta$ is the Laplacian with periodic boundary conditions on $\T$,
$\xi_{\varepsilon}$ is a Gaussian process obtained by regularizing $\xi$
by convoluting it with a smooth kernel $\rho$, that is $\xi_{\varepsilon} =
\rho_{\varepsilon} \ast \xi$ with $\rho_{\varepsilon} (x) = \varepsilon^{- 1}
\rho (x / \varepsilon)$, and where the constant $C_{\varepsilon}$ has to be
chosen in such a way that it diverges to infinity as $\varepsilon \rightarrow
0$. While it had been known since Bertini and Giacomin's work~{\cite{BertiniGiacomin1997}} that the solutions to~(\ref{eq:kpz-eps})
converge to a unique limit, the key point of Hairer's
construction is that using rough path theory he is able to give an \emph{intrinsic} meaning to
the product $(\mathD h (t, x))^2$ and to obtain good bounds for that operation
on suitable function spaces, which ultimately allowed him to solve the equation.

The solution of the KPZ equation was one of the stepping stones in the development of Hairer's regularity structures~\cite{Hairer2014Regularity}, and now it is of course also possible to solve the equation using regularity structures rather than rough paths; see for example~\cite{FrizHairer2014}.

Using regularity structures, Hairer and Quastel~\cite{HairerQuastel} recently proved the universality of the KPZ equation in the following sense: they consider a regularized version $\xi_\varepsilon$ of the white noise and an arbitrary even polynomial $P$ and show that under the right rescaling the solution $h_\varepsilon$ to
\begin{equation}\label{eq:hairer-quastel}
   \partial_t h_\varepsilon = \Delta h_\varepsilon + \sqrt{\varepsilon} P(\mathD h_\varepsilon) + \xi_\varepsilon
\end{equation}
converges to the KPZ equation with (possibly) a non--universal constant in front of the non--linearity. Let us also mention the recent work~\cite{HairerShen2015}, where a similar problem is studied except that $P(x) = x^2$ but $\xi_\varepsilon$ is not necessarily Gaussian.

In the same period Hairer was developing his theory,  together with P.~Imkeller we proposed the notion of \emph{paracontrolled distributions}~\cite{gubinelli_paraproducts_2012, GubinelliImkeller2014, GubinelliPerkowski2015} as a tool to give a meaning to and solve a large class of singular SPDEs. The class of models which can be analyzed using this method includes the generalized Parabolic Anderson model in 2 dimensions and the stochastic quantization equation in 3 dimensions and other interesting singular SPDEs \cite{CatellierChouk2013, ChoukFriz2014, Zhu2014NS, Zhu2014NSdiscrete, Bailleul20152d, CannizzaroChouk2015, Bailleul20153d, Proemel2015}. 

After Hairer's breakthrough, another set of tools which should allow to solve the KPZ equation was developed by Kupiainen~\cite{Kupiainen2014}, who builds on renormalization group techniques to make sense of the three dimensional stochastic quantization equation $\phi^4_3$.

The main aim of the present paper is to describe the paracontrolled approach to the KPZ equation. While the analysis of the KPZ equation itself is quite simple and  shorter versions of the present paper circulated informally among people interested in paracontrolled equations, we kept delaying a public version in order to use the full versatility of paracontrolled analysis to explore various questions around KPZ and related equations and more importantly to find simple and direct arguments for most of the proofs. 

Indeed, despite paracontrolled calculus being currently less powerful than the full fledged regularity structure theory, we believe that it is lightweight enough to be effective in exploring various questions related to the qualitative behaviour of solutions and in particular in bridging the gap between stochastic and deterministic PDE theories.

We start in Section~\ref{sec:paracontrolled} by recalling some basic results from paracontrolled distributions, which we then apply in Section~\ref{sec:burgers} to solve the conservative Burgers equation
\[
   \LL u = \mathD u^2 + \mathD \xi
\]
driven by an irregular signal $\xi$ in a pathwise manner. In Section~\ref{sec:kpz} we indicate how to adapt the arguments to also solve the KPZ equation and the rough heat equation pathwise, and we show that the formal links between these equations ($u = \mathD h$, $w = e^h$) can indeed be justified rigorously, also for a driving signal that only has the regularity of the white noise and not just for the equations driven by mollified signals.

We then show in Section~\ref{sec:interpretation} that by working in a space of paracontrolled distributions one can find a natural interpretation for the nonlinearity in the stochastic Burgers equation, an observation which was not yet made in the works based on rough paths or regularity structures, where the focus is more on the continuous dependence of the solution map on the data rather than on the declaration of the nonlinear operations appearing in the equation. Our result allows us to build a partial bridge between energy solutions and paracontrolled solutions, although the full picture remains unclear.

Section~\ref{sec:singular initial} extends our previous results to start the equations in more irregular initial conditions. In the case of the KPZ equation we can start in any $\CC^\beta$ function for $\beta>0$, but the linear heat equation can be started in $w_0 \in B_{p,\infty}^{-\gamma}$ whenever $\gamma < 1/2$ and $p \in [1,\infty]$; in particular we can take $w_0$ to be a fractional derivative of order $\gamma$ of the Dirac delta.

In Section~\ref{sec:hjb} we develop yet another approach to the KPZ equation. We show that its solution is given by the value function of the stochastic control problem
\begin{equation}\label{eq:hjb introduction}
  h(t,x) = \sup_v \E_x \left[ \bar{h}(\gamma^v_t) + \int_0^t \Big(\arr \xi (s, \gamma^v_s) - \infty - \frac{1}{4} | v_s |^2 \Big) \mathd s \right],
\end{equation}
where $\arr \xi(s) = \xi(t-s)$ and under $\E_x$ we have
\[
   \gamma^v_s = x + \int_0^s v_r \dd r + \sqrt{2} W_s
\]
with a standard Brownian motion $W$. Such a representation has already proved very powerful in the case of a spatially smooth noise $\xi$, see for example~\cite{EKhanin2000,BakhtinCator2014}, but of course it is not obvious how to even make sense of it when $\xi$ is a space-time white noise. Based on paracontrolled distributions and the techniques of~\cite{DelarueDiel2014,CannizzaroChouk2015} we can give a rigorous interpretation for the optimization problem and show that the identity~\eqref{eq:hjb introduction} is indeed true. Immediate consequences are a comparison principle for the KPZ equation and a pathwise global existence result. Recall that in~\cite{hairer_solving_2011} global existence could only be shown by identifying the rough path solution with the Cole--Hopf solution, which means that the null set where global existence fails may depend on the initial condition. Here we show that this is not the case, and that for any $\omega$ for which $\xi(\omega)$ can be enhanced to a \emph{KPZ-enhancement} (see Definition~\ref{def:kpz rough distribution}) and for any initial condition $h_0 \in \CC^\beta$ with $\beta>0$ there exists a unique global in time paracontrolled solution $h(\omega)$ to the equation
\[
   \LL h(t,x,\omega) = (\mathD h(t,x,\omega))^2 - \infty + \xi(t,x,\omega), \qquad h(0,\omega) = h_0.
\]
and that the $L^\infty$ norm of $h$ is controlled only in terms of the KPZ-enhancement of the noise and the $L^\infty$ norm of the initial condition.

A surprising byproduct of these estimates is a positivity result for the solution of the SHE which is independent of the precise details of the noise. This is at odds with currently available proofs of positivity~\cite{Mueller1991, Moreno-Flores2014, ChenKim2014}, starting from the original proof of Mueller~\cite{Mueller1991}, which all use heavily the Gaussian nature of the noise and proceed via rather indirect arguments. Ours is a direct PDE argument showing that the ultimate reason for positivity does not lie in the law of the noise but somehow in its space-time regularity (more precisely in the regularity of its enhancement).  

Section~\ref{sec:SS} is devoted to the study of Sasamoto-Spohn~\cite{SasamotoSpohn2009} type discretizations of the conservative stochastic Burgers equation. We consider a lattice model $u_N \colon [0,\infty) \times \T_N \to \R$ (where $\T_N = (2\pi \Z / N) / (2 \pi \Z)$), defined by an SDE
\begin{align*}
       \mathd u_{N}(t,x) & = \Delta_N u_N(t,x) \mathd t + \big(\mathD_{N} B_{N} (u_N(t), u_N(t))\big)(x) \mathd t +  \mathd (\mathD_{N} \varepsilon^{-1/2} W_N(t,x)) \\
       u_N(0,x) & = u_0^N(x),
\end{align*}
where $\Delta_{N}$, $\mathD_{N}$ are approximations of Laplacian and spatial derivative, respectively, $B_N$ is a bilinear form approximating the pointwise product, $(W_{N}(t,x))_{t \in \R_+, x \in \T_N}$ is an $N$--dimensional standard Brownian motion, and $u_0^N$ is independent of $W_N$. We show that if $(u_0^N)$ converges in distribution in $\CC^{-\beta}$ for some $\beta < 1$ to $u_0$, then $(u_N)$ converges weakly to the paracontrolled solution $u$ of
\begin{equation}\label{eq:perturbed burgers intro}
   \LL u = \mathD u^2 + c \mathD u + \mathD \xi,\qquad u(0) = u_0,
\end{equation}
where $c \in \R$ is a constant depending on the specific choice of $\Delta_N$, $\mathD_N$, and $B_N$. If $\Delta_N$ and $\mathD_N$ are the usual discrete Laplacian and discrete gradient and $B_N$ is the pointwise product, then $c = 1/2$. However, if we replace $B_N$ by the bilinear form introduced in~\cite{Krug1991}, then $c=0$. It was observed before in the works of Hairer, Maas and Weber~\cite{Hairer_Maas_2012, Hairer_Maas_2014} that such a ``spatial Stratonovich corrector'' can arise in the limit of discretizations of Hairer's~\cite{Hairer2011Rough} generalized Burgers equation $\LL u = g(u) \partial_x u + \xi$, and indeed one of the motivations for the work~\cite{Hairer_Maas_2014} was that it would provide a first step towards the understanding of discretizations of the KPZ equation. To carry out this program we simplify many of the arguments in~\cite{Hairer_Maas_2014}, replacing rough paths by paracontrolled distributions. A key difference to~\cite{Hairer_Maas_2014} is that here the ``paracontrolled derivative'' is not constant, which introduces tremendous technical difficulties that were absent in all discretizations of singular PDEs which were studied so far, for example~\cite{Zhu2014NSdiscrete,MourratWeber2014}. We overcome these difficulties by introducing a certain random operator which we bound using stochastic computations.

As an application of our convergence result, we show that  the distribution of $m + 2^{-1/2} \eta$, where $\eta$ is a space white noise and $m \in \R$, is invariant under the evolution of the conservative stochastic Burgers equation. While this is well known (see~\cite{BertiniGiacomin1997} and also the recent work~\cite{FunakiQuastel2014} for an elegant proof in the more complicated setting of the non-periodic KPZ equation), ours seems to be the first proof which does not rely on the Cole--Hopf transform.

In Section~\ref{sec:stochastics} we construct the enhanced white noise which we needed in our pathwise analysis. We try to build a link with the Feynman diagrams from quantum field theory. The required bounds are shown by reducing everything to a few basic integrals that can be controlled by a simple recursive algorithm. In Section~\ref{sec:discrete stochastics} we indicate how to adapt these calculations to also obtain the convergence result for the enhanced data in the lattice models. We also calculate the explicit form of the correction constant $c$ appearing in~\eqref{eq:perturbed burgers intro}.

\paragraph{Acknowledgments.} We are very grateful to Khalil Chouk, who pointed out that by working in more general Besov spaces than $B_{\infty,\infty}$ one can start the rough heat equation in the Dirac delta. We would also like to thank Rongchan and Xiangchan Zhu, who found a mistake in a previous version of the paper.

The main part of the work was carried out while N.P. was employed by Universit\'e Paris Dauphine and supported by the Fondation Sciences Math\'ematiques de Paris (FSMP) and by a public grant overseen by the French National Research Agency (ANR) as part of the ``Investissements d'Avenir'' program (reference: ANR-10-LABX-0098).

\section{Paracontrolled calculus}\label{sec:paracontrolled}

Paracontrolled calculus and the relevant estimates which will be needed in our
study of the KPZ (and related) equations have been introduced
in~{\cite{gubinelli_paraproducts_2012}}. In this section we will recall the
notations and the basic results of paracontrolled calculus without proofs. For
more details on Besov spaces, Littlewood--Paley theory, and Bony's paraproduct
the reader can refer to the nice recent monograph~{\cite{Bahouri2011}}.

\subsection{Notation and conventions}

Throughout the paper, we use the notation $a \lesssim b$ if there exists a
constant $c > 0$, independent of the variables under consideration, such that
$a \leqslant c \cdot b$, and we write $a \simeq b$ if $a \lesssim b$ and $b
\lesssim a$. If we want to emphasize the dependence of $c$ on the variable
$x$, then we write $a (x) \lesssim_x b (x)$. For index variables $i$ and $j$
of Littlewood-Paley decompositions (see below) we write $i \lesssim j$ if
there exists $N \in \mathbb{N}$, independent of $j$, such that $i \leqslant j
+ N$ (or in other words if $2^i \lesssim 2^j$), and we write $i \sim j$ if $i \lesssim j$ and $j \lesssim i$.

An \tmtextit{annulus} is a set of the form $\CA = \{x \in \mathbb{R}^d : a
\leqslant |x| \leqslant b\}$ for some $0 < a < b$. A {\tmem{ball}} is a set of
the form $\CB = \{x \in \mathbb{R}^d : |x| \leqslant b\}$. $\mathbb{T}
=\R/ (2 \pi \Z)$ denotes the torus.

If $f$ is a map from $A \subset \R$ to the linear space $Y$, then we
write $f_{s, t} = f (t) - f (s)$. For $f \in L^p (\mathbb{T})$ we write $\|f
(x)\|^p_{L^p_x (\mathbb{T})} \assign \int_{\mathbb{T}} |f (x) |^p \mathd x$.

Given a Banach space $X$ with norm $\| \cdummy \|_X$ and $T > 0$, we write
$C_T X = C ([0, T], X)$ for the space of continuous maps from $[0, T]$ to $X$,
equipped with the supremum norm $\lVert \cdummy \rVert_{C_T X}$, and we set $C
X = C (\R_+, X)$, equipped with the topology of uniform convergence on compacts. Similarly $C(\R,X)$ will always be equipped with the locally uniform topology. For $\alpha \in (0, 1)$ we also define
$C^{\alpha}_T X$ as the space of $\alpha$-H{\"o}lder continuous functions from
$[0, T]$ to $X$, endowed with the seminorm $\|f\|_{C^{\alpha}_T X} = \sup_{0
\leqslant s < t \leqslant T} \|f_{s, t} \|_X / |t - s|^{\alpha}$, and we write
$C^{\alpha}_{\tmop{loc}} X$ for the space of locally $\alpha$-H{\"o}lder
continuous functions from $\R_+$ to $X$.

The space of distributions on the torus is denoted by $\DD' (\T)$ or
$\DD'$. The Fourier transform is defined with the normalization
\[ \CF u (k) = \hat{u} (k) = \int_{\T} e^{- i \langle k, x
   \rangle} u (x) \mathd x, \hspace{2em} k \in \Z, \]
so that the inverse Fourier transform is given by $\CF^{- 1} v (x) = (2
\pi)^{- 1} \sum_k e^{i \langle k, x \rangle} v (k)$. We denote Fourier multipliers by $\varphi(\mathD) u = \CF^{-1} (\varphi \CF u)$ whenever the right hand side is well defined.

Throughout the paper, $(\chi, \rho)$ will denote a dyadic partition of unity
such that $\tmop{supp} (\rho (2^{- i} \cdummy)) \cap \tmop{supp} (\rho (2^{-
j} \cdummy)) = \emptyset$ for $|i - j| > 1$. \ The family of operators
$(\Delta_j)_{j \geqslant - 1}$ will denote the Littlewood-Paley projections
associated to this partition of unity, that is $\Delta_{- 1} u = \CF^{- 1}
\left( \chi \CF u \right)$ and $\Delta_j = \CF^{- 1} \left( \rho (2^{- j}
\nosymbol \cdummy) \CF u \right)$ for $j \geqslant 0$. We also use the notation $S_j
= \sum_{i < j} \Delta_i$. We write $\rho_j = \rho(2^{-j} \cdot)$ for $j \geqslant 0$ and $\rho_{-1} = \chi$, and $\chi_j = \chi(2^{-j} \cdot)$ for $j \geqslant 0$ and $\chi_{j} = 0$ for $j < 0$. We define $\psi_{\prec} (k, \ell) = \sum_j \chi_{j - 1} (k) \rho_j (\ell)$ and $\psi_{\circ} (k, \ell) = \sum_{| i - j | \leqslant 1} \rho_i(k) \rho_j (\ell)$.
The H{\"o}lder-Besov space $B^{\alpha}_{\infty,
\infty} (\T, \R)$ for $\alpha \in \R$ will be
denoted by $\CC^{\alpha}$ and equipped with the norm
\[ \lVert f \rVert_{\alpha} = \lVert f \rVert_{B^{\alpha}_{\infty, \infty}} =
   \sup_{i \geqslant - 1} (2^{i \alpha} \| \Delta_i f \|_{L^{\infty}
   (\T)}) . \]
If $f$ is in $\CC^{\alpha - \varepsilon}$ for all $\varepsilon > 0$, then we
write $f \in \CC^{\alpha -}$. For $\alpha \in (0, 2)$, we also define the
space $\LL_T^{\alpha} = C^{\alpha / 2}_T L^{\infty} \cap C_T \CC^{\alpha}$,
equipped with the norm
\[ \| f \|_{\LL_T^{\alpha}} = \max \left\{ \| f \|_{C^{\alpha / 2}_T
   L^{\infty}}, \| f \|_{C_T \CC^{\alpha}} \right\} . \]
The notation is chosen to be reminiscent of $\LL = \partial_t - \Delta$, by which we
will always denote the heat operator with periodic boundary conditions on
$\T$. We also write $\LL^{\alpha} = C^{\alpha / 2}_{\tmop{loc}}
L^{\infty} \cap C \CC^{\alpha}$.

\subsection{Bony--Meyer paraproducts}\label{sec:bony}

Paraproducts are bilinear operations introduced by J.~M.~Bony~{\cite{Bony1981,Meyer1981}} in
order to linearize a class of nonlinear PDE problems. Paraproducts allow
to describe functions which ``look like'' some given reference functions and
to perform detailed computations on their singular behavior. They also appear
naturally in the analysis of the product of two Besov distributions. In terms
of Littlewood--Paley blocks, the product $fg$ of two distributions $f$ and $g$
can be (at least formally) decomposed as
\[ fg = \sum_{j \geqslant - 1} \sum_{i \geqslant - 1} \Delta_i f \Delta_j g =
   f \para g + f \lpara g + f \reso g. \]
Here $f \para g$ is the part of the double sum with $i < j - 1$, and $f \lpara
g$ is the part with $i > j + 1$, and $f \reso g$ is the ``diagonal'' part,
where $|i - j| \leqslant 1$. More precisely, we define
\[ f \para g = g \lpara f = \sum_{j \geqslant - 1} \sum_{i = - 1}^{j - 2}
   \Delta_i f \Delta_j g \hspace{2em} \text{and} \hspace{2em} f \reso g =
   \sum_{|i - j| \leqslant 1} \Delta_i f \Delta_j g. \]
This decomposition behaves nicely with respect to Littlewood--Paley theory. Of
course, the decomposition depends on the dyadic partition of unity used to
define the blocks $\{ \Delta_j \}_{j \geqslant - 1}$, and also on the
particular choice of the set of pairs $(i, j)$ which contributes to the
diagonal part. Our choice of taking all $(i, j)$ with $|i - j| \leqslant 1$
into the diagonal part corresponds to a dyadic partition of unity which
satisfies $\tmop{supp} (\rho (2^{- i} \cdummy)) \cap \tmop{supp} (\rho (2^{-
j} \cdummy)) = \emptyset$ for $|i - j| > 1$. This implies that every term
$S_{j - 1} f \Delta_j g$ in the series $f \para g = \sum_j S_{j - 1} f
\Delta_j g$ has a Fourier transform which is supported in an annulus $2^j
\CA$, and of course the same holds true for $f \lpara g$. On the other side,
the terms in the diagonal part $f \reso g$ have Fourier transforms which are
supported in balls. We call $f \para g$ and $f \lpara g$
\tmtextit{paraproducts}, and $f \reso g$ the \tmtextit{resonant} term.

Bony's crucial observation is that $f \para g$ (and thus $f \lpara g$) is
always a well-defined distribution. In particular, if $\alpha > 0$ and $\beta
\in \mathbb{R}$, then $(f, g) \mapsto f \para g$ is a bounded bilinear
operator from $\CC^{\alpha} \times \CC^{\beta}$ to $\CC^{\beta}$.
Heuristically, $f \para g$ behaves at large frequencies like $g$ (and thus
retains the same regularity), and $f$ provides only a modulation of $g$ at
larger scales. The only difficulty in defining $fg$ for arbitrary
distributions lies in handling the resonant term $f \reso g$. The basic result
about these bilinear operations is given by the following estimates.

\begin{lemma}[Paraproduct estimates]
  \label{thm:paraproduct} For any $\beta \in \mathbb{R}$ we have
  \begin{equation}
    \label{eq:para-1} \|f \para g\|_{\beta} \lesssim_{\beta}
    \|f\|_{L^{\infty}} \|g\|_{\beta},
  \end{equation}
  and for $\alpha < 0$ furthermore
  \begin{equation}
    \label{eq:para-2} \|f \para g\|_{\alpha + \beta} \lesssim_{\alpha, \beta}
    \|f\|_{\alpha} \|g\|_{\beta} .
  \end{equation}
  For $\alpha + \beta > 0$ we have
  \begin{equation}
    \label{eq:para-3} \|f \reso g\|_{\alpha + \beta} \lesssim_{\alpha, \beta}
    \|f\|_{\alpha} \|g\|_{\beta} .
  \end{equation}
\end{lemma}

A natural corollary is that the product $fg$ of two elements $f \in
\CC^{\alpha}$ and $g \in \CC^{\beta}$ is well defined as soon as $\alpha +
\beta > 0$, and that it belongs to $\CC^{\gamma}$, where $\gamma = \min
\{\alpha, \beta, \alpha + \beta\}$.

\subsection{Product of paracontrolled distributions and other useful results}\label{sec:preliminaries}

The key result of paracontrolled calculus is the fact that we are able to multiply
paracontrolled distributions, extending Bony's results beyond the case $\alpha
+ \beta > 0$. We present here a simplified version which is adapted to our
needs. Let us start with the following meta-definition:

\begin{definition}
  Let $\beta > 0$ and $\alpha \in \R$. A distribution $f \in
  \CC^{\alpha}$ is called paracontrolled by $u \in \CC^{\alpha}$ if there
  exists $f' \in \CC^{\beta}$ such that $f^{\sharp} = f - f' \para u \in
  \CC^{\alpha + \beta}$.
\end{definition}

Of course in general the derivative $f'$ is not uniquely determined by $f$ and
$u$, so more correctly we should say that $(f, f')$ is paracontrolled by $u$.

\begin{theorem}\label{thm:paracontrolled product}
   Let $\alpha, \beta \in (1 / 3, 1 / 2)$.
  Let $u \in \CC^{\alpha}$, $v \in \CC^{\alpha - 1}$, and let $(f, f')$ be
  paracontrolled by $u$ and $(g, g')$ be paracontrolled by $v$. Assume that $u
  \reso v \in \CC^{2 \alpha - 1}$ is given as limit of $(u_n \reso v_n)$ in
  $\CC^{2 \alpha - 1}$, where $(u_n)$ and $(v_n)$ are sequences of smooth
  functions that converge to $u$ in $\CC^{\alpha}$ and to $v$ in $\CC^{\alpha
  - 1}$ respectively. Then $f g$ is well defined and satisfies
  \[ \| f g - f \para g \|_{2 \alpha - 1} \lesssim (\| f' \|_{\beta} \| u
     \|_{\alpha} + \| f^{\sharp} \|_{\alpha + \beta}) (\| g' \|_{\beta} \| v
     \|_{\alpha - 1} + \| g^{\sharp} \|_{\alpha + \beta - 1}) + \| f' g'
     \|_{\beta} \| u \reso v \|_{2 \alpha - 1} . \]
  Furthermore, the product is locally Lipschitz continuous: Let $\tilde{u} \in
  \CC^{\alpha}$, $\tilde{v} \in \CC^{\alpha - 1}$ with $\tilde{u} \reso
  \tilde{v} \in \CC^{2 \alpha - 1}$ and let $(\tilde{f}, \tilde{f}')$ be
  paracontrolled by $\tilde{u}$ and $(\tilde{g}, \tilde{g}')$ be
  paracontrolled by $\tilde{v}$. Assume that $M > 0$ is an upper bound for the
  norms of all distributions under consideration. Then
  \[ \| (f g - f \para g) - (\tilde{f} \tilde{g} - \tilde{f} \para \tilde{g})
     \|_{2 \alpha - 1} \lesssim (1 + M^3) [\| f' - \tilde{f}' \|_{\beta} + \|
     g' - \tilde{g}' \|_{\beta} \nobracket + \| u - \tilde{u} \|_{\alpha} + \|
     v - \tilde{v} \|_{\alpha - 1} \]
  \[ + \| f^{\sharp} - \tilde{f}^{\sharp} \|_{\alpha + \beta} + \| g^{\sharp}
     - \tilde{g}^{\sharp} \|_{\alpha + \beta - 1} + \nobracket \| u \reso v -
     \tilde{u} \reso \tilde{v} \|_{2 \alpha - 1}] . \]
  If $f' = \tilde{f}' = 0$ or $g' = \tilde{g}' = 0$, then $M^3$ can be
  replaced by $M^2$.
\end{theorem}

\begin{remark}\label{rmk:paracontrolled commutator}
   The proof is based on the simple result of~\cite{gubinelli_paraproducts_2012} that the commutator
   \[
      C(f,g,h) := (f \para g) \reso h - f (g \reso h)
   \]
   is a bounded trilinear operator from $\CC^\beta \times \CC^\alpha \times \CC^{\alpha-1}$ to $\CC^{2\alpha + \beta - 1}$.
\end{remark}

We will write $f \cdummy g$ instead of $f g$ whenever we want to stress the
fact that we are considering the product of paracontrolled distributions. For
some computations we will also need a paralinearization result which allows us
to control nonlinear functions of the unknown in terms of a paraproduct.

\begin{lemma}[Bony--Meyer paralinearization theorem]
  \label{lemma:paralinearization} Let $\alpha \in (0, 1)$ and let $F \in C^2$.
  There exists a locally bounded map $R_F : \CC^{\alpha} \rightarrow \CC^{2
  \alpha}$ such that
  \begin{equation}
    \label{eq:para-linearization} F (f) = F' (f) \para f + R_F (f)
  \end{equation}
  for all $f \in \CC^{\alpha}$. If $F \in C^3$, then $R_F$ is locally
  Lipschitz continuous.
\end{lemma}

We will also need the following lemma which allows to replace the paraproduct
by a pointwise product in certain situations:

\begin{lemma}[A further commutator estimate]
  \label{lem:bony commutator}Let $\alpha > 0$, $\beta \in \R$, and
  let $f, g \in \CC^{\alpha}$, and $h \in \CC^{\beta}$. Then
  \[ \| f \para (g \para h) - (f g) \para h \|_{\alpha + \beta} \lesssim \| f
     \|_{\alpha} \| g \|_{\alpha} \| h \|_{\beta} . \]
\end{lemma}

When dealing with paraproducts in the context of parabolic equations it would
be natural to introduce parabolic Besov spaces and related paraproducts. But
to keep a simpler setting, we choose to work with space--time distributions
belonging to the scale of spaces $\left( C \CC^{\alpha} \right)_{\alpha \in
\R}$. To do so efficiently, we will use a modified paraproduct which
introduces some smoothing in the time variable that is tuned to the parabolic
scaling. Let therefore $\varphi \in C^{\infty} (\R, \R_+)$
be nonnegative with compact support contained in $\R_+$ and with
total mass $1$, and define for all $i \geqslant - 1$ the operator
\[ Q_i : C \CC^{\beta} \rightarrow C \CC^{\beta}, \hspace{2em} Q_i f (t) =
   \int_{\R} 2^{- 2 i} \varphi (2^{2 i} (t - s)) f (s \vee 0) \mathd
   s. \]
We will often apply $Q_i$ and other operators on $C \CC^{\beta}$ to functions
$f \in C_T \CC^{\beta}$ which we then simply extend from $[0, T]$ to
$\R_+$ by considering $f (\cdummy \wedge T)$. With the help of $Q_i$,
we define a modified paraproduct
\[ f \mpara g = \sum_i (Q_i S_{i - 1} f) \Delta_i g \]
for $f, g \in C \left( \R_+, \CD' (\T) \right)$. It is easy
to see that for $f \mpara g$ we have essentially the same estimates as for
the pointwise paraproduct $f \para g$, only that we have to bound $f$
uniformly in time. More precisely:

\begin{lemma}
  For any $\beta \in \mathbb{R}$ we have
  \begin{equation}
    \label{eq:mod para-1} \|f \mpara g (t) \|_{\beta} \lesssim \|f\|_{C_t
    L^{\infty}} \|g (t) \|_{\beta},
  \end{equation}
  for all $t > 0$, and for $\alpha < 0$ furthermore
  \begin{equation}
    \label{eq:mod para-2} \|f \mpara g (t) \|_{\alpha + \beta} \lesssim
    \|f\|_{C_t \CC^{\alpha}} \|g (t) \|_{\beta} .
  \end{equation}
\end{lemma}

We will also need the following two commutation results.

\begin{lemma}\label{lem:modified paraproduct commutators}
Let $\alpha \in (0, 2)$ and
  $\beta \in \R$. Then
  \[ \| (f \mpara g - f \para g) (t) \|_{\alpha + \beta} \lesssim \| f
     \|_{\LL^{\alpha}_t} \| g (t) \|_{\beta} \]
  for all $t \geqslant 0$. If $\alpha \in (0, 1)$, then we also have
  \[ \left\| \left( \LL (f \mpara g) - f \mpara \left( \LL g \right)
     \right) (t) \right\|_{\alpha + \beta - 2} \lesssim \| f
     \|_{\LL^{\alpha}_t} \| g (t) \|_{\beta} . \]
\end{lemma}

For a proof see {\cite{gubinelli_paraproducts_2012}}. As a consequence, every distribution on $\R_+ \times \T$ which is paracontrolled in terms of the modified paraproduct is also paracontrolled using the original paraproduct, at least if the derivative is in $\LL^{\alpha}$ and not just in $C \CC^{\alpha}$.

\

Moreover, we introduce the linear operator $I : C \left( \R_+, \DD'
(\T) \right) \rightarrow C \left( \R_+, \DD' (\T)
\right)$ given by
\[ I f (t) = \int_0^t P_{t - s} f (s) \mathd s, \]
for which we have standard estimates that are summarized in the following
Lemma.

\begin{lemma}[Schauder estimates]\label{lemma:schauder}
  For $\alpha \in (0, 2)$ we have
  \begin{equation}
    \label{eq:schauder-heat} \|I f\|_{\LL^{\alpha}_T} \lesssim \| f \|_{C_T
    \CC^{\alpha - 2}}
  \end{equation}
  for all $T > 0$, as well as
  \begin{equation}
    \label{eq:schauder initial contribution} \| s \mapsto P_s u_0
    \|_{\LL^{\alpha}_T} \lesssim \| u_0 \|_{\alpha} .
  \end{equation}
\end{lemma}

Combining the commutator estimate above with the Schauder estimates, we are
able to control the modified paraproduct in $\LL^{\alpha}_T$ spaces rather
than just $C_T \CC^{\alpha}$ spaces:

\begin{lemma}
  \label{lem:paraproduct parabolic space}Let $\alpha \in (0, 2)$, $\delta > 0$
  and let $f \in \LL_T^{\delta}$, $g \in C_T \CC^{\alpha}$, and $\LL g \in C_T
  \CC^{\alpha - 2}$. Then
  \[ \| f \mpara g \|_{\LL^{\alpha}_T} \lesssim \| f \|_{\LL^{\delta}_T}
     \left( \| g \|_{C_T \CC^{\alpha}} + \left\| \LL g \right\|_{C_T
     \CC^{\alpha - 2}} \right) . \]
\end{lemma}

\begin{proof}
  We apply the Schauder estimates for $\LL$ to obtain
  \begin{align*}
     \| f \mpara g \|_{\LL^{\alpha}_T} & \lesssim \| f \mpara g (0) \|_{\CC^{\alpha}} + \left\| \LL (f \mpara g) \right\|_{C_T \CC^{\alpha- 2}} \\
     & \leqslant \left\| \LL (f \mpara g) - f \mpara \LL g \right\|_{C_T \CC^{\alpha + \delta - 2}} + \left\| f \mpara \left( \LL g \right) \right\|_{C_T \CC^{\alpha - 2}} \\
     & \lesssim \| f \|_{\LL^{\delta}_T} ( \| g \|_{C_T \CC^{\alpha}} + \left\| \LL g \right\|_{C_T \CC^{\alpha - 2}} ).
  \end{align*}
  
\end{proof}

It is possible to replace $\| f \|_{\LL^{\delta}_T}$ on the right hand side by $\| f \|_{C_T L^{\infty}}$, see Lemma~2.16 of~{\cite{furlan2014}}. But we will only apply Lemma~\ref{lem:paraproduct parabolic space} in the form stated above, so that we do not need this improvement.

At the end of this section, let us observe that in $\left( \LL^{\alpha}_T
\right)_{\alpha}$ spaces we can gain a small scaling factor by passing to a
larger space.

\begin{lemma}
  \label{lem:scaling factor smaller norm}Let $\alpha \in (0, 2)$, $T > 0$, and
  let $f \in \LL^{\alpha}_T$. Then for all $\delta \in (0, \alpha]$ we have
  \[ \| f \|_{\LL^{\delta}_T} \lesssim \| f (0) \|_{\delta} + T^{(\alpha -
     \delta) / 2} \| f \|_{\LL^{\alpha}_T} . \]
\end{lemma}

\begin{proof}
  It suffices to observe that by interpolation $\| f \|_{C^{(\alpha - \delta)
  / 2}_T \CC^{\delta}} \lesssim \| f \|_{\LL^{\alpha}_T}$.
\end{proof}

\section{Stochastic Burgers equation}\label{sec:burgers}

\subsection{The strategy}

Instead of dealing directly with the KPZ equation we find it convenient to
consider its derivative $u = \mathD h$ which solves the Stochastic Burgers
equation (SBE)
\begin{equation}
  \LL u = \mathD u^2 + \mathD \xi . \label{eq:burgers}
\end{equation}
In a later section we will show that the two equations are equivalent. The
strategy we adopt is to consider a regularized version for this equation,
where the noise $\xi$ has been replaced by a noise $\xi_{\varepsilon}$ that is
smooth in space. Our aim is then to show that the solution $u_{\varepsilon}$
of the equation
\[ \LL u_{\varepsilon} = \mathD u_{\varepsilon}^2 + \mathD \xi_{\varepsilon}
\]
converges in the space of distributions to a limit $u$ as $\varepsilon
\rightarrow 0$. Moreover, we want to describe a space of distributions where
the nonlinear term $\mathD u^2$ is well defined, and we want to show that $u$
is the unique element in this space which satisfies $\LL u = \mathD u^2 +
\mathD \xi$ and has the right initial condition.

That it is non-trivial to study the limiting behavior of the sequence
$(u_{\varepsilon})$ can be understood from the following considerations.
First, the limiting solution $u$ cannot be expected to behave better (in terms
of spatial regularity) than the solution $X$ of the linear equation
\[ \LL X = \mathD \xi \]
(for example with zero initial condition at time 0). It is well known that,
almost surely, $X \in C \CC^{- 1 / 2 -}$ but not better, at least on the scale
of $\left( C \CC^{\alpha} \right)_{\alpha}$ spaces. In that case, the term
$u^2$ featured on the right hand side of~(\ref{eq:burgers}) is not well
defined since the product $(f, g) \mapsto fg$ is a continuous map from $C
\CC^{\alpha} \times C \CC^{\beta}$ to $C \CC^{\alpha \wedge \beta}$ only if
$\alpha + \beta > 0$, a condition which is violated here. Thus, we cannot hope
to directly control the limit of $(u_{\varepsilon}^2)_{\varepsilon}$ as
$\varepsilon \rightarrow 0$.

What raises some hope to have a well--defined limit is the observation that
if $X_{\varepsilon}$ is the solution of the linear equation with regularized
noise $\xi_{\varepsilon}$, then $(\mathD X_{\varepsilon}^2)_{\varepsilon}$
converges to a well defined space--time distribution $\mathD X^2 $ which
belongs to $\LL C \CC^{0 -}$ (the space of all distributions of the form $\LL
v$ for some $v \in C \CC^{0 -}$). This hints to the fact that there are
stochastic cancellations going into $\mathD X^2$ due to the correlation
structure of the white noise, cancellations which result in a well behaved
limit. Taking these cancellations properly into account in the full solution
$u$ is the main aim of the strategy we will implement below.

In order to prove the convergence and related results it will be convenient to
take a more general point of view on the problem and to study the
\tmtextit{solution map} $\Phi_{\tmop{rbe}} : (\theta, \upsilon) \mapsto u$
mapping well-behaved (for example smooth, but here it is enough to consider
elements of $C (\R_+, C^1)$ and $C^1$ respectively) functions
$\theta$ and initial conditions $\upsilon$ to the classical solution of the
evolution problem
\begin{equation}
  \LL u = \mathD u^2 + \mathD \theta, \hspace{2em} u (0) = \upsilon,
  \label{eq:sbe-theta}
\end{equation}
which we will call the Rough Burgers Equation (RBE).

To recover the original problem our aim is to understand the properties of
$\Phi$ when $\theta$ belongs to spaces of distributions of the same regularity
as the space--time white noise, that is for $\theta \in C \CC^{- 1 / 2 -}$. In order
to do so we will introduce an $n$-tuple of distributions $\X (\theta)$
constructed recursively from the data $\theta$ of the problem and living in a
complete metric space $\mathcal{X}_{\tmop{rbe}}$ and show that for every bounded ball
$\CB$ around 0 in $\mathcal{X}_{\tmop{rbe}}$ and every bounded ball $\tilde \CB$ in $\CC^{-1-}$ there exist $T > 0$ and a continuous map
$\Psi_{\tmop{rbe}} : \CB \times \tilde \CB \rightarrow C_T \CC^{- 1 -}$
satisfying $\Phi_{\tmop{rbe}} (\theta, \upsilon) = \Psi_{\tmop{rbe}}
(\X (\theta), \upsilon) |_{[0, T]}$ for all $(\X (\theta), \upsilon) \in
\CB \times \tilde \CB$. This will allow us to obtain the local in time convergence of
$u_{\varepsilon} = \Phi_{\tmop{rbe}} (\xi_{\varepsilon}, \varphi(\varepsilon \mathD) u_0)$ by showing that the sequence $(\X
(\xi_{\varepsilon}))_{\varepsilon}$ converges as $\varepsilon \rightarrow 0$
to a well--defined element of $\mathcal{X}_{\tmop{rbe}}$; here $\varphi$ is a compactly supported, infinitely smooth function with $\varphi(0) = 1$. That last step will
be accomplished via direct probabilistic estimates with respect to the law of
the white noise $\xi$. The described approach thus allows us to decouple the
initial problem into two parts:
\begin{enumeratenumeric}
  \item a functional analytic part related to the study of particular spaces
  of distributions that are suitable to analyze the structure of solutions to
  equation~(\ref{eq:sbe-theta});
  
  \item a probabilistic part related to the construction of the
  \tmtextit{RBE-enhancement} $\X (\xi) \in
  \mathcal{X}_{\tmop{rbe}}$ associated to the white noise $\xi$. 
\end{enumeratenumeric}

\subsection{Structure of the solution}

In this discussion we consider the case of zero initial condition. Let us
first define a linear map $J (f)$ for each smooth $f$ on $\R_+ \times
\T$ as the solution to $\LL J (f) = \mathD f$ with initial condition
$J (f) (0) = 0$. Explicitly, we have
\[ J (f) (t) = I (\mathD f) (t) = \int_0^t P_{t - s} \mathD (f (s)) \mathd s,
\]
where $(P_t)_{t \geqslant 0}$ is the semi--group generated by the periodic
Laplacian on $\T$. Using the estimates on $I$ given in
Lemma~\ref{lemma:schauder}, we see that $J$ is a bounded linear operator from
$C_T \CC^{\alpha}$ to $C_T \CC^{\alpha + 1}$ for all $\alpha \in \R$ and $T>0$.

Let us now expand the solution to the rough burgers
equation~(\ref{eq:sbe-theta}) around the solution $X = J (\theta)$ to the
linear equation $\LL X = \mathD \theta$. Setting $u = X + u^{\geqslant 1}$, we
have
\[ \LL u^{\geqslant 1} = \mathD (u^2) = \mathD (X^2) + 2 \mathD (Xu^{\geqslant
   1}) + \mathD ((u^{\geqslant 1})^2) . \]
We can proceed by performing a further change of variables in order to remove
the term $\mathD (X^2)$ from the equation by setting
\begin{equation}
  u = X + J (X^2) + u^{\geqslant 2} \label{eq:naive-exp} .
\end{equation}
Now $u^{\geqslant 2}$ satisfies
\begin{equation}
  \begin{array}{lll}
    \LL u^{\geqslant 2} & = & 2 \mathD (XJ (X^2)) + \mathD (J (X^2) J (X^2))\\
    &  & + 2 \mathD (Xu^{\geqslant 2}) + 2 \mathD (J (X^2) u^{\geqslant 2}) +
    \mathD ((u^{\geqslant 2})^2) .
  \end{array} \label{eq:naive-exp-2}
\end{equation}
We can imagine to make a similar change of variables to get rid of the term $2
\mathD (XJ (X^2))$. As we proceed in this inductive expansion, we generate a
certain number of explicit terms, obtained via various combinations of the
function $X$ and of the bilinear map
\[ B (f, g) = J (fg) = \int_0^{\cdot} P_{\cdot - s} \mathD (f (s) g (s))
   \mathd s. \]
Since we will have to deal explicitly with at least some of these terms, it is
convenient to represent them with a compact notation involving binary trees. A
binary tree $\tau \in \mathcal{T}$ is either the root $\bullet$ or the
combination of two smaller binary trees $\tau = (\tau_1 \tau_2)$, where the
two edges of the root of $\tau$ are attached to $\tau_1$ and $\tau_2$
respectively. The natural grading $\bar d \colon \mathcal{T} \rightarrow \N$ is
given by $\bar d (\bullet) = 0$ and $\bar d ((\tau_1 \tau_2)) = 1 + \bar d (\tau_1) + \bar d
(\tau_2)$. However, for our purposes it is more convenient to work with the degree $d(\tau)$, which we define as the number of leaves of $\tau$. By induction we easily see that $d = \bar d +1$. Then we define recursively a map $X : \mathcal{T} \rightarrow C
\left( \R_+, \CD' \right)$ by
\[ X^{\bullet} = X, \hspace{2em} X^{(\tau_1 \tau_2)} = B (X^{\tau_1},
   X^{\tau_2}), \]
giving
\[ X^{\zzone} = B (X, X), \hspace{1em} X^{\zztwo} = B (X, X^{\zzone}),
   \hspace{1em} X^{\zzthree} = B (X, X^{\zztwo}), \hspace{1em} X^{\zzfour} = B
   (X^{\zzone}, X^{\zzone}) \]
and so on, where
\[ (\bullet \bullet) =
   \text{\resizebox{.8em}{!}{\includegraphics{trees-1.eps}}} \comma
   \hspace{1em} (\text{\resizebox{.8em}{!}{\includegraphics{trees-1.eps}}}
   \bullet) = \text{\resizebox{.8em}{!}{\includegraphics{trees-2.eps}}} \comma
   \hspace{1em} (\bullet
   \text{\resizebox{.8em}{!}{\includegraphics{trees-2.eps}}}) =
   \text{\resizebox{1em}{!}{\includegraphics{trees-3.eps}}} \comma
   \hspace{1em} (\text{\resizebox{.8em}{!}{\includegraphics{trees-1.eps}}} 
   \resizebox{.8em}{!}{\includegraphics{trees-1.eps}}) =
   \text{\resizebox{1.25em}{!}{\includegraphics{trees-4.eps}}} \comma
   \hspace{1em} \ldots \]
In this notation the expansion~(\ref{eq:naive-exp})-(\ref{eq:naive-exp-2})
reads
\begin{equation}
  u = X + X^{\zzone} + u^{\geqslant 2}, \label{eq:u-expansion}
\end{equation}
\begin{equation}
  u^{\geqslant 2} = 2 X^{\zztwo} + X^{\zzfour} + 2 B (X, u^{\geqslant 2}) + 2
  B (X^{\zzone}, u^{\geqslant 2}) + B (u^{\geqslant 2}, u^{\geqslant 2}) .
  \label{eq:u-ge-2}
\end{equation}
\begin{remark}
  We observe that formally the solution $u$ of equation~(\ref{eq:sbe-theta})
  can be expanded as an infinite sum of terms labelled by binary trees:
  \[ u = \sum_{\tau \in \mathcal{T}} c (\tau) X^{\tau}, \]
  where $c (\tau)$ is a combinatorial factor counting the number of planar
  trees which are isomorphic (as graphs) to $\tau$. For example $c (\bullet) =
  1$, $c ( \text{\resizebox{.8em}{!}{\includegraphics{trees-1.eps}}}) = 1$, $c (
  \text{\resizebox{.8em}{!}{\includegraphics{trees-2.eps}}} ) = 2$, $c
  (\text{\resizebox{1em}{!}{\includegraphics{trees-3.eps}}}) = 4$, $c (
  \text{\resizebox{1.2em}{!}{\includegraphics{trees-4.eps}}} ) = 1$ and
  in general $c (\tau) = \sum_{\tau_1, \tau_2 \in \mathcal{T}}
  \1_{(\tau_1 \tau_2) = \tau} c (\tau_1) c (\tau_2)$. Alternatively,
  truncating the summation at trees of degree at most $n$ and setting
  \[ u = \sum_{\tau \in \mathcal{T}, d (\tau) < n} c (\tau) X^{\tau} +
     u^{\geqslant n}, \]
  we obtain a remainder $u^{\geqslant n}$ that satisfies the equation
  \[ u^{\geqslant n} = \sum_{\tmscript{\begin{array}{c}
       \tau_1, \tau_2 : d (\tau_1) < n, d (\tau_2) < n\\
       d ((\tau_1 \tau_2)) \geqslant n
     \end{array}}} c (\tau_1) c (\tau_2) X^{(\tau_1 \tau_2)} + \sum_{\tau : d
     (\tau) < n} c (\tau) B (X^{\tau}, u^{\geqslant n}) + B (u^{\geqslant n},
     u^{\geqslant n}) . \]
\end{remark}

\

Our aim is to control the truncated expansion under the natural regularity
assumptions for the white noise case. These regularities turn out to be
\begin{equation}
  X \in C \CC^{- 1 / 2 -}, \hspace{1em} X^{\zzone} \in C \CC^{0 -},
  \hspace{1em} X^{\zztwo}, X^{\zzthree} \in C \CC^{1 / 2 -}, \hspace{1em}
  X^{\zzfour} \in C \CC^{1 -} . \label{eq:reg-X}
\end{equation}
Moreover, for any $f \in C \CC^{\alpha}$ with $\alpha > 1 / 2$ we have that $B
(X, f) \in C \CC^{1 / 2 -}$: indeed, the bilinear map $B$ satisfies
\begin{equation}
  \| B (f, g) \|_{C_T \CC^{\delta}} \lesssim \| f \|_{C_T \CC^{\alpha}} \| g
  \|_{C_T \CC^{\beta}} \label{eq:bilinear-estimate}
\end{equation}
for any $T > 0$, where $\delta = \min (\alpha, \beta, \alpha + \beta) + 1$,
but only if $\alpha + \beta > 0$. \ Equation~(\ref{eq:u-ge-2}) implies that
$u^{\geqslant 2}$ has at least regularity $(1 / 2) \um$. In this case $B
(u^{\geqslant 2}, u^{\geqslant 2})$ is well defined and in $C \CC^{3 / 2 -}$,
but $B (X, u^{\geqslant 2})$ is not well defined because the sum of the
regularities of $X$ and of $u^{\geqslant 2}$ just fails to be positive. To
make sense of $B (X, u^{\geqslant 2})$, we continue the expansion. Setting
$u^{\geqslant 2} = 2 X^{\zztwo} + u^{\geqslant 3}$, we obtain
from~(\ref{eq:u-ge-2}) that $u^{\geqslant 3}$ solves
\[ u^{\geqslant 3} = 4 X^{\zzthree} + X^{\zzfour} + 2 B (X, u^{\geqslant 3}) +
   2 B (X^{\zzone}, u^{\geqslant 2}) + B (u^{\geqslant 2}, u^{\geqslant 2}) .
\]
At this stage, if we assume that $u^{\geqslant 3}, u^{\geqslant 2} \in C
\CC^{1 / 2 -}$, all the terms but $4 X^{\zzthree} + 2 B (X, u^{\geqslant 3})$
are of regularity $C \CC^{1 -}$: indeed
\[ B (X^{\zzone}, u^{\geqslant 2}) \in C \CC^{1 -}, \hspace{1em} B
   (u^{\geqslant 2}, u^{\geqslant 2}) \in C \CC^{3 / 2 -}, \]
and we already assumed that $X^{\zzfour} \in C \CC^{1 -}$. This observation
implies that $u^{\geqslant 3}$ has the form
\[ u^{\geqslant 3} = 2 B (X, u^{\geqslant 2}) + 4 X^\zzthree + C \CC^{1 -}, \]
by which we mean $u^{\geqslant 3} - 2 B (X, u^{\geqslant 2}) - 4X^\zzthree \in C \CC^{1 -}$.
Of course, the problem still lies in obtaining an expression for $B (X,
u^{\geqslant 2})$ which is not well defined since our a priori estimates are
insufficient to handle it. This problem leads us to finding a suitable
paracontrolled structure to overcome the lack of regularity of $u^{\geqslant
2}$.

\subsection{Paracontrolled ansatz}

Here we discuss formally our approach to solving the rough Burgers
equation~(\ref{eq:sbe-theta}). The rigorous calculations will be performed in
the following section.

Inspired by the partial tree series expansion of $u$, we set up a
paracontrolled ansatz of the form
\begin{equation}
  u = X + X^{\zzone} + 2 X^{\zztwo} + u^Q, \hspace{2em} u^Q = u' \mpara Q +
  u^{\sharp}, \label{eq:paracontrolled-structure}
\end{equation}
where the functions $u', Q$ and $u^{\sharp}$ are for the moment arbitrary and
later will be chosen so that $u$ solves the RBE~(\ref{eq:sbe-theta}). The idea
is that the term $u' \mpara Q$ takes care of the less regular contributions
to $u$ of all the terms of the form $B (X, f)$. In this sense, we assume that
$Q \in C \CC^{1 / 2 -}$ and that $u' \in C \CC^{1 / 2 -} \cap C^{1 / 4 -}
L^{\infty}$. The remaining term $u^{\sharp}$ will contain the more regular
part of the solution, more precisely we assume that $u^{\sharp} \in C \CC^{1
-}$. Any such $u$ is called paracontrolled.

For paracontrolled distributions $u$, the nonlinear term takes the form
\[ \mathD u^2 = \mathD (XX + 2 XX^{\zzone} + X^{\zzone} X^{\zzone} + 4
   XX^{\zztwo}) + 2 \mathD (Xu^Q) + 2 \mathD (X^{\zzone} (u^Q + 2 X^{\zztwo}))
   + \mathD ((u^Q + 2 X^{\zztwo})^2) . \]
Of all the term involved, the only one which is problematic is $\mathD
(Xu^Q)$, since by our assumptions $2 \mathD (X^{\zzone} (u^Q + X^{\zztwo})) +
\mathD ((u^Q + X^{\zztwo})^2)$ are well defined and
\[ \mathD (XX + 2 XX^{\zzone} + X^{\zzone} X^{\zzone} + 4 XX^{\zztwo}) = \LL
   (X^{\zzone} + 2 X^{\zztwo} + X^{\zzfour} + 4 X^{\zzthree}) . \]
But using the paracontrolled structure of $u$, the term $\mathD (X u^Q)$ is
defined as $\mathD (X \cdummy u^Q)$ provided that $Q \reso X \in C \CC^{0 -}$
is given. Under this assumption we easily deduce from
Theorem~\ref{thm:paracontrolled product} that
\[
   \mathD (Xu^Q) - u^Q \para \mathD X \in C \CC^{- 1 -}.
\]
In other words, for paracontrolled distributions the nonlinear operation on
the right hand side of~(\ref{eq:sbe-theta}) is well defined as a bounded
operator.

Next, we should specify how to choose $Q$ and which form $u'$ will take for
the solution $u$ of~(\ref{eq:sbe-theta}). To do so, we derive an equation for
$u^{\sharp}$:
\[ \LL u^{\sharp} = \LL (u - X - X^{\zzone} - 2 X^{\zztwo} - u' \mpara Q) =
   4 \mathD (X X^{\zztwo}) + 2 u^Q \para \mathD X - u' \mpara \LL Q + C
   \CC^{- 1 -}, \]
where we used the commutator estimate Lemma~\ref{lem:modified paraproduct
commutators} to replace $\LL (u' \mpara Q)$ by $u' \mpara \LL Q$. If we
now assume that $X^{\zztwo} \reso X \in C \CC^{0 -}$, then we obtain
\[ \LL u^{\sharp} = (2 u^Q + 4 X^{\zztwo}) \mpara \mathD X - u' \mpara
   \LL Q + C \CC^{- 1 -}, \]
which shows that we should choose $Q$ so that $\LL Q = \mathD X$ and $u' = 2
u^Q + 4 X^{\zztwo}$. While this equation looks weird, it can be easily solved
by a fixed point argument. However, in the next section we take a different
approach and start with the solution to the equation driven by regular data
(of which we know that it exists), and use it to derive a priori estimates
which allow us to pass to the limit.

\subsection{A priori estimates on small time intervals}\label{sec:burgers smooth initial}

We are now ready to turn these formal discussions into rigorous mathematics
and to solve Burgers equation on small time intervals. For simplicity, we
first treat the case of initial conditions that are smooth perturbations of
the white noise. In Section~\ref{sec:singular initial} we indicate how to
adapt the arguments to treat more general initial conditions.

Throughout this section we fix $\alpha \in (1 / 3, 1 / 2)$. First, let us
specify which distributions will play the role of enhanced data for us. We
will usually consider regularized versions of the white noise which leave the
temporal variable untouched and convolute the spatial variable against a
smooth kernel. So in the following definition we are careful to only require a
smooth $X$, but not a smooth $\theta$.

\begin{definition}(RBE-enhancement)
  Let
  \[
     \mathcal{X}_{\tmop{rbe}} \subseteq C \CC^{\alpha - 1} \times C \CC^{2 \alpha - 1} \times \LL^{\alpha} \times \LL^{2 \alpha} \times \LL^{2 \alpha} \times C \CC^{2 \alpha - 1}
  \]
  be the closure of the image of the map $\Theta_{\tmop{rbe}} : \LL C(\R, C^{\infty} (\T)) \rightarrow \mathcal{X}_{\tmop{rbe}}$ given by
  \begin{equation}\label{eq:enhanced-theta}
     \Theta_{\tmop{rbe}}(\theta) = (X (\theta), X^{\zzone} (\theta), X^{\zztwo} (\theta), X^\zzthreereso (\theta), X^{\zzfour} (\theta), (Q \reso X) (\theta)),
  \end{equation}
  where
  \begin{equation}\label{eq:pol-X}
    \begin{array}{rll}
      \LL X (\theta) & = & \mathD \theta,\\
      \LL X^{\zzone} (\theta) & = & \mathD (X (\theta)^2),\\
      \LL X^{\zztwo} (\theta) & = & \mathD (X (\theta) X^{\zzone} (\theta)),\\
      \LL X^{\zzthreereso} (\theta) & = & \mathD (X^\zztwo (\theta) \reso X (\theta)),\\
      \LL X^{\zzfour} (\theta) & = & \mathD (X^{\zzone} (\theta) X^{\zzone}
      (\theta)),\\
      \LL Q (\theta) & = & \mathD X (\theta),
    \end{array}
  \end{equation}
  all with zero initial conditions except $X (\theta)$ for which we choose the ``stationary'' initial condition
  \[
     X (\theta) (0) = \int_{- \infty}^0 P_{- s} \mathD \theta (s) \mathd s.
  \]
  We call $\Theta_{\tmop{rbe}} (\theta)$ the \tmtextit{RBE-enhancement}
  of the driving distribution $\theta$. For $T > 0$ we define
  $\mathcal{X}_{\tmop{rbe}} (T) =\mathcal{X}_{\tmop{rbe}} |_{[0, T]}$ and we
  write $\| \X \|_{\mathcal{X}_{\tmop{rbe}} (T)}$ for the norm of
  $\X$ in the Banach space $C_T \CC^{\alpha - 1} \times C_T \CC^{2
  \alpha - 1} \times \LL^{\alpha}_T \times \LL_T^{2 \alpha} \times \LL^{2\alpha}_T \times C_T \CC^{2 \alpha - 1}$. Moreover, we set
  $d_{\mathcal{X}_{\tmop{rbe}} (T)} (\X, \tilde{\X}) = \|
  \X- \tilde{\X} \|_{\mathcal{X}_{\tmop{rbe}} (T)}$.
\end{definition}

For every RBE-enhancement we have an associated space of paracontrolled distributions.

\begin{definition}
  Let $\X \in \mathcal{X}_{\tmop{rbe}}$ and $\delta > 1 - 2 \alpha$.
  We define the space $\CD^{\delta}_{\tmop{rbe}} = \DD^{\delta}_{\tmop{rbe},
  \X}$ of distributions paracontrolled by $\X$ as the set of
  all $(u,u',u^\sharp) \in C \CC^{\alpha-1} \times \LL^{\delta} \times \LL^{\alpha + \delta}$ such that
  \[
     u = X + X^{\zzone} + 2 X^{\zztwo} + u' \mpara Q + u^{\sharp}.
  \]
  For $\delta = \alpha$ we will usually write $\DD_{\tmop{rbe}} =
  \CD^{\alpha}_{\tmop{rbe}}$. For every $T > 0$ we set
  $\CD^{\delta}_{\tmop{rbe}} (T) = \CD^{\delta}_{\tmop{rbe}} |_{[0, T]}$, and
  we define
  \[ \| u \|_{\DD^{\delta}_{\tmop{rbe}} (T)} = \| u' \|_{\LL^{\delta}_T} + \|
     u^{\sharp} \|_{\LL_T^{\alpha + \delta}} . \]
  We will often use the notation $u^Q = u' \mpara Q + u^{\sharp}$. Since $u$ is a function of $u'$ and $u^\sharp$, we will also write $(u',u^\sharp) \in \CD_{\mathrm{rbe}}$, and if there is no ambiguity then we also abuse notation by writing $u \in \CD_{\mathrm{rbe}}$. We call $u'$ the \emph{derivative} and $u^\sharp$ the \emph{remainder}.
\end{definition}

Strictly speaking it is not necessary to assume $u^{\sharp} \in
\LL^{\alpha + \delta}$, the weaker condition $u^{\sharp} \in C \CC^{\alpha +
\delta}$ would suffice to define $\mathD u^2$. But this stronger assumption will
simplify the notation and the uniqueness argument below, and it will always be
satisfied by the paracontrolled solution to Burgers equation.

Our a priori estimate below will bound $\| u \|_{\CD_{\tmop{rbe}} (T)}$ in
terms of the weaker norm $\| u \|_{\CD_{\tmop{rbe}}^{\delta} (T)}$ for
suitable $\delta < \alpha$. This will allow us to obtain a scaling factor
$T^{\varepsilon}$ on small time intervals $[0, T]$, from where we can then
derive the locally Lipschitz continuous dependence of the solution $u$ on the
data $(\X, u_0)$. Throughout this section we will work under the following assumption:

\begin{assumption}[T,M]
  Assume that $\theta, \tilde{\theta} \in \LL C (\R,
  C^{\infty} (\T))$ and $u_0, \tilde{u}_0 \in \CC^{2\alpha}$, and that $u$ is the unique global in time
  solution to the Burgers equation
  \begin{equation}\label{eq:rbe smooth initial}
     \LL u = \mathD u^2 + \mathD \xi, \hspace{2em} u (0) = X(0) + u_0.
  \end{equation}
  We define $\X= \Theta_{\tmop{rbe}} (\theta)$, $u^Q = u - X -
  X^{\zzone} - 2 X^{\zztwo}$, and $u' = 2 u^Q + 4 X^{\zztwo}$, and we set
  $u^{\sharp} = u^Q - u' \mpara Q$. Similarly we define $\tilde{\X}, \tilde{u}, \tilde{u}^Q, \tilde{u}', \tilde{u}^{\sharp}$. Finally we assume that $T,M>0$ are such that
  \[
     \max \big\{ \|u_0\|_{2\alpha}, \|\tilde{u}_0\|_{2\alpha}, \|\X\|_{\mathcal{X}_{\tmop{rbe}} (T)}, \| \tilde{\X} \|_{\mathcal{X}_{\tmop{rbe}} (T)} \big\} \leqslant M.
  \]
\end{assumption}

\begin{lemma}\label{lem:burgers a priori}
   Make Assumption (T,M) and let $\delta > 1 - 2 \alpha$. Then
  \begin{equation}\label{eq:usharp a priori}
     \| \LL u^{\sharp} - \LL X^{\zzfour} - 4\LL X^\zzthreereso \|_{C_T \CC^{2 \alpha - 2}} + \| u' \|_{\LL^{\alpha}_T} \lesssim (1 + M^2) ( 1 + \| u \|^2_{\CD^{\delta}_{_{\tmop{rbe}}} (T)} ).
  \end{equation}
  If further also $\| u \|_{\CD_{\tmop{rbe}, \X} (T)}, \| \tilde{u} \|_{\CD_{\tmop{rbe}, \tilde{\X}} (T)} \leqslant M$, then
  \begin{align} \label{eq:usharp a priori difference}
     & \|(\LL u^{\sharp} - \LL X^{\zzfour} - 4\LL X^\zzthreereso) - (\LL \tilde{u}^{\sharp} - \LL \tilde{X}^{\zzfour} - 4\LL \tilde X^\zzthreereso)\|_{C_T \CC^{2 \alpha - 2}} + \| u' - \tilde{u}'  \|_{\LL^{\alpha}_T} \\ \nonumber
     &\hspace{130pt} \lesssim M^2 \big(d_{\mathcal{X}_{\tmop{rbe}} (T)} (\X, \tilde{\X}) + ( \| u' - \tilde{u}' \|_{\LL^{\delta}_T} + \| u^{\sharp} - \tilde{u}^{\sharp} \|_{\LL^{2 \delta}_T} ) \big) .
  \end{align}
\end{lemma}

\begin{proof}
  Let us decompose $\LL u^{\sharp}$ into a sum of bounded multilinear
  operators:
  \begin{align*}
    \LL u^{\sharp} & = \mathD u^2 - \mathD X^2 - 2 \mathD X^{\zzone} X - \LL
     (u' \mpara Q) \\
     & = \LL X^{\zzfour} + 2 \mathD [u^Q X - u^Q \para X] + 2 [\mathD (u^Q \para X) - u^Q \para \mathD X] \\
     &\quad + 4 \LL X^\zzthreereso  + 4 \mathD [X^{\zztwo} \lpara X] + 4 [\mathD (X^{\zztwo} \para X) - X^{\zztwo} \para \mathD X] + [u' \para \mathD X - u' \mpara \mathD X] \\
     &\quad + \mathD [2 X^{\zzone} (2 X^{\zztwo} + u^Q) + (2 X^{\zztwo} + u^Q)^2] - \left[ \LL (u' \mpara Q) - u' \mpara \left( \LL Q \right) \right],
  \end{align*}
  where we used that $u' = 2 u^Q + 4 X^{\zztwo}$. From here it is
  straightforward to show that
  \begin{align}\label{eq:usharp a priori pr1}
     \| \LL u^{\sharp} - \LL X^{\zzfour} - 4 \LL X^\zzthreereso \|_{C_T \CC^{2 \alpha - 2}} & \lesssim \| \X \|_{\mathcal{X}_{\tmop{rbe}} (T)} ( \| u^Q \|_{C_T \CC^{\alpha}} + \| u' \|_{\LL^{\alpha}_T} ) + \| u^Q \|_{C_T L^{\infty}}^2 \\ \nonumber
     &\quad + (\| \X \|_{\mathcal{X}_{\tmop{rbe}} (T)} + \| \X \|_{\mathcal{X}_{\tmop{rbe}} (T)}^2) ( 1 + \| u^{\sharp} \|_{C_T \CC^{\alpha + \delta}} + \| u' \|_{C_T \CC^{\delta}} ) .
  \end{align}
  Now $\| u^Q \|_{C_T \CC^{\alpha}} \lesssim \| u' \|_{\LL^{\delta}_T} \|
  \X \|_{\mathcal{X}_{\tmop{rbe}} (T)} + \| u^{\sharp} \|_{C_T
  \CC^{\alpha}}$ by Lemma~\ref{lem:paraproduct parabolic space}  (observe that $\| Q
  \|_{\LL^{\alpha}_T} \lesssim \| \X \|_{\mathcal{X}_{\tmop{rbe}}
  (T)}$ by the Schauder estimates, Lemma~\ref{lemma:schauder}), and therefore
  the estimate~(\ref{eq:usharp a priori}) follows as soon as we can bound $(1 + \| \X \|_{\mathcal{X}_{\tmop{rbe}} (T)}) \| u' \|_{\LL^{\alpha}_T}$ by the right hand side of~(\ref{eq:usharp a priori}).
  For this purpose it suffices to use the explicit form $u' = 2 u^Q + 4
  X^{\zztwo}$ and then $u^Q = u' \mpara Q + u^{\sharp}$ in combination with
  Lemma~\ref{lem:paraproduct parabolic space}.
  
  The same arguments combined with the multilinearity of our operators also yield~(\ref{eq:usharp a priori difference}).
\end{proof}

The main result of this section now follows:

\begin{theorem}\label{thm:burgers smooth initial}
  For every $(\X,u_0) \in \mathcal{X}_{\tmop{rbe}} \times \CC^{2\alpha}$ there exists $T^\ast \in (0,
  \infty]$ such that for all $T<T^\ast$ there is a unique solution $(u,u',u^\sharp) \in \CD_{\mathrm{rbe}}(T)$ to the rough Burgers equation~\eqref{eq:rbe smooth initial} on the interval $[0,T]$. Moreover, we can choose
  \[
     T^\ast = \inf\{t \geqslant 0: \|u\|_{\CD_{\mathrm{rbe}}(t)} = \infty \}.
  \]
\end{theorem}

\begin{proof}
  Let $M = 2 \max\{\|u_0\|_{2\alpha}, \|\X \|_{\mathcal{X}_{\tmop{rbe}}(1)}\}$. Consider a sequence $(\theta_n, u_0^n)\subset \LL C(\R,C^\infty) \times \CC^{2\alpha}$ such that $\X_n=\Theta_{\tmop{rbe}}(\theta_n)$ approximates $\X|_{[0,T]}$ in $\mathcal{X}_{\tmop{rbe}}(T)$ for all $T>0$ and $(u_0^n)$ approximates $u_0$ in $\CC^{2\alpha}$. Without loss of generality assume that $\|u_0^n\|_{2\alpha}\vee\|\X_n \|_{\mathcal{X}_{\tmop{rbe}}(1)} \leqslant M$ for all $n$. Denote by $u_n$ the solution to Burgers equation driven by $\X_n$ and started in $u_0^n$ and define $u_n^Q$, $u_n'$, $u_n^\sharp$ as in Assumption~($T,M$). In particular, for all $n$ we are in the setting of Assumption~($1,M$).
  
  For fixed $n$, the remainder has the initial condition $u^\sharp_n(0) = u^n_0 \in \CC^{2\alpha}$, so that Lemma~\ref{lem:burgers a priori} and the Schauder estimates
  (Lemma~\ref{lemma:schauder}) yield
  \[
     \| u_n \|_{\CD_{\tmop{rbe}} (\tau)} \lesssim (1 + M^2) ( 1 + \| u_n \|^2_{\CD^{\delta}_{\tmop{rbe}} (\tau)} )
  \]
  for all $\tau \in (0, 1]$. Therefore, we can apply Lemma~\ref{lem:scaling factor smaller norm} to gain a small scaling factor:
  \begin{equation}\label{eq:burgers smooth initial pr1}
     \| u_n \|_{\CD_{\tmop{rbe}} (\tau)} \lesssim (1 + M^2) ( 1 + \tau^{\alpha - \delta} \| u_n \|^2_{\CD_{\tmop{rbe}} (\tau)} ) .
  \end{equation}
  Now $\| u_n \|_{\CD_{\tmop{rbe}} (\tau)} \leqslant \| u_n \|_{\CD_{\tmop{rbe}}
  (1)}$ for all $\tau \in (0, 1]$, and therefore there exists a universal
  constant $C > 0$ such that $\| u_n \|_{\CD_{\tmop{rbe}} (\tau)} \leqslant C (1 +
  M^2)$ for all sufficiently small $\tau \leqslant \tau^*$. However, a priori $\tau^*$ may
  depend on stronger norms of $\theta_n$ and not just on $\| \X_n
  \|_{\mathcal{X}_{\tmop{rbe}} (1)}$. To see that we may choose $\tau^*$ only
  depending on $M$, it suffices to note that $\tau \mapsto \| u_n \|_{\CD_{\tmop{rbe}} (\tau)}$ is a
  continuous function (which follows from the fact that $u^{\sharp}_n$ and $u'_n$
  are actually more regular than $\LL^{2 \alpha}$ and $\LL^{\alpha}$
  respectively), to select the smallest $\tau^*$ with $\| u_n
  \|_{\CD_{\tmop{rbe}} (\tau^*)} = C (1 + M^2)$, and to plug $\| u_n
  \|_{\CD_{\tmop{rbe}} (\tau^*)}$ into~(\ref{eq:burgers smooth initial pr1}).
  
  In other words $\| u_n \|_{\CD_{\tmop{rbe}} (\tau)} \leqslant C(1+M^2)$ for all $n$, so that the second statement of Lemma~\ref{lem:burgers a priori} in conjunction with the Schauder estimates for the Laplacian shows that 
  \[
     \|u^{\sharp}_m - u^{\sharp}_n \|_{\LL^{2 \alpha}_\tau} + \| u'_m -  u_n' \|_{\LL^{\alpha}_\tau} \lesssim (1 + M^2) \big(d_{\mathcal{X}_{\tmop{rbe}} (\tau)} (\X_m, \X_n) + \|u_0^m - u_0^n\|_{2\alpha} \big)
  \]
  for all $m,n$ (possibly after further decreasing $\tau = \tau(M)$).
  
  So if
  \[
     \CB = \left\{ (\Theta_{\tmop{rbe}} (\theta), u_0) : \theta \in \LL C(\R, C^{\infty} (\T)), v \in \CC^{2\alpha},  \| \Theta_{\tmop{rbe}} (\theta) \|_{\mathcal{X}_{\tmop{rbe}} (1)}, \|v\|_{2\alpha} \leqslant M \right\},
  \]
  and if $\Psi : \CB \rightarrow \LL^{\alpha}_{\tau(M)} \times \LL^{2 \alpha}_{\tau(M)}$ is the map
  which assigns to $(\theta, v)$ the local solution $(u', u^{\sharp})$ of \eqref{eq:rbe smooth initial}, then $\Psi$ has a unique
  continuous extension from $\CB$ to its closure in $\mathcal{X}_{\tmop{rbe}}(1) \times \CC^{2\alpha}$, which is just the ball $\CB_M (\mathcal{X}_{\tmop{rbe}} (1) \times \CC^{2\alpha})$ of size $M$ in this space. By construction, $u = (u', u^{\sharp}) = \Psi
  (\X,u_0)$ solves Burgers equation driven by $\X$ and with initial condition $u_0$ if we interpret the product $\mathD u^2$
  using the paracontrolled structure of $u$. To see that $u$ is the unique
  element in $\CD_{\tmop{rbe}, \X} (\tau (M))$ which solves the
  equation, it suffices to repeat the proof of estimate~(\ref{eq:usharp a priori difference}).
  
  It remains to iterate this argument in order to obtain solutions until the blow up time $T^\ast$. Assume that we constructed a paracontrolled solution $u$ on $[0,\tau]$ for some $\tau>0$ and consider the time interval $[\tau,\tau+1]$. Denote the solution on this interval by $\tilde u$, and write $\tilde \X(t) = \X(t+\tau)$. The initial condition for $\tilde u$ is $\tilde u(0) = u(\tau) \in \CC^\alpha$, so that we are not in the same setting as before. But in fact we only needed a smooth initial condition  to obtain an initial condition of regularity $2\alpha$ for the remainder, and now we have
  \begin{align*}
     \tilde u^\sharp(0) & = \tilde u(0) - \tilde X(0) - \tilde X^{\zzone} (0) - 2\tilde X^{\zztwo} (0) - \tilde u' \mpara \tilde Q(0) \\
     & = u(\tau) - X(0) - X^{\zzone} (\tau) - 2 X^{\zztwo} (\tau) - u'(\tau) \para Q(\tau).
  \end{align*}
  According to Lemma~\ref{lem:modified paraproduct commutators} we have $u'(\tau) \para Q(\tau) - u'\mpara Q(\tau) \in \CC^{2\alpha}$, so that the paracontrolled structure of $u$ on the time interval $[0,\tau]$ shows that the initial condition for $\tilde u^\sharp$ is in $\CC^{2\alpha}$. As a consequence, there exists $\tilde M>0$ which is a function of $\|\X\|_{\mathcal{X}_{\tmop{rbe}}(\tau+1)}$, $\|u\|_{\CD_{\tmop{rbe}}(\tau)}$ such that the unique paracontrolled solution $\tilde u$ can be constructed on $[0,\tau(\tilde M)]$. We then extend $u$ and $u'$ from $[0,\tau]$ to $[0,\tau + \tau(\tilde M)]$ by setting $u(t) = \tilde u(t-\tau)$ for $t \geqslant \tau$ and similarly for $u'$. Then by definition $u' \in \LL^{\alpha}_{\tau + \tau(\tilde M)}$. For the remainder we know from our construction that on the interval $[\tau,\tau+\tau(\tilde M)]$ we have $u^Q - u'\mpara_{\tau} Q \in \LL^{2\alpha}$, where
  \[
      u'\mpara_{\tau} Q(t) = \sum_j \int_{-\infty}^t 2^{2j} \varphi(2^{2j}(t-s)) S_{j-1} u'(s \vee \tau) \mathd s \Delta_j Q(t).
  \]
  It therefore suffices to show that $(u'\mpara_{\tau}Q - u'\mpara Q)|_{[\tau,\tau+\tau(\tilde M)]} \in \LL^{2\alpha}$. But we already showed that $(u'\mpara_{\tau}Q - u'\mpara Q)(\tau) \in \CC^{2\alpha}$, and for $t \in [\tau,\tau+\tau(\tilde M)]$ we get
  \begin{align*}
     \LL(u'\mpara_{\tau}Q - u'\mpara Q)(t) & = (u' \mpara_{\tau}\LL Q - u'\mpara \LL Q)(t) + \CC^{2\alpha} \\
     & = (u' \mpara_{\tau}\mathD X - u'\mpara \mathD X)(t) + \CC^{2\alpha}.
  \end{align*}
  Now Lemma~\ref{lem:modified paraproduct commutators} shows that $u'\mpara \mathD X - u' \para \mathD X \in C_{\tau+\tau(\tilde M)}\CC^{2\alpha-2}$, and by the same arguments it follows that also $(u'\mpara_{\tau} \mathD X - u' \para \mathD X \in \CC^{2\alpha-2})|_{[\tau,\tau+\tau(\tilde M)]} \in C([\tau,\tau+\tau(\tilde M)], \CC^{2\alpha})$, so that the $\LL^{2\alpha}$ regularity of $u^\sharp$ on $[0,\tau+\tau(\tilde M)]$ follows from the Schauder estimates for the heat flow, Lemma~\ref{lemma:schauder}.
  
  Reversing these differences between the different paraproducts, we also obtain the uniqueness of $u$ on $[0,\tau + \tilde \tau(M)]$.
\end{proof}

\begin{remark}
   There is a subtlety concerning the uniqueness statement. It is to be interpreted in the sense that there exists (locally in time) a unique fixed point $(u,u',u^\sharp)$ in $\CD_{\mathrm{rbe}}$ for the map $\CD_{\mathrm{rbe}} \ni (v,v',v^\sharp) \mapsto (w,w',w^\sharp) \in \CD_{\mathrm{rbe}}$, where
   \[
      w = P_\cdot(u_0) + I(\mathD v^2) + X, \qquad w' = 2 v^Q + 4 X^{\zztwo}, \qquad w^\sharp = w - X - X^\zzone - 2 X^\zzthree - w'\mpara Q.
   \]
   For general $\X \in \mathcal{X}_{\mathrm{rbe}}$ it is \emph{not true} that there is a unique $u \in C_T \CC^{\alpha-1}$ for which there exist $u', u^\sharp$ with $(u,u',u^\sharp) \in \CD_{\mathrm{rbe,\X}}(T)$ and such that $u = P_\cdot(u_0) + I(\mathD u^2) + X$, where the square is calculated using $u'$ and $u^\sharp$. The problem is that $u'$ and $u^\sharp$ are in general not uniquely determined by $u$ (see however~\cite{HairerPillai2013,FrizShekhar2012} for conditions that guarantee uniqueness of the derivative in the case of controlled rough paths). Nonetheless, we will see in Section~\ref{sec:interpretation} below that this stronger notion of uniqueness always holds as long as $\X$ is ``generated from convolution smoothings'', and in particular when $\X$ is constructed from the white noise.
   
   Moreover, without additional conditions on $\X \in \mathcal{X}_{\mathrm{rbe}}$ there always exists a unique $u \in C_T \CC^{\alpha-1}$ so that if $(\theta_n)$ is a sequence in $\LL C(\R,C^\infty)$ for which $(\Theta_{\mathrm{rbe}}(\theta_n))$ converges to $\X$, then the classical solutions $u_n$ to Burgers equation driven by $\theta_n$ converge to $u$.
\end{remark}

The solution constructed in Theorem~\ref{thm:burgers smooth initial} depends continuously on the data $\X$ and on the initial condition $u_0$. To formulate this, we have to be a little careful because for now we cannot exclude the possibility that the blow up time $T^\ast$ is finite (although we will see in Section~\ref{sec:hjb} below that in fact $T^\ast = \infty$ for all $\X \in \mathcal{X}_{\rbe}$, $u_0 \in \CC^{2\alpha}$).

We define two distances: write
\begin{equation}\label{eq:local Xrbe distance}
   d_{\mathcal{X}_\rbe}(\X, \tilde \X) = \sum_{m=1}^\infty 2^{-m} (1 \wedge d_{\mathcal{X}_\rbe(m)}(\X, \tilde \X)).
\end{equation}
For $u \in \CD_{\rbe}$ and $c>0$ we define
\[
   \tau_c(u) = \inf\{ t \geqslant 0: \| u\|_{\CD_\rbe(t)} \geqslant c\}
\]
and then for $u \in \CD_{\rbe,\X}$, $\tilde u \in \CD_{\rbe,\tilde \X}$
\[
   d_{\rbe,c}(u,\tilde u) = \| u' - \tilde u' \|_{\LL^\alpha_{c \wedge \tau_c(u) \wedge \tau_c (\tilde u)}} + \| u^\sharp - \tilde u^\sharp \|_{\LL^{2\alpha}_{c \wedge \tau_c(u) \wedge \tau_c (\tilde u)}}
\]
and
\begin{equation}\label{eq:local Drbe distance}
   d_{\rbe}(u,\tilde u) = \sum_{m=1}^\infty 2^{-m} (1 \wedge (d_{\rbe,m}(u,\tilde u) + |\tau_m(u) - \tilde \tau_m(\tilde u)|)).
\end{equation}

\begin{theorem}\label{thm:rbe continuous from data}
   Let $(\X_n)_n, \X \in \CX_\rbe$ be such that $\lim_n d_{\CX_\rbe}(X_n, X) = 0$ and let $(u_0^n), u_0 \in \CC^{2\alpha}$ be such that $\lim_n \| u_0^n - u_0 \|_\alpha = 0$. Denote the solution to the rough Burgers equation~\eqref{eq:rbe smooth initial} driven by $\X_n$ and started in $X_n(0) + u_0^n$ by $u_n$, and write $u$ for the solution driven by $\X$ and started in $X(0) + u_0$. Then
   \[
      \lim_n d_{\rbe}(u_n, u) = 0.
   \]
\end{theorem}

\begin{proof}
   Let $m\in \N$ and define $M = \max\{ \|\X\|_{\CX_\rbe(m)},  \| \X_n\|_{\CX_\rbe(m)}, \|u_0\|_{2\alpha}, \| u_0^n \|_{2\alpha}\}$. Using the same arguments as in the proof of Theorem~\ref{thm:burgers smooth initial}, we see that there exists a constant $C>0$, depending only on $m$ and $M$, such that
   \[
      d_{\rbe,m}(u,u_n) \leqslant C( d_{\CX_{\rbe}(m)}(\X,\X_n) + \|u_0 - u_0^n\|_{2\alpha}).
   \]
   It remains to show that $\tau_m(u_n)$ converges to $\tau_m(u)$ for all $m$. Write $T = \tau_m(u)$. Using the same arguments as in Lemma~\ref{lem:burgers a priori}, we obtain that $z_n = u^\sharp - u^\sharp_n$, $y_n = u^Q - u^Q_n$ solves an equation of the type
   \[
      \LL z_n = \nabla (y_n^2) + A_n(y_n) + B_n(z_n) + \nabla h_n, \qquad z_n(0) = u_0 - u_0^n,
   \]
   where $(A_n)$ and $(B_n)$ are sequences of bounded operators from $\LL_T^\alpha$ to $C_T \CC^{2\alpha-2}$ and from $\LL_T^{2\alpha}$ to $C_T \CC^{2\alpha-2}$ respectively, with operator norms bounded uniformly in $n$, and where $( I (\nabla h_n))$ converges to zero in $\LL_T^{2\alpha}$. Furthermore, we have $y_n = 2 y_n \mpara Q_n + z_n + \varepsilon(n)$, where $\varepsilon(n)$ converges to zero in $\LL^\alpha_T$. Since also $z_n(0)$ converges to zero in $\CC^{2\alpha}$, we can use the same arguments as above to get that $(y_n, z_n)$ converge to zero in $\LL_T^{\alpha} \times \LL_T^{2\alpha}$, and therefore $\lim_n \tau_m(u_n) = T$.
\end{proof}

\section{KPZ and Rough Heat Equation}\label{sec:kpz}

We discuss the relations of the RBE with the KPZ and the multiplicative heat equation in the paracontrolled setting.

\subsection{The KPZ equation}

For the KPZ equation
\[
   \LL h = (\mathD h)^2 + \xi, \hspace{2em} h (0) = h_0,
\]
one can proceed, at least formally, as for the SBE equation by introducing a
series of driving terms $(Y^{\tau})_{\tau \in \mathcal{T}}$ defined
recursively by
\[ \LL Y^{\bullet} = \xi, \hspace{2em} \LL Y^{(\tau_1 \tau_2)} = (\mathD
   Y^{\tau_1} \mathD Y^{\tau_2}) \]
and such that $\mathD Y^{\tau} = X^{\tau}$ for all $\tau \in \mathcal{T}$; we will usually write $Y = Y^\bullet$.
However, now we have to be more careful, because it will turn out that for the
white noise $\xi$ some of the terms $Y^{\tau}$ can only be constructed after a
suitable renormalization. Let us consider for example $Y^{\zzone}$. If $\varphi\colon \R \to \R$
is a smooth compactly supported function with $\varphi(0) = 1$ and $\xi_{\varepsilon} = \varphi(\varepsilon \mathD) \xi$, then as $\varepsilon \rightarrow 0$ we only obtain the convergence of
$Y^{\zzone}_{\varepsilon}$ after subtracting suitable diverging constants
(which are deterministic): there exist $(c_{\varepsilon})_{\varepsilon > 0}$,
with $\lim_{\varepsilon \rightarrow 0} c_{\varepsilon} = \infty$, such that
$(Y_{\varepsilon}^{\zzone} (t) - c_{\varepsilon} t)_{t \geqslant 0}$ converges
to a limit $Y^{\zzone}$. We stress the fact that while $c_{\varepsilon}$
depends on the specific choice of $\varphi$, the limit $Y^{\zzone}$ does not. To
make the constant $c_{\varepsilon}$ appear in the equation, we should solve
for
\[
   \LL h_{\varepsilon} = (\mathD h_{\varepsilon})^2 - c_{\varepsilon} + \xi_{\varepsilon}, \hspace{2em} h_{\varepsilon} (0) = \varphi(\varepsilon\mathD) h_0.
\]
In that sense, the limiting object $h$ will actually solve
\begin{equation}
  \label{eq:kpz-renormalized} \LL h = (\mathD h)^{\diamond 2} + \xi,
  \hspace{2em} h (0) = h_0,
\end{equation}
where $(\mathD h)^{\diamond 2} = (\mathD h)^2 - \infty$ denotes a renormalized
product. The reason why we did not see this renormalization appear in Burgers
equation is that there we differentiated after the multiplication, considering
for example $X^{\zzone}_{\varepsilon} (t) = \mathD (Y^{\zzone}_{\varepsilon}
(t) - c_{\varepsilon} t) = \mathD Y^{\zzone}_{\varepsilon} (t)$. There are two
other terms which need to be renormalized, $Y^{\zzthree}$ and $Y^{\zzfour}$.
But on the level of the equation these renormalizations cancel exactly, so
that it suffices to introduce the renormalization of
$Y^{\zzone}_{\varepsilon}$ by subtracting $c_{\varepsilon}$.
We still consider a fixed $\alpha \in (1 / 3, 1 / 2)$.

\begin{definition}\label{def:kpz rough distribution}(KPZ-enhancement)
   Let
  \[
     \mathcal{Y}_{\tmop{kpz}} \subseteq \LL^{\alpha} \times \LL^{2 \alpha} \times \LL^{\alpha + 1} \times \LL^{2 \alpha + 1} \times \LL^{2\alpha+1} \times C \CC^{2 \alpha - 1}
  \]
  be the closure of the image of the map
  \[
     \Theta_{\tmop{kpz}} : \LL C^{\alpha / 2}_{\mathrm{loc}} (\R, C^{\infty} (\T)) \times \R \times \R \rightarrow \mathcal{Y}_{\tmop{kpz}},
  \]
  given by
  \begin{equation}\label{eq:enhanced-kpz}
    \Theta_{\tmop{kpz}} (\theta, c^{\zzone}, c^{\zzfour}) = (Y (\theta), Y^{\zzone} (\theta), Y^{\zztwo} (\theta), Y^\zzthreereso(\theta), Y^{\zzfour} (\theta), (\mathD P \reso \mathD Y) (\theta)), 
  \end{equation}
  where
  \begin{equation}
    \begin{array}{rll}
      \LL Y (\theta) & = & \theta,\\
      \LL Y^{\zzone} (\theta) & = & (\mathD Y (\theta))^2 - c^{\zzone},\\
      \LL Y^{\zztwo} (\theta) & = & \mathD Y (\theta) \mathD Y^{\zzone} (\theta),\\
      \LL Y^\zzthreereso(\theta) & = & \mathD Y^\zztwo \reso \mathD Y + c^\zzfour/4 \\
      \LL Y^{\zzfour} (\theta) & = & \mathD Y^{\zzone} (\theta) \mathD Y^{\zzone} (\theta) - c^{\zzfour},\\
      \LL P (\theta) & = & \mathD Y (\theta),
    \end{array} \label{eq:pol-Y}
  \end{equation}
  all with zero initial conditions except $Y(\theta)$ for which we choose the
  ``stationary'' initial condition
  \[
     Y (\theta) (0) = \int_{- \infty}^0 P_{- s} \Pi_{\neq 0} \theta (s) \mathd s.
  \]
  Here and in the following $\Pi_{\neq 0}$ denotes the projection on the (spatial) Fourier modes $\neq 0$. We call $\Theta_{\tmop{kpz}} (\theta,c^\zzone, c^\zzfour)$ the \ \tmtextit{renormalized KPZ-enhancement} of the driving
  distribution $\theta$. For $T > 0$ we define $\mathcal{Y}_{\tmop{kpz}} (T)
  =\mathcal{Y}_{\tmop{kpz}} |_{[0, T]}$ and we write $\| \Y
  \|_{\mathcal{Y}_{\tmop{kpz}} (T)}$ for the norm of $\Y$ in the
  Banach space $\LL_T^{\alpha} \times \LL_T^{2 \alpha} \times \LL^{\alpha +
  1}_T \times \LL_T^{2 \alpha + 1} \times \LL_T^{2\alpha+1} \times C_T
  \CC^{2 \alpha - 1}$. Moreover, we define the distance
  $d_{\mathcal{Y}_{\tmop{kpz}} (T)} (\Y, \tilde{\Y}) = \|
  \Y- \tilde{\Y} \|_{\mathcal{Y}_{\tmop{kpz}} (T)}$.
\end{definition}

Observe the link between the renormalizations for $Y^{\zzfour}$ and $Y^\zzthreereso$. It is chosen so that $c^{\zzfour}$ will never appear in the equation. Also note that for every $\Y \in \mathcal{Y}_{\tmop{kpz}}$ there exists an associated $\X \in \mathcal{X}_{\tmop{rbe}}$, obtained by differentiating the first five entries of $\Y$ and keeping the sixth entry fixed. Abusing notation, we write $\X = \mathD \Y$.

The paracontrolled ansatz for KPZ is $h \in \DD_{\tmop{kpz}} =
\CD_{\tmop{kpz}, \Y}$ if
\begin{equation}
  h = Y + Y^{\zzone} + 2 Y^{\zztwo} + h^P, \hspace{2em} h^P = h' \mpara P +
  h^{\sharp}, \label{eq:paracontrolled-structure-kpz}
\end{equation}
where $h' \in \LL^{\alpha}$, and $h^{\sharp} \in \LL^{2 \alpha + 1}$. For $h
\in \CD_{\tmop{kpz}}$ it now follows from the same arguments as in the case of
Burgers equation that the term
\[ (\mathD h)^{\diamond 2} = (\mathD h)^2 - c^{\zzone} \]
is well defined as a bounded multilinear operator. Now we obtain the local
existence and uniqueness of paracontrolled solutions to the KPZ
equation~(\ref{eq:kpz-renormalized}) exactly as for the Burgers equation, at least as long as the initial condition is of the form $Y(0) + h(0)$ with $h(0) \in \CC^{2\alpha+1}$.

Moreover, let $\X= \mathD \Y \in \mathcal{X}_{\tmop{rbe}}$,
and let $h \in \CD_{\tmop{kpz}, \Y}$. Then $\mathD h = X +
X^{\zzone} + 2 X^{\zztwo} + \mathD h^P$, where
\[ \hspace{2em} \mathD h^P = (\mathD h') \mpara P + h' \mpara \mathD P +
   \mathD h^{\sharp} = h' \mpara Q + (\mathD h)^{\sharp} \]
with $(\mathD h)^{\sharp} \in \LL^{2 \alpha - \varepsilon}$ for all
$\varepsilon > 0$. This follows from the fact that $\LL^{2 \alpha + 1}
\subseteq C^{(2 \alpha - \varepsilon) / 2} \CC^{1 + \varepsilon}$. In other
words $\mathD h \in \CD_{\tmop{rbe}, \X}^{\delta}$ for all $\delta <
\alpha$ (using an interpolation argument one can actually show that $(\mathD h)^{\sharp} \in \LL^{2 \alpha}$, but we will not need this). Moreover, if $h \in \DD_{\tmop{kpz}}$ solves the KPZ equation with
initial condition $h_0$, then $u = \mathD h \in \DD_{\tmop{rbe}}^\delta$ solves the
Burgers equation with initial condition $\mathD h_0$.

Conversely, let $u \in \DD_{\tmop{rbe}}$. It is easy to see that $u = \mathD
\tilde{h} + f$, where $\tilde{h} \in \CD_{\tmop{kpz}}$ and $f \in C \CC^{2\alpha}$. Therefore, the renormalized product
$u^{\diamond 2} = u^2 - c^{\zzone}$ is well defined. Note that $u^{\diamond
2}$ does not depend on the decomposition $u = \mathD \tilde{h} + f$. The
linear equation
\[
   \LL h = u^{\diamond 2} + \theta, \hspace{2em} h (0) = Y(0) + h_0,
\]
has a unique global in time solution $h$ for all $h_0 \in \CC^{2\alpha+1}$. Setting $h^P = h - Y - Y^{\zzone} -
2 Y^{\zztwo}$, we get
\[ \LL h^P = \LL (h - Y - Y^{\zzone} - 2 Y^{\zztwo}) = u^{\diamond 2} -
   ((\mathD Y)^2 - c^{\zzone}) - 2 \mathD Y \mathD Y^{\zzone} . \]
Recalling that $u^{\diamond 2} = (\mathD \tilde{h})^{\diamond 2} + 2 f \mathD \tilde{h} + f^2 = ((\mathD Y)^2 - c^{\zzone}) + 2 \mathD Y \mathD Y^{\zzone} + \LL (\LL^{\alpha + 1})$, we deduce that $h^P \in \LL^{\alpha + 1}$. Furthermore,
if we set $h' = 2 u^Q + 4 X^{\zztwo}$ and expand $u^{\diamond 2}$, then
\begin{align*}
   \LL (h^P - h' \mpara P) &= u^{\diamond 2} - ((\mathD Y)^2 - c^{\zzone}) - 2 \mathD Y \mathD Y^{\zzone}\\
   &\qquad - \left( \LL (h' \mpara P) - h' \mpara \LL P \right) + h' \mpara \LL P \\
   & = (h' \para X - h' \mpara X) + \LL (\LL^{2 \alpha + 1}).
\end{align*}
Using Lemma~\ref{lem:modified paraproduct commutators}, we get that $h^{\sharp} \in \LL^{2 \alpha + 1}$ and therefore $h \in \DD_{\tmop{kpz}}$. By construction, $\mathD h$ solves the equation $\LL
(\mathD h) = \mathD u^2 + \mathD \theta$ with initial condition $\mathD h (0)
= X(0) + \mathD h_0$. So if $u$ solves the Burgers equation with initial condition $X(0) + \mathD h_0$, then
$\LL (\mathD h - u) = 0$ and $(\mathD h - u) (0) = 0$, and therefore $\mathD h
= u$. As a consequence, $h$ solves the KPZ equation with initial condition $Y(0) + h_0$. 

We can also introduce analogous distances $d_{\Ykpz}(\Y, \tilde \Y)$ and $d_{\kpz}(h,\tilde h)$ as in Section~\ref{sec:burgers smooth initial}, equations~\eqref{eq:local Xrbe distance} and~\eqref{eq:local Drbe distance}. We omit the details since the definition is exactly the same, replacing every occurrence of ``$\mathrm{rbe}$'' by ``$\kpz$'' and adapting the regularities to the KPZ setting.

\begin{theorem}\label{thm:kpz smooth initial}
  For every $(\Y, h_0) \in \mathcal{X}_{\tmop{kpz}} \times \CC^{2\alpha+1}$ there exists $T^\ast \in (0,
  \infty]$ such that for all $T<T^\ast$ there is a unique solution $(h,h',h^\sharp) \in \CD_{\mathrm{kpz}}(T)$ to the KPZ equation
  \begin{equation}\label{eq:kpz smooth initial}
     \LL h = (\mathD h)^{\diamond 2} + \theta, \qquad h (0) = Y(0)+h_0,
  \end{equation}
  on the interval $[0,T]$. The map that sends $(\Y, h_0) \in \Ykpz \times \CC^{2\alpha+1}$ to the solution $h \in \CD_{\kpz,\Y}$ is continuous with respect to the distances $(d_{\Ykpz} + \|\cdot \|_{2\alpha+1}, d_{\kpz})$, and we can choose
  \[
     T^\ast = \inf\{t \geqslant 0: \|h\|_{\CD_{\mathrm{kpz}}(t)} = \infty \}.
  \]
  If $\X= \mathD \Y$, then $\X \in \mathcal{X}_{\tmop{rbe}}$ and $u = \mathD h \in \CD_{\tmop{rbe}}(T)$ solves the rough Burgers equation with data $\X$ and initial condition $u (0) = X(0) + \mathD h_0$.
  Conversely, given the solution $u$ to Burgers equation driven by
  $\X$ and started in $X(0) + \mathD h_0$, the solution $h \in \CD_{\tmop{kpz}}(T)$ to the linear PDE $\LL h = u^{\diamond 2} + \theta$, $h (0) = Y(0)+h_0$, solves the KPZ equation driven by $\Y$ and with initial condition $Y(0)+h_0$.
\end{theorem}

\subsection{The Rough Heat Equation}
\label{sec:rhe}
Let us now solve
\begin{equation}\label{eq:stochastic heat equation}
   \LL w = w \diamond \xi, \hspace{2em} w(0) = w_0,
\end{equation}
where $\diamond$ denotes again a suitably renormalized product and where $\xi$
is the space-time white noise. We replace for the moment $\xi$ by some $\theta \in C (\R, C^{\infty} (\T))$ and let $\Y= \Theta_{\tmop{kpz}}
(\theta, c^{\zzone}, c^{\zzfour})$ for some $c^{\zzone}, c^{\zzfour} \in \R$ be as in Definition~\ref{def:kpz rough distribution}. Then we know from the Cole--Hopf transform that the unique
classical solution to the heat equation $\LL w = w (\theta - c^\zzone)$, $h(0) = e^{h_0}$ is given by $w = \exp (h)$, where $h$ solves the KPZ equation with initial condition $h_0$. Indeed, we have
\[
   \LL w = \exp (h) \LL h - \exp (h) (\mathD h)^2 = \exp (h) ((\mathD h)^2 - c^\zzone + \theta) - \exp (h) (\mathD h)^2 = w (\theta - c^\zzone).
\]
Extending the Cole--Hopf transform to the paracontrolled setting, we suspect
that the solution will be of the form $w = e^{Y +
Y^{\zzone} + 2 Y^{\zztwo} + h^P}$ where $h^P = h' \mpara P + h^{\sharp}$
with $h' \in \LL^{\alpha}$ and $h^{\sharp} \in \LL^{2 \alpha + 1}$, or in
other words
\[
   w = e^{Y + Y^{\zzone} + 2 Y^{\zztwo}} w^P, \hspace{2em} w^P = w' \mpara P + w^{\sharp},
\]
with $w' \in \LL^{\alpha}$ and $w^{\sharp} \in \LL^{2 \alpha + 1}$. So we call $w$ paracontrolled and write $w \in \CD_{\tmop{rhe}, \Y} =   \DD_{\tmop{rhe}}$ if it is of this form.

Our first objective is to make sense of the term $w \diamond \theta$ for $w \in
\DD_{\tmop{rhe}}$. In the renormalized smooth case $\Y= \Theta_{\tmop{kpz}}(\theta, c^{\zzone}, c^{\zzfour})$ for $\theta \in C (\R_+, C^{\infty} (\T))$ and $c^{\zzone}, c^{\zzfour} \in \R$  we have for $w \in \DD_{\tmop{rhe}}$
\begin{align*}
  \LL w & = e^{Y + Y^{\zzone} + 2 Y^{\zztwo}} \left[ \LL w^P + \LL (Y + Y^{\zzone} + 2 Y^{\zztwo}) w^P - (\mathD Y + \mathD Y^{\zzone} + 2 \mathD Y^{\zztwo})^2 w^P \right] \\
  &\qquad - 2 e^{Y + Y^{\zzone} + 2 Y^{\zztwo}} \mathD (Y + Y^{\zzone} + 2 Y^{\zztwo}) \mathD w^P .
\end{align*}
Taking into account that $\LL Y =  \theta$, $\LL Y^{\zzone} = (\mathD Y)^2 -
c^{\zzone}$, and $\LL Y^{\zztwo} = \mathD Y \mathD Y^{\zzone}$, we get
\begin{align}\label{eq:heat equation product definition} \nonumber
   w (\theta - c^{\zzone}) = \LL w - e^{Y + Y^{\zzone} + 2 Y^{\zztwo}} & \Big[- [4 (\LL Y^\zzthreereso + \mathD Y^\zztwo \para \mathD Y + \mathD Y^\zztwo \lpara \mathD Y) + \LL Y^\zzfour ] w^p \\ \nonumber
   &\qquad + [4 \mathD Y^{\zzone} \mathD Y^{\zztwo} + (2 \mathD Y^{\zztwo})^2] w^P \\
   &\qquad + \LL w^P - 2 \mathD (Y + Y^{\zzone} + 2 Y^{\zztwo}) \mathD w^P \Big].
\end{align}
At this point we can simply take the right hand side as the
\tmtextit{definition} of $w \diamond \theta = w (\theta - c^{\zzone})$. For smooth $\theta$ and $\Y= \Theta_{\tmop{rbe}} (\theta, c^{\zzone}, c^{\zzfour})$, we have
just seen that $w \diamond \theta$ is nothing but the renormalized pointwise
product $w (\theta - c^{\zzone})$. Moreover, the operation $w \diamond \theta$
is continuous:

\begin{lemma}
  The distribution $w \diamond \theta$ is well defined for all $\Y =  \Theta_{\tmop{rbe}} (\theta, c^{\zzone}, c^{\zzfour})$ with $\theta \in \LL C^{\alpha / 2}_{\mathrm{loc}} (\R, C^{\infty} (\T))$, $c^\zzone, c^\zzfour \in \R$, $w' \in \LL^\alpha$ and $w^\sharp \in \LL^{2\alpha}$. Moreover, $w \diamond \theta$ depends continuously on $(\Y, w', w^{\sharp}) \in \Ykpz \times \LL^\alpha \times \LL^{2\alpha}$.
\end{lemma}

\begin{proof}
  Consider~(\ref{eq:heat equation product definition}). Inside the big square bracket on the right hand side, there are three terms which are not immediately well defined: $\mathD Y \mathD w^P$, $\LL Y^\zzthreereso w^P$, and $\LL Y^\zzfour w^P$. For the first one, we can use the fact that $w^P = w' \mpara P + w^\sharp$ and that $\Y$ ``contains'' $\mathD P \reso \mathD Y$, just as we did when solving the KPZ equation. The product $\Delta Y^\zzthreereso w^P$ is well defined because $\Delta Y^\zzthreereso \in C \CC^{2\alpha-1}$ and $w^P \in C \CC^\alpha$. The product $\partial_t Y^\zzthreereso w^P$ can be defined using Young integration: we have $Y^\zzthreereso \in C^{\alpha+1/2}_{\mathrm{loc}} L^\infty$ and $w^P \in C^{\alpha/2}_{\mathrm{loc}} L^\infty$ and $3\alpha/2 + 1/2 > 1$. The same arguments show that $\LL Y^\zzfour w^P$ is well defined.
  
  So taking into account that $w^P = w' \mpara P + w^{\sharp}$ and using
  Lemma~\ref{lem:modified paraproduct commutators} as well as Bony's
  commutator estimate Lemma~\ref{lem:bony commutator}, we see that the term in
  the square brackets is of the form
  \[
      \partial_t (w^{\sharp} + Y^\zzthreereso + Y^\zzfour) + (w' - 4 w^P \mathD Y^{\zztwo} - 2 \mathD w^P) \para \mathD Y + C \CC^{2 \alpha - 1}.
  \]
  The second term is paracontrolled by $\mathD Y$, and since $e^{Y +
  Y^{\zzone} + 2 Y^{\zztwo}}$ is paracontrolled by $Y$, the product between
  the two is well defined as long as $Y \reso \mathD Y$ is given. But of
  course $Y \reso \mathD Y = \mathD (Y \reso Y) / 2$. Finally, the product $e^{Y +
  Y^{\zzone} + 2 Y^{\zztwo}} \partial_t (w^{\sharp} + Y^\zzthreereso + Y^\zzfour)$ can be defined using Young
  integration as described above.
\end{proof}

Besides it agreeing with the renormalized pointwise product in the smooth case and being continuous in suitable paracontrolled spaces, there is another indication that our definition of $w \diamond \theta$ is, despite its unusual appearance, quite useful.

\begin{proposition}
  Let $\xi$ be a space-time white noise on $\R \times \T$, and let $(\mathcal{F}_t)_{t \geqslant 0}$ be its natural filtration,
  \[
     \mathcal{F}_t^{\reso} = \sigma \left( \xi( \1_{(s_1, s_2]} \psi : -\infty < s_1 < s_2 \leqslant t, \psi \in C^{\infty} (\T,\R) \right),
     \hspace{2em} t \geqslant 0, \]
  and $(\mathcal{F}_t)_{t \geqslant 0}$ is its right-continuous completion. Let $\Y \in \mathcal{Y}_{\tmop{kpz}}$ be constructed from $\xi$ as described in Theorem~\ref{thm:kpz data}, and let $(w_t)_{t \geqslant 0}$ be an adapted process with values in $\LL^{\alpha + 1}$, such that almost surely $w \in \CD_{\tmop{kpz}, \Y}$ with adapted $w'$ and $w^{\sharp}$. Let $\psi : \R_+ \times \T \rightarrow \R$ be a
  smooth test function of compact support. Then
  \[
     \int_{\R_+ \times \T} \psi (s, x) (w \diamond \xi) (s, x) \mathd x \mathd s = \int_{\R_+ \times \T} (\psi (s, x) w (s, x)) \xi (s, x) \mathd x \mathd s,
  \]
  where the left hand side is to be interpreted using our definition of $w
  \diamond \xi$, and the right hand side denotes the It{\^o} integral.
\end{proposition}

\begin{proof}
  Let $\varphi$ be a smooth compactly supported even function with $\varphi(0) = 1$ and set $\xi_{\varepsilon} = \varphi(\varepsilon \mathD) \xi$. Theorem~\ref{thm:kpz data} states that for $c_\varepsilon^\zzone = \frac{1}{4 \pi \varepsilon} \int_\R \varphi^2(x) \dd x$ and an appropriate choice of $c^{\zzfour}_{\varepsilon}$, we have that $\Theta_{\tmop{kpz}} (\xi_{\varepsilon}, c^{\zzone}_{\varepsilon}, c^{\zzfour}_{\varepsilon})$ converges in probability to $\Y$ as $\varepsilon \rightarrow 0$. Moreover, by the stochastic Fubini theorem ({\cite{daprato_redbook_1992}},
  Theorem~4.33) and the continuity properties of the It{\^o} integral,
  \[ \int_{\R_+ \times \T} (\psi (s, x) w (s, x)) \xi (s, x)
     \mathd x \mathd s = \lim_{\varepsilon \rightarrow 0} \int_{\R_+
     \times \T} (\psi (s, x) w_{\varepsilon} (s, x))
     \xi_{\varepsilon} (s, x) \mathd x \mathd s, \]
  where we set
  \[
     w_{\varepsilon} = e^{Y_{\varepsilon} + Y_{\varepsilon}^{\zzone} + 2 Y^{\zztwo}_{\varepsilon}} (w' \mpara P_{\varepsilon} + w^{\sharp}),
  \]
  and where the $Y_{\varepsilon}^{\tau}$ and $P_{\varepsilon}$ are constructed from $\varphi$, $c^{\zzone}_{\varepsilon}$, and $c^{\zzfour}_{\varepsilon}$ as in Definition~\ref{def:kpz rough
  distribution}. But by definition of $\Y$ and the continuity
  properties of our product operator, also $w_{\varepsilon} \diamond
  \xi_{\varepsilon}$ converges in probability to $w \diamond \xi$ in the sense
  of distributions on $\R_+ \times \T$. It therefore
  suffices to show that
  \[ \int_0^t (w_{\varepsilon} \diamond \xi_{\varepsilon}) (s, x) \mathd s =
     \int_0^t w_{\varepsilon} (s, x) \xi_{\varepsilon} (s, x) \mathd s =
     \int_0^t w_{\varepsilon} (s, x) \mathd_s B_{\varepsilon} (s, x) \]
  for all $(t, x) \in \R_+ \times \T$, where we wrote
  $B_{\varepsilon} (t, x) = \int_0^t \xi_{\varepsilon} (s, x) \mathd s$ which,
  as a function of $t$, is a Brownian motion with covariance $(2\pi)^{-1} \sum_k \varphi(\varepsilon k)^2 = 2 c^\zzone_\varepsilon + o(1)$ for every $x \in \T$. The first term on the right hand side of~(\ref{eq:heat equation product definition}) is $\LL w_{\varepsilon}$. Let us therefore apply It{\^o}'s formula: Writing $w^{P_\varepsilon} = w' \mpara P_{\varepsilon} + w^{\sharp}$, we have for fixed $x \in \T$
  \begin{align*}
    \mathd_t  w_{\varepsilon} - \Delta w_{\varepsilon} \mathd t & = e^{Y_\varepsilon + Y_\varepsilon^{\zzone} + 2 Y_\varepsilon^{\zztwo}} \Big[ \mathd_t w^{P_\varepsilon} - \Delta w^{P_\varepsilon}\mathd t \\
    &\hspace{65pt} + w^{P_\varepsilon} (\mathd_t(Y_\varepsilon + Y_\varepsilon^{\zzone} + 2 Y_\varepsilon^{\zztwo}) - \Delta(Y_\varepsilon + Y_\varepsilon^{\zzone} + 2 Y_\varepsilon^{\zztwo}) \mathd t) \\
    &\hspace{65pt} - (\mathD Y_\varepsilon + \mathD Y_\varepsilon^{\zzone} + 2 \mathD Y_\varepsilon^{\zztwo})^2 w^{P_\varepsilon} \mathd t - 2 \mathD (Y_\varepsilon + Y_\varepsilon^{\zzone} + 2 Y_\varepsilon^{\zztwo}) \mathD w^{P_\varepsilon} \mathd t \Big] \\
  &\qquad + \frac{1}{2} w_\varepsilon \mathd\langle Y_\varepsilon(\cdot, x) \rangle_t,
  \end{align*}
  where we used that all terms except $Y_\varepsilon$ have zero quadratic variation. But now $\mathd\langle Y_\varepsilon(\cdot, x) \rangle_t = \mathd \langle B_\varepsilon(\cdot, x) \rangle_t = (2 c^{\zzone}_\varepsilon + o(1)) \mathd t$, and therefore the It{\^o} correction term cancels in the limit with
  the term $- w_{\varepsilon} c_{\varepsilon}^{\zzone} \mathd t$ that
  we get from $w_{\varepsilon} (\mathd_t Y_\varepsilon^\zzone - \Delta Y_\varepsilon^\zzone \mathd t - (\mathD Y_\varepsilon)^2 \dd t)$. So if we compare the expression we obtained with the right hand side of~\eqref{eq:heat equation product definition}, we see that
  \[
     w_\varepsilon \diamond \xi_\varepsilon \dd t = w_{\varepsilon} (\mathd_t Y_\varepsilon - \Delta Y_\varepsilon \mathd t) + o(1) = w_\varepsilon \mathd_t B_{\varepsilon} + o(1),
  \]
  and the proof is complete.
\end{proof}

With our definition~(\ref{eq:heat equation product definition}) of $w \diamond
\theta$, the function $w \in \DD_{\tmop{rhe}}$ solves the rough heat
equation~(\ref{eq:stochastic heat equation}) if and only if
\begin{align*}
   e^{Y + Y^{\zzone} + 2 Y^{\zztwo}} & \Big[- [4 (\LL Y^\zzthreereso + \mathD Y^\zztwo \para \mathD Y + \mathD Y^\zztwo \lpara \mathD Y) + \LL Y^\zzfour + 4 \mathD Y^{\zzone} \mathD Y^{\zztwo} + (2 \mathD Y^{\zztwo})^2] w^P \\
   &\qquad + \LL w^P - 2 \mathD (Y + Y^{\zzone} + 2 Y^{\zztwo}) \mathD w^P \Big] = 0.
\end{align*}
Since $e^{Y + Y^{\zzone} + 2 Y^{\zztwo}}$ is a strictly positive
continuous function, this is only possible if the term in the square brackets
is constantly equal to zero, that is if
\begin{align}\label{eq:rhe reinterpretation} \nonumber
   \LL w^P & = [4 (\LL Y^\zzthreereso + \mathD Y^\zztwo \para \mathD Y + \mathD Y^\zztwo \lpara \mathD Y) + \LL Y^\zzfour + 4 \mathD Y^{\zzone} \mathD Y^{\zztwo} + (2 \mathD Y^{\zztwo})^2] w^p \\
   &\quad + 2 \mathD (Y + Y^{\zzone} + 2 Y^{\zztwo}) \mathD w^P.
\end{align}

We deduce the following equation for $w^{\sharp}$:
\[
   \LL w^{\sharp} = (2 \mathD w^P + 4 w^P \mathD Y^{\zztwo} - w') \para \mathD Y + (4 \LL Y^\zzthreereso + \LL Y^\zzfour) w^P + \LL \LL^{2 \alpha + 1},
\]
where we used again Bony's commutator estimate Lemma~\ref{lem:bony commutator}. Thus, we should choose $w' = 2
\mathD w^P + 4 \mathD Y^{\zztwo} w^P$. Now we can use similar arguments as for Burgers and KPZ equation to obtain the existence and uniqueness of
solutions to the rough heat equation for every $\Y \in
\mathcal{Y}_{\tmop{kpz}}$ with initial condition $e^{Y_0} w_0$ with $w_0 \in \CC^{2\alpha}$. Since the equation is linear, here we even obtain the global existence and uniqueness of solutions. Moreover, as limit of nonnegative functions $w$ is nonnegative whenever the initial condition $w_0$ is nonnegative.

\begin{theorem}\label{thm:she smooth initial}
  For all $(\Y, w_0) \in \mathcal{X}_{\tmop{kpz}} \times \CC^{2\alpha+1}$ and all $T>0$ there is a unique solution $(w,w',w^\sharp) \in \CD_{\mathrm{kpz}}(T)$ to the rough heat equation
  \begin{equation}\label{eq:she smooth initial}
     \LL w = w \diamond \theta, \qquad w(0) = e^{Y(0)} w_0,
  \end{equation}
  on the interval $[0,T]$. The solution is nonnegative whenever $w_0$ is nonnegative. If $\Y$ is sampled from the space-time white noise $\xi$ as described in Theorem~\ref{thm:kpz data}, then $w$ is almost surely equal to the It{\^o} solution of~\eqref{eq:she smooth initial}.
\end{theorem}

\subsection{Relation between KPZ and the Rough Heat Equation}

\paragraph{From KPZ to heat equation.}

Let now $h \in \DD_{\tmop{kpz}}$ and set $w = e^h$. Then $w = e^{Y +
Y^{\zzone} + 2 Y^{\zztwo}} w^P$ with $w^P = e^{h^P}$. To see that $w \in
\CD_{\tmop{rhe}}$, we need the following lemma:

\begin{lemma}\label{lem:paralinearization in time}
  Let $F \in C^2 (\R,\R)$, $h' \in \LL^{\alpha}$, $P \in \LL^{\alpha + 1}$ with $\LL P \in C \CC^{\alpha - 1}$, and $h^{\sharp} \in \LL^{2 \alpha + 1}$. Write $h^P = h' \mpara P + h^{\sharp} \in \LL^{\alpha + 1}$. Then we have for all $\varepsilon > 0$
  \[
     F (h^P) - (F' (h^P) h') \mpara P \in \LL^{2 \alpha + 1 - \varepsilon}.
  \]
\end{lemma}

\begin{proof}
  It is easy to see that the difference is in $C \CC^{2\alpha+1}$, the difficulty is to get the right temporal regularity. 
  If we apply the
  heat operator to the difference and use Lemma~\ref{lem:modified paraproduct
  commutators} together with Lemma~\ref{lem:bony commutator}, we get
  \begin{align*}
     \LL (F (h^P) - (F' (h^P) h') \mpara P) & = F' (h^P) \LL (h' \mpara P + h^{\sharp}) - F'' (h^P) (\mathD h^P)^2 \\
     &\qquad - (F' (h^P) h') \mpara \LL P + C \CC^{2 \alpha - 1} \\
     & = F' (h^P)\partial_t h^\sharp + C \CC^{2 \alpha - 1}.
  \end{align*}
  So an application of the Schauder estimates for $\LL$ shows that
  \[
     (F (h^P) - (F' (h^P) h') \mpara P - \int_0^\cdot P_{\cdot - s} \{F' (h^P(s))\partial_s h^\sharp(s)\} \mathd s \in \LL^{2\alpha+1}.
  \]
  Since $F'(h^P) \in C^{(\alpha+1)/2}L^\infty$ and $h^\sharp \in C^{(2\alpha+1)/2} L^\infty$, the Young integral
  \[
     I(t) = \int_0^t F' (h^P(s)) \mathd h^\sharp(s)
  \]
  is well defined and in $C^{(2\alpha+1)/2} L^\infty$. We can therefore apply Theorem~1 of \cite{Gubinelli2006} to see that $\int_0^\cdot P_{\cdot -s} \mathd I(s) \in C^{(2\alpha+1-\varepsilon)/2}L^\infty$ for all $\varepsilon>0$, and the proof is complete.
\end{proof}

As a consequence of this lemma we have $e^h \in \CD_{\tmop{rhe}}^\delta$ for every $h
\in \CD_{\tmop{kpz}}$ and $\delta < \alpha$, with derivative $e^{h^P} h' \in \LL^{\alpha}$.

\begin{lemma}
  Let $\Y \in \mathcal{Y}_{\tmop{kpz}}$, $T>0$, and let $h \in \DD_{\tmop{kpz},\Y}$ solve the KPZ equation on $[0,T]$ with initial condition $h(0) = h_0$, and set $w = e^h \in \CD^{\alpha-\varepsilon}_{\tmop{rhe}, \Y}$ with derivative $e^{h^P} h'$. Then $w$ solves the rough heat equation on $[0,T]$ with initial condition $w(0) = e^{h_0}$.
\end{lemma}

\begin{proof}
  Recall that
  \[ \LL h^P = \LL (h - Y - Y^{\zzone} - 2 Y^{\zztwo}) = (\mathD h)^{\diamond
     2} - ((\mathD Y)^2 - c^{\zzone}) - 2 \mathD Y \mathD Y^{\zzone} . \]
  From the chain rule for the heat operator we obtain
  \[ \LL w^P = w^P \LL h^P - w^P (\mathD h^P)^2 = w^P ((\mathD h)^{\diamond 2}
     - ((\mathD Y)^2 - c^{\zzone}) - 2 \mathD Y \mathD Y^{\zzone} - (\mathD
     h^P)^2) . \]
  But now
  \begin{align*}
     &(\mathD h)^{\diamond 2} - ((\mathD Y)^2 - c^{\zzone}) - 2 \mathD Y \mathD
     Y^{\zzone} - (\mathD h^P)^2 \\
     &\hspace{30pt} = 4 (\LL Y^\zzthreereso + \mathD Y^\zztwo \para \mathD Y + \mathD Y^\zztwo \lpara Y) + Y^\zzfour + 4 \mathD Y^{\zztwo} \mathD Y^{\zzone} + (2 \mathD Y^{\zztwo})^2 \\
     &\hspace{30pt} \quad + 2 \mathD (Y + Y^{\zzone} + 2 Y^{\zztwo}) \mathD h^P,
  \end{align*}
  and $\mathD h^P = \mathD (\log w^P) = \mathD w^P / w^P$, from where we
  deduce that
  \begin{align*}
     \LL w^P &= [4 (\LL Y^\zzthreereso + \mathD Y^\zztwo \para \mathD Y + \mathD Y^\zztwo \lpara Y) + Y^\zzfour + 4 \mathD Y^{\zztwo} \mathD Y^{\zzone} + (2
     \mathD Y^{\zztwo})^2] w^P \\
     &\quad + 2 \mathD (Y + Y^{\zzone} + 2 Y^{\zztwo}) \mathD w^P.
  \end{align*}
  In other words, $w^P$ satisfies~(\ref{eq:rhe reinterpretation}) and $w$
  solves the rough heat equation.
\end{proof}

\paragraph{From the heat equation to KPZ.}

Conversely, let $w = e^{Y + Y^{\zzone} + 2 Y^{\zztwo}} w^P \in
\DD_{\tmop{rhe}}$ be strictly positive and define $h = \log w$. Then $h = Y +
Y^{\zzone} + 2 Y^{\zztwo} + \log w^P$, and Lemma~\ref{lem:paralinearization in
time} yields $\log w^P = (w' / w^P) \mpara P + \LL^{2 \alpha + 1 -
\varepsilon}$ for all $\varepsilon > 0$, in other words $h$ is paracontrolled
with derivative $w' / w^P$.

\begin{lemma}
  Let $\Y \in \mathcal{Y}_{\tmop{kpz}}$, $T>0$, and let $w \in \DD_{\tmop{rhe},\Y}$ be a strictly positive solution to the rough heat equation with initial condition $w(0) = w_0 > 0$ on the time interval $[0,T]$. Set $h = \log w \in \CD^{\alpha-\varepsilon}_{\tmop{rhe},\Y}$ with derivative $w' / w^P$. Then $h$ solves the KPZ equation on $[0,T]$ with initial condition $h_0 = \log w_0$.
\end{lemma}

\begin{proof}
  We have
  \begin{align*}
     \LL h^P & = \LL (\log w^P) = \frac{1}{w^P} \LL w^P + \frac{1}{(w^P)^2} (\mathD w^P)^2 \\
     & = \frac{1}{w^P} [4 (\LL Y^\zzthreereso + \mathD Y^\zztwo \para \mathD Y + \mathD Y^\zztwo \lpara \mathD Y) + \LL Y^\zzfour + 4 \mathD Y^{\zzone} \mathD Y^{\zztwo} + (2 \mathD Y^{\zztwo})^2] w^p \\
   &\quad + \frac{1}{w^P} 2 \mathD (Y + Y^{\zzone} + 2 Y^{\zztwo}) \mathD w^P + (\mathD h^P)^2,
  \end{align*}
  where for the last term we used that $\mathD w^P = w^P \mathD h^P$. Using this relation once more, we arrive at
  \begin{align*}
     \LL h^P & = 4 (\LL Y^\zzthreereso + \mathD Y^\zztwo \para \mathD Y + \mathD Y^\zztwo \lpara \mathD Y) + \LL Y^\zzfour + 4 \mathD Y^{\zzone} \mathD Y^{\zztwo} + (2 \mathD Y^{\zztwo})^2 \\
     &\quad + 2 \mathD (Y + Y^{\zzone} + 2 Y^{\zztwo}) \mathD h^P + (\mathD h^P)^2,
  \end{align*}
  or in other words $\LL h = (\mathD h)^{\diamond 2} + \theta$.
\end{proof}

Let us summarize our observations:

\begin{theorem}\label{thm:heat to kpz}
     Let $\Y, w_0, h$ be as in Theorem~\ref{thm:she smooth initial} and set
     \[
        T^\ast = \inf\{ t \geqslant 0: \min_{x\in \T} w(t,x)=0\}.
     \]
     Then for all $T<T^\ast$ the function $\log w|_{[0,T]}\in \CD_{\tmop{kpz}}(T)$ solves the KPZ equation~\eqref{eq:kpz smooth initial} driven by $\Y$ and started in $Y(0) + \log w_0$. Conversely, if $h \in \CD_{\tmop{kpz}}(T)$ solves the KPZ equation driven by $\Y$ and started in $Y(0) + h_0$, then $w = \exp(h)$ solves~\eqref{eq:she smooth initial} with $w_0 = \exp(h_0)$.
\end{theorem}

As an immediate consequence we obtain a better characterization of the blow up time for the solution to the KPZ equation:

\begin{corollary}\label{cor:kpz explosion time}
   In the context of Theorem~\ref{thm:kpz smooth initial}, the explosion time $T^\ast$ of the paracontrolled norm of the solution $h$ to the KPZ equation is given by
   \[
      T^\ast = \inf\{ t \geqslant 0: \min_{x \in \T} \exp(h(t,x)) = 0\} = \inf \{t\geqslant 0: \|h(t)\|_{L^\infty} = \infty\}.
   \]
\end{corollary}

\section{Interpretation of the nonlinearity and energy solutions}\label{sec:interpretation}

The purpose of this section is to give a more natural interpretation of the nonlinearity $\mathD u^2$ that appears in the formulation of Burgers equation. We will show that if $\xi$ is a space-time white noise, then $u$ is the only stochastic process for which there exists $u'$ with $(u, u') \in \CD_{\tmop{rbe}}$ almost surely and such that almost surely
\[
   \partial_t u = \Delta u + \lim_{\varepsilon \rightarrow 0} \mathD (\varphi(\varepsilon \mathD) u)^2 + \mathD \xi, \hspace{2em} u (0) = u_0,
\]
whenever $\varphi$ is an even smooth function of compact support with $\varphi(0) = 1$, where the convergence holds in probability in $\CD'(\R_+ \times \T)$.

This is a simple consequence of the following theorem.

\begin{theorem}\label{thm:natural square}
  Let $\X \in \mathcal{X}_{\tmop{rbe}}$ and $\varphi \in C^{\infty}(\R)$ with compact support and $\varphi(0) = 0$. Define for $\varepsilon > 0$ the mollification $X^\tau_\varepsilon = \varphi(\varepsilon \mathD) X^\tau$, $Q_\varepsilon = \varphi(\varepsilon \mathD) Q$, and then
  \[
     \X_{\varepsilon} = \big( X_\varepsilon, I (\mathD [(X_\varepsilon)^2]), I (\mathD [X^{\zzone}_\varepsilon X_\varepsilon]), I( \mathD[X^{\zztwo}_\varepsilon \reso X_\varepsilon]), I(\mathD [( X^{\zzone}_\varepsilon)^2]), Q_\varepsilon \reso X_\varepsilon \big).
  \]
  Let $T>0$ and assume that $\|\X_{\varepsilon} - \X\|_{\Xrbe(T)}$ converges to 0 as $\varepsilon \to 0$. Let $u \in \CD_{\tmop{rbe}} (T)$. Then
  \begin{equation}\label{eq:natural paracontrolled square}
     \lim_{\varepsilon \rightarrow 0} I (\mathD [(\varphi(\varepsilon \mathD) u)^2]) = I (\mathD u^2),
  \end{equation}
  where $\mathD u^2$ on the right hand side denotes the paracontrolled square and where the convergence takes place in $C \left( [0,T], \CD' \right)$.
\end{theorem}

\begin{proof}
  The right hand side of (\ref{eq:natural paracontrolled square}) is defined
  as
  \[ X^{\zzone} + 2 X^{\zztwo} + X^{\zzfour} + 4 X^{\zzthree} + I \left( 2
     \mathD (u^Q X) + 2 \mathD (X^{\zzone} (u^Q + 2 X^{\zztwo})) + \mathD
     ((u^Q + 2 X^{\zztwo})^2) \right), \]
  where
  \begin{equation}
    \label{eq:natural paracontrolled square pr1} u^Q X = u \para X + u \lpara X
    + u^{\sharp} \reso X + ((u' \mpara Q) - u' \para Q) \reso X + C (u', Q,
    X) + u' (Q \reso X) .
  \end{equation}
  On the other side we have $\varphi(\varepsilon \mathD) u = (X_\varepsilon + X^{\zzone}_\varepsilon + 2 X^{\zztwo}_\varepsilon + \varphi(\varepsilon \mathD) u^Q)$, and our assumptions are chosen exactly so that the convergence of all terms of $I(\mathD [(\varphi(\varepsilon \mathD) u)^2])$ is trivial, except that of $I (\mathD [(\varphi(\varepsilon \mathD) (u' \mpara Q)) \reso X_\varepsilon])$ to $I (\mathD [(u' \mpara Q) \reso X])$, where the second resonant product is interpreted in the paracontrolled sense. Since $f \mapsto I (\mathD f)$ is a continuous operation on $C ( [0,T], \CD')$, it suffices to show that $(\varphi(\varepsilon \mathD) (u' \mpara Q)) \reso X_\varepsilon$ converges in $C( [0,T], \CD')$ to $(u' \mpara Q) \reso X$. We decompose 
  \[
     (\varphi(\varepsilon \mathD) u^Q) \reso X_{\varepsilon} = (\varphi(\varepsilon \mathD) [(u' \mpara Q) - (u' \para Q)]) \reso X_{\varepsilon} +  (\varphi(\varepsilon \mathD) (u' \para Q)) \reso X_{\varepsilon}
  \]
  and use the continuity properties of the resonant term in combination with Lemma~\ref{lem:modified paraproduct commutators} to conclude that the first term on the right hand side converges to its ``without $\varepsilon$ counterpart''. It remains to treat $(\varphi(\varepsilon \mathD) (u' \para Q)) \reso X_{\varepsilon}$. But now Lemma 5.3.20 of~\cite{Perkowski2014} states that
  \[
     \| \varphi(\varepsilon \mathD) (u' \para Q) - u' \para Q_{\varepsilon}  \|_{C_T \CC^{2 \alpha - \delta}} \lesssim \varepsilon^{\delta} \| u' \|_{C_T \CC^{\alpha}} \| Q \|_{C_T \CC^{\alpha}}
  \]
  whenever $\delta \leqslant 1$. If we choose $\delta > 0$ small enough so that $3 \alpha - \delta > 1$, this allows us to replace $(\varphi(\varepsilon \mathD) (u' \para Q)) \reso X_{\varepsilon}$ by $(u' \para Q_\varepsilon) \reso X_{\varepsilon}$. Then we get
  \[
     (u' \para Q_{\varepsilon}) \reso X_{\varepsilon} = C (u',Q_{\varepsilon}, X_{\varepsilon}) + u' (Q_{\varepsilon} \reso X_{\varepsilon}),
  \]
  and since $C$ is continuous and by assumption $Q_{\varepsilon} \reso X_{\varepsilon}$ converges to $Q \reso X$ in $C_T \CC^{2 \alpha - 1}$, this completes the proof.
\end{proof}

\begin{corollary}\label{cor:unique energy solution}
   Let $\xi$ be a space-time white noise on $\R \times \T$, let $\X \in \Xrbe$ be constructed from $\xi$ as described in Theorem~\ref{thm:kpz data}, and let $u_0 \in \CC^{2\alpha}$ almost surely. Assume that $u$ is a stochastic process for which there exists $u'$ with $(u, u') \in \CD_{\tmop{rbe}}$ almost surely and such that almost surely
\[
   \partial_t u = \Delta u + \lim_{\varepsilon \rightarrow 0} \mathD (\varphi(\varepsilon \mathD) u)^2 + \mathD \xi, \hspace{2em} u (0) = u_0 + Y(0),
\]
whenever $\varphi$ is an even smooth function of compact support with $\varphi(0) = 1$, where the convergence holds in probability in $\CD'(\R_+ \times \T)$. Then $u$ is almost surely equal to the unique paracontrolled solution to the rough Burgers equation driven by $\X$ and started in $u_0 + Y(0)$.
\end{corollary}

\begin{proof}
   Theorem~\ref{thm:natural square} shows that almost surely
   \[
      \LL u = \mathD u^2 + \mathD \xi, \qquad u(0) = u_0 + Y(0),
   \]
   where the square $\mathD u^2$ is interpreted using the paracontrolled structure of $(u,u')$. This implies that also $(u, 2 u^Q + 4 X^\zztwo)$ is almost surely paracontrolled, where $u^Q = u - X - X^\zzone - 2 X^\zztwo$. However, another application of Theorem~\ref{thm:natural square} shows that almost surely the paracontrolled square $\mathD u^2$ is the same for $(u, u')$ and for $(u, 2u^Q, 4 X^\zztwo)$, and therefore $u$ is the paracontrolled solution of the equation.
\end{proof}

\begin{remark}\label{rmk:energy solution}
   There have been several efforts to find formulations of the Burgers equation based on interpretations like~\eqref{eq:natural paracontrolled square}, see for example~\cite{Assing2002, GoncalvesJara2014, Assing2012, Catuogno2012, GubinelliJara2013}. In~\cite{GoncalvesJara2014} the concept of \emph{energy solutions} is introduced. A process $u$ is an energy solution to the stationary Burgers equation if $u(t,\cdot)$ is ($1/\sqrt{2}$ times) a space white noise for all times $t \geqslant 0$, if it satisfies some ``energy estimates'', and if for any $\varphi$ as in Theorem~\ref{thm:natural square} we have
   \[
      u(t) = P_t u_0 + \lim_{\varepsilon \to 0} I (\mathD [(\varphi(\varepsilon \mathD) u)^2])(t) + I(\xi)(t)
   \]
   almost surely for all $t \geqslant 0$, where the convergence takes place in $L^2$ in $\CD'(\T)$. In~\cite{GoncalvesJara2014} it was also shown that energy solutions are a useful tool for studying the convergence of rescaled particle systems to the stochastic Burgers equation: it can often be shown that a given rescaled particle system is tight, and that any limit point in distribution is an energy solution of Burgers equation. However, so far it is not known whether energy solutions are unique and this appears to be a difficult question. On the other side, Corollary~\ref{cor:unique energy solution} shows that there exists at most one energy solution which is almost surely paracontrolled.
\end{remark}

\begin{remark}
   For simplicity we formulated the result under the assumption $u(0) = Y(0) + u_0$ for some relatively regular perturbation $u_0$. Using the techniques of Section~\ref{sec:singular initial} it is possible to extend this to general $u(0) \in \CC^{-1+\varepsilon}$.
\end{remark}

\section{Singular initial conditions}\label{sec:singular initial}

In this section we extend the results of Section~\ref{sec:burgers} to allow for more general initial conditions. For that purpose we adapt the definition of paracontrolled distributions. To allow for initial conditions that are not necessarily regular perturbations of $X(0)$, we will introduce weighted norms which permit a possible explosion of the paracontrolled norm at 0. Moreover, as was kindly pointed out to us by Khalil Chouk, in the case of the rough heat equation it is very convenient to work on
\[
   \CC_p^\alpha = B_{p,\infty}^\alpha
\]
for $\alpha \in \R$, rather than restricting ourselves to $p = \infty$. For example the Dirac delta has regularity $\CC^{1/p-1}_p$ on $\T$, so by working in $\CC^\alpha_1$ spaces we will be able to start the rough heat equation from $(-\Delta)^{\beta} \delta_0$ for any $\beta < 1/4$. In fact for the same reasons and in the setting of regularity structures an approach based on Besov spaces with finite index was developed recently in~\cite{HairerLabbe2015}.

For $p \in [1,\infty]$ and $\gamma \geqslant 0$ we define $\mathcal{M}^{\gamma}_T L^p = \{ v \colon [0, T] \rightarrow \CD'(\T) : \| v \|_{\mathcal{M}^{\gamma}_T L^p} < \infty \}$, where
\[
   \| v \|_{\mathcal{M}^{\gamma}_T L^p} = \sup_{t \in [0, T]} \{ \| t^{\gamma} v(t) \|_{L^p} \} .
\]
In particular $\| v \|_{\mathcal{M}^{\gamma}_T L^p} = \| v \|_{C_T L^p}$ for $\gamma=0$. If further $\alpha \in (0,
2)$ and $T > 0$ we define the norm
\[
   \| f \|_{\LL^{\gamma, \alpha}_p(T)} = \max \big\{ \| t \mapsto t^{\gamma} f(t) \|_{C^{\alpha / 2}_T L^{p}}, \| f \|_{\mathcal{M}^{\gamma}_T \CC_p^{\alpha}} \big\}
\]
and the space $\LL^{\gamma, \alpha}_p(T) = \{ f \colon [0, T] \rightarrow
\CD': \| f \|_{\LL^{\gamma, \alpha}_p(T)} < \infty \}$ as well as
\[
   \LL^{\gamma,\alpha}_p = \big\{f\colon \R_+\to \CD': f|_{[0,T]} \in \LL^{\gamma,\alpha}_p(T) \text{ for all } T > 0\big\}.
\]
In particular, we have $\LL^{0,\alpha}_\infty(T) = \LL^{\alpha}_T$.

\subsection{Preliminary estimates}\label{sec:general initial preliminary}

Here we translate the results of Section~\ref{sec:preliminaries} to the setting of $\LL_p^{\gamma,\alpha}$ spaces. The estimates involving only one fixed time remain essentially unchanged, but we carefully have to revisit every estimate that involves the modified paraproduct $\mpara$, as well as the Schauder estimates for the Laplacian.

Let us start with the paraproduct estimates:
\begin{lemma}\label{thm:paraproduct exp} For any $\beta \in \mathbb{R}$ and $p \in [1,\infty]$ we have
  \[
    \|f \para g\|_{\CC^\beta_p} \lesssim_{\beta} \min\{ \|f\|_{L^{\infty}} \|g\|_{\CC^\beta_p}, \|f \|_{L^p} \|g\|_\beta \},
  \]
  and for $\alpha < 0$ furthermore
  \[
     \|f \para g\|_{\CC^{\alpha + \beta}_p} \lesssim_{\alpha, \beta} \min\{ \|f\|_{\CC^\alpha_p} \|g\|_{\beta}, \|f\|_{\alpha} \|g\|_{\CC^\beta_p}\}.
  \]
  For $\alpha + \beta > 0$ we have
  \[
    \|f \reso g\|_{\CC^{\alpha + \beta}_p} \lesssim_{\alpha, \beta} \min\{ \|f\|_{\CC_p^\alpha} \|g\|_{\beta}, \|f\|_{\alpha} \|g\|_{\CC_p^\beta} \} .
  \]
\end{lemma}

To adapt the multiplication theorem of paracontrolled distributions to our setting, we first adapt the meta-definition of a paracontrolled distribution:

\begin{definition}
  Let $\beta > 0$, $\alpha \in \R$, and $p \in [1,\infty]$. A distribution $f \in
  \CC^{\alpha}_p$ is called paracontrolled by $u \in \CC^{\alpha}$ and we write $f \in \CD^{\beta}(u)$, if there
  exists $f' \in \CC^{\beta}_p$ such that $f^{\sharp} = f - f' \para u \in
  \CC^{\alpha + \beta}_p$.
\end{definition}

Let us stress that $u \in \CC^\alpha$ is not a typo, we do not weaken the integrability assumptions on the reference distribution $u$.

\begin{theorem}\label{thm:paracontrolled product exp}
   Let $\alpha, \beta \in (1 / 3, 1 / 2)$.
  Let $u \in \CC^{\alpha}$, $v \in \CC^{\alpha - 1}$, and let $(f, f') \in \CD^\alpha_\infty(u)$ and $(g, g') \in \CD^\alpha_{p}(v)$. Assume that $u \reso v \in \CC^{2 \alpha - 1}$ is given as limit of $(u_n \reso v_n)$ in
  $\CC^{2 \alpha - 1}$, where $(u_n)$ and $(v_n)$ are sequences of smooth
  functions that converge to $u$ in $\CC^{\alpha}$ and to $v$ in $\CC^{\alpha - 1}$ respectively. Then $f g$ is well defined and satisfies
  \[
     \| f g - f \para g \|_{\CC^{2 \alpha - 1}_p} \lesssim (\| f' \|_{\beta} \| u \|_{\alpha} + \| f^{\sharp} \|_{\alpha + \beta}) (\| g' \|_{\CC_p^\beta} \| v  \|_{\alpha - 1} + \| g^{\sharp} \|_{\CC_p^{\alpha + \beta - 1}}) + \| f' g' \|_{\CC_p^\beta} \| u \reso v \|_{2 \alpha - 1} .
  \]
  Furthermore, the product is locally Lipschitz continuous: Let $\tilde{u} \in
  \CC^{\alpha}$, $\tilde{v} \in \CC^{\alpha - 1}$ with $\tilde{u} \reso
  \tilde{v} \in \CC^{2 \alpha - 1}$ and let $(\tilde{f}, \tilde{f}')\in \CD^\alpha_\infty(\tilde{u})$ and $(\tilde{g}, \tilde{g}') \in \CD^\alpha_{p}(\tilde{v})$. Assume that $M > 0$ is an upper bound for the
  norms of all distributions under consideration. Then
  \begin{align*}
     \| (f g - f \para g) - (\tilde{f} \tilde{g} - \tilde{f} \para \tilde{g}) \|_{\CC_p^{2 \alpha - 1}} &  \lesssim (1 + M^3) \Big[\| f' - \tilde{f}' \|_{\beta} + \| g' - \tilde{g}' \|_{\CC_p^\beta} \nobracket + \| u - \tilde{u} \|_{\alpha} \\
     &\hspace{65pt} + \| v - \tilde{v} \|_{\alpha - 1} + \| f^{\sharp} - \tilde{f}^{\sharp} \|_{\alpha + \beta}  \\
     &\hspace{65pt} +  \| g^{\sharp} - \tilde{g}^{\sharp} \|_{\CC_p^{\alpha + \beta - 1}} \| u \reso v - \tilde{u} \reso \tilde{v} \|_{2 \alpha - 1}\Big].
  \end{align*}
  If $f' = \tilde{f}' = 0$ or $g' = \tilde{g}' = 0$, then $M^3$ can be replaced by $M^2$.
\end{theorem}

In this setting, the proof is a straightforward adaptation of the arguments in~\cite{gubinelli_paraproducts_2012}, see also~\cite{Allez2015}. For an extension to much more general Besov spaces see~\cite{Proemel2015}.

Next we get to the modified paraproduct, which we recall was defined as
\[
   f \mpara g = \sum_i (Q_i S_{i - 1} f) \Delta_i g \qquad \text{with} \qquad Q_i f (t) = \int_{\R} 2^{- 2 i} \varphi (2^{2 i} (t - s)) f (s \vee 0) \mathd s.
\]
Since for $f \in \CM^\gamma_T L^p$ we have in general $f(0) \notin L^p$, this definition might lead to problems. Recalling that $f \in \LL^{\gamma,\alpha}_\infty(T)$ if and only if $t \mapsto t^\gamma f(t) \in \LL^\gamma_T$, it seems reasonable to replace $f$ by $ f(t) \1_{t > 0}$, so that $\int_{-\infty}^0 \varphi(2^{2i}(t-s))f(s\vee 0) \mathd s = 0$. So in what follows we shall always silently perform this replacement when evaluating $\mpara$ on elements of $\LL^{\gamma,\alpha}_p(T)$ or $\CM^\gamma_T L^p$.

\begin{lemma}\label{lem:modified paraproduct exp}
  For any $\beta \in \mathbb{R}$, $p \in [1,\infty]$, and $\gamma \in [0, 1)$ we have
  \begin{equation}\label{eq:mod para-1 exp}
    t^{\gamma} \|f \mpara g (t) \|_{\CC_p^\beta} \lesssim \min\{\|f\|_{\mathcal{M}^{\gamma}_t L^{p}} \|g (t) \|_{\beta},\|f\|_{\mathcal{M}^\gamma_t L^\infty} \|g(t)\|_{\CC^\beta_p} \}
  \end{equation}
  for all $t > 0$, and for $\alpha < 0$ furthermore
  \begin{equation}\label{eq:mod para-2 exp}
     t^{\gamma} \|f \mpara g (t) \|_{\CC_p^{\alpha + \beta}} \lesssim \min\{ \|f\|_{\mathcal{M}^{\gamma}_t \CC^{\alpha}_p} \|g (t) \|_{\beta}, \|f\|_{\mathcal{M}^{\gamma}_t \CC^{\alpha}} \|g (t) \|_{\CC_p^\beta} \} .
  \end{equation}
\end{lemma}

\begin{lemma}\label{lem:modified paraproduct commutators exp}
   Let $\alpha \in (0, 2)$, $\beta \in \R$, $p \in [1,\infty]$, and $\gamma \in [0, 1)$. Then
   \[
      t^{\gamma} \| (f \mpara g - f \para g) (t) \|_{\CC_p^{\alpha + \beta}} \lesssim \| f \|_{\LL^{\gamma, \alpha}_p(t)} \| g (t) \|_{\beta}
   \]
  for all $t > 0$, as well as
  \[
     t^{\gamma} \left\| \left( \LL (f \mpara g) - f \mpara \left( \LL g \right) \right) (t) \right\|_{\CC_p^{\alpha + \beta - 2}} \lesssim \| f   \|_{\LL^{\gamma, \alpha}_p(t)} \| g (t) \|_{\beta}.
  \]
\end{lemma}

These two lemmas are not very difficult to show, but at least the proof for the second one is slightly technical; the proofs can be found in Appendix~\ref{app:some proofs}.

Recall the definition of the operator $I f (t) = \int_0^t P_{t - s} f (s) \mathd s$. In the singular case we can adapt the Schauder estimates for $I$ as follows:

\begin{lemma}[Schauder estimates]\label{lemma:schauder exp}
   Let $\alpha \in (0, 2)$, $p \in [1,\infty]$, and $\gamma \in [0, 1)$. Then
  \begin{equation}\label{eq:schauder-heat exp}
     \|I f\|_{\LL^{\gamma, \alpha}_p(T)} \lesssim \| f \|_{\mathcal{M}^{\gamma}_T \CC^{\alpha - 2}_p}
  \end{equation}
  for all $T > 0$. If further $\beta \geqslant - \alpha$, then
  \begin{equation}\label{eq:schauder initial contribution exp}
     \| s \mapsto P_s u_0 \|_{\LL^{(\beta + \alpha) / 2, \alpha}_p(T)} \lesssim \| u_0 \|_{\CC_p^{- \beta}} .
  \end{equation}
  For all $\alpha \in \R$, $\gamma \in [0, 1)$, and $T > 0$ we have
  \begin{equation}\label{eq:schauder without time hoelder exp}
     \|I f\|_{\mathcal{M}^{\gamma}_T \CC^{\alpha}_p} \lesssim \| f \|_{\mathcal{M}^{\gamma}_T \CC^{\alpha - 2}_p}.
  \end{equation}
\end{lemma}
To a large extent, the proof is contained in~\cite{gubinelli_paraproducts_2012}. We indicate in Appendix~\ref{app:some proofs} how to adapt the arguments therein to obtain the estimates above.

Just as in the ``non-explosive'' setting, the Schauder estimates allow us to bound $f\mpara g$ in the parabolic space $\LL^{\gamma,\alpha}_p(T)$:

\begin{lemma}\label{lem:paraproduct parabolic space exp}
Let $\alpha \in (0, 2)$, $p \in [1,\infty]$, $\gamma \in [0,1)$ and $\delta > 0$. Let $f \in \LL^{\gamma,\delta}_p(T)$, $g \in C_T \CC^{\alpha}$, and $\LL g \in C_T \CC^{\alpha - 2}$. Then
  \[
     \| f \mpara g \|_{\LL^{\gamma,\alpha}_p(T)} \lesssim \| f \|_{\LL^{\gamma,\delta}_p(T)} ( \| g \|_{C_T \CC^{\alpha}} + \left\| \LL g \right\|_{C_T \CC^{\alpha - 2}} ).
  \]
\end{lemma}
The proof is completely analogous to the one of Lemma~\ref{lem:paraproduct parabolic space}. Finally, we will need a lemma which allows us to pass between different $\LL^{\gamma,\alpha}_p$ spaces, the proof of which can be also found in Appendix~\ref{app:some proofs}.

\begin{lemma}\label{lem:lower explosive regularity}
   Let $\alpha \in (0, 2)$, $\gamma \in (0, 1)$, $p\in [1,\infty]$, $T > 0$, and let $f \in \LL^{\gamma, \alpha}_T$. Then
   \[
      \| f \|_{\LL^{\gamma - \varepsilon / 2, \alpha - \varepsilon}_p(T)} \lesssim \| f \|_{\LL^{\gamma, \alpha}_p(T)} . \]
  for all $\varepsilon \in [0, \alpha \wedge 2 \gamma)$.
\end{lemma}

\subsection{Burgers and KPZ equation with singular initial conditions}\label{sec:rbe singular}

Let us indicate how to modify our arguments to solve Burgers equation with
initial condition $u_0 \in \CC^{- \beta}$ for arbitrary $\beta < 1$. Throughout this section we fix $\alpha \in (1/3,1/2)$ and $\beta \in (1-\alpha,2\alpha)$. For $\varepsilon \geqslant 0$ we write
\[
    \gamma_{\varepsilon} = \frac{\beta + \varepsilon}{2}.
\]

\begin{definition}\label{def:paracontrolled singular}
  Let $\X \in \mathcal{X}_{\tmop{rbe}}$ and $\delta > 1 - 2 \alpha$.
  We define the space $\CD^{\tmop{exp},\delta}_{\tmop{rbe}} = \DD^{\tmop{exp},\delta}_{\tmop{rbe},
  \X}$ of distributions paracontrolled by $\X$ as the set of
  all $(u,u',u^\sharp) \in C \CC^{\beta} \times \LL^{\gamma_\delta,\delta}_\infty \times \LL^{\gamma_{\alpha+\delta},(\alpha + \delta)}_\infty$ such that
  \[
     u = X + X^{\zzone} + 2 X^{\zztwo} + u' \mpara Q + u^{\sharp}.
  \]
  For $\delta = \alpha$ we will usually write $\DD^{\tmop{exp}}_{\tmop{rbe}} =
  \CD^{\tmop{exp},\alpha}_{\tmop{rbe}}$. For $T > 0$ we set
  $\CD^{\tmop{exp},\delta}_{\tmop{rbe}} (T) = \CD^{\tmop{exp},\delta}_{\tmop{rbe}} |_{[0, T]}$, and
  we define
  \[ \| u \|_{\DD^{\tmop{exp},\delta}_{\tmop{rbe}} (T)} = \| u' \|_{\LL^{\gamma_\delta,\delta}_\infty(T)} + \|
     u^{\sharp} \|_{\LL_\infty^{\gamma_{\alpha+\delta},(\alpha + \delta)}(T)} . \]
  We will often use the notation $u^Q = u' \mpara Q + u^{\sharp}$, and we will also write $(u',u^\sharp) \in \CD^{\tmop{exp}}_{\mathrm{rbe}}$ or $u \in \CD^{\tmop{exp}}_{\mathrm{rbe}}$. We call $u'$ the \emph{derivative} and $u^\sharp$ the \emph{remainder}.
\end{definition}

The $\tmop{exp}$ in $\CD^{\tmop{exp}}_{\tmop{rbe}}$ stands for ``explosive'', by which we mean that the paracontrolled norm of $u(t)$ is allowed to blow up as $t$ approaches 0. Throughout this section we work under the following assumption:

\begin{assumption}[T,M] Assume that $\theta, \tilde{\theta} \in \LL C (\R, C^{\infty} (\T))$ and $u_0, \tilde{u}_0 \in \CC^{-\beta}$, and that $u$ is the unique global in time solution to the Burgers equation
  \begin{equation}\label{eq:rbe general initial}
     \LL u = \mathD u^2 + \mathD \xi, \hspace{2em} u (0) = u_0.
  \end{equation}
  We define $\X= \Theta_{\tmop{rbe}} (\theta)$, $u^Q = u - X -
  X^{\zzone} - 2 X^{\zztwo}$, and $u' = 2 u^Q + 4 X^{\zztwo}$, and we set
  $u^{\sharp} = u^Q - u' \mpara Q$. Similarly we define $\tilde{\X}, \tilde{u}, \tilde{u}^Q, \tilde{u}', \tilde{u}^{\sharp}$. Finally we assume that $T,M>0$ are such that
  \[
     \max \big\{ \|u_0\|_{-\beta}, \|\tilde{u}_0\|_{-\beta}, \|\X\|_{\mathcal{X}_{\tmop{rbe}} (T)}, \| \tilde{\X} \|_{\mathcal{X}_{\tmop{rbe}} (T)} \big\} \leqslant M.
  \]
\end{assumption}

From the paracontrolled ansatz we derive the following equation for
$u^{\sharp}$:
\[ \LL u^{\sharp} = \mathD u^2 - \mathD X^2 - 2 \mathD X^{\zzone} X - \LL (u'
   \mpara Q), \hspace{2em} u^{\sharp} (0) = u_0 - X (0) - X^{\zzone} (0) -
   2 X^{\zztwo} (0), \]
where for the initial condition we used that $u' \mpara Q (0) = 0$. Using
similar arguments as for Lemma~\ref{lem:burgers a priori}, we deduce the
following result.

\begin{lemma}\label{lem:u sharp bound exp}
  Under Assumption (T,M) we have
  \[ \|(\LL u^{\sharp} - \LL X^{\zzfour} - \LL X^\zzthreereso) (t) \|_{2 \alpha - 2} \lesssim t^{-
     \gamma_\beta}  (1 + M^2) \big( 1
     + \| u \|_{\CD^{\tmop{exp}}_{\tmop{rbe}} (T)}^2 \big) . \]
  for all $t \in (0, T]$. If further also $\| u \|_{\CD^{\tmop{exp}}_{\tmop{rbe}, \X} (T)}, \| \tilde{u} \|_{\CD^{\tmop{exp}}_{\tmop{rbe}, \tilde{\X}} (T)} \leqslant M$, then
  \begin{align}\label{eq:explosive a priori difference} 
      &\|((\LL u^{\sharp} - \LL X^{\zzfour} - \LL X^\zzthreereso) - (\LL \tilde{u}^{\sharp} - \LL \tilde{X}^{\zzfour} - \LL \tilde X^\zzthreereso)) (t) \|_{2 \alpha - 2} \\ \nonumber
      &\hspace{50pt} \lesssim t^{- \gamma_\beta} (1 + M^2) \big( d_{\mathcal{X}_{\tmop{rbe}} (T)} (\X, \tilde{\X})  + \| u' - \tilde{u}' \|_{\LL^{\gamma_{\alpha}, \alpha}_\infty(T)} + \| u^{\sharp} - \tilde{u}^{\sharp} \|_{\LL^{\gamma_{2 \alpha, 2 \alpha}}_\infty(T)} \big).
  \end{align}
\end{lemma}

\begin{proof}[Sketch of Proof]
   We use the same decomposition for $(\LL u^{\sharp} - \LL X^{\zzfour} - \LL X^\zzthreereso) (t)$ as in the proof of Lemma~\ref{lem:burgers a priori}. The most tricky term to bound is
   \[
      \| \mathD (u^Q(t))^2 \|_{2\alpha -2} \lesssim \| (u^Q(t))^2 \|_{2\alpha-1} \lesssim \| (u^Q(t)) \|^2_{L^\infty} \lesssim (t^{-\gamma_0} \|u^Q\|_{\mathcal{M}^{\gamma_0}_T L^\infty})^2.
   \]
   Now $2\gamma_0 = \gamma_\beta$ and Lemma~\ref{lem:modified paraproduct exp} and Lemma~\ref{lem:lower explosive regularity} yield
   \begin{align*}
      \|u^Q\|_{\mathcal{M}^{\gamma_0}_T L^\infty} & \leqslant \|u' \mpara Q \|_{\mathcal{M}^{\gamma_0}_T L^\infty} + \| u^\sharp \|_{\mathcal{M}^{\gamma_0}_T L^\infty} \lesssim \| u' \|_{\CM^{\gamma_0}_T L^\infty} \|Q\|_{C_T \CC^\alpha} + \| u^\sharp \|_{\LL^{\gamma_{2\alpha},2\alpha}_\infty(T)} \\
      & \lesssim (\| u' \|_{\LL^{\gamma_{\alpha},\alpha}_\infty(T)} + \| u^\sharp \|_{\LL^{\gamma_{2\alpha},2\alpha}_\infty(T)}) (1 + \| \X \|_{\mathcal{X}_{\mathrm{rbe}}(T)}).
   \end{align*}
   Every other term in the decomposition of $(\LL u^{\sharp} - \LL X^{\zzfour} - \LL X^\zzthreereso) (t)$ that does not explicitly depend on $u^\sharp$ can be estimated with a factor $t^{-\gamma_\alpha}$ and since $2\gamma_0 = \gamma_\beta$ and $\alpha<\beta$ we get $t^{-\gamma_\alpha} \lesssim_T t^{-2\gamma_0}$ for all $t \in [0,T]$. The last remaining term is then
   \[
       \| \mathD (u^\sharp \reso X(t)) \|_{\CC^{\alpha+\beta-2}} \lesssim t^{-\gamma_{\beta}} \|u^\sharp \|_{\LL^{\gamma_{\beta},\beta}_\infty(T)} \|X(t) \|_{\alpha-1} \lesssim t^{-\gamma_{\beta}} \|u^\sharp \|_{\LL^{\gamma_{2\alpha},2\alpha}_\infty (T)} \| \X \|_{\mathcal{X}_{\mathrm{rbe}}(T)},
   \]
   where we used that $\alpha+\beta>1$ and that $2\alpha > \beta$.
\end{proof}

\begin{corollary}
  Under Assumption (T,M), we have
  \begin{align*}
     \| u \|_{\CD^{\tmop{exp}}_{\tmop{rbe}} (T)} & \lesssim M+ T^{\gamma_{2 \alpha} - \gamma_\beta} (1 + M^2) \big( 1 +\| u \|_{\CD^{\tmop{exp}}_{\tmop{rbe}} (T)}^2 \big)
  \end{align*}
  If further also $\| u \|_{\CD^{\mathrm{exp}}_{\tmop{rbe}, \X} (T)}, \| \tilde{u} \|_{\CD^{\mathrm{exp}}_{\tmop{rbe}, \tilde{\X}} (T)} \leqslant M$, then
  \begin{align*}
     & \|u^{\sharp} - \tilde{u}^{\sharp} \|_{\LL^{\gamma_{2 \alpha}, 2 \alpha}_\infty(T)} + \| u' - \tilde{u}' \|_{\LL^{\gamma_{\alpha}, \alpha}_\infty(T)} \lesssim \| u_0 - \tilde{u}_0 \|_{- \beta} + d_{\mathcal{X}_{\tmop{rbe}}(T)} (\X, \tilde{\X}) \\
     &\hspace{50pt} + T^{\gamma_{2 \alpha} - \gamma_\beta} (1 + M^2) \big(d_{\mathcal{X}_{\tmop{rbe}} (T)} (\X, \tilde{\X}) + \|
     u' - \tilde{u}' \|_{\LL^{\gamma_{\alpha}, \alpha}_\infty(T)} + \| u^{\sharp} - \tilde{u}^{\sharp} \|_{\LL^{\gamma_{2 \alpha, 2 \alpha}}_\infty(T)} \big).
  \end{align*}
\end{corollary}

\begin{proof}
  Lemma~\ref{lem:u sharp bound exp} and Lemma~\ref{lemma:schauder exp} yield
  \[
     \| u^{\sharp}\|_{\LL^{\gamma_{2\alpha},2\alpha}_\infty(T)}  \lesssim \|u_0\|_{-\beta} + \| \X \|_{\mathcal{X}_{\tmop{rbe}} (T)} + T^{\gamma_{2 \alpha} - \gamma_\beta} (1 + M^2) \big( 1 +\| u \|_{\CD^{\tmop{exp}}_{\tmop{rbe}} (T)}^2 \big)
  \]
  and similarly for $u^{\sharp} - \tilde{u}^{\sharp}$. It remains to control $u'$ and $u' - \tilde{u}'$. We get
  \begin{align*}
     \| u' \|_{\LL^{\gamma_{\alpha}, \alpha}_\infty(T)} & \leqslant 2 \| u^Q  \|_{\LL^{\gamma_{\alpha}, \alpha}_\infty(T)} + 4 \|X^{\zztwo}  \|_{\LL^{\gamma_{\alpha}, \alpha}_\infty(T)} \\
     &\lesssim \| u' \mpara Q  \|_{\LL^{\gamma_{\alpha}, \alpha}_\infty(T)} + \| u^{\sharp} \|_{\LL^{\gamma_{2 \alpha}, 2 \alpha}_\infty(T)} + \| \X \|_{\mathcal{X}_{\tmop{rbe}} (T)},
  \end{align*}
  where we used Lemma~\ref{lem:lower explosive regularity} to bound $\| u^{\sharp} \|_{\LL^{\gamma_{\alpha}, \alpha}_\infty(T)} \lesssim \| u^{\sharp} \|_{\LL^{\gamma_{2 \alpha}, 2 \alpha}_\infty(T)}$. Now Lemma~\ref{lem:paraproduct parabolic space exp} and then once more Lemma~\ref{lem:lower explosive regularity} yield
  \begin{align*}
     \| u' \mpara Q \|_{\LL^{\gamma_{\alpha}, \alpha}_\infty(T)} & \lesssim T^{\gamma_{\alpha}-\gamma_{\beta-\alpha}} \| u' \mpara Q \|_{\LL^{\gamma_{\beta-\alpha}, \alpha}_\infty(T)} \\
     & \lesssim T^{\gamma_{\alpha}-\gamma_{\beta-\alpha}}  \| u' \|_{\LL^{\gamma_{\beta-\alpha}, \beta-\alpha}_\infty(T)} \| \X \|_{\mathcal{X}_{\tmop{rbe}} (T)} \\
     & \lesssim     T^{\gamma_{2 \alpha} - \gamma_\beta} \| u' \|_{\LL^{\gamma_{\alpha}, \alpha}_\infty(T)}  \| \X \|_{\mathcal{X}_{\tmop{rbe}} (T)}.
  \end{align*}
  Combining this with our estimate for $\| u^{\sharp} \|_{\LL^{\gamma_{2
  \alpha}, 2 \alpha}_\infty(T)}$, we obtain the bound for $\| u'
  \|_{\LL^{\gamma_{\alpha}, \alpha}_\infty(T)}$. The same arguments allow us to treat
  $\| u' - \tilde{u}' \|_{\LL^{\gamma_{\alpha}, \alpha}_\infty(T)}$.
\end{proof}

The main result of this section now immediately follows. Before we state it, let us recall that the distance $d_{\mathcal{X}_\rbe}(\X, \tilde \X)$ was defined in~\eqref{eq:local Xrbe distance} and define for $u \in \CD^\expp_{\rbe}$ and $c>0$
\[
   \tau^\expp_c(u) = \inf\{ t \geqslant 0: \| u\|_{\CD^{\expp}_\rbe(t)} \geqslant c\}
\]
and then for $u \in \CD^\expp_{\rbe,\X}$ and $\tilde u \in \CD^\expp_{\rbe,\tilde \X}$
\[
   d^{\expp}_{\rbe,c}(u,\tilde u) = \| u' - \tilde u' \|_{\LL^{\gamma_\alpha,\alpha}_\infty(c \wedge \tau_c(u) \wedge \tau_c (\tilde u))} + \| u^\sharp - \tilde u^\sharp \|_{\LL^{\gamma_{2\alpha},2\alpha}_\infty(c \wedge \tau_c(u) \wedge \tau_c (\tilde u))}
\]
and
\[
   d^{\expp}_{\rbe}(u,\tilde u) = \sum_{m=1}^\infty 2^{-m} (1 \wedge (d^{\expp}_{\rbe,m}(u,\tilde u) + |\tau_m(u) - \tilde \tau_m(\tilde u)|)).
\]

\begin{theorem}\label{thm:burgers general initial}
  For every $(\X,u_0) \in \mathcal{X}_{\tmop{rbe}} \times \CC^{-\beta}$ there exists $T^\ast \in (0, \infty]$ such that for all $T<T^\ast$ there is a unique solution $(u,u',u^\sharp) \in \CD^{\tmop{exp}}_{\mathrm{rbe}}(T)$ to the rough Burgers equation~\eqref{eq:rbe general initial} on the interval $[0,T]$. The map that sends $(\X, u_0) \in \Xrbe \times \CC^{-\beta}$ to the solution $(u,u',u^\sharp) \in \CD^{\tmop{exp}}_{\mathrm{rbe}}$ is continuous with respect to the distances $(d_{\Xrbe} + \| \cdot \|_{-\beta}, d_{\rbe})$, and we can choose
  \begin{equation}\label{eq:Tstar burgers general initial}
     T^\ast = \inf\{t \geqslant 0: \|u\|_{C_t\CC^{-\beta}} = \infty \}.
  \end{equation}
\end{theorem}

\begin{proof}
   The proof is essentially the same as the one for Theorem~\ref{thm:burgers smooth initial}. The only difference is in the iteration argument. Assume that we constructed a paracontrolled solution $u$ on $[0,\tau]$ for some $\tau>0$. Then $\tau^{\gamma_{\alpha}}(u^Q(\tau) - u'\mpara Q(\tau)) \in \CC^{2\alpha}$ by assumption, and Lemma~\ref{lem:modified paraproduct commutators exp} shows that
   \begin{equation}\label{eq:burgers general initial pr1}
      \tau^{\gamma_{\alpha}} (u' \mpara Q(\tau) - u' (\tau)\para Q(\tau)) \in \CC^{2\alpha}.
   \end{equation}
   The initial condition for $u^\sharp$ in the iteration on $[\tau, \tau+1]$ is given by $u^Q(\tau) - u'(\tau)\para Q(\tau)$, which as we just saw is in $\CC^{2\alpha}$, and therefore we can use the arguments of Section~\ref{sec:burgers smooth initial} to construct a paracontrolled solution $(\tilde u, \tilde u', \tilde u^\sharp) \in C_{\tilde \tau}\CC^{\alpha-1} \times \LL^\alpha_{\tilde \tau} \times \LL^{2\alpha}_{\tilde \tau}$ on the time interval $[0,\tilde \tau]$ for some $\tilde \tau >0$. Extend now $u$ and $u'$ from $[0,\tau]$ to $[0,\tau + \tilde \tau]$ by setting $u(t) = \tilde u(t-\tau)$ for $t \geqslant \tau$ and similarly for $u'$. Since on $[\tau, \tau + \tilde \tau]$ the function $t \mapsto t^{\gamma_\alpha}$ is infinitely differentiable, we obtain that $(u,u') \in C_{\tau+\tilde\tau} \CC^{\beta} \times \LL^{\gamma_\alpha,\alpha}_{\tau+\tilde \tau}$. Moreover, $u^Q - u'\mpara_{\tau} Q \in \LL^{2\alpha}$ on the interval $[\tau,\tau+\tilde\tau]$, where
  \[
      u'\mpara_{\tau} Q(t) = \sum_j \int_{-\infty}^t 2^{2j} \varphi(2^{2j}(t-s)) S_{j-1} u'(s \vee \tau) \mathd s \Delta_j Q(t).
  \]
  Using the smoothness of $t \mapsto t^{\gamma_{2\alpha}}$ on $[\tau,\tau+ \tilde \tau]$, it therefore suffices to show that $(u'\mpara_{\tau}Q - u'\mpara Q)|_{[\tau,\tau+\tilde \tau]} \in \LL^{2\alpha}$. But we already saw in equation~\eqref{eq:burgers general initial pr1} above that $(u'\mpara_{\tau}Q - u'\mpara Q)(\tau) \in \CC^{2\alpha}$, and for $t \in [\tau,\tau+\tilde \tau]$ we get from the second estimate of Lemma~\ref{lem:modified paraproduct commutators exp} that
  \[
     \LL(u'\mpara_{\tau}Q - u'\mpara Q)(t) = (u' \mpara_{\tau}\mathD X - u'\mpara \mathD X)(t) + t^{-\gamma_\alpha}\CC^{2\alpha-2}.
  \]
  On $[\tau,\tau+\tilde \tau]$ the factor $t^{-\gamma_\alpha}$ poses no problem and therefore it suffices to control the difference between the two modified paraproducts. 
  Now the first estimate of Lemma~\ref{lem:modified paraproduct commutators exp} shows that $(u'\mpara \mathD X - u' \para \mathD X)|_{[\tau,\tau+\tilde \tau]} \in C \CC^{2\alpha-2} $, and from Lemma~\ref{lem:modified paraproduct commutators} it follows that also $(u'\mpara_{\tau} \mathD X - u' \para \mathD X)|_{[\tau,\tau+\tilde \tau]} \in C\CC^{2\alpha-2}$, so that the $\LL^{2\alpha}$ regularity of $u^\sharp|_{[\tau,\tau+\tilde\tau]}$ follows from the Schauder estimates for the heat flow, Lemma~\ref{lemma:schauder}. Uniqueness and continuous dependence on the data can be handled along the same lines.
  
  However, so far the construction only works up to time $\tilde T^\ast = \inf\{t \geqslant 0: \|u\|_{\CD^{\tmop{exp}}_{\mathrm{rbe}}(t)} = \infty\}$, and it remains to show that $\tilde T^\ast = T^\ast$ for $T^\ast$ defined in~\eqref{eq:Tstar burgers general initial}. Assume therefore that $\|u\|_{C_{\tilde T^\ast} \CC^{-\beta}} < \infty$. Then we can solve the equation starting in $u(\tilde T^\ast - \varepsilon)$ for some very small $\varepsilon$ for which we can perform the Picard iteration on an interval $[\tilde T^\ast - \varepsilon, \tilde T^\ast + \varepsilon]$, and in particular the solution $(u,u',u^\sharp)$ restricted to the time interval $[\tilde T^\ast - \varepsilon/2, \tilde T^\ast + \varepsilon/2]$ is in $C([\tilde T^\ast - \varepsilon/2, \tilde T^\ast + \varepsilon/2],\CC^{\alpha-1}) \times \LL^\alpha([\tilde T^\ast - \varepsilon/2, \tilde T^\ast + \varepsilon/2]) \times \LL^{2\alpha}([\tilde T^\ast - \varepsilon/2, \tilde T^\ast + \varepsilon/2])$ (with the natural interpretation for the $\LL^\delta([a,b])$ spaces). Using the uniqueness of the solution we get a contradiction to $\|u\|_{\CD^{\tmop{exp}}_{\mathrm{rbe}}(\tilde T^\ast)} = \infty$, and the proof is complete.
\end{proof}

For the KPZ equation and $Y\in \Ykpz$ we can introduce analogous spaces $\CD^\expp_{\kpz,\Y}$ with distance $d^\expp_\kpz(h,\tilde h)$ and obtain the analogous result:

\begin{theorem}\label{thm:kpz general initial}
  For every $(\Y,h_0) \in \Ykpz \times \CC^{1-\beta}$ there exists $T^\ast \in (0, \infty]$ such that for all $T<T^\ast$ there is a unique solution $(h,h',h^\sharp) \in \CD^{\tmop{exp}}_{\mathrm{kpz}}(T)$ to the rough KPZ equation
  \[
     \LL h = (\mathD h)^{\diamond 2} + \xi, \qquad h(0) = h_0
  \]
   on the interval $[0,T]$. The map that sends $(\Y, h_0) \in \Xrbe \times \CC^{1-\beta}$ to the solution $(h,h',h^\sharp) \in \CD^{\tmop{exp}}_{\mathrm{rbe}}$ is continuous with respect to the distances $(d_{\Ykpz} + \| \cdot\|_{1-\beta}, d_{\kpz})$, and we can choose
  \begin{equation}\label{eq:Tstar burgers general initial}
     T^\ast = \inf\{t \geqslant 0: \|u\|_{C_t\CC^{1-\beta}} = \infty \}.
  \end{equation}
\end{theorem}

\subsection{Heat equation with singular initial conditions}\label{sec:heat singular}

The rough heat equation is linear, and therefore in that case it sometimes turns out to be advantageous to work in $\CC^\alpha_p$ spaces for general $p$ in order to allow for more general initial conditions. More precisely, we will be able to handle initial conditions $w_0 \in \CC^{-\beta}_p$ for arbitrary $\beta < 1/2$ and $p \in [1,\infty]$. Taking $p = 1$, this allows us for example to start in $(-\Delta)^{\gamma} \delta$ for $\gamma < 1/4$, where $\delta$ denotes the Dirac delta.

We fix $p \in [1,\infty]$, $\alpha \in (1/3,1/2)$, and $\beta \in (0,\alpha)$, and we want to solve the paracontrolled equation
\[
   \LL w = w\diamond \xi, \qquad w(0) = w_0,
\]
for $w_0 \in \CC^{-\beta}$. For $\varepsilon \geqslant 0$ let us write
\[
    \gamma_{\varepsilon} = \frac{\beta + \varepsilon}{2}.
\]

\begin{definition}
  Let $\Y \in \mathcal{Y}_{\tmop{kpz}}$ and $\delta \in (1 - 2 \alpha, 1 - \alpha - \beta)$.
  We define the space $\CD^{\tmop{exp},\delta}_{\tmop{rhe}} = \DD^{\tmop{exp},\delta}_{\tmop{rhe},
  \Y}$ of distributions paracontrolled by $\Y$ as the set of
  all $(u,u',u^\sharp) \in C \CC^{\beta}_p \times \LL^{\gamma_\delta,\delta}_p \times \LL^{\gamma_{1+\alpha+\delta},(1+\alpha + \delta)}_p$ such that
  \[
     u = e^{Y + Y^{\zzone} + 2 Y^{\zztwo}} (w' \mpara P + w^{\sharp}).
  \]
  For $T > 0$ we set
  $\CD^{\tmop{exp},\delta}_{\tmop{rhe}} (T) = \CD^{\tmop{exp},\delta}_{\tmop{rhe}} |_{[0, T]}$, and
  we define
  \[ \| w \|_{\DD^{\tmop{exp},\delta}_{\tmop{rhe}} (T)} = \| w' \|_{\LL^{\gamma_\delta,\delta}_p(T)} + \|
     w^{\sharp} \|_{\LL_p^{\gamma_{1+\alpha+\delta},(1+\alpha + \delta)}(T)} . \]
  We will often write $w^P = w' \mpara P + w^{\sharp}$.
\end{definition}
Recall from~\eqref{eq:heat equation product definition} the definition of the renormalized product $w\diamond \xi$ for $\Y = \Theta_{\tmop{kpz}}(\xi, c^\zzone, c^\zzfour)$:
\begin{align*}
      w \diamond \xi = \LL w - e^{Y + Y^{\zzone} + 2 Y^{\zztwo}} & \Big[- [4 (\LL Y^\zzthreereso + \mathD Y^\zztwo \para \mathD Y + \mathD Y^\zztwo \lpara \mathD Y) + \LL Y^\zzfour ] w^p \\ \nonumber
   &\qquad + [4 \mathD Y^{\zzone} \mathD Y^{\zztwo} + (2 \mathD Y^{\zztwo})^2] w^P \\
   &\qquad + \LL w^P - 2 \mathD (Y + Y^{\zzone} + 2 Y^{\zztwo}) \mathD w^P \Big].
\end{align*}
It is then a simple exercise to apply the results of Section~\ref{sec:general initial preliminary} to see that $w\diamond \xi$ depends continuously on $(\Y,w) \in \mathcal{Y}_{kpz} \times \DD^{\tmop{exp},\delta}_{\tmop{rhe}, \Y}$. Moreover, $w \in \DD^{\tmop{exp},\delta}_{\tmop{rhe}, \Y}$ solves
\[
   \LL w = w \diamond \xi, \qquad w(0) = w_0,
\]
if and only if
\begin{align*}
   \LL w^P & = [4 (\LL Y^\zzthreereso + \mathD Y^\zztwo \para \mathD Y + \mathD Y^\zztwo \lpara \mathD Y) + \LL Y^\zzfour + 4 \mathD Y^{\zzone} \mathD Y^{\zztwo} + (2 \mathD Y^{\zztwo})^2] w^p \\
   &\quad + 2 \mathD (Y + Y^{\zzone} + 2 Y^{\zztwo}) \mathD w^P,
\end{align*}
$w^P(0) = w_0 e^{-Y(0) - Y^\zzone(0) - 2 Y^\zztwo(0)}$. Since $w_0 \in \CC_p^{-\beta}$ and $e^{-Y(0) - Y^\zzone(0) - 2 Y^\zztwo(0)} \in \CC^\alpha$ with $\alpha > \beta$, the latter product is well defined and in $\CC^{-\beta}_p$.

We define the distance $d^\expp_\rhe$ analogously to the case of the KPZ or Burgers equation. Solutions to the rough heat construction can now be constructed using the same arguments as in the previous section, so that we end up with the following result:

\begin{theorem}\label{thm:she general initial}
  Let $\delta \in (1-2\alpha, 1-\alpha-\beta)$. For all $(\Y,w_0,T) \in \mathcal{Y}_{\tmop{kpz}} \times \CC^{-\beta}_p \times [0,\infty)$ there is a unique solution $(w,w',w^\sharp) \in \CD^{\tmop{exp,\delta}}_{\mathrm{rhe},\Y}(T)$ to the rough heat equation~\eqref{eq:rbe general initial} on the interval $[0,T]$. The solution depends continuously on $(\Y,w_0)$ with respect to $(d_{\Ykpz} + \| \cdot \|_{\CC^{-\beta}_p}, d^\expp_\rhe)$. Moreover, if there exist $(w_0^n) \subset C^\infty$ with $\| w_0 - w_0^n \|_{-\beta} \to 0$ as $n \to \infty$ and such that $w_0^n \geqslant 0$ for all $n$, then $w(t,x) \geqslant 0$ for all $(t,x)$.
\end{theorem}

\section{Variational representation and global existence of solutions}\label{sec:hjb}

\subsection{KPZ as Hamilton-Jacobi-Bellman equation}\label{sec:HJB1}

Here we show that for every $\Y \in \mathcal{Y}_{\mathrm{rbe}}$ and every $h_0 \in \CC^\beta$ for $\beta > 0$ there are global in time solutions to the KPZ equation. The idea is to interpret the solution as value function of an optimal control problem, and to ``guess'' an expansion of the optimal control. This can be made rigorous in the case of smooth data and allows us to obtain a priori bounds which show that no explosion occurs as we let the smooth data converge to $\Y$.

Let $h$ solve the KPZ equation
\begin{equation}\label{eq:kpz hjb section}
   \LL h = (\mathD h)^2 - c^\zzone + \theta,\qquad h(0) = \bar{h},
\end{equation}
for $\theta \in \LL C^{\alpha/2}_{\mathrm{loc}}(\R, C^\infty)$, $c^\zzone \in \R$, and $\bar{h} \in C^\infty$. Then by the Cole--Hopf transform we have $h = \log w$, where
\begin{equation}\label{eq:rhe hjb section}
   \LL w = w(\theta - c^\zzone), \qquad w(0) = e^{\bar h}.
\end{equation}
We specified in Section~\ref{sec:rhe} how to interpret this equation. But as long as $Y$ (see Definition~\ref{def:kpz rough distribution}) is in $C(\R_+, C^\infty)$, we have the following simpler characterization. The function $w \colon \R_+ \times \T \to \R$ solves~\eqref{eq:rhe hjb section} in the sense of Section~\ref{sec:rhe} if and only if $w = e^Y w^1$, where $w^1$ is a classical solution to
\[
   \LL w^1 = (|\mathD Y|^2 - c^\zzone) w^1 + 2 \mathD Y \mathD w^1, \qquad w^1(0) = e^{\bar h - Y(0)}.
\]
For the rest of this section it will be convenient to reverse time. So fix $T>0$ and define $\arr{\varphi} = \varphi(T-t)$ for all appropriate $\varphi$. Then
\[
   (\partial_t + \Delta) \arr h = -((\mathD \arr h)^2 - c^\zzone) - \arr \theta, \qquad \arr h(T) = \bar{h},
\]
and $\arr w = e^{\arr Y} \arr w^1$ with
\[
   (\partial_t + \Delta + 2 \mathD \arr Y \mathD) \arr w^1 = -(|\mathD \arr Y|^2 - c^\zzone) \arr w^1, \qquad \arr w^1(T) = e^{\bar h - \arr Y(T)}.
\]
The Feynman--Kac formula (\cite{Karatzas1988}, Theorem 5.7.6) shows that for $t \in [0,T]$ we have
\begin{equation}\label{eq:feynman-kac}
   \arr w^1(t,x) = \E_{t,x}\Big[e^{\bar{h}(\gamma_{T}) - \arr Y(T,\gamma_T) + \int_{t}^{T} (|\mathD \arr Y|^2(s,\gamma_s) - c^\zzone) \mathd s}\Big],
\end{equation}
where under $\P_{t,x}$ the process $\gamma$ solves
\[
   \gamma_s = \int_t^s 2 \mathD \arr Y(r, \gamma_r) \dd r + B_s, \quad s \geqslant t, 
\]
with a Brownian motion $B$ with variance 2 (that is $\mathd \langle B \rangle_s = 2 \mathd s$), started in $B_t = x$. An application of Girsanov's theorem gives
\begin{align*}
   \arr w(t,x) &= e^{\arr Y(t,x)} \E_{t,x}\Big[e^{\bar{h}(B_{T}) - \arr Y(T,B_T) + \int_{t}^{T} (|\mathD \arr Y|^2(s,B_s) - c^\zzone) \mathd s} e^{\int_t^T \mathD \arr Y(s,B_s) \dd B_s - \int_t^T |\mathD \arr Y|^2(s,B_s) \dd s}\Big] \\
   & = \E_{t,x}\Big[e^{\bar{h}(B_{T}) - (\arr Y(T,B_T)  - \arr Y(t,x) - \int_t^T \mathD \arr Y(s,B_s) \dd B_s)- c^\zzone(T-t)}\Big].
\end{align*}
Now if we formally apply It\^o's formula, we get
\[
   \arr Y(T,B_T)  - \arr Y(t,x) - \int_t^T \mathD \arr Y(s,B_s) \dd B_s = \int_t^T (\partial_s + \Delta) \arr Y(s,B_s) \dd s = - \int_t^T \arr \theta(s,B_s) \dd s.
\]
Using the temporal continuity of $\arr Y$ (and thus of $\arr W(t,x) = \int_t^T \arr \theta(s,x) \dd s$), it is actually possible to show that the right hand side makes sense as an $L^2(\P_{t,x})$--limit. But we will not need this and simply take the left hand side as the definition of the right hand side. We will however need the following generalization of the Bou\'e--Dupuis~\cite{Boue1998} formula which has been recently established by \"Ust\"unel:
\begin{theorem}[\cite{Ustunel2014}, Theorem~7]\label{thm:boue-dupuis}
  Let $B \colon [0, T] \rightarrow \R^d$ be a Brownian motion with variance $\sigma^2$ and let $F$ be a measurable functional on $C ([0, T] ; \R^d)$ such that $F(B) \in L^2$ and $e^{-F(B)} \in L^2$. Then
  \[
     - \log \E [e^{- F (B)}] = \inf_v \E \Big[ F\Big(B + \int_0^{\cdot} v_s \mathd s\Big) + \frac{1}{2\sigma^2} \int_0^T | v_s|^2 \mathd s \Big],
  \]
  where the infimum runs over all processes $v$ that are progressively measurable with respect to $(\F_t)$, the augmented filtration generated by $B$, and that are such that $\omega \mapsto v_s(\omega)$ is $\F_s$--measurable for Lebesgue--almost all $s \in [0,T]$.
\end{theorem}
\begin{corollary}
   Let $h$ solve~\eqref{eq:kpz hjb section}, let $T>0$, and let $B$ be a Brownian motion with variance $\langle B \rangle_t = 2 t$, started in $B_0 = 0$. Then
    \begin{align}\label{eq:Dupuis-KPZ} \nonumber
       h(T,x) = \log \arr w(0,x) & = \log \E\Big[\exp\Big( \bar{h}(x+B_{T}) + \int_{0}^{T} (\arr \theta(s, x+ B_s) - c^\zzone) \mathd s \Big)\Big] \\ \nonumber
       & = - \inf_v \E \left[ - \bar{h}(\gamma^v_T) - \int_0^T (\arr \theta (s, \gamma^v_s) - c^\zzone) \mathd s + \frac{1}{4} \int_0^T | v_s |^2 \mathd s \right] \\
       & = \sup_v \E \left[ \bar{h}(\gamma^v_T) + \int_0^T \Big(\arr \theta (s, \gamma^v_s) - c^\zzone - \frac{1}{4} | v_s |^2 \Big) \mathd s \right],
    \end{align}
    where we wrote $\gamma^v_t = x + B_t + \int_0^t v_s \mathd s$ and the supremum runs over the same processes $v$ as in Theorem~\ref{thm:boue-dupuis}.
\end{corollary}
Define now for any $\gamma$ of the form $\dd \gamma_t = v_t \dd t + \dd B_t$ the payoff functional
\begin{align*}
   \Phi(\gamma,v) & = \bar{h}(\gamma_T) + \int_0^T \Big( \arr \theta (s, \gamma_s) - c^\zzone - \frac{1}{4} | v_s |^2 \Big) \mathd s \\
   & := \bar{h}(\gamma_T) - \arr Y(T,\gamma_T)  + \arr Y(0,x) + \int_0^T \mathD \arr Y(s,\gamma_s) \dd \gamma_s - \int_0^T (c^\zzone + \frac{1}{4} | v_s |^2) \mathd s,
\end{align*}
so that $h(T,x) = \sup_v \E[\Phi(\gamma^v,v)]$. Plugging in $\dd \gamma^v_t = v_t \dd t + \dd B_t$, we get
\begin{align*}
   \Phi(\gamma^v,v) & = \bar{h}(\gamma^v_T) - \arr Y(T,\gamma^v_T)  + \arr Y(0,x) - \int_0^T \Big( - v \mathD \arr Y + c^\zzone + \frac{1}{4} | v |^2 \Big)(s, \gamma^v_s) \mathd s + \text{mart.} \\
   & = \bar{h}(\gamma^v_T) - \arr Y(T,\gamma^v_T)  + \arr Y(0,x) - \int_0^T \Big( - |\mathD \arr Y|^2 + c^\zzone + \frac{1}{4} | v - 2 \mathD \arr Y|^2 \Big)(s, \gamma^v_s) \mathd s \\
   &\quad + \text{mart.},
\end{align*}
where we write ``mart.'' for an arbitrary martingale term whose expectation vanishes under $\E$. Now change the optimization variable to $v^1_t = v_t - 2\mathD \arr Y(t,\gamma^v_t)$, so that
\[
   \gamma^v_t = x+ B_t + \int_0^t (v^1 + 2 \mathD \arr Y)(s,\gamma^v_s) \mathd s,
\]
and the payoff becomes
\[
   \Phi(\gamma^v,v) = \bar{h}(\gamma^v_T) - \arr Y(T,\gamma^v_T)  + \arr Y(0,x) + \int_0^T \Big( | \mathD \arr Y |^2 - c^\zzone - \frac{1}{4} | v^1 |^2 \Big)(s,\gamma^v_s) \mathd s + \tmop{mart}.
\]
In the following denote $X^i = \mathD Y^i$. We can iterate the process by considering $\arr Y^\zzone$, where $Y^\zzone$ is as in Definition~\ref{def:kpz rough distribution} (i.e. $\arr Y^\zzone$ solves $(\partial_t + \Delta) \arr Y^\zzone = - (| \arr X |^2 - c^\zzone)$ with terminal condition $\arr Y^\zzone_T = 0$), which allows us to represent the payoff function as
\begin{align*}
   \Phi(\gamma^v,v) & =  \bar{h}(\gamma^v_T) - \arr Y(T,\gamma^v_T)  + \arr Y(0,x) + \arr Y^\zzone(0,x) \\
   &\quad + \int_0^T \Big( v^1 \arr X^\zzone + 2 \arr X \arr X^\zzone - \frac{1}{4} | v^1 |^2 \Big)(s,\gamma^v_s) \mathd s + \tmop{mart}. \\
   & = \bar{h}(\gamma^v_T) - \arr Y(T,\gamma^v_T)  + \arr Y(0,x) + \arr Y^\zzone(0,x) \\
   &\quad + \int_0^T \Big( |\arr X^\zzone |^2 + 2 \arr X \arr X^\zzone - \frac{1}{4} | v^1 - 2 \arr X^\zzone |^2 \Big)(s,\gamma^v_s) \mathd s + \tmop{mart}.
\end{align*}
We change the optimization strategy to $v^2_t = v^1_t - 2 \arr X^\zzone(t,\gamma^v_t)$, so that the
dynamics of $\gamma^v$ read $\mathd \gamma^v_t = \mathd B_t +(v^2 + 2\arr X + 2\arr X^\zzone)(t,\gamma^v_t) \mathd t$. Now let
\begin{equation}\label{eq:YR}
   \LL Y^R = | X^\zzone |^2 + 2 X X^\zzone + 2(X + X^\zzone) \mathD Y^R,\qquad Y^R(0) = 0.
\end{equation}
This is a linear paracontrolled equation whose solution is of the form $Y^R = Y^\zztwo + Y'\mpara P + Y^\sharp$ and depends continuously on $\Y$. Moreover, another application of It\^o's formula yields
\begin{align*}
   \Phi(\gamma^v,v) & =  \bar{h}(\gamma^v_T) - \arr Y(T,\gamma^v_T)  + \arr Y(0,x) + \arr Y^\zzone(0,x) + \arr Y^R(0,x) \\
   &\quad + \int_0^T \Big( v^2 \arr X^R - \frac{1}{4} | v^2 |^2 \Big)(s,\gamma^v_s) \mathd s + \tmop{mart}.
\end{align*}

Let us summarize the result of our calculation:

\begin{theorem}\label{thm:hjb}
    Let $\theta \in \LL C^{\alpha/2}_{\mathrm{loc}}(\R, C^\infty)$ and $c^\zzone, c^\zzfour \in \R$, and let $\Y = \Theta_{\mathrm{kpz}} (\theta, c^\zzone, c^\zzfour) \in \mathcal{Y}_{\mathrm{kpz}}$ (see Definition~\ref{def:kpz rough distribution}). Let $h$ solve the KPZ equation~\eqref{eq:kpz hjb section} driven by $\Y$, started in $\bar{h} \in L^\infty$, let $Y^R$ solve~\eqref{eq:YR}, and let $T>0$. Then
    \begin{align}\label{eq:HJB} \nonumber
       &(h - Y - Y^\zzone - Y^R)(T,x) \\
       &\hspace{60pt}= \sup_v \E \Big[\bar{h}(\zeta^v_T) - Y(0,\zeta^v_T) + \int_0^T \Big(  |\arr X^R|^2 -  \frac{1}{4} | v - 2 \arr X^R |^2 \Big)(s,\zeta^v_s) \mathd s  \Big],
    \end{align}
    where
    \begin{equation}\label{eq:singular diffusion}
       \zeta^v_t = x + \int_0^t (2\arr X + 2\arr X^\zzone + v)(s,\zeta^v_s) \mathd s + B_t
    \end{equation}
    and the supremum is taken over the same $v$ as in Theorem~\ref{thm:boue-dupuis}. 
\end{theorem}

This representation is very useful for deriving a priori bounds on $h$.

\begin{corollary}\label{cor:hjb bound}
   In the setting of Theorem~\ref{thm:hjb}, for all $T> 0$ there exists a constant $C>0$ depending only on $T$ and $\| \Y\|_{\mathcal{Y}_{\tmop{kpz}}(T)}$, such that
   \[
      \| h \|_{C_T L^\infty} \leqslant C(1 + \| \bar{h}\|_{L^\infty}).
   \]
   In particular, if $(\Theta_{\mathrm{kpz}} (\theta_n, c^\zzone_n, c^\zzfour_n))_n$ is a converging sequence in $\mathcal{Y}_{\mathrm{kpz}}$ and every $\theta_n$ is as above, then the corresponding solutions $(h_n)$ stay uniformly bounded in $C_T L^\infty$ for all $T>0$.
\end{corollary}

\begin{proof}
    First, we have $-  \frac{1}{4} | v - 2 \arr X^R |^2(s,\zeta^v_s) \leqslant 0$, independently of $v$, and therefore
    \begin{align}\label{eq:hjb comparison with 0 initial} \nonumber
       &\sup_v \E \Big[\bar{h}(\zeta^v_T) - Y(0,\zeta^v_T) + \int_0^T \Big(  |\arr X^R|^2 -  \frac{1}{4} | v - 2 \arr X^R |^2 \Big)(s,\zeta^v_s) \mathd s \Big] \\
        &\hspace{150pt} \leqslant \| \bar h \|_{L^\infty} + \|Y(0)\|_{L^\infty} + \|X^R\|_{L^2_T L^\infty}^2.
    \end{align}
    Moreover, choosing the specific control $v(t) \equiv 0$ we get
    \[
       \sup_v \E \Big[\bar{h}(\zeta^v_T) - Y(0,\zeta^v_T) + \int_0^T \Big(  |\arr X^R|^2 -  \frac{1}{4} | v - 2 \arr X^R |^2 \Big)(s,\zeta^v_s) \mathd s \Big] \geqslant - \|\bar h\|_{L^\infty} - \|Y(0)\|_{L^\infty}.
    \]
    In combination with Theorem~\ref{thm:hjb} and~\eqref{eq:hjb comparison with 0 initial}, this yields
    \begin{equation}
       \|h(T)\|_{L^\infty} \leqslant \|(Y + Y^\zzone + Y^R)(T)\|_{L^\infty} + \|Y(0)\|_{L^\infty} + \|\bar h\|_{L^\infty} + \|X^R\|_{L^2_T L^\infty}^2,
    \end{equation}
    which gives us a bound on the solution of the KPZ equation which is linear in terms of the data $\Y \in \mathcal{Y}_{\tmop{kpz}}$, and quadratic in terms of the solution $Y^R$ to a linear paracontrolled equation that can in turn be bounded by the data $\Y$.
\end{proof}

An immediate consequence is the global in time existence of solutions to the KPZ equation:

\begin{corollary}\label{cor:no explosion}
   In Theorem~\ref{thm:kpz smooth initial} and Theorem~\ref{thm:heat to kpz} we have $T^\ast = \infty$. If in Theorem~\ref{thm:burgers smooth initial} and Theorem~\ref{thm:burgers general initial} the initial condition is $u_0 = \mathD h_0$ for some $h_0 \in \CC^{2\alpha+1}$ (respectively $h_0 \in \CC^{-\beta+1}$), then also here $T^\ast = \infty$.
\end{corollary}

\begin{proof}
   To see that $T^\ast = \infty$ in Theorem~\ref{thm:kpz smooth initial}, Theorem~\ref{thm:heat to kpz} and in Theorem~\ref{thm:burgers smooth initial} under the condition $u_0 = \mathD h_0$, it suffices to combine Theorem~\ref{thm:hjb} with Corollary~\ref{cor:kpz explosion time}. As for Theorem~\ref{thm:burgers general initial}, recall from the proof of this theorem that in order to extend a solution $u \in \CD^{\tmop{exp}}_{\mathrm{rbe}}(T)$ from $[0,T]$ to $[0,T']$ with $T'>T$, it suffices to solve Burgers equation on $[0,T'-T]$ with the initial condition $u^Q(T) - u'(T)\para Q(T) \in \CC^{2\alpha}$ for the remainder. Moreover, if $u_0 = \mathD h_0$, then also $u(T) = \mathD h(T)$, so that the first part of the theorem tells us that $u$ can be extended from $[0,T]$ to $[0,\infty)$.
\end{proof}

\begin{remark}\label{rmk:explosion with nongradient initial}
   The condition $u_0 = \mathD h_0$ is equivalent to $\int_\T u_0(x) \dd x = 0$. In case the integral is equal to $c \in \R \setminus \{ 0 \}$ we can consider $\tilde u = u - c$ which solves
   \[
      \LL \tilde u = \mathD \tilde u^2 + 2 c \mathD \tilde u + \mathD \xi, \qquad \tilde u(0) = \mathD h_0
   \]
   for some $h_0$. This is a paracontrolled equation which we can solve up to some explosion time, and we have $\tilde u = \mathD \tilde h$ for the solution $\tilde h$ to
   \[
      \LL \tilde h = |\mathD \tilde h|^{\diamond 2} + 2 c \mathD \tilde h + \xi, \qquad \tilde h(0) = h_0.
   \]
   The Cole--Hopf transform then shows that $\tilde h = \log \tilde w$, where
   \[
      \LL \tilde w = 2 c \mathD w + w \diamond \xi, \qquad \tilde w(0) = e^{h_0},
   \]
   and based on these observations we could perform the same analysis as above to show that the explosion time of $\tilde u$ (and thus of $u$) is infinite. We would only have to replace the Brownian motion $B$ by the process $(B_t + 2 c t)_{t \geqslant 0}$ which corresponds to the generator $\Delta + 2 c \nabla$.
\end{remark}

Another simple consequence of~\eqref{eq:HJB} is a quantitative comparison result for the KPZ equation.
\begin{lemma}[``Comparison principle'']
   In the setting of Theorem~\ref{thm:kpz general initial} let $\bar h_1, \bar h_2 \in \CC^{1-\beta}$, and let $h_1$ solve
   \[
      \LL h_1 = |\mathD h_1|^{\diamond 2} + \xi, \qquad h_1(0) = \bar h_1,
   \]
   and $h_2$ analogously with $\bar h_1$ replaced by $\bar h_2$. Then
   \begin{equation}\label{eq:KPZ-comparison}
      h_1(t,x) + \inf_{x} (\bar h_2(x) - \bar h_1(x)) \leqslant h_2(t,x) \leqslant h_1(t,x) + \sup_x (\bar h_2(x) - \bar h_1(x))
   \end{equation}
   for all $(t,x)$. In particular, $\| h_1 - h_2\|_{C_T L^\infty} \leqslant \| \bar h_1 - \bar h_2\|_{L^\infty}$ for all $T>0$.
\end{lemma}
\begin{proof}
   Consider regular data $(\Y^n, \bar h^n_1, \bar h^n_2)$ that converges to $(\Y, \bar h_1, \bar h_2)$ in $\mathcal{Y}_{\tmop{kpz}} \times \CC^{1-\beta} \times \CC^{1-\beta}$, and denote the corresponding solutions by $h_1^n$ and $h_2^n$ respectively. For every $n$ the representation~\eqref{eq:HJB} and the decomposition $\bar h_n^2 = \bar h_n^1 + (\bar h_n^2 - \bar h_n^1)$ gives
   \[
      h^n_1(t,x) + \inf_{x} (\bar h^n_2(x) - \bar h^n_1(x)) \leqslant h^n_2(t,x) \leqslant h^n_1(t,x) + \sup_x (\bar h^n_2(x) - \bar h^n_1(x)).
   \]
   Letting $n$ tend to infinity, we get~\eqref{eq:KPZ-comparison}.
\end{proof}

\begin{remark}
   Corollary~\ref{cor:no explosion} gives a pathwise proof for the strict positivity of solutions to the rough heat equation started in strictly positive initial data. The classical proof for the stochastic heat equation is due to Mueller~\cite{Mueller1991}, whose arguments are rather involved; see also the recent works~\cite{Moreno-Flores2014, ChenKim2014}. Compared to these, our proof has the advantage that it does not use the structure of the white noise at all, so that it is applicable in a wide range of scenarios. The disadvantage is that we need to start in strictly positive data, whereas Mueller's result allows to start in nonpositive, nonzero data.
\end{remark}

\subsection{Partial Girsanov transform and random directed polymers}\label{sec:partial girsanov}

To formulate the optimization problem~\eqref{eq:HJB} for non-smooth elements $\Y$ of $\mathcal Y_{\mathrm{kpz}}$, we first need to make sense of the diffusion equation~\eqref{eq:singular diffusion} for $\zeta^v$ in that case. This can be done with the techniques of~\cite{DelarueDiel2014} or~\cite{CannizzaroChouk2015}, which we will apply with a slight twist.

Let us start by formally deriving the dynamics of the coordinate process under the random directed polymer measure. This is the measure given by
\[
   \dd \Q_{T,x} = \exp\Big( \int_0^T (\arr \xi(t, B_t) - \infty) \dd t \Big) \dd \P_x,
\]
where $\xi$ is a space-time white noise (and thus $\arr \xi$ as well), and under $\P_x$ the process $B$ is a Brownian motion started in $x$ and with variance $2$. The term $-\infty T$ is chosen so that $\Q_{T,x}$ has total mass $1$. If now $h$ solves the KPZ equation with $h(0) = 0$ and if $\arr h(s) = h(T-s)$, then we can write
\begin{align*}
   0 & = \arr h(0,x) + \int_0^T (\partial_t + \Delta) \arr h(t, B_t) \dd t + \int_0^T \mathD \arr h(t,B_t) \dd B_t \\
   & = \arr h(0,x) + \int_0^T (-\arr \xi(t, B_t) + \infty) \dd t + \int_0^T \mathD \arr h(t,B_t) \dd B_t - \int_0^T |\mathD \arr h(t,B_t)|^2 \dd t,
\end{align*}
which shows that under $\Q_{T,x}$ the process $B$ solves the SDE
\[
   B_t = x + \int_0^T 2 \mathD \arr h(t,B_t) \dd t + \dd W_t,
\]
where $W$ is a $\Q_{T,x}$-Brownian motion with variance 2.

Of course, a priori such a diffusion equation does not make any sense because $\mathD \arr h(t,\cdot)$ is a distribution and not a function. But let $\arr Y$, $\arr Y^\zzone$, and $\arr Y^R$ be as in Section~\ref{sec:HJB1}. Then we can rewrite the term in the exponential as
\begin{align*}
    \int_0^T (\arr \xi(t, B_t) - \infty) \dd t & = Y(0,x) + Y^\zzone(0,x) + Y^R(0,x) - Y(T,B_T) \\
    &\quad + \int_0^T (\arr X + \arr X^\zzone)(t,B_t) \dd B_t - \int_0^T | (\arr X + \arr X^\zzone)(t,B_t)|^2 \dd t \\
    &\quad + \int_0^T \arr X^R(s,B_s) (\dd B_s - 2 (X + X^\zzone)(s,B_s)\dd s).
\end{align*}
Let us define
\begin{equation}\label{eq:Z def}
   Z = (X + X^\zzone)
\end{equation}
   and set
\[
   \dd \P^Z_{T,x} = \exp\Big(\int_0^T \arr Z (t,B_t) \dd B_t - \int_0^T | \arr Z(t,B_t)|^2 \dd t \Big) \dd \P.
\]
Then under $\P^Z_{T,x}$ the coordinate process $B$ solves
\[
   B_t = x + \int_0^t 2 \arr Z(s,B_s) \dd s + \tilde B_t,
\]
where $\tilde B$ is a (variance 2) Brownian motion under $\P_{T,x}$. The advantage of this splitting of the Radon-Nikodym density is that now we singled out the singular part of the measure, and $\Q_{T,x}$ is absolutely continuous with respect to $\P_{T,x}^Z$, with
\[
   \dd \Q_{T,x} = \frac{\exp(- Y(T,B_T) +  \int_0^T \arr X^R(s,B_s) \dd \tilde B_s )}{\E_{\P^Z_{T,x}}[\exp(- Y(T,B_T) +  \int_0^T \arr X^R(s,B_s) \dd \tilde B_s )]} \dd \P^Z_{T,x}.
\]
This density is strictly positive and in $L^p(\Q_{T,x})$ for all $p \in [1,\infty)$, even for general $\Y \in \mathcal{Y}_{\mathrm{kpz}}$ (not necessarily smooth). It remains to construct the measure $\P^Z_{T,x}$ for general $\Y \in \mathcal{Y}_{\mathrm{kpz}}$. In general $\P^Z_{T,x}$ will be singular with respect to the Wiener measure.

\begin{theorem}\label{thm:partial girsanov}
   Let $\Y \in \mathcal{Y}_{\mathrm{kpz}}$ and $T>0$. Consider for given $\varphi_T \in \CC^{\alpha+1}$ and $f \in C([0,T], L^\infty)$ the solution $\varphi$ to the paracontrolled equation
   \[
      (\partial_t + \Delta) \varphi = - 2 \arr Z \mathD \varphi + f, \qquad \varphi(T) = \varphi_T.
   \]
   Then for every $x \in \T$ there exists a unique probability measure $\P^Z_{T,x}$ on $\Omega = C([0,T], \T)$, such that $\P^Z_{T,x}(\gamma_0 = x) = 1$ and for all $\varphi$ as above the process
   \[
      \varphi(t,\gamma_t) - \int_0^t f(s,\gamma_s) \dd s, \qquad t \in [0,T],
   \]
   is a square integrable martingale. Here $\gamma$ denotes the coordinate process on $\Omega$. Moreover, assume that $\Y_n = \Theta_{\mathrm{kpz}} (\theta_n, c^\zzone_n, c^\zzfour_n)$ for a sequence $\theta_n \in \LL C^{\alpha/2}_{\mathrm{loc}}(\R, C^\infty)$ and $c^\zzone_n, c^\zzfour_n \in \R$ is such that $(\Y_n)$ converges to $\Y$ in $\mathcal{Y}_{\mathrm{kpz}}$. Let for every $n$ the process $\gamma^n$ solve the SDE
   \[
      \gamma^n_t = x + \int_0^t 2 \arr Z_n(s, \gamma^n_s) \dd s + B_t
   \]
   for $Z_n = X_n + X_n^\zzone = \mathD(Y_n + Y_n^\zzone)$, which is well posed because $Z_n$ is Lipschitz continuous by assumption, uniformly in $t \in [0,T]$. Here $B$ denotes a variance 2 Brownian motion. Then there exists a Brownian motion $\tilde B$ on $\Omega$ with respect to the measure $\P^Z_{T,x}$, such that $(\gamma_n, B)$ converges weakly to $(\gamma, \tilde B)$. Finally, $\gamma$ is a time-inhomogeneous strong Markov process under the family of measures $(\P^Z_{T,x})_{x \in \T}$.
\end{theorem}

\begin{proof}
   Everything follows from Theorem~2.6 and Section~3 in~\cite{CannizzaroChouk2015}, provided we can show that $Z$ is a ``ground drift'' as defined in that paper. In other words we need to show that $Q^Z \reso Z$ is given, for the solution $Q^Z$ to $\LL Q^Z = \mathD Z$, $Q^Z(0) = 0$. But $Q^Z = Q + Q^\zzone$, where $\LL Q^\zzone = \mathD X^\zzone$. Since $X^\zzone \in C \CC^{2\alpha-1}$ and thus $Q^\zzone \in C \CC^{2\alpha}$, the only problematic term in the definition of $Q^Z \reso Z$ is $Q \reso X$. And since $Q \reso X$ is contained in $\Y$, the proof is complete.
\end{proof}

\begin{remark}
   We expect that for $0 \le s < t \le T$ under $(\P^Z_{T,x})_{x \in \T}$ the transition function
   \[
      T_{s,t} g(x) = \E_{\P^Z_{T}}[g(\gamma_t)| \gamma_s = x]
   \]
   has a density $p_{s,t}(x,y)$ which is jointly continuous in $(s,t,x,y)$. For that purpose we would have to solve the generator equation
   \begin{equation}\label{eq:polymer density discussion} 
      (\partial_r + \Delta + \arr Z \mathD) \varphi = f, \, r \in [0,t), \qquad \varphi(t) = \varphi_t
   \end{equation}
   for general $t \in [0,T]$. Taking $f \equiv 0$, we would have $T_{s,t} \varphi_t(x) = \varphi(s,x)$. Then we could enlarge the class of terminal conditions $\varphi^t$, and using the techniques of Section~\ref{sec:heat singular} it should be possible to take $\varphi_t$ as the Dirac delta in an arbitrary point $y \in \T$. But then the transition density $p_{s,t}(x,y)$ of $T_{s,t}$ must be given by $\varphi(s,x)$, where $\varphi$ solves~\eqref{eq:polymer density discussion} with terminal condition $\delta(y)$.
   
   We also expect that the density is strictly positive: In the upcoming work~\cite{Cannizzaro2015} it will be shown that if $u$ solves a linear paracontrolled equation and $u(0) \in C(\T)$ is a continuous, nonnegative and non-zero function, then $u(t,x) > 0$ for all $t> 0$, $x \in \T$. Combining this with the smoothing effect of the linear equation~\eqref{eq:polymer density discussion}, which takes as the terminal condition at time $t$ the Dirac delta and returns a nonnegative non-zero continuous function for all sufficiently large $r<t$, the strict positivity of the density should follow.
\end{remark}

\begin{remark}
   In the setting of Theorem~\ref{thm:partial girsanov}, we can construct the ``full random directed polymer measure'' $\Q_{T,x}$ by setting
   \[
      \frac{\dd \Q_{T,x}}{\dd \P^Z_{T,x}} = \frac{\exp(- Y(T,\gamma_T) +  \int_0^T \arr X^R(s,\gamma_s) \dd \tilde B_s )}{\E_{\P^Z_{T,x}}[\exp(- Y(T,\gamma_T) +  \int_0^T \arr X^R(s,\gamma_s) \dd \tilde B_s )]},
   \]
   which is now a perfectly well defined expression.
   
   Similarly we can also construct the measure $\Q_{T,(t,x)}$ under which the coordinate $\gamma$ process formally solves
   \[
      \gamma_s = x + \int_t^s 2 \mathD \arr h(r, \gamma_r) \dd r + (B_s - B_t), \qquad s \in [t,T].
   \]
   With respect to $\Q_{T,(t,x)}$ we have the following Feynman-Kac representation for the solution $w$ to the paracontrolled equation heat equation driven by $\Y$ and started in $w(0) = \bar w$:
   \[
      w(t,x) = \E_{\Q_{T,(t,x)}}[\bar w(\gamma_T)] \tilde w(t,x),
   \]
   where $\tilde w(t,x)$ solves the same equation as $w$ but started in $\tilde w(0) \equiv 1$.
\end{remark}

\subsection{Rigorous control problem}\label{sec:HJB2}

The control problem that we formulated in Section~\ref{sec:HJB1} worked only for ``smooth'' elements of $\Ykpz$, that is for $\Y = \Theta_{\mathrm{kpz}} (\theta, c^\zzone, c^\zzfour)$ with $\theta \in \LL C^{\alpha/2}_{\mathrm{loc}}(\R, C^\infty)$ and $c^\zzone, c^\zzfour \in \R$. Let us use the construction of Section~\ref{sec:partial girsanov} to make it rigorous also for general $\Y \in \Ykpz$.

For that purpose fix $T>0$ and set $\Omega_T = C([0,T], \T)$, equipped with the canonical filtration $(\F_t)_{t \in [0,T]}$. Recall that a function $v\colon [0,T] \times \Omega_T \to \R$ is called \emph{progressively measurable} if for all $t \in [0,T]$ the map $[0,t] \times \Omega_T \ni (s,\omega) \mapsto v(s,\omega)$ is $\CB[0,t] \otimes \F_t$--measurable.

From now on we also fix $\Y \in \Ykpz$ and define $Z$ as in~\eqref{eq:Z def}.

\begin{definition}
   Let $v$ be a progressively measurable process and let $(\tilde \Omega, \tilde \F, (\tilde \F_t)_{t \in [0,T]}, \P)$ be a filtered probability space. A stochastic process on this space is called a \emph{martingale solution} to the equation
   \begin{equation}\label{eq:martingale solution}
      \gamma_t = x + \int_0^t (2\arr Z(s, \gamma_s) + v(s, \gamma)) \dd s + B_t
   \end{equation}
   if $\P(\gamma_0 = x) = 1$ and whenever $\varphi_T \in \CC^{\alpha+1}$ and $f \in C([0,T], L^\infty)$ and $\varphi$  solves the paracontrolled equation
   \[
      (\partial_t + \Delta + 2 \arr Z \mathD) \varphi = f, \qquad \varphi(T) = \varphi_T.
   \]
   on $[0,T]$, then
   \[
      \varphi(t,\gamma_t) - \int_0^t (f(s,\gamma_s) + \mathD \varphi(s,\gamma_s) v(s,\gamma)) \dd s, \qquad t \in [0,T],
   \]
   is a martingale.
   
   We write $\prog$ for the set of progressively measurable processes, and for a given $v \in \prog$ and $x \in \T$ we write $\FM(v,x)$ for the collection of all martingale solutions to~\eqref{eq:martingale solution}. Note that here we explicitly allow the probability space to vary for different martingale solutions, and also that we do not make any claim about the existence or uniqueness (in law) of martingale solutions.
\end{definition}

\begin{theorem}\label{thm:variational kpz}
   Let $\bar h - Y(0) \in \CC^{2\alpha+1}$ and let $h$ be a paracontrolled solution to the KPZ equation with initial condition $\bar h$. Then
   \begin{align}\label{eq:rigorous hjb}
      &(h - Y - Y^\zzone - Y^R)(T,x) \\ \nonumber
      &\hspace{10pt} = \sup_{v \in \prog} \sup_{\gamma \in \FM(v,x)} \E \Big[\bar h(\gamma_T) - Y(0,\gamma_T) + \int_0^T \Big(  |\arr X^R|^2 -  \frac{1}{4} | v - 2 \arr X^R |^2 \Big)(s,\gamma) \mathd s  \Big]
    \end{align}
    and the optimal $v$ is
    \[
       v(t,\gamma) = 2 \mathD (\arr h - \arr X - \arr X^\zzone) (t, \gamma_t) = 2(\arr X^R + \mathD \arr h^R)(t, \gamma_t),
    \]
    where $h^R$ solves the paracontrolled equation
    \begin{equation}\label{eq:hR}
       \LL h^R = |X^R|^2 + 2( X + X^\zzone + X^R) \mathD h^R + | \mathD h^R|^2,\qquad h^R(0) = \bar h - Y(0).
    \end{equation}
    For this $v$ and all $x \in \T$, the set $\FM(v,x)$ is non-empty and every $\gamma \in \FM(v,x)$ has the same law.
\end{theorem}

\begin{proof}
   Clearly $h$ solves the KPZ equation driven by $\Y$ and started in $\bar h$ if and only if $h^R = h - X - X^\zzone - X^R$ solves~\eqref{eq:hR}. Moreover, the paracontrolled structure of $h^R$ is
   \[
      h^R = 2 (\mathD h^R) \mpara P + h^{R,\sharp}
   \]
   with $h^{R,\sharp} \in \LL^{2\alpha+1}$, and in particular $h \in C\CC^{\alpha+1}$. Reversing time, we get
   \[
      (\partial_t + \Delta + 2 \arr Z \mathD) \arr h^R = - |\arr X^R|^2 - 2 \arr X^R \mathD \arr h^R - | \mathD \arr h^R|^2,\qquad \arr h^R(T) = \bar h - Y(0).
   \]
   Let now $v \in \prog$ and let $\gamma$ be a martingale solution to
   \[
      \gamma_t = x + \int_0^t (2 \arr Z(s, \gamma_s) + v(s,\gamma)) \dd s + B_t.
   \]
   Since $\bar{h} - Y(0) \in \CC^{2\alpha+1}$, we can take $\arr h^R$ as test function in the martingale problem and get
   \begin{align*}
      \bar{h}(\gamma_T) - Y(0,\gamma_T) & = \arr h^R(0,x) + \int_0^T (- |\arr X^R|^2 - 2 \arr X^R \mathD \arr h^R - | \mathD \arr h^R|^2)(s,\gamma_s) \dd s \\
      &\quad + \int_0^T \mathD \arr h^R(s,\gamma_s) v(s,\gamma) \dd s + \text{mart.},
   \end{align*}
   so writing $\tilde v(s,\gamma) = v(s,\gamma) - 2 \arr X^R(s, \gamma_s) - 2 \mathD \arr h^R(s,\gamma_s)$ we obtain
   \begin{align*}
      & \E \Big[\bar h(\gamma_T) - Y(0,\gamma_T) + \int_0^T \Big(  |\arr X^R|^2 -  \frac{1}{4} | v - 2 \arr X^R |^2 \Big)(s,\gamma) \mathd s  \Big] - \arr h^R(0,x) \\
      &\hspace{20pt} = \E \Big[ \int_0^T (- |\arr X^R|^2 - 2 \arr X^R \mathD \arr h^R - | \mathD \arr h^R|^2 + \mathD \arr h^R(2 \arr X^R + 2 \mathD \arr h^R + \tilde v))(s,\gamma)  \dd s  \Big] \\
      &\hspace{20pt}\quad + \E\Big[ \int_0^T \Big(  |\arr X^R|^2 -  \frac{1}{4} | \tilde v + 2 \mathD \arr h^R |^2 \Big)(s,\gamma_s) \mathd s  \Big] \\
      &\hspace{20pt} = \E\Big[ \int_0^T - \frac{1}{4} |\tilde v(s,\gamma)|^2 \dd s\Big].
   \end{align*}
   This shows that
   \[
      \sup_{v \in \prog} \sup_{\gamma \in \FM(v,x)} \E \Big[\bar h(\gamma_T) - Y(0,\gamma_T) + \int_0^T \Big(  |\arr X^R|^2 -  \frac{1}{4} | v - 2 \arr X^R |^2 \Big)(s,\gamma) \mathd s  \Big] \leqslant h^R(T,x).
   \]
   On the other side, taking $v = 2 \arr X^R + 2 \mathD \arr h^R$ we obtain a ground drift in the terminology of~\cite{CannizzaroChouk2015}, and therefore for all $x \in \T$ there exists a $\gamma \in \FM(v,x)$ and its law is unique. For any such $\gamma$ we obtain
   \[
      \E \Big[\bar h(\gamma_T) - Y(0,\gamma_T) + \int_0^T \Big(  |\arr X^R|^2 -  \frac{1}{4} | v - 2 \arr X^R |^2 \Big)(s,\gamma) \mathd s  \Big] = h^R(T,x),
   \]
   and thus the proof is complete.
\end{proof}

\begin{remark}
   We only needed $\bar{h} - Y(0) \in \CC^{2\alpha+1}$ to apply the results of~\cite{CannizzaroChouk2015}. By using similar arguments as in Section~\ref{sec:singular initial} it will be possible to weaken the assumptions on the ground drift (allowing a possible singularity at 0) and on the terminal condition in the martingale problem in~\cite{CannizzaroChouk2015}, and then the variational representation of Theorem~\ref{thm:variational kpz} will extend   to $\bar h \in \CC^{1-\beta}$ for $\beta < 1$ as in Theorem~\ref{thm:kpz general initial}.
   
   Actually, the extension to $\bar h \in \CC^\alpha$ is immediate because we can simply start $Y^R$ in $\bar h - Y(0)$ (which still gives us a ground drift $\arr Z + 2 \arr X^R + 2 \mathD \arr h^R$, where $h^R$ is now started in 0), and change the control problem to
   \[
      \sup_{v \in \prog} \sup_{\gamma \in \FM(v,x)} \E \Big[ \int_0^T \Big(  |\arr X^R|^2 -  \frac{1}{4} | v - 2 \arr X^R |^2 \Big)(s,\gamma) \mathd s  \Big].
   \]
   Similarly, if we start $Y^R$ in $- Y(0)$ the control problem has the more appealing form
   \[
      \sup_{v \in \prog} \sup_{\gamma \in \FM(v,x)} \E \Big[ \bar{h}(\gamma_T) + \int_0^T \Big(  |\arr X^R|^2 -  \frac{1}{4} | v - 2 \arr X^R |^2 \Big)(s,\gamma) \mathd s  \Big].
   \]
\end{remark}

\section{Convergence of Sasamoto-Spohn lattice models}\label{sec:SS}

In this section we consider the weak universality conjecture in the context of weakly asymmetric interface models $\varphi_N\colon \R_+ \times \Z_N \to \R$ (where $\Z_N = \Z/ (N\Z)$) with
\begin{align}\label{eq:WASS}\nonumber
   \mathd \varphi_{N}(t,x) &= \Delta_{\Z_N} \varphi_N(t,x) \mathd t + \sqrt{\varepsilon} \big(B_{\Z_N} (\mathD_{\Z_N} \varphi_N(t), \mathD_{\Z_N} \varphi_N(t))\big)(x) \mathd t + \mathd W_N(t,x), \\
   \varphi_N(0,x) & = \varphi_0^N(x),
\end{align}
where $\Delta_{\Z_N}$ and $\mathD_{\Z_N}$ are discrete versions of Laplacian and spatial derivative respectively, $B_{\Z_N}$ is
a bilinear form taking the role of the pointwise product, $(W_{N}(t,x))_{t \in \R_+, x \in \Z_N}$ is an $N$--dimensional standard Brownian motion, and $\varphi_0^N$ is independent of $W_N$. We assume throughout this section that
\[
   \varepsilon = \frac{2\pi}{N}.
\]
Equation~\eqref{eq:WASS} is a generalization of the Sasamoto-Spohn discretization of the KPZ equation, see Remark~\ref{rmk:sasamoto-spohn} below. To simplify things (eliminating the need to introduce renormalization constants), let us look at the flux $\mathD_{\Z_N} \varphi_N$. Assume that there exists $\beta<1$ for which $(x \mapsto (\mathD_{\Z_N} \varphi_0^N)( x /\varepsilon))_N$ converges weakly to 0 in $\CC^{-\beta}$ with rate of convergence $\varepsilon^{1/2}$. Then $((t,x)\mapsto \mathD_{\Z_N} \varphi_N(t/\varepsilon^2,x/\varepsilon))$ converges to 0 (this will be a consequence of our analysis below), and we can study the fluctuations defined by
\[
   u_N(t,x) = \varepsilon^{-1/2} \mathD_{\Z_N} \varphi_N(t/\varepsilon^2,x/\varepsilon).
\]
This is a stochastic process on $\R_+ \times \T_N$ with $\T_N = (\varepsilon \Z) / (2 \pi \Z)$ which solves the SDE
\begin{align}\label{eq:lattice burgers} \nonumber
       \mathd u_{N}(t,x) & = \Delta_N u_N(t,x) \mathd t + \big(\mathD_{N} B_{N} (u_N(t), u_N(t))\big)(x) \mathd t +  \mathd (\mathD_{N} \varepsilon^{-1/2} W_N(t,x)) \\
       u_N(0) & = u_0^N.
\end{align}
where $\Delta_{N}$, $\mathD_{N}$, and $B_{N}$ are approximations of Laplacian, spatial derivative, and pointwise product respectively,  $\partial_t \mathD_{N} \varepsilon^{-1/2} W_N$ converges to $\mathD \xi$, where $\xi$ is a space-time white noise and $u_0^N(x) = \mathD_N \varepsilon^{1/2} \varphi_0^N(x / \varepsilon)$. We show that if $ \mathD_N \varepsilon^{1/2} \varphi_0^N(x / \varepsilon)$ converges in distribution in $\CC^{-\beta}$, then $(u_N)$ converges in distribution to the solution of a modified Burgers equation involving a sort of It\^o-Stratonovich corrector. 

Another way of reading our result is that~\eqref{eq:lattice burgers} is a lattice discretization of the Burgers equation and we show that it might converge to a different equation in the limit, depending on how we choose $\Delta_N$, $\mathD_N$, and $B_N$.

\bigskip
There are two problems that we have to deal with before we can study its convergence. First, it is not obvious whether $u_N$ blows up in finite time, because the equation contains a quadratic nonlinearity. Let therefore $\zeta$ be a cemetery state and define the space
\[
   C_N = \{ \varphi\colon \R_+ \to \R^{\T_N} \cup \{\zeta\}, \varphi \text{ is continuous on } [0,\tau_\zeta(\varphi)) \text{ and } \varphi(\tau_\zeta(\varphi)+t) = \zeta, t \geqslant 0\},
\]
where
\[
   \tau_\zeta(\varphi) = \inf\{t \geqslant 0: \varphi(t) = \zeta\} \quad \text{and for } c>0 \quad \tau_c(\varphi) = \inf\{t \geqslant 0: \|\varphi(t)\|_{L^\infty (\T_N)} \geqslant c\},
\]
with $\|\zeta\|_{L^\infty (\T_N)} = \infty$. Then a stochastic process $u_N$ with values in $C_N$ is a solution to~\eqref{eq:lattice burgers} if $\tau_\zeta(u_N) = \sup_{c>0} \tau_c(u_N)$ and $u_N|_{[0,\tau_\zeta(u_N))}$ solves~\eqref{eq:lattice burgers} on $[0,\tau_\zeta(u_N))$. It is a classical result that there exists a unique solution in that sense (which is always adapted to the filtration generated by $u_0^N$ and $W_N$, but we will not need this).

The next problem that we face is that $u_N(t)$ is only defined on the grid $\T_N$ and not on the entire torus $\T$. Since we will only obtain convergence in a space of distributions and not of continuous functions, some care has to be exercised when choosing an extension of $u_N$ to $\T$. For $\delta >0$ one can easily define sequences of smooth functions $(f_N)$ and $(g_N)$ on $\T$ such that $f_N$ and $g_N$ agree in the lattice points $\T_N$, but both converge to different limits in $\CC^{-\delta}$. Here we will work with a particularly convenient extension of $(u_N)$ that can be constructed using discrete Fourier transforms~{\cite{Hairer_Maas_2012,Hairer_Maas_2014}}. Since it will simplify the notation, we make the following assumption from now on:
\[
   N \text{ is odd.}
\]
Of course, our results do not depend on that assumption and we only make it for convenience. We define for $\varphi \colon \T_N \rightarrow \C$ the discrete Fourier transform
\[
   \CF_{N} \varphi (k) = \varepsilon \sum_{|\ell | <N/2} \varphi (\varepsilon\ell) e^{- i \varepsilon \ell k}, \hspace{2em} k \in \Z_N,
\]
(for even $N$ we would have to adapt the domain of summation), and then
\[
   \mathcal{E}_N \varphi (x) = (2 \pi)^{- 1} \sum_{| k | < N / 2} \CF_{N} \varphi (k) e^{i k x}, \hspace{2em} x \in \T.
\]
Then
$\mathcal{E}_N \varphi (x) = \varphi (x)$ for all $x \in \T_N$, and by construction $\mathcal{E}_N \varphi$ is a smooth function with Fourier transform $\CF \mathcal{E}_N \varphi (k) = \CF_{N} \varphi (k)
\1_{| k | < N / 2}$. If $\varphi$ is real valued, then
so is $\mathcal{E}_N \varphi$.

\begin{remark}
   A quite generic method of extending a function $\varphi\colon \T_N \to \R$ to $\T$ is as follows: write
   \[
       \psi = \sum_{x \in \T_N} \varphi(x) \varepsilon \delta_x,
   \]
   let $\eta_N \in \CD'(\T)$, and define $\bar \varphi(x) = \psi \ast \eta_N$, which yields for all $k \in \Z$
   \[
      \CF \bar \varphi(k) = \CF \psi(k) \CF \eta_N(k) = \CF_N \varphi(k) \CF \eta_N(k).
   \]
   Now assume that $(\varphi_N)$ is a sequence of functions on $\T_N$ such that $(\mathcal{E}_N \varphi_N)$ converges in $\CD'(\T)$ to some limit $\varphi_\infty$. Convergence in $\CD'(\T)$ is equivalent to the convergence of all Fourier modes, together with a uniform polynomial bound on their growth, and thus we get $\lim_{N} \CF_N \varphi_N(k) = \CF\varphi_\infty(k)$ for all $k \in \Z$. So if $(\eta_N)$ converges to $\delta_0$ in $\CD'(\T)$, then $(\bar \varphi_N)$ converges in $\CD'(\T)$ to $\varphi_\infty$. Typical examples for interpolations that can be constructed in this way are the ``Dirac delta extension'' (take $\eta_N = \delta_0$ for all $N$), the ``piecewise constant extension'' (take $\eta_N = \varepsilon^{-1} \1_{[0,\varepsilon)}$), or the ``piecewise linear extension'' (take $\eta_N(x) = \varepsilon^{-1} ( (\varepsilon^{-1} x+1)\1_{[-\varepsilon,0]}(x) + (1-\varepsilon^{-1} x) \1_{(0,\varepsilon]}(x))$). In particular, our convergence result below also implies the convergence of all these interpolations.
\end{remark}

For $\varphi \in C_N$ we then get $\mathcal{E}_N \varphi \colon \R_+ \to \CD'(\T) \cup \{\zeta\}$ and for general $\psi \colon \R_+ \to \CD'(\T) \cup \{\zeta\}$ we define
\[
   \tau_c^\beta(\psi) = \inf\{ t \geqslant 0: \|\psi(t)\|_{\beta} \ge c\}
\]
whenever $\beta \in \R$, with $\| \zeta \|_{\beta} = \infty$. Then write $d^\beta_c(\psi,\psi') = \| \psi - \psi' \|_{C_{c \wedge \tau^\beta_c(\psi) \wedge \tau^\beta_c(\psi')} \CC^\beta}$ as well as
\[
   d_{\beta}(\psi,\psi') = \sum_{m=1}^\infty 2^{-m} (1 \wedge (d_m^\beta(\psi,\psi') + |\tau^\beta_m(\psi) - \tau^\beta_m(\psi')|)).
\]

We will need some assumptions on the operators $\Delta_N$, $\mathD_N$, $B_N$: Let
\begin{gather}\label{eq:discrete operators def}
   \Delta_{N} \varphi (x) = \varepsilon^{-2} \int_{\Z} \varphi (x + \varepsilon y) \pi (\mathd y), \qquad \mathD_{N} \varphi (x) = \varepsilon^{-1} \int_{\Z} \varphi (x + \varepsilon y) \nu (\mathd y), \\ \nonumber
   B_{N} (\varphi, \psi) (x) = \int_{\Z^2} \varphi (x + \varepsilon y) \psi (x + \varepsilon z) {\mu} (\mathd y, \mathd z),
\end{gather}
where $\pi$, $\nu$, and $\mu$ are finite signed measures on $\Z$ and a probability measure on $\Z^2$, respectively. We define
\[
   f (x) = \frac{\int_{\Z} e^{i x y} \pi (\mathd y)}{- x^2}, \hspace{1em} g (x) = \frac{\int_{\Z} e^{i x y} \nu (\mathd y)}{i x}, \hspace{1em} h (x_1, x_2) = \int_{\Z^2} e^{i (x_1 z_1 + x_2 z_2)} {\mu} (\mathd z_1, \mathd z_2),
\]
and make the following assumptions on the measures $\pi$, $\nu$, $\mu$:
\begin{description}
  \item[(H$_f$)] The finite signed measure $\pi$ on $\Z$ is symmetric, has total mass zero, finite fourth moment, and satisfies $\int_{\Z} y^2 \pi (\mathd y) = 2$. Moreover, there exists $c_f > 0$ such that $f (x) > c_f$ for all $x \in [- \pi, \pi]$.
  \item[(H$_g$)] The finite signed measure $\nu$ on $\Z$ has total mass zero, finite second moment, and satisfies $\int_{\Z} y \nu (\mathd y) = 1$.
  \item[(H$_h$)] The probability measure ${\mu}$ has a finite first moment on $\Z^2$ and satisfies ${\mu} (A \times B) ={\mu} (B \times A)$ for all $A, B \subseteq \Z$.
\end{description}
The constant $c_f$ in (H$_f$) exists for example if $\pi = \sum_{k \geqslant 1} p_k (\delta_k + \delta_{- k}) - c \delta_0$ where $p_k \geqslant 0$,
$\sum_{k \geqslant 1} 2 p_k = c$ and $p_1 > 0$.

\begin{theorem}\label{thm:SS Burgers}
   Make assumption (H$_f$), (H$_g$), (H$_h$) and let $\beta \in (-1,-1/2)$. Consider for all $N\in \N$ an $N$-dimensional standard Brownian motion $(W_N(t,x))_{t\geqslant 0, x \in \T_N}$ and an independent random variable $(u_0^N(x))_{x \in \T_N}$ and denote the solution to~\eqref{eq:lattice burgers} by $u_N$. If there exists a random variable $u_0$ such that $\mathcal{E}_N u_0^N$ converges to $u_0$ in distribution in $\CC^\beta$, then $(\mathcal{E}_N u_N)$ converges in distribution with respect to the metric $d_\beta$ to $u$, the unique paracontrolled solution of
   \begin{equation}\label{eq:burgers with transport}
      \LL u = \mathD u^2 + 4 c \mathD u + \mathD \xi,\qquad u(0) = u_0,
   \end{equation}
   where $\xi$ is a space-time white noise which is independent of $u_0$, and where
  \begin{equation*}
     c = - \frac{1}{4\pi} \int_0^\pi \frac{\mathrm{Im} (g (x) \bar{h} (x))}{x} \frac{h (x, - x) | g (x) |^2}{| f (x) |^2} \dd x \in \R.
  \end{equation*}
\end{theorem}

\begin{proof}
   We have
   \begin{equation*}
   \CF_{N} B_{N} (\varphi, \psi) (k) = (2 \pi)^{- 1} \sum_{ \ell \in \Z_N} \CF_{N} \varphi (\ell) \CF_{N} \psi (k - \ell) \int_{\Z^2} e^{i \left( \varepsilon \ell y + \varepsilon (k - \ell) z \right)} {\mu}(\mathd y, \mathd z),
\end{equation*}
from where we deduce that $\mathcal{E}_N B_N(\varphi,\psi) = \Pi_N B_N(\mathcal E_N \varphi, \mathcal E_N \psi)$ for all $\varphi,\psi \in \R^{\T_N}$, with
\[
   \Pi_N \varphi (x) = (2\pi)^{- 1} \sum_{k} e^{i k^N x} \CF \varphi (k),
\]
where $k^N = \arg\min\{ |\ell| : \ell \in \Z,  \ell = k + j N \text{ for some } j \in \Z\} \in (-N/2, N/2)$. Moreover, we obtain for $|k| < N/2$
\begin{align*}
   &\E[\varepsilon^{-1 / 2} \CF \mathcal{E}_N W_N(t,k_1) \varepsilon^{-1 / 2} \CF \mathcal{E}_N W_N (t,k_2)] \\
   &\hspace{80pt} = \varepsilon \sum_{|\ell_1|, |\ell_2| < N/2}  \E [W_N(t,\varepsilon \ell_1) W_N(t,\varepsilon \ell_2)] e^{- i \varepsilon (\ell_1 k_1 + \ell_2 k_2)} \\
   &\hspace{80pt} = \varepsilon \sum_{|\ell| < N/2} t e^{- i \ell (k_1 + k_2) \varepsilon} = t 2 \pi \delta_{k_1 + k_2 = 0},
\end{align*}
which shows that $\partial_t \varepsilon^{-1 / 2} \mathcal{E}_N W_N(t)$ has the same distribution as $\PC_N \xi$, where $\xi$ is a space-time white noise and $\PC_N \varphi = \1_{(-N/2,N/2)}(\mathD) = \CF^{-1} (\1_{(-N/2,N/2)} \CF \varphi)$ is a Fourier cutoff operator.

We now place ourselves on a probability space where such a space-time white noise is given and where $(\tilde u_0^N)$ is a sequence of random variables with values in $\CD'(\T)$, which is independent of $\xi$, such that $\tilde u_0^N$ has the same distribution as $\mathcal E_N u_0^N$ for all $N$, and such that $u_0^N$ converges to a random variable $u_0$ in probability in $\CC^\beta$ (which is then also independent of $\xi$). Then $u_N$ has the same distribution as $\tilde u_N$, the solution to
\[
   (\partial_t - \Delta_N) \tilde u_N (t, x) = \mathD_N \Pi_N B_N (\tilde u_N,\tilde u_N) (t, x) + \mathD_N \PC_N \xi (t, x), \qquad \tilde u_N(0) = \tilde u_0^N,
\]
and therefore it suffices to study $\tilde u_N$. The pathwise analysis of this equation is carried out in Section~\ref{sec:paracontrolled SS} below, and the convergence result assuming convergence of the data $(\X_N(\xi))$ and the random operators $(A_N)$ is formulated in Theorem~\ref{thm:discrete burgers}. The convergence of $(\X_N(\xi))$ and $(A_N)$ in $L^p$ spaces is shown in Theorem~\ref{thm:discrete stochastic} and Theorem~\ref{thm:random operator}, respectively. Thus, we get that $(\tilde u_N)$ converges in probability with respect to $d_{\beta}$ to $u$, the solution of~\eqref{eq:burgers with transport}.
\end{proof}

\begin{remark}\label{rmk:sasamoto-spohn}
  If we take $B_N$ as the pointwise product for all $N$, and $\Delta_N$ as the discrete Laplacian, $\Delta_N f(x) = \varepsilon^{-2} (f(x+\varepsilon) + f(x-\varepsilon) - 2f(x))$ and $\mathD_{N} f (x) = \varepsilon^{-1}(f (x) - f (x - \varepsilon))$, then we get $c = 1/8$, so the additional term in equation~\eqref{eq:burgers with transport} is $1/2 \mathD u$.
   
   However, if we take the same $\Delta_N$ and $\mathD_N$ but replace the pointwise product by
   \[
      B_{N} (\varphi, \psi) (x) = \frac{1}{2 (\kappa + \lambda)} (\kappa \varphi (x) \psi (x) + \lambda (\varphi (x) \psi (x + \varepsilon) + \varphi (x + \varepsilon) \psi (x)) + \kappa \varphi (x + \varepsilon) \psi (x + \varepsilon))
   \]
   for some $\kappa, \lambda \in [0, \infty)$ with $\kappa + \lambda > 0$, then one can check that $c=0$. Here the Sasamoto--Spohn discretization~\cite{Krug1991,Lam1998,SasamotoSpohn2009} corresponds to $\kappa = 1$, $\lambda = 1/2$. In that case one furthermore has
   \[
      \langle \varphi, \mathD_N B_N(\varphi,\varphi) \rangle_{\T_N} = \sum_{x \in \T_N} \varphi(x) \mathD_N B_N(\varphi,\varphi)(x) = 0,
   \]
   which entails that already for fixed $N$ there is no blow up in the system, i.e. $u_N$ is well defined for all times. Moreover, now we can explicitly write down a family of stationary measures for $u_N$: For all $m \in \R$, the evolution of $u_N$ is invariant under
   \[
      \mu_m^\varepsilon (\dd x) = \prod_{j=0}^{N-1} \frac{\exp(- \varepsilon x_j^2 + m x_j)}{Z^\varepsilon_{m}} \dd x_j,
   \]
   where $Z^\varepsilon_{m}$ is a constant normalizing the mass of $\mu_m^\varepsilon$ to 1; see~\cite{SasamotoSpohn2009} or simply verify that the $L^2(\T)$--adjoint of the generator of $u_N$ applied to the density of $\mu^\varepsilon_m$ equals 0 and then use Echeverr\'ia's criterion to obtain the invariance of $\mu^\varepsilon_m$ from its infinitesimal invariance~\cite{Echeverria1982}. If $u_N(0) \sim \mu_m^\varepsilon$, then for all $t \geqslant 0$ the vector $(u_N(t,x))_{x \in \T_N}$ consists of independent Gaussian random variables with variance $\varepsilon^{-1}/2$ and mean $m$. Therefore, for all $t>0$ the $\CD'$-valued random variable $(\mathcal{E}_N u_N(t,\cdot))$ converges in distribution to a space white noise with mean $m$ and variance $1/2$. It is also straightforward to verify that $\sup_N \E[\| \mathcal{E}_N u_N(0,\cdot)\|_{B_{p,p}^\alpha}^p] < \infty$ whenever $\alpha < -1/2$, and then the Besov embedding theorem shows that the convergence actually takes place in distribution in $\CC^\beta$, for $\beta$ as required in Theorem~\ref{thm:SS Burgers}. But if $\mathcal{E}_N u_N$ is a stationary process for all $N$, then any limit in distribution must be stationary as well, and this shows that the white noise with mean $m$ and variance $1/2$ is an invariant distribution for the stochastic Burgers equation. This is of course well known, see for example~\cite{BertiniGiacomin1997} or~\cite{FunakiQuastel2014}. But to the best of our knowledge, ours is the first proof which does not rely on the Cole--Hopf transform.
\end{remark}

We now take
\begin{equation}\label{eq:discrete sbe}
   \LL_N u_N (t, x) = (\partial_t - \Delta_N) u_N (t, x) = \mathD_N \Pi_N B_N (u_N, u_N) (t, x) + \mathD_N \PC_N \xi (t, x)
\end{equation}
as the starting point for our analysis, where we assume that $u_N(0) = \PC_N u_N(0)$. Recall that
\[
   \Pi_N \varphi (x) = (2\pi)^{- 1} \sum_{k} e^{i k^N x} \CF \varphi (k)
\]
with
\[
   k^N = \arg\min\{ |\ell| : \ell \in \Z,  \ell = k + j N \text{ for some } j \in \Z\} \in (-N/2, N/2),
\]
and that
\[
   \PC_N \varphi = \1_{(-N/2,N/2)}(\mathD) = \CF^{-1} (\1_{(-N/2,N/2)} \CF \varphi).
\]
The operators $\Delta_N$, $\mathD_N$, and $B_N$ are given in terms of finite signed measures $\pi,\nu,\mu$ as described in~\eqref{eq:discrete operators def}.

\begin{lemma}\label{lem:f nice}
  Let $\pi$ be a finite signed measure on $\R$ that satisfies (H$_f$). Then the function
  \[ f (x) = - \frac{\int_{\R} e^{i x y} \pi (\mathd y)}{x^2} =
     \int_{\R} \frac{1 - \cos (x y)}{(\nobracket x y \nobracket)^2}
     y^2 \pi (\mathd y) \]
  is in $C^2_b$ and such that $f (0) = 1$.
\end{lemma}

\begin{proof}
  The function $\varphi (x) = (1 - \cos (x)) / x^2 = 2 \sin^2 (x / 2) / x^2$ is nonnegative, bounded by $1 / 2$, and satisfies $\varphi (0) = 1 / 2$. Therefore, $f$ is bounded and $f (0) = 1 / 2 \int_{\R} y^2 \pi(\mathd y) = 1$. Furthermore, it is easy to check that $\varphi \in C^2_b$, and thus
\[
     | f' (x) | \leqslant \int_{\R} | \varphi' (x y) | | y |^3 | \pi | (\mathd y) \lesssim \int_{\R} | y |^3 | \pi | (\mathd y),   \hspace{2em} | f'' (x) | \lesssim \int_{\R} | y |^4 | \pi | (\mathd y) .
  \]
  As $\pi$ has a finite fourth moment, this shows that $f \in C^2_b$.
\end{proof}

\begin{lemma}
  Let $\nu$ be a finite signed measure on $\R$ that satisfies (H$_g$). Then the function
  \[
     g (x) = \frac{\int_{\R} e^{i x y} \nu (\mathd y)}{i x} =  \int_{\R} \frac{e^{i x y} - 1}{i x y} y \nu (\mathd y)
  \]
  is in $C^1_b$ and such that $g (0) = 1$.
\end{lemma}

\begin{proof}
  It suffices to observe that the function $\varphi (x) = (e^{i x} - 1) / (i x)$ is in $C^1_b$ and satisfies $\varphi (0) = 1$, and then to copy the proof of Lemma~\ref{lem:f nice}.
\end{proof}

The next lemma is a simple and well known statement about characteristic
functions of probability measures.

\begin{lemma}
  Let ${\mu}$ be a probability measure on $\R^2$ that satisfies (H$_h$). Then the
  function
  \[ h (x, y) = \int_{\R^2} e^{i (x z_1 + y z_2)} {\mu} (\mathd
     z_1, \mathd z_2) \]
  is in $C^{1}_b$ and such that $h (0, 0) = 1$.
\end{lemma}

Our general strategy is to find a paracontrolled structure
for~\eqref{eq:discrete sbe} and then to follow the same steps as in the
continuous setting. To do so, we need to translate all steps of our continuous
analysis to the discrete setting.

\subsection{Preliminary estimates}\label{sec:preliminary sasamoto-spohn}

\paragraph{Fourier cutoff.}

The cutoff operator $\PC_N$ is not a bounded operator on $\CC^{\alpha}$ spaces (at least not uniformly in $N$) and will lead to a small loss of regularity.

\begin{lemma}\label{lem:cutoff bound}
  We have
  \[ \| \PC_N \varphi \|_{L^{\infty}} \lesssim \log N \| \varphi
     \|_{L^{\infty}}, \]
  and in particular we get for all $\delta > 0$
  \[ \| \PC_N \varphi - \varphi \|_{\alpha - \delta} \lesssim  N^{- \delta} \log N
     \| \varphi \|_{\alpha} \lesssim \| \varphi \|_{\alpha} . \]
\end{lemma}

\begin{proof}
  We have
  \[ \| \PC_N \varphi \|_{L^{\infty} (\T)} \lesssim \left\| \CF^{- 1}
     \1_{(- N / 2, N / 2)} \right\|_{L^1 (\T)} \| \varphi
     \|_{L^{\infty} (\T)}, \]
  and using that $N$ is odd we get
  \[ \CF^{- 1} \1_{(- N / 2, N / 2)} (x) = (2 \pi)^{- 1} \sum_{| k |
     < N / 2} e^{i k x} = (2 \pi)^{- 1} \frac{\cos (x (N - 1) / 2) - \cos (x
     (N + 1) / 2)}{1 - \cos (x)} . \]
  Now
  \[ \left| \frac{\cos (x (N - 1) / 2) - \cos (x (N + 1) / 2)}{1 - \cos (x)}
     \right| \lesssim \min \left\{ \frac{N x^2}{x^2}, \frac{| x |}{x^2}
     \right\}, \]
  and therefore
  \[ \int_{- \pi}^{\pi} \left| \CF^{- 1} \1_{(- N / 2, N / 2)} (x)
     \right| \mathd x \lesssim \log N. \]
  To obtain the bound for $\PC_N \varphi - \varphi$ it suffices to note that $\PC_N$ acts trivially on $\Delta_j$
  (either as identity or as zero) unless $2^j \simeq N$.
\end{proof}

\begin{lemma}\label{lem:periodic cutoff bound}
  Let $\alpha \geqslant 0$ and $\varphi \in \CC^{\alpha}$. Then for any $\delta
  \geqslant 0$
  \[ \| \Pi_N \varphi - \varphi \|_{\alpha - \delta} \lesssim N^{- \delta} \log N \| \varphi
     \|_{\alpha} . \]
  If $\tmop{supp} ( \CF \varphi ) \subset [- c N, c N]$ for some $c \in
  (0, 1)$, then this inequality extends to general $\alpha \in \R$.
\end{lemma}

\begin{proof}
  We already know that $\| \PC_N \varphi - \varphi \|_{\alpha - \delta} \lesssim N^{-
  \delta} \log N \| \varphi \|_{\alpha}$. So since $\Pi_N \varphi = \PC_N( (1 + e^{i N \cdot} + e^{-i N \cdot}) \varphi)$ we get for $\alpha \geqslant 0$
  \begin{align*}
     \| \Delta_q \PC_N ((e^{- i N \cdummy} + e^{i N \cdummy}) \varphi) \|_{L^{\infty}}  & \leqslant \1_{2^q \lesssim N} \| \PC_N ((e^{- i N \cdummy} + e^{i  N \cdummy}) \varphi) \|_{L^{\infty}} \\
     & \lesssim \sum_{j : 2^j \simeq N} \1_{2^q \lesssim N} \log (N) \| (e^{- i N \cdummy} + e^{i N \cdummy}) \Delta_j \varphi \|_{L^{\infty}} \\
     & \lesssim  \sum_{j : 2^j \simeq N} \1_{2^q \lesssim N} \log (N) 2^{- j  \alpha} \| \varphi \|_{\alpha} \\
     & \lesssim \1_{2^q \lesssim N} \log (N) N^{- \alpha} \| \varphi \|_{\alpha} \lesssim 2^{- q (\alpha - \delta)} \log (N) N^{- \delta} \| \varphi  \|_{\alpha}.
  \end{align*}
  If also $\tmop{supp} \left( \CF u \right) \subset [- c N, c N]$, then the
  spectrum of $2 \cos (N x) u$ is contained in an annulus $N \CA$, and
  therefore we can replace the indicator function $\1_{2^q \lesssim
  N}$ by $\1_{2^q \simeq N}$ in the calculation above, from where the
  claim follows.
\end{proof}

\begin{remark}\label{rmk:PiN on product}
  There exists $c \in (0, 1)$, independent of $N$, such that if $\tmop{supp}
  \left( \CF \psi \right) \subset [- N / 2, N / 2]$, then $\tmop{supp} \left( \CF
  (\varphi \para \psi) \right) \subset [- c N, c N]$. This means that we can always
  bound $\Pi_N (\varphi \para \psi) - \varphi \para \psi$, even if the paraproduct has negative
  regularity. On the other side the best statement we can make about the
  resonant product is that if $\varphi$ and $\psi$ are both spectrally supported in $[-
  N / 2, N / 2]$, then $\tmop{supp} \left( \CF (\varphi \reso \psi) \right) \subset [-N, N]$. A simple consequence is that if $\alpha + \beta > 0$, $\varphi \in
  \CC^{\alpha}$, $\psi \in \CC^{\beta}$, and $\mathrm{supp}( \CF \varphi) \cup \mathrm{supp}( \CF \psi) \subset [- N / 2, N / 2]$, then
  \[
     \| \Pi_N (\varphi\psi) - \varphi\psi \|_{\alpha \wedge \beta - \delta} \lesssim N^{- \delta} \log (N) \| \varphi \|_{\alpha} \| \psi \|_{\beta} .
  \]
\end{remark}

\paragraph{Estimates for the discrete Laplacian.}

\begin{lemma}\label{lem:discrete laplacian}
   Let $\pi$ satisfy (H$_f$). Let $\alpha < 1$,
  $\beta \in \R$, and let $\varphi \in \CC^{\alpha}$ and $\psi \in
  \CC^{\beta}$. Then for all $\delta \in [0, 1]$ and $N \in \N$
  \[ \| \Delta_N \psi - \Delta \psi \|_{\beta - 2 - \delta} \lesssim N^{-
     \delta} \| \psi \|_{\beta} \hspace{2em} \tmop{and} \]
  \[ \| \Delta_N (\varphi \para \psi) - \varphi \para \Delta_N \psi \|_{\alpha
     + \beta - 2} \lesssim \| \varphi \|_{\alpha} \| \psi \|_{\beta} . \]
\end{lemma}

\begin{proof}
  Using that $\pi$ has zero mass, zero first moment, and finite second moment,
  we have
  \begin{align*}
     | \Delta_j \Delta_N \psi (x) | & = \left| \varepsilon^{- 2} \int_{\R} \Delta_j \psi (x + \varepsilon y) \pi (\mathd y)\right| \\
     & = \left| \varepsilon^{- 2} \int_{\R} (\Delta_j \psi (x + \varepsilon y) - \Delta_j \psi (x) - \mathD \Delta_j \psi (x) \varepsilon y) \pi (\mathd y) \right| \\
     & \leqslant \| \mathD^2 \Delta_j \psi \|_{L^{\infty}} \int_{\R} y^2 | \pi | (\mathd y) \lesssim 2^{j (2 - \beta)} \| \psi \|_{\beta}.
  \end{align*}
  On the other side,
  \begin{align*}
     | \Delta_j (\Delta_N \psi - \Delta \psi) (x) | & = \left| \varepsilon^{- 2} \int_{\R} (\Delta_j \psi (x + \varepsilon y) - \Delta \Delta_j \psi (x)) \pi (\mathd y) \right| \\
     & \leqslant \varepsilon \| \mathD^3 \Delta_j \psi \|_{L^{\infty}} \int_{\R} | y |^3 | \pi | (\mathd y) \lesssim 2^{- j     (\beta - 3)} \varepsilon \| \psi \|_{\beta},
  \end{align*}
  and thus by interpolation $\| \Delta_N \psi - \Delta \psi \|_{\beta - 2 - \delta} \lesssim \varepsilon^\delta \| \psi \|_{\beta}$.
  
  For the commutator between discrete Laplacian and paraproduct it suffices to
  control $\Delta_N (S_{j - 1} \varphi \Delta_j \psi) - S_{j - 1} \varphi \Delta_j \Delta_N \psi$ in $L^\infty$, which is given by
  \begin{align*}
     &| (\Delta_N (S_{j - 1} \varphi \Delta_j \psi) - S_{j - 1} \varphi \Delta_j \Delta_N \psi) (x) | \\
     &\hspace{80pt} = \left| \varepsilon^{- 2}     \int_{\R} (S_{j - 1} \varphi (x + \varepsilon y) - S_{j - 1}  \varphi (x)) \Delta_j \psi (x + \varepsilon y) \pi (\mathd y) \right| \\
     &\hspace{80pt} \leqslant \| \mathD^2 S_{j - 1} \varphi \|_{L^{\infty}} \| \Delta_j \psi  \|_{L^{\infty}} \int_{\R} y^2 | \pi | (\mathd y) \\
     &\hspace{80pt}\quad + \varepsilon^{- 1} \left| \int_{\R} \mathD S_{j - 1} \varphi (x) y (\Delta_j \psi (x + \varepsilon y) - \Delta_j \psi (x)) \pi (\mathd y)
     \right| \\
     &\hspace{80pt} \lesssim 2^{j (2 - \alpha - \beta)} \| \varphi \|_{\alpha} \| \psi  \|_{\beta} \int_{\R} y^2 | \pi | (\mathd y).
  \end{align*}
  
\end{proof}

While the semigroup generated by the discrete Laplacian $\Delta_N$ does not
have good regularizing properties, we will only apply it to functions with
spectral support contained in $[- N / 2, N / 2]$, where it has the same
regularizing effect as the heat flow. It is here that we need the assumption
that $f (x) \geqslant c_f > 0$ for $x \in [- \pi, \pi]$.

\begin{lemma}\label{lem:discrete heat flow}
  Assume that $\pi$ satisfies (H$_f$). Let $\alpha \in \R$, $\beta \geqslant 0$, and let $\varphi \in \CD'$ with $\tmop{supp} \left( \CF \varphi \right) \subset [-N / 2, N / 2]$. Then we have for all $T > 0$ uniformly in $t \in (0, T]$
  \begin{equation}\label{eq:discrete heat flow}
     \| e^{t \Delta_N} \varphi \|_{\alpha + \beta} \lesssim t^{- \beta / 2} \| \varphi \|_{\alpha}.
  \end{equation}
  If $\alpha<0$, then we also have
  \[
     \| e^{t \Delta_N} \varphi \|_{L^{\infty}} \lesssim t^{\alpha / 2} \| \varphi \|_{\alpha}
  \]
\end{lemma}

\begin{proof}
  Let $\chi$ be a compactly supported smooth function with $\chi \equiv 1$ on $[-\pi, \pi]$ and such that $f(x) \geqslant c_f /2$ for all $x \in \mathrm{supp}(\chi)$. Then $e^{t \Delta_N} \varphi = e^{- t f (\varepsilon \mathD) \mathD^2} \chi (\varepsilon \mathD) \varphi = : \psi_{\varepsilon} (\mathD)\varphi$. According to Lemma~2.2 in {\cite{Bahouri2011}}, it suffices to show that
  \[
     \max_{k = 0, 1, 2} \sup_{x \in \R} | \mathD^k \psi_{\varepsilon} (x) | t^{\beta / 2} | x |^{\beta + k} \leqslant C < \infty,
  \]
  uniformly in $\varepsilon \in (0, 1]$. For $\psi_{\varepsilon}$ itself we
  have
  \[
     | \psi_{\varepsilon} (x) | t^{\beta / 2} | x |^{\beta} \lesssim e^{- \frac{c_f}{2} | \sqrt{t} x |^2} | \sqrt{t} x |^{\beta} \lesssim 1.
  \]
  To calculate the derivatives, note that
  \begin{gather*}
      (e^{- \varphi (x)} \rho (x))' = e^{- \varphi (x)} [- \varphi' (x) \rho(x) + \rho' (x)], \\
      (e^{- \varphi (x)} \rho (x))'' = e^{- \varphi (x)} [\varphi' (x)^2 \rho(x) - 2 \varphi' (x) \rho' (x) - \varphi'' (x) \rho (x) + \rho'' (x)].
  \end{gather*}
  In our case we set $\varphi_{\varepsilon} (x) = - t f (\varepsilon x) x^2$
  and $\rho_{\varepsilon} (x) = \chi (\varepsilon x)$, and obtain
  \[ \varphi_{\varepsilon}' (x) = - t x (f' (\varepsilon x) \varepsilon x +
     2 f (\varepsilon x)), \hspace{2em} \varphi_{\varepsilon}'' (x) = - t (f''
     (\varepsilon x) (\varepsilon x)^2 + 4 f' (\varepsilon x) \varepsilon x
     + 2 f (\varepsilon x)) . \]
  Since $| \varepsilon x \chi (\varepsilon x) | \lesssim 1$ and similarly for
  $(\varepsilon x)^2 \chi (\varepsilon x)$ and $\varepsilon x \chi'(\varepsilon x)$, we get
  \[
     | x | | \varphi'_{\varepsilon} (x) \rho_{\varepsilon} (x) | \lesssim | \sqrt{t} x |^2 \1_{\chi(\varepsilon x) \neq 0} \hspace{2em} \tmop{and} \hspace{2em} | x | | \rho'_{\varepsilon} (x) | \lesssim \1_{\chi(\varepsilon x) \neq 0},
  \]
  from where we deduce that
  \[
     | \mathD_x \psi_{\varepsilon} (x) | t^{\beta / 2} | x |^{\beta + 1}  \lesssim e^{- \frac{c_f}{2} ( \sqrt{t} x )^2} \left( | \sqrt{t} x |^{2 + \beta} + | \sqrt{t} x |^{\beta} \right) \lesssim 1.
  \]
  Similar arguments show that also $| \mathD^2_x \psi_{\varepsilon} (x) |t^{\beta / 2} | x |^{\beta + 2} \lesssim 1$, which concludes the proof of~\eqref{eq:discrete heat flow}. The $L^\infty$ estimate now follows from an interpolation argument.
\end{proof}

\begin{corollary}\label{cor:discrete heat flow time}
  Let $\pi$ satisfy (H$_f$), let $\alpha \in (0, 2)$, and let $\varphi \in \CC^{\alpha}$ with spectral support in $[- N / 2, N / 2]$. Then
  \[
     \| (e^{t \Delta_N} - \tmop{id}) \varphi \|_{L^{\infty}} \lesssim t^{\alpha / 2} \| \varphi \|_{\alpha}.
  \]
\end{corollary}

\begin{proof}
  By definition of $e^{t \Delta_N}$ we have
  \[
     \| (e^{t \Delta_N} - \tmop{id}) \varphi \|_{L^{\infty}} \leqslant \int_0^t \| e^{s \Delta_N} \Delta_N \varphi \|_{L^{\infty}} \dd s \lesssim \int_0^t s^{- (2 - \alpha) / 2} \| \Delta_N \varphi \|_{\alpha - 2} \mathd s \lesssim t^{\alpha / 2} \| \varphi \|_{\alpha},
  \]
  where we used Lemma~\ref{lem:discrete heat flow} and Lemma~\ref{lem:discrete laplacian} in the second step.
\end{proof}

Combining Lemma~\ref{lem:discrete heat flow} and Corollary~\ref{cor:discrete heat flow time}, we can apply the same arguments as in the continuous setting to derive analogous Schauder estimates for $(e^{t\Delta_N})$ as in Lemma~\ref{lemma:schauder} or Lemma~\ref{lemma:schauder exp} -- of course always restricted to elements of $\CS'$ that are spectrally supported in $[- N / 2, N / 2]$.

\paragraph{Estimate for the discrete derivative.}

\begin{lemma}\label{lem:discrete derivative}
   Let $\nu$ satisfy (H$_g$). Let $\alpha \in
  (0, 1)$, $\beta \in \R$, and let $\varphi \in \CC^{\alpha}$ and
  $\psi \in \CC^{\beta}$. Then for all $\delta \in [0, 1]$ and $N \in
  \N$
  \[
     \| \mathD_N \psi - \mathD \psi \|_{\beta - 1 - \delta} \lesssim  N^{- \delta} \| \psi \|_{\beta} \hspace{2em} \tmop{and}
  \]
  \[
     \| \mathD_N (\varphi \para \psi) - \varphi \para \mathD_N \psi \|_{\alpha + \beta - 1} \lesssim \| \varphi \|_{\alpha} \| \psi \|_{\beta}.
  \]
\end{lemma}

The proof is the same as the one of Lemma~\ref{lem:discrete laplacian}, and we omit it.

\paragraph{Estimates for the bilinear form.}

Let us define paraproduct and resonant term with respect to $B_N$:
\[ B_N (\varphi \para \psi) = \sum_j B_N (S_{j - 1} \varphi, \Delta_j \psi),
   \hspace{2em} B_N (\varphi \lpara \psi) = \sum_j B_N (\Delta_j \varphi, S_{j
   - 1} \psi), \]
\[ B_N (f \reso g) = \sum_{| i - j | \leqslant 1} B_N (\Delta_i f, \Delta_j g)
   . \]
\[ \  \]
We have the same estimates as for the usual product:

\begin{lemma}
  \label{lem:discrete product}Let ${\mu}$ satisfy (H$_h$). For any $\beta
  \in \mathbb{R}$ and $\delta \in [0, 1]$ we have
  \[ \| B_N (\varphi \para \psi) - \varphi \para \psi \|_{\beta - \delta}
     \lesssim N^{- \delta} \| \varphi \|_{L^{\infty}} \| \psi \|_{\beta}, \]
  and for $\alpha < 0$ furthermore
  \[ \|B_N (\varphi \para \psi) - \varphi \para \psi \|_{\alpha + \beta -
     \delta} \lesssim N^{- \delta} \| \varphi \|_{\alpha} \| \psi \|_{\beta} .
  \]
  For $\alpha + \beta - \delta > 0$ we have
  \[ \|B_N (\varphi \reso \psi) - \varphi \reso \psi \|_{\alpha + \beta -
     \delta} \lesssim N^{- \delta} \| \varphi \|_{\alpha} \| \psi \|_{\beta} .
  \]
\end{lemma}

\begin{proof}
  It suffices to note that $\CF B_N (\Delta_i \varphi, \Delta_j \psi)$ and
  $\CF (\Delta_i \varphi \Delta_j \psi)$ have the same support and that $\|
  B_N (\Delta_i \varphi, \Delta_j \psi) \|_{L^{\infty}} \leqslant \| \Delta_i
  \varphi \|_{L^{\infty}} \| \Delta_j \psi \|_{L^{\infty}}$, whereas
  \[
     \| B_N (\Delta_i \varphi, \Delta_j \psi) - \Delta_i \varphi \Delta_j \psi \|_{L^{\infty}} \lesssim N^{- 1} (2^i + 2^j) \| \Delta_i \varphi \|_{L^{\infty}} \| \Delta_j \psi \|_{L^{\infty}}.
  \]
\end{proof}

To invoke our commutator estimate, we have to pass from $B_N (\cdummy \para \cdummy)$ to the usual paraproduct, which can be done using the following commutator lemma. Here we write
\[
   \overline{B}_N \varphi = \overline{B}_N (\varphi) = B_N (\varphi, 1) = B_N(1, \varphi).
\]
   
\begin{lemma}\label{lem:bilinear paraproduct to normal}
  Let ${\mu}$ satisfy (H$_h$). Let $\alpha < 1$, $\beta \in \R$, and $\varphi \in \CC^{\alpha}$, $\PC_N \psi \in \CC^{\beta}$. Then
  \[
     \left\| B_N (\varphi \para \PC_N \psi) - \varphi \para \overline{B}_N  (\PC_N \psi) \right\|_{\alpha + \beta} \lesssim \| \varphi \|_{\alpha} \| \PC_N \psi \|_{\beta}.
  \]
\end{lemma}

\begin{proof}
  By spectral support properties it suffices to control
  \begin{align*}
     &\left| \left( B_N (S_{j - 1} \varphi, \Delta_j \PC_N \psi) - S_{j - 1} \varphi \overline{B}_N (\Delta_j \PC_N \psi) \right) (x) \right| \\
     &\hspace{80pt} = \left| \int_{\R^2} (S_{j - 1} \varphi (x + \varepsilon y) - S_{j - 1} \varphi (x)) \Delta_j \PC_N \psi (x + \varepsilon z){\mu} (\mathd y, \mathd z) \right| \\
     &\hspace{80pt} \lesssim \varepsilon 2^{j (1 - \alpha - \beta)} \| \varphi \|_{\alpha} \| \PC_N \psi \|_{\beta}.
  \end{align*}
  But now $\Delta_j \PC_N \equiv 0$ unless $2^j \lesssim N$, and therefore we may estimate $\varepsilon 2^j \lesssim 1$.
\end{proof}

\begin{lemma}\label{lem:bilinear form reso}
   Let ${\mu}$ satisfy (H$_h$). Let $\alpha + \beta + \gamma > 0$, $\beta+\gamma<0$, assume that $\alpha \in (0,1)$, and let $\varphi \in \CC^{\alpha}$, $\PC_N \psi \in \CC^{\beta}$, $\chi \in \CC^\gamma$. Define the operator
   \begin{align*}
      A_N^{\psi, \chi}(\varphi) := &\, \int \big( \Pi_N (\Pi_N ( \varphi \para \tau_{-\varepsilon y} \PC_N \psi )) \reso \tau_{-\varepsilon z} \chi ) - \PC_N((\varphi \para \tau_{-\varepsilon y} \PC_N \psi ) \reso \tau_{-\varepsilon z} \chi) \big) \mu(\dd y, \dd z).
   \end{align*}
   Then for all $\delta \in [0,\alpha + \beta + \gamma)$
   \begin{align*}
       &\| \Pi_N B_N( \Pi_N  (\varphi \para \PC_N \psi) \reso \chi) - \PC_N C(\varphi, \PC_N \psi, \chi) - \PC_N (\varphi B_N( \PC_N \psi \reso \chi)) \|_{\beta+\gamma} \\
       &\hspace{50pt} \lesssim N^{-\delta} \log(N)^2 \| \varphi \|_{\alpha} \| \PC_N \psi \|_{\beta} \| \gamma \|_\gamma + \| A_N^{\psi, \chi}\|_{L(\CC^\alpha,\CC^{\beta+\gamma})} \| \varphi \|_\alpha ,
   \end{align*}
   where $C$ is the commutator of Remark~\ref{rmk:paracontrolled commutator} and $L(U,V)$ denotes the space of bounded operators between the Banach spaces $U$ and $V$, equipped with the operator norm.
\end{lemma}

\begin{proof}
   We decompose the difference as follows:
   \begin{align*}
      &\| \Pi_N B_N( \Pi_N  ( \varphi \para \PC_N \psi) \reso \chi) - \PC_N C(\varphi, \PC_N \psi, \chi) - \PC_N(\varphi B_N( \PC_N \psi \reso \chi)) \|_{\beta+\gamma} \\
      &\hspace{15pt}\leqslant \Big\| \Pi_N \Big( B_N( \Pi_N ( \varphi \para \PC_N \psi ) \reso \chi) - \int (\Pi_N ( \varphi \para \tau_{-\varepsilon y} \PC_N \psi )) \reso \tau_{-\varepsilon z} \chi \mu(\dd y, \dd z) \Big) \Big\|_{\beta+\gamma} \\
      &\hspace{15pt} \quad + \| A_N^{\psi,\chi} (\varphi) \|_{\beta+\gamma-\delta} + \int \| \PC_N(C(\varphi, \tau_{-\varepsilon y} \PC_N \psi, \tau_{-\varepsilon z} \chi) - C(\varphi, \PC_N \psi, \chi) ) \|_{\beta+\gamma} \mu(\dd y, \dd z).
   \end{align*}
   For the first term on the right hand side the same arguments as in the proof of Lemma~\ref{lem:bilinear paraproduct to normal} show that for all $\delta \in [0,1]$ with $\alpha + \beta + \gamma - \delta > 0$
  \begin{align*}
     &\left\| \Pi_N \Big(B_N (\Pi_N (\varphi \para \PC_N \psi))\reso \chi) - \int (\Pi_N (\varphi \para \tau_{-\varepsilon y}\PC_N \psi))\reso \tau_{-\varepsilon z}\chi) \mu(\dd y, \dd z)\Big) \right\|_{\alpha + \beta + \gamma - \delta} \\
     &\hspace{130pt} \lesssim N^{-\delta} \log(N)^2 \| \varphi \|_{\alpha} \| \PC_N \psi \|_{\beta} \| \chi \|_\gamma
  \end{align*}
  (the factor $\log(N)^2$ is due to the operator $\Pi_N$ which appears twice, see Lemma~\ref{lem:periodic cutoff bound} and Remark~\ref{rmk:PiN on product}). The second term is trivial to bound, and for the last term we simply use that
  \[
      \|\tau_{-\varepsilon y} u - u \|_{\kappa - \delta} \lesssim N^{-\delta} |y|^\delta \| u \|_{\kappa}
  \]
  whenever $\kappa \in \R$, $u \in \CC^\kappa$, and $\delta \in [0,1]$.
\end{proof}

\subsection{Paracontrolled analysis of the discrete equation}\label{sec:paracontrolled SS}

We now have all tools at hand that are required to describe the paracontrolled structure of the solution $u_N$ to equation~\eqref{eq:discrete sbe} which as we recall is given by
\[
   \LL_N u_N = \mathD_N \Pi_N B_N (u_N, u_N) + \mathD_N \PC_N \xi, \hspace{2em} u_N (0) = \PC_N u_0^N.
\]
We set
\begin{equation*}
     \X_N(\xi) = (X_N (\xi), X^{\zzone}_N(\xi), X^{\zztwo}_N(\xi), X^{\zzthreereso}_N(\xi), X^{\zzfour}_N (\xi), B_N(Q_N \reso X_N) (\xi)), 
  \end{equation*}
  where
  \begin{equation}\label{eq:discrete data}
    \begin{array}{rll} 
      \LL X_N (\xi) & = & \mathD_N \PC_N \xi,\\
      \LL X^{\zzone}_N (\xi) & = & \mathD_N \Pi_N B_N (X_N(\xi),X_N(\xi)),\\
      \LL X^{\zztwo}_N (\xi) & = & \mathD_N \Pi_N B_N (X_N (\xi), X^{\zzone}_N (\xi)),\\
      \LL X^{\zzthreereso}_N (\xi) & = & \mathD_N \Pi_N B_N(X^{\zztwo}_N(\xi) \reso X_N(\xi)), \\
      \LL X^{\zzfour}_N (\xi) & = & \mathD_N \Pi_N B_N (X^{\zzone}_N (\xi), X^{\zzone}_N (\xi)),\\
      \LL Q_N (\xi) & = & \mathD_N B_N (X_N (\xi), 1),
    \end{array}
  \end{equation}
    all with zero initial conditions except $X_N(\xi)$ for which we choose the ``stationary'' initial condition
    \[
       X_N(\xi)(0) = \int_{-\infty}^0 e^{-s f_\varepsilon | \cdot|^2}(\mathrm{D}) \mathD_N \PC_N \xi(s) \dd s.
    \]
As in Section~\ref{sec:rbe singular} we fix $\alpha \in (1/3,1/2)$ and $\beta \in (1-\alpha, 2\alpha)$. We will consider initial conditions that converge in $\CC^{-\beta}$, which means that we expect the solutions $(u_N)$ to converge in $C_T \CC^{-\beta}$ whenever $T>0$ and the stochastic data $(\X_N)$ converges in an appropriate sense. However, we have to be careful because while by now we know that no blow up can occur for solutions to the continuous Burgers equation, this is not so obvious in the discrete setting. So let $\zeta$ be a cemetery state and denote by $(u_N)$ the unique solution to~\eqref{eq:discrete sbe}, defined up to the blow up time $T^\ast_N = \inf\{ t \geqslant 0: \|u_N(t)\|_{L^\infty} = \infty\}$ and extended via $u_N|_{[T^\ast_N,\infty)} = \zeta$. Note that the spectrum of $u_N$ is contained in $[-N/2,N/2]$, and therefore all $\CC^\gamma$ norms for $\gamma \in \R$ are equivalent to its $L^\infty$ norm, so that we could as well have defined $T^\ast_N$ as the blow up time of $\| u_N(t)\|_{-\beta}$.

\begin{theorem}\label{thm:discrete burgers}
   Let $(\X_N)$ be as in~\eqref{eq:discrete data} and assume that the sequence is uniformly bounded in $C_T \CC^{\alpha-1} \times C_T \CC^{2\alpha-1} \times \LL_T^{\alpha} \times \LL_T^{2\alpha} \times \LL^{2\alpha}_T \times C_T \CC^{2\alpha-1}$ for all $T>0$, and converges to
   \[
      (X, X^\zzone, X^\zztwo + 2cQ, X^\zzthreereso + c Q^\zzone + 2c Q^{Q\reso X}, X^\zzfour, Q\reso X + c)
   \]
   in $C_T \CC^{\alpha-1} \times C_T \CC^{2\alpha-1} \times \LL_T^{\alpha} \times \LL_T^{2\alpha} \times \LL^{2\alpha}_T \times C([\delta,T], \CC^{2\alpha-1})$ for all $0< \delta < T$, where $c \in \R$,
   \[
      \X = (X, X^\zzone, X^\zztwo, X^\zzthreereso, X^\zzfour, Q\reso X) \in \X_{\rbe},
   \]
   and
  \[
     \LL Q^{\zzone} = \mathD X^\zzone, \qquad \LL Q^{Q\reso X} = \mathD (Q\reso X),
  \]
  both with $0$ initial condition. Assume also that the operator $A_N = A_N^{Q_N,X_N}$ in Lemma~\ref{lem:bilinear form reso} converges to 0 in $C_T(L(\CC^{\bar{\alpha}}, \CC^{2\bar{\alpha}-1}))$ for all $T>0$, $\bar{\alpha} \in (1/3,\alpha)$. Finally, let $(\PC_N u_0^N), u_0 \in \CC^{-\beta}$ and assume that $\lim_N \| \PC_N u_0^N - u_0\|_{-\beta} = 0$.
  
  Let $(u_N)$ be the solution to~\eqref{eq:discrete sbe}. Then $\lim_N d_{-\beta}(u_N,u) = 0$, where $u \in \CD_{\rbe,\X}^{\exp}$ (see Definition~\ref{def:paracontrolled singular}) is the unique solution to
  \begin{equation}\label{eq:sbe additional transport}
     \LL u = \mathD u^2 + 4c \mathD u + \mathD \xi, \qquad u(0) = u_0.
  \end{equation}
\end{theorem}

\begin{remark}
   According to Remark~\ref{rmk:explosion with nongradient initial}, equation~\eqref{eq:sbe additional transport} has a unique solution in $\CD_{\rbe,\X}^{\exp}$ which does not blow up. In particular $\lim_N T_N^\ast = \infty$, even if for fixed $N$ we cannot exclude the possibility of a blow up.
\end{remark}

\begin{proof}
    Throughout the proof we fix $\bar{\alpha} \in (1/3,\alpha)$ and we define $\gamma_\delta = (\beta + \delta)/2$ whenever $\delta \geqslant 0$. We would like to perform a paracontrolled analysis of the equation, working in spaces modeled after the $\CD_{\rbe}^{\exp}$ of Definition~\ref{def:paracontrolled singular}. For that purpose, we decompose the nonlinearity as follows:
    \begin{align*}
       \mathD_N \Pi_N B_N (u_N, u_N)  & = \LL_N ( X_N^{\zzone} + 2 X_N^{\zztwo} + 4 X_N^{\zzthreereso} + X_N^{\zzfour}) + 4 \mathD_N \Pi_N B_N (X_N^{\zztwo} \para X_N) \\
        &\qquad  + 4 \mathD_N \Pi_N B_N (X_N^{\zztwo} \lpara X_N) + 2 \mathD_N \Pi_N B_N (u_N^{Q}, X_N) \\
        &\qquad + 2 \mathD_N \Pi_N B_N ( X_N^\zzone, 2 X_N^{\zztwo} + u_N^Q) \\
        &\qquad + \mathD_N \Pi_N B_N (2 X_N^{\zztwo} + u_N^Q, 2 X_N^{\zztwo} + u_N^Q).
    \end{align*}
    The term $\mathD_N \Pi_N B_N (u_N^Q, X_N)$ can be further decomposed as
    \begin{align*}
       \mathD_N \Pi_N B_N (u_N^Q, X_N) & = \mathD_N \Pi_N B_N (u_N^Q \para X_N) + \mathD_N \Pi_N B_N (u_N^Q \lpara X_N) \\
       &\qquad + \mathD_N \Pi_N B_N (u_N^Q \reso X_N),
    \end{align*}
    and the critical term is of course $\mathD_N \Pi_N B_N (u_N^Q \reso X_N)$. Using Lemma~\ref{lem:discrete derivative}, Lemma~\ref{lem:bilinear paraproduct to normal}, Lemma~\ref{lem:cutoff bound}, and Remark~\ref{rmk:PiN on product} we have for all $T>0$
    \begin{align*}
       &\| \mathD_N \Pi_N B_N ((2 u_N^Q + 4 X^\zztwo_N) \para X_N) - \Pi_N ((2 u_N^Q + 4 X^\zztwo_N) \mpara \overline{B}_N (\mathD_N X_N)) \|_{\CM^{\gamma_{\bar{\alpha}}}_T \CC^{2 \bar{\alpha} - 2}} \\
       &\hspace{70pt} \lesssim \big(\| u_N^Q \|_{\LL^{\gamma_{\bar{\alpha}},\bar{\alpha}}_\infty(T)} + \|X^\zztwo_N\|_{\LL^{\alpha}_T}\big) \| X_N \|_{\alpha - 1}.
    \end{align*}
    However, the Fourier cutoff operator $\Pi_N$ does not commute with the paraproduct (at least not allowing for bounds that are uniform in $N$), and in particular $u^Q_N$ is not paracontrolled. Rather, we have
    \[
       u_N^Q = \Pi_N ( u'_N \mpara Q_N ) + u_N^{\sharp}
    \]
    with
    \[
       u'_N = 2 u_N^Q + 4X_N^\zztwo \in \LL^{\gamma_{\bar{\alpha}},\bar{\alpha}}_\infty(T),\qquad u_N^\sharp \in \LL^{\gamma_{2\bar{\alpha}},2\bar{\alpha}}_\infty(T).
    \]
    This means that we need an additional ingredient beyond the paracontrolled tools in order to control the term $\mathD_N \Pi_N B_N ( \Pi_N ( u'_N \para Q_N ) \reso X_N)$, and it is here that we need our assumption on the operator $A_N$. Under this assumption we can apply Lemma~\ref{lem:bilinear form reso} to write
    \begin{equation}\label{eq:sasamoto-spohn pr1}
       \mathD_N \Pi_N B_N ( \Pi_N ( u'_N \para Q_N ) \reso X_N) = R_N + \PC_N C(u'_N,Q_N,X_N) + \PC_N( u_N' B_N(Q_N\reso X_N)),
    \end{equation}
    where $R_N$ is a term that converges to zero in $\CM^{\gamma_{\bar{\alpha}}}_T \CC^{2 \bar{\alpha} - 2}$ if $u'_N$ stays uniformly bounded in $\LL^{\gamma_{\bar{\alpha}},\bar{\alpha}}_\infty(T)$. Denote now
    \[
       (\tilde X, \tilde X^\zzone, \tilde X^\zztwo, \tilde X^\zzthreereso, \tilde X^\zzfour, \tilde \eta, \tilde Q) = \lim_N (X_N, X^\zzone_N, X^\zztwo_N, X^\zzthreereso_N, X^\zzfour_N, B_N(Q_N\reso X_N), Q_N).
    \]
    Based on the above representation of the nonlinearity, it is not difficult to repeat the arguments of Section~\ref{sec:rbe singular} in order to show that $d_{-\beta}(u_N,u)$ converges to 0, where
    \[
       u = \tilde X + \tilde X^\zzone + 2 \tilde X^\zztwo + \tilde u^Q,\qquad \tilde u^Q = \tilde u'\mpara \tilde Q + \tilde u^\sharp,
    \]
    with $(\tilde u^Q, \tilde u', \tilde u^\sharp) = \lim_N (u_N^Q, u_N', u_N^\sharp)$, where $u(0) = u_0$ and
    \begin{align}\label{eq:sasamoto-spohn pr2} \nonumber
       \LL \tilde u^Q & = \LL \tilde X^\zzfour + 4 \LL \tilde X^\zzthreereso + 2 \mathD ((2 \tilde X^\zztwo + \tilde u^Q) \para \tilde X) + 2 \mathD ((2 \tilde X^\zztwo + \tilde u^Q) \lpara \tilde X)\\ \nonumber
       &\quad + 2 \mathD (u^\sharp \reso \tilde X) + 2 \mathD ( (\tilde u' \mpara \tilde Q - \tilde u' \para \tilde Q) \reso \tilde X) + 2 \mathD C(\tilde u', \tilde Q, \tilde X) + 2 \mathD (\tilde u' \tilde \eta) \\
       &\quad + 2 \mathD (\tilde X^\zzone (2 \tilde X^\zztwo + \tilde u^Q)) + \mathD (2 \tilde X^\zztwo + \tilde u^Q)^2.
    \end{align}
    We also have
    \[
       \tilde u' = 2\tilde u^Q + 4 \tilde X^\zztwo.
    \]
    The fact that $(B_N(Q_N \reso X_N))$ converges not uniformly but only uniformly on intervals $[\delta,T]$ for $\delta>0$ poses no problem, because at the same time the sequence is uniformly bounded, so that given $\kappa>0$ we can fix a small $\delta >0$ with $\sup_N \| u_N(\delta) - \PC_N u_0^N\|_{-\beta} < \kappa$, and then use the uniform convergence of the data on $[\delta,T]$ and the convergence of $(\PC_N u_0^N)$ to $u_0$.
    
    Now observe that
    \[
       u = X + X^\zzone + 2 X^\zztwo + u^Q
    \]
    with $u^Q - \tilde u^Q = 2 \tilde X^\zztwo - 2 X^\zztwo = 4cQ$, and moreover
    \[
       \tilde u' = 2 \tilde u^Q + 4 X^\zztwo + 8cQ = 2 u^Q + 4 X^\zztwo.
    \]
    Plugging this as well as the specific form of the $\tilde X^\tau$ into~\eqref{eq:sasamoto-spohn pr2}, we get
    \begin{align*}
       \LL u^Q & = 4 c \mathD X + \LL X^\zzfour + 4 \LL X^\zzthreereso + 4 c \mathD X^\zzone + 8 c \mathD (Q\reso X) + 2 \mathD ((2  X^\zztwo + u^Q) \para X) \\
       &\quad + 2 \mathD ((2 X^\zztwo + u^Q) \lpara X)+ 2 \mathD (u^\sharp \reso X) + 2 \mathD ( (\tilde u' \mpara Q - \tilde u' \para Q) \reso X) \\
       &\quad + 2 \mathD C(\tilde u', Q, X) + 2 \mathD (\tilde u' (Q\reso X + c)) + 2 \mathD ( X^\zzone (2 X^\zztwo + u^Q)) + \mathD (2 X^\zztwo + u^Q)^2.
    \end{align*}
    Since
    \begin{align*}
       & 8 c \mathD (Q\reso X)  + 2 \mathD (u^\sharp \reso X) + 2 \mathD ( (\tilde u' \mpara Q - \tilde u' \para Q) \reso X) + 2 \mathD C(\tilde u', Q, X) + 2 \mathD (\tilde u' (Q\reso X)) \\
       &\hspace{70pt }= 8 c \mathD (Q\reso X)  + 2 \mathD (\tilde u^Q \reso X) = 2 \mathD (u^Q \reso X),
    \end{align*}
    we end up with
    \begin{align*}
       \LL u^Q & = \mathD u^2 - \LL X - \LL X^\zzone - 2 \LL X^\zztwo + 4 c \mathD X + 4 c \mathD X^\zzone + 2 c \mathD \tilde u' \\
       & = \mathD u^2 - \LL X - \LL X^\zzone - 2 \LL X^\zztwo + 4 c \mathD u,
    \end{align*}
    which completes the proof.
\end{proof}

It remains to study the convergence of the discrete stochastic data, which will be done in Section~\ref{sec:discrete stochastics} below.

\section{The stochastic driving terms}\label{sec:stochastics}

In this section we study the random fields
\[
   X, X^{\zzone}, X^{\zztwo}, X^{\zzthreereso}, X^{\zzfour}, Q\reso X
\]
which appear in the definition of $\X \in \Xrbe$. Our main results are the following two theorems, whose proofs will cover the next subsections.

\begin{theorem}\label{thm:rbe data}
   Let $\xi$ be a space-time white noise on $\R \times \T$ and define
   \begin{equation*}
     \begin{array}{rll}
       \LL X & = & \mathD \xi,\\
       \LL X^{\zzone} & = & \mathD (X^2),\\
       \LL X^{\zztwo} & = & \mathD (X  X^{\zzone} ),\\
       \LL X^{\zzthreereso} & = & \mathD (X^\zztwo  \reso X ),\\
       \LL X^{\zzfour} & = & \mathD (X^{\zzone}  X^{\zzone}),\\
       \LL Q & = & \mathD X ,
     \end{array}
   \end{equation*}
   all with zero initial condition except $X$ for which we choose the stationary initial condition
   \[
      X(0) = \int_{- \infty}^0 P_{- s} \mathD \xi (s) \mathd s.
   \]
   Then almost surely $\X = (X, X^\zzone, X^\zztwo, X^\zzthreereso, X^\zzfour, Q\circ X) \in \Xrbe$. If $\varphi \colon \R \to \R$ is a measurable, bounded, even function of compact support, such that $\varphi(0) = 1$ and $\varphi$ is continuous in a neighborhood of $0$, and if
   \[
      \xi_\varepsilon = \CF^{-1} ( \varphi(\varepsilon \cdot) \CF \xi) = \varphi(\varepsilon\mathD) \xi
   \]
   (here $\CF$ denotes the spatial Fourier transform) and $\X_\varepsilon = \Theta_{\mathrm{rbe}}(\xi_\varepsilon)$, then for all $T,p>0$
   \[
      \lim_{\varepsilon \to 0} \E[\| \X - \X_\varepsilon \|_{\Xrbe(T)}^p] = 0.
   \]
   Similarly, if $\tilde \X_\varepsilon = (\tilde X_\varepsilon, \tilde X^\zzone_\varepsilon, \tilde X^\zztwo_\varepsilon, \tilde X^\zzthreereso_\varepsilon, \tilde X^\zzfour_\varepsilon, \tilde Q_\varepsilon \circ \tilde X_\varepsilon)$ for $\tilde X_\varepsilon = \varphi(\varepsilon\mathD) X$ and
   \begin{equation*}
     \begin{array}{rll}
       \LL \tilde X^{\zzone}_\varepsilon & = & \mathD ((\tilde X_\varepsilon)^2),\\
       \LL \tilde X^{\zztwo}_\varepsilon & = & \mathD ( \tilde X_\varepsilon  \varphi(\varepsilon\mathD) X^{\zzone} ),\\
       \LL \tilde X^{\zzthreereso}_\varepsilon & = & \mathD (\varphi(\varepsilon\mathD) X^\zztwo  \reso \tilde X_\varepsilon ),\\
       \LL \tilde X^{\zzfour}_\varepsilon & = & \mathD (\varphi(\varepsilon\mathD) X^{\zzone} \varphi(\varepsilon\mathD) X^{\zzone}),\\
       \LL \tilde Q_\varepsilon & = & \mathD X_\varepsilon ,
     \end{array}
   \end{equation*}
   all with zero initial conditions, then for all $T,p > 0$
   \[
      \lim_{\varepsilon \to 0} \E[\| \X - \tilde \X_\varepsilon \|_{\Xrbe(T)}^p] = 0.
   \]
\end{theorem}

\begin{remark}
   The theorem would be easier to formulate if we assumed $\varphi$ to be continuous, of compact support, and with $\varphi(0)=1$. The reason why we chose the complicated formulation above is that we do not want to exclude the function $\varphi(x) = \1_{[-1,1]}(x)$.
\end{remark}

\begin{theorem}\label{thm:kpz data}
   Let $\xi$ be a space-time white noise on $\R \times \T$. Then there exists an element $\Y \in \Ykpz$ such that for every even, compactly supported function $\varphi \in C^1(\R,\R)$ with $\varphi(0) = 1$ there are diverging constants $c_\varepsilon^\zzone$ and $c_\varepsilon^\zzfour$ for which
   \[
      \lim_{\varepsilon \to 0} \E[\| \Y - \Y_\varepsilon \|_{\Ykpz(T)}^p] = 0
   \]
   for all $T,p>0$, where $\Y_\varepsilon = \Theta_{\mathrm{kpz}}( \varphi(\varepsilon \mathD) \xi, c_\varepsilon^\zzone, c_\varepsilon^\zzfour)$. Moreover, we have
   \[
      c_\varepsilon^\zzone = \frac{1}{4 \pi \varepsilon} \int_\R \varphi^2(x) \dd x.
   \]
\end{theorem}

\begin{remark}
   We will only worry about the construction of $\X$ and $\Y$. The convergence result then follows easily from the dominated convergence theorem, because since $\varphi$ is an even function all the symmetries in the kernels that we will use below also hold for the kernels corresponding to $\X_\varepsilon$, $\tilde \X_\varepsilon$, $\Y_\varepsilon$.
\end{remark}

\subsection{Kernels}

We can represent the white noise in terms of its spatial Fourier transform. More precisely, let $E = \Z \setminus \{ 0\}$ and let $\tilde W$ be a complex valued centered Gaussian process on $\R \times E$ defined by the covariance
\[
   \E \left[ \int_{\R \times E} f (\eta) \tilde W (\mathd \eta) \int_{\R \times E} g (\eta') \tilde W (\mathd \eta') \right] = (2 \pi)^{-1}   \int_{\R \times E} g (\eta_1) f (\eta_{- 1}) \mathd \eta_1,
\]
where $\eta_a = (s_a, k_a)$, $s_{- a} = s_a$, $k_{- a} = - k_a$ and the measure $\mathd \eta_a = \mathd s_a \mathd k_a$ is the product of the Lebesgue measure $\mathd s_a$ on $\R$ and of the counting measure $\mathd k_a$ on $E$. The functions $f, g$ are complex valued and in $L^2 (\R \times E)$. Then the process $\tilde \xi(\varphi) = \tilde W(\CF \varphi)$, where $\varphi \in L^2(\R \times \T)$ and $\CF$ denotes the spatial Fourier transform, is a white noise on $L^2_0(\R \times \T)$, the space of all $L^2$ functions $\varphi$ with $\int_\T \varphi(x,y) \dd y = 0$ for almost all $x$.

\textbf{Convention:} To eliminate many constants of the type $(2\pi)^p$ in the following calculations, let us rather work with $\sqrt{2\pi} \tilde W$, which we denote by the same symbol $\tilde W$. Of course all qualitative results that we prove for this transformed noise stay true for the original noise, and we only have to pay attention in Theorem~\ref{thm:kpz data} to get the constant $c_\varepsilon^\zzone$ right.

The process $X$ then has the following representation as an integral
\[
   X (t, x) = \int_{\R \times E} e^{i kx} H_{t - s} (-k) \tilde W (\mathd\eta),
\]
where $\eta = (s, k) \in \R \times E$ and
\[
    h_t (k) = e^{- k^2 t} \1_{t \geqslant 0}, \hspace{2em} H_t (k) = i k h_t (k).
 \]
 This means that $H_{t-s}(-k) = - H_{t-s}(k)$, and it will simplify the notation if we work with $W = -\tilde W$ and $\xi =  -\tilde \xi$, which of course have the same distribution as $\tilde W$ and $\tilde \xi$ and for which
 \[
   X (t, x) = \int_{\R \times E} e^{i kx} H_{t - s} (k) W (\mathd\eta).
\]
The space Fourier transform $\hat{X} (t, k) = \hat{X}_t (k)$ of $X(t, \cdot)$ reads
\[
   \hat{X} (t, k) = \int_{\R} e^{i kx} H_{t - s} (k) W_k (\mathd s),
\]
where $W_{k'} (\mathd s) = \int_E \delta_{k, k'} W (\mathd s \mathd k)$ is
just a countable family of complex time white noises satisfying $W_k (\mathd
s)^{\ast} = W_{- k} (\mathd s)$ and $\E [W_{k'} (\mathd s) W_k
(\mathd s')] =  \delta_{k,- k'} \delta (s - s') \mathd s \mathd s'$.

Note that if $s \leqslant t$
\begin{equation}\label{eq:ou-cov-int}
  \int_{\R} H_{s - \sigma} (k) H_{t - \sigma} (- k) \mathd \sigma = \frac{e^{- k^2 | t - s |}}{2},
\end{equation}
from where we read the covariance of $X$:
\begin{align*}
   \E [X (t, x) X (s, y)] & = \E \left[ \int_{\R \times E} e^{i k_1 x} H_{t - s_1} (k_1) W (\mathd \eta_1) \int_{\R \times E} e^{i k_2 y} H_{s - s_2} (k_2) W (\mathd \eta_2) \right] \\
   & = \int_E \mathd k_1 e^{i k_1 (x - y)} \int_{\R} H_{t - s_1} (k_1) H_{s - s_1} (- k_1) \mathd s_1 \\
   & = \int_E e^{i k_1 (x - y)} \frac{e^{- k_1^2 | t - s |}}{2} \mathd k_1 = \frac{1}{2} p_{| t - s |} (x - y),
\end{align*}
where $p_t (x)$ is the kernel of the heat semigroup: $P_t f = p_t \ast f =
\int_{\T} p_t (\cdummy - y) f (y) \mathd y$. In Fourier space we have
\[ \E [\hat{X}_t (k) \hat{X}_s (k')] = \delta_{k + k' = 0} \frac{e^{-
   k^2 | t - s |}}{2} \]
as expected.

These notations and preliminary results will be useful below in relation to
the representation of elements in the chaos of $W$ and the related Gaussian
computations. Recall that $X^{\bullet} = X$ and $X^{(\tau_1 \tau_2)} = B
(X^{\tau_1}, X^{\tau_2})$. Then
\[ X^{\tau} (t, x) = \int_{(\R \times E)^n} G^{\tau} (t, x,
   \eta_{\tau}) \prod_{i = 1}^n W (\mathd \eta_i) \]
where $n = d (\tau) + 1$, $\eta_{\tau} = \eta_{1 \cdots n} = (\eta_1, \ldots,
\eta_n) \in (\R \times E)^n$ and $\mathd \eta_{\tau} = \mathd \eta_{1
\cdots n} = \mathd \eta_1 \cdots \mathd \eta_n$. Here we mean that each of the
$X^{\tau}$ is a polynomial in the Gaussian variables $W (\mathd \eta_i)$, and in
the next section we study how these polynomials decompose into the chaoses of
$W$. For the moment we are interested in the analysis of the kernels
$G^{\tau}$ involved in this representation. These kernels are defined
recursively by
\[
   G^{\bullet} (t, x, \eta) = e^{i kx} H_{t - s} (k),
\]
and then
\begin{align*}
   G^{(\tau_1 \tau_2)} (t, x, \eta_{(\tau_1 \tau_2)}) & = B (G^{\tau_1} (\cdot, \cdot, \eta_{\tau_1}), G^{\tau_2} (\cdot, \cdot, \eta_{\tau_2})) (t, x) \\
   & = \int_0^t \mathd \sigma \mathD P_{t - \sigma} (G^{\tau_1} (\sigma, \cdot, \eta_{\tau_1}) G^{\tau_2} (\sigma, \cdot, \eta_{\tau_2})) (x).
\end{align*}
In the first few cases this gives
\begin{align*}
   G^{\zzone} (t, x, \eta_{12}) & = \int_0^t \mathd \sigma \mathD P_{t - \sigma} (G^{\bullet} (\sigma, \cdot, \eta_1) G^{\bullet} (\sigma, \cdot, \eta_2))(x) \\
   & = e^{i k_{[12]} x} \int_0^t \mathd \sigma H_{t - \sigma} (k_{[12]}) H_{\sigma - s_1} (k_1) H_{\sigma - s_2} (k_2),
\end{align*}
where we set $k_{[1 \cdots n]} = k_1 + \cdots + k_n$, and
\begin{align*}
   &G^{\zztwo} (t, x, \eta_{123}) = \int_0^t \mathd \sigma \mathD P_{t - \sigma} (G^{\zzone} (\sigma, \cdot, \eta_{12}) G^{\bullet} (\sigma, \cdot,
   \eta_3)) (x) \\
   &\hspace{15pt} = e^{i k_{[123]} x} \int_0^t \mathd \sigma H_{t - \sigma} (k_{[123]}) \left( \int_0^{\sigma} \mathd \sigma' H_{\sigma - \sigma'} (k_{[12]})
   H_{\sigma' - s_1} (k_1) H_{\sigma' - s_2} (k_2) \right) H_{\sigma - s_3} (k_3).
\end{align*}
In both cases, the kernel has the factorized form
\begin{equation}
  G^{\tau} (t, x, \eta_{\tau}) = e^{i k_{[\tau]} x} H^{\tau} (t, \eta_{\tau}),
  \label{eq:G-kernel-factoriz}
\end{equation}
where we further denote $k_{[\tau]} = k_{[1 \cdots n]} = k_1 + \cdots + k_n$,
and this factorization holds for all $G^{\tau}$. In fact, it is easy to show
inductively that
\[
   G^{(\tau_1 \tau_2)} (t, x, \eta_{(\tau_1 \tau_2)}) = e^{i k_{[\tau]} x} \int_0^t \mathd \sigma H_{t - \sigma} (k_{[\tau]}) H^{\tau_1} (\sigma, \eta_{\tau_1}) H^{\tau_2} (\sigma, \eta_{\tau_2}),
\]
from where we read
\begin{align*}
   G^{\zzthree} (t, x, \eta_{1234}) & = e^{i k_{[1234]} x} \int_0^t \mathd \sigma H_{t - \sigma} (k_{[1234]}) \int_0^{\sigma} \mathd \sigma' H_{\sigma - \sigma'} (k_{[123]}) \times \\
   &\quad\times \Big( \int_0^{\sigma'} \mathd \sigma'' H_{\sigma' - \sigma''} (k_{[12]}) H_{\sigma'' - s_1} (k_1) H_{\sigma'' - s_2} (k_2) \Big) H_{\sigma' - s_3} (k_3) H_{\sigma - s_4} (k_4)
\end{align*}
and thus
\begin{align*}
   G^{\zzthreereso} (t, x, \eta_{1234}) & = e^{i k_{[1234]} x} \int_0^t \mathd \sigma H_{t - \sigma} (k_{[1234]}) \psi_\circ(k_{[123]},k_4) \int_0^{\sigma} \mathd \sigma' H_{\sigma - \sigma'} (k_{[123]}) \times \\
   &\quad\times \Big( \int_0^{\sigma'} \mathd \sigma'' H_{\sigma' - \sigma''} (k_{[12]}) H_{\sigma'' - s_1} (k_1) H_{\sigma'' - s_2} (k_2) \Big) H_{\sigma' - s_3} (k_3) H_{\sigma - s_4} (k_4),
\end{align*}
where we recall that $\psi_\circ(k,\ell) = \sum_{|i-j| \leqslant 1} \rho_i(k) \rho_j (\ell)$. Similarly, we have
\begin{align*}
   G^{\zzfour} (t, x, \eta_{1234}) & = e^{i k_{[1234]} x} \int_0^t \mathd \sigma
   \int_0^{\sigma} \mathd \sigma' \int_0^{\sigma} \mathd \sigma'' H_{t -
   \sigma} (k_{[1234]}) H_{\sigma - \sigma'} (k_{[12]}) H_{\sigma - \sigma''}
   (k_{[34]}) \times \\
   &\quad \times H_{\sigma' - s_1} (k_1) H_{\sigma' - s_2} (k_2) H_{\sigma'' - s_3}  (k_3) H_{\sigma'' - s_4} (k_4).
\end{align*}

\subsection{Chaos decomposition and diagrammatic representation}\label{sec:chaos}

The representation
\[ X^{\tau} (t, x) = \int_{(\R \times E)^n} G^{\tau} (t, x,
   \eta_{\tau}) \prod_{i = 1}^n W (\mathd \eta_i) \]
is not very useful for the analysis of the properties of the random fields
$X$. It is more meaningful to separate the components in different chaoses.
Denote by
\[ \int_{(\R \times E)^n} f (\eta_{1 \cdots n}) W (\mathd \eta_{1
   \cdots n}) \]
a generic element of the $n$-th chaos of the white noise $W$ on $(\R
\times E)$. We find convenient not to symmetrize the kernels in the chaos
decomposition. If we follow this convention then we should recall that the
variance of the chaos elements will be given by
\[
   \E \left[ \left| \int_{(\R \times E)^n} f (\eta_{1 \cdots n}) W (\mathd \eta_{1 \cdots n}) \right|^2 \right] =  \sum_{\sigma \in  \mathcal{S}_n} \int_{(\R \times E)^n} f (\eta_{1 \cdots n}) f  (\eta_{(- \sigma (1)) \cdots (- \sigma (n))}) \mathd \eta_{1 \cdots n}
\]
where the sum runs over all the permutations of $\{ 1, \ldots, n \}$ and where we introduce the notation $\eta_{- 1} = (s_1, - k_1)$ to describe the contraction of the Gaussian variables. By the Cauchy-Schwarz inequality
\[
   \E \left[ \left| \int_{(\R \times E)^n} f (\eta_{1 \cdots n}) W (\mathd \eta_{1 \cdots n}) \right|^2 \right] \leqslant (n!)   \int_{(\R \times E)^n} | f (\eta_{1 \cdots n}) |^2 \mathd \eta_{1 \cdots n},
\]
so that for the purpose of bounding the variance of the chaoses it is enough
to bound the $L^2$ norm of the unsymmetrized kernels. The general explicit
formula for the chaos decomposition of a polynomial
\[ \int_{(\R \times E)^n} f (\eta_{1 \cdots n}) \prod_{i = 1}^n W
   (\mathd \eta_i) \]
is given by
\[ \int_{(\R \times E)^n} f (\eta_{1 \cdots n}) \prod_{i = 1}^n W
   (\mathd \eta_i) = \sum_{\ell = 0}^n \int_{(\R \times E)^k}
   f_{\ell} (\eta_{1 \cdots \ell}) W (\mathd \eta_{1 \cdots \ell}) \]
with $f_{\ell} (\eta_{1 \cdots \ell}) = 0$ if $n - \ell$ is odd. To give the expression for $f_{\ell}$ with $n - \ell$ even, let us introduce some notation: We write $\mathcal{S} (\ell, n)$ for the shuffle permutations of $\{1, \ldots, n \}$ which leave the order of the first $\ell$ and the last $n -\ell$ terms intact. We also write
\[
   \eta \sqcup \upsilon = \eta_{1 \ldots \ell} \sqcup \upsilon_{1 \ldots m} = (\eta_1, \ldots, \eta_{\ell}, \upsilon_1, \ldots, \upsilon_m)
\]
for the concatenation of $\eta$ and $\upsilon$. So if $n - \ell = 2 m$ for
some $m$, then
\[
   f_{\ell} (\eta_{1 \ldots \ell}) = 2^{-m}\sum_{\sigma \in \mathcal{S} (\ell, n)} \sum_{\substack{\tau \in \mathcal{S}(m, 2m),  \\ \tilde{\tau} \in \mathcal{S} (m)}} \int_{(\R \times E)^m} f (\sigma [\eta_{1 \ldots \ell} \sqcup \tau(\eta_{(\ell + 1) \ldots (\ell + m)} \sqcup \tilde{\tau} (\eta_{- (\ell + 1) \ldots - (\ell + m)}))]) \mathd \eta_{(\ell + 1) \ldots (\ell + m)},
\]
where $\sigma (\eta_{1 \ldots n}) = \eta_{\sigma (1) \ldots \sigma (n)}$. This is just a formal manner of writing that we sum over all ways in which $2m$ of the $2m + \ell$ variables can be paired and integrated out. For example, we have
\[ \text{} W (\mathd \eta_1) W (\mathd \eta_2) = W (\mathd \eta_{12}) + \delta
   (\eta_1 + \eta_{- 2}) \mathd \eta_1 \mathd \eta_2 . \]
In general we will denote by $G^{\tau}_{\ell}$ the kernel of the $\ell$-th
chaos arising from the decomposition of $X^{\tau}$:
\[ X^{\tau} (t, x) = \sum_{\ell = 0}^n \int_{(\R \times E)^{\ell}}
   G^{\tau}_{\ell} (t, x, \eta_{1 \cdots \ell}) W (\mathd \eta_{1 \cdots
   \ell}) . \]
Terms $X^{\tau}$ of odd degree have zero mean by construction while the terms
of even degree have zero mean due to the fact that if $d (\tau) = 2 n$ we have
\[
   \E [X^{\tau} (t, x)] = 2^{-n} \sum_{\substack{\tau \in \mathcal{S}(n,2n) \\ \tilde \tau \in \mathcal{S}_n}} \int_{(\R \times E)^n} G^{\tau} (t, x, \tau (\eta_{1 \cdots n }\sqcup \tilde \tau(\eta_{(- 1) \cdots (- n)}))) \mathd \eta_{1 \ldots n}.
\]
But now $k_{[1 \cdots n (- \tilde \tau(1)) \cdots (- \tilde \tau(n))]} = k_1 + \cdots + k_n - k_{\tilde \tau(1)} \cdots
- k_{\tilde \tau(n)} = 0$ and we always have $G^{\tau} (t, x, \eta_{1 \cdots 2 n}) \propto
k_{[1 \cdots (2 n)]}$, which implies that
\[
   G^{\tau} (t, x, \tau (\eta_{1 \cdots n }\sqcup \tilde \tau(\eta_{(- 1) \cdots (- n)}))) = 0.
\]
This is a special simplification of considering the stochastic Burgers
equation instead of the KPZ equation. Later we will study the kernel functions
for the KPZ equation to understand some subtle cancellations which appear in
the terms belonging to the $0$-th chaos.

Applying these considerations to the first nontrivial case given by
$X^{\zzone}$, we obtain
\[ X^{\zzone} (t, x) = \int_{(\R \times E)^2} G^{\zzone} (t, x,
   \eta_{12}) W (\mathd \eta_1 \mathd \eta_2) + G^{\zzone}_0 (t, x) \]
with
\[ G^{\zzone}_0 (t, x) = \int_{(\R \times E)^2} G^{\zzone} (t, x,
   \eta_{1 (- 1)}) \mathd \eta_1 . \]
But as already remarked
\[ G^{\zzone} (t, x, \eta_{1 (- 1)}) = e^{i k_{[1 (- 1)]} x} \int_0^t H_{t -
   \sigma} (k_{[1 (- 1)]}) H_{\sigma - s_1} (k_1) H_{\sigma - s_2} (k_{- 1})
   \mathd \sigma = 0 \]
since $H_{t - \sigma} (0) = 0$. Consider the next term
\[
   X^{\zztwo} (t, x) = \int_{(\R \times E)^3} G^{\zztwo} (t, x, \eta_{123}) W (\mathd \eta_1 \mathd \eta_2 \mathd \eta_3) + \int_{\R \times E} G^{\zztwo}_1 (t, x, \eta_1) W (\mathd \eta_1).
\]
In this case we have three possible contractions contributing to
$G^{\zztwo}_1$, which result in
\[ G^{\zztwo}_1 (t, x, \eta_1) =  \int_{\R \times E} (G^{\zztwo} (t,
   x, \eta_{12 (- 2)}) + G^{\zztwo} (t, x, \eta_{21 (- 2)}) + G^{\zztwo} (t,
   x, \eta_{2 (- 2) 1})) \mathd \eta_2 . \]
But note that \ $G^{\zztwo} (t, x, \eta_{2 (- 2) 1}) = 0$ since, as above,
this kernel is proportional to $k_{2 (- 2)} = 0$. Moreover, by symmetry
$G^{\zztwo} (t, x, \eta_{12 (- 2)}) = G^{\zztwo} (t, x, \eta_{21 (- 2)})$ and
we remain with
\[
   G^{\zztwo}_1 (t, x, \eta_1) = 2 G^{\zzfive} (t, x, \eta_1) = 2 \int_{\R \times E} G^{\zztwo} (t, x, \eta_{12 (- 2)}) \mathd \eta_2,
\]
where we introduced the intuitive notation $G^{\zzfive} (t, x, \eta_1)$ which
is useful to keep track graphically of the Wick contractions of the kernels
$G^{\tau}$ by representing them as arcs between leaves of the binary tree.

Now an easy computation gives
\[ G^{\zzfive} (t, x, \eta_1) = e^{i k_1 x} \int_0^t \mathd \sigma
   \int_0^{\sigma} \mathd \sigma' H_{t - \sigma} (k_1) H_{\sigma' - s_1} (k_1)
   V^{\zzfive} (\sigma - \sigma', k_1), \]
where
\[
   V^{\zzfive} (\sigma, k_1) = \int H_{\sigma} (k_{[12]}) H_{\sigma - s_2} (k_2) H_{- s_2} (k_{- 2}) \mathd \eta_2 = \int \mathd k_2 H_{\sigma} (k_{[12]}) \frac{e^{- | \sigma | k_2^2}}{2}.
\]
We call the functions $V_n^{\tau}$ vertex functions. They are useful to
compare the behavior of different kernels.

By similar arguments we can establish the decomposition for the last two
terms, that is
\[
   X^{\zzthreereso} (t, x) = \int_{(\R \times E)^3} G^{\zzthreereso} (t, x, \eta_{1234}) W (\mathd \eta_{1234}) + \int_{(\R \times E)^2} G^{\zzthreereso}_2 (t, x, \eta_{12}) W (\mathd \eta_{12})
\]
and
\[ X^{\zzfour} (t, x) = \int_{(\R \times E)^3} G^{\zzfour} (t, x,
   \eta_{1234}) W (\mathd \eta_{1234}) + \int_{(\R \times E)^2}
   G^{\zzfour}_2 (t, x, \eta_{12}) W (\mathd \eta_{12}) \]
with
\begin{align*}
   G^{\zzthreereso}_2 (t, x, \eta_{12}) & = \int_{\R \times E} (G^{\zzthreereso} (t, x, \eta_{123 (- 3)}) + 2 G^{\zzthreereso} (t, x, \eta_{13 (- 3) 2}) + 2 G^{\zzthreereso} (t, x, \eta_{132 (- 3)})) \mathd \eta_3 \\
   & = (G^{\zzsixreso} (t, x, \eta_{12}) + 2 G^{\zzsevenreso} (t, x, \eta_{12}) + 2 G^{\zzeightreso} (t, x, \eta_{12}))
\end{align*}
and
\[
   G^{\zzfour}_2 (t, x, \eta_{12}) = 4 \int_{\R \times E} G^{\zzfour}(t, x, \eta_{13 (- 3) 2}) \mathd \eta_3 = 4 G^{\zznine} (t, x, \eta_{12}) .
\]
Here the contributions associated to $G^{\zzsixreso} (t, x, \eta_{12})$ and $G^{\zzsevenreso} (t, x, \eta_{12})$ are ``reducible'' since they can be conveniently factorized as follows:
\begin{align*}
   G^{\zzsixreso} (t, x, \eta_{12}) & = \int_{\R \times E} G^{\zzthreereso} (t, x, \eta_{123 (- 3)}) \mathd \eta_3 \\
   & = e^{i k_{[12]} x} \int_0^t \mathd \sigma \int_0^{\sigma} \mathd \sigma' \int_0^{\sigma'} \mathd \sigma'' H_{t - \sigma} (k_{[12]}) H_{\sigma' - \sigma''} (k_{[12]}) H_{\sigma'' - s_1} (k_1) \\
   &\hspace{130pt} \times H_{\sigma'' - s_2} (k_2)V^{\zzfivereso} (\sigma - \sigma', k_{[12]})
\end{align*}
with
\begin{align*}
   V^{\zzfivereso} (\sigma, k_1) & = \int  \psi_\circ(k_{[12]}, k_{-2}) H_{\sigma} (k_{[12]}) H_{\sigma - s_2} (k_2) H_{- s_2} (k_{- 2}) \mathd \eta_2 \\
   &= \int \mathd k_2 \psi_\circ(k_{[12]}, k_{-2}) H_{\sigma} (k_{[12]}) \frac{e^{- | \sigma | k_2^2}}{2}.
\end{align*}
Also,
\begin{align*}
   G^{\zzsevenreso} (t, x, \eta_{12}) & = \int_{\R \times E} G^{\zzthreereso} (t, x, \eta_{13 (- 3) 2}) \mathd \eta_3 \\
   & = e^{i k_{[12]} x} \psi_\circ(k_1, k_2) \int_0^t \mathd \sigma H_{t - \sigma} (k_{[12]}) H_{\sigma - s_2} (k_2) \int_0^{\sigma} \mathd \sigma' \int_0^{\sigma'} \mathd \sigma'' H_{\sigma - \sigma'} (k_1) \\
   &\hspace{220pt} \times H_{\sigma'' - s_1} (k_1) V^{\zzfive} (\sigma' - \sigma'', k_1) \\
   & = e^{i k_{[12]} x} \psi_\circ(k_1, k_2) \int_0^t \mathd \sigma H_{t - \sigma} (k_{[12]}) H_{\sigma - s_2} (k_2) e^{- i k_1 x} G^{\zzfive} (\sigma, x, \eta_1).
\end{align*}
On the other side, the term $G^{\zzeightreso} (t, x, \eta_{12})$ cannot be reduced to a form involving $V^{\zzfive}$ or $V^{\zzfivereso}$, and instead we have for it
\begin{align*}
   G^{\zzeightreso} (t, x, \eta_{12}) & = \int_{\R \times E} G^{\zzthreereso} (t, x, \eta_{132 (- 3)}) \mathd \eta_3 \\
   & = e^{i k_{[12]} x} \int_0^t \mathd \sigma \int_0^{\sigma} \mathd \sigma' \int_0^{\sigma'} \mathd \sigma'' H_{t - \sigma} (k_{[12]}) H_{\sigma' -  s_2} (k_2) H_{\sigma'' - s_1} (k_1) \times \\
   &\hspace{150pt} \times V^{\zzeightreso} (\sigma - \sigma', \sigma - \sigma'', k_{12})
\end{align*}
with
\begin{align*}
   V^{\zzeightreso} (\sigma - \sigma', \sigma - \sigma'', k_{12}) & = \int_{\R \times E} \mathd \eta_3 \psi_\circ(k_{[132]}, k_{-3}) H_{\sigma - s_3} (k_{- 3}) H_{\sigma - \sigma'} (k_{[132]}) \times \\
   &\hspace{150pt} \times H_{\sigma' - \sigma''} (k_{[13]}) H_{\sigma'' - s_3} (k_3) \\
   & = \int_E \mathd k_3 \psi_\circ(k_{[123]}, k_{-3}) H_{\sigma - \sigma'} (k_{[123]}) H_{\sigma' - \sigma''} (k_{[13]}) \frac{e^{- k_3^2 | \sigma - \sigma'' |}}{2}.
\end{align*}
Similarly we have for $G^{\zznine}$
\begin{align*}
   G^{\zznine} (t, x, \eta_{12}) & = \int_{\R \times E} G^{\zzfour} (t, x, \eta_{13 (- 3) 2}) \mathd \eta_3 \\
   & = e^{i k_{[12]} x} \int_0^t \mathd \sigma \int_0^{\sigma} \mathd \sigma' \int_0^{\sigma} \mathd \sigma'' H_{t - \sigma} (k_{[12]}) H_{\sigma' - s_1} (k_1) H_{\sigma'' - s_2} (k_2) \times \\
   &\hspace{150pt} \times V^{\zznine} (\sigma - \sigma', \sigma - \sigma'', k_{12})
\end{align*}
with a vertex function which we choose to write in symmetrized form:
\begin{align*}
   V^{\zznine} (\sigma - \sigma', \sigma - \sigma'', k_{12}) & = \frac{1}{2} \int_{\R \times E} \mathd \eta_3 H_{\sigma - \sigma'} (k_{[13]}) H_{\sigma - \sigma''} (k_{[2 (- 3)]}) H_{\sigma' - s_3} (k_3) H_{\sigma'' - s_3} (k_{- 3}) \\
   &\quad + \frac{1}{2} \int_{\R \times E} \mathd \eta_3 H_{\sigma - \sigma''} (k_{[13]}) H_{\sigma - \sigma'} (k_{[2 (- 3)]}) H_{\sigma' - s_3}(k_3) H_{\sigma'' - s_3} (k_{- 3}) \\
   & = \frac{1}{2} \int_E \mathd k_3 H_{\sigma - \sigma'} (k_{[13]}) H_{\sigma - \sigma''} (k_{[2 (- 3)]})  \frac{e^{- k_3^2 | \sigma' - \sigma'' |}}{2} \\
   &\quad + \frac{1}{2} \int_E \mathd k_3 H_{\sigma - \sigma''} (k_{[13]}) H_{\sigma - \sigma'} (k_{[2 (- 3)]})  \frac{e^{- k_3^2 | \sigma' - \sigma'' |}}{2}.
\end{align*}

\subsection{Feynman diagrams}

While here we only need to analyze few of the random fields $X^{\tau}$, we
think it useful to give a general perspective on their structure and in
particular on the structure of the kernels $G^{\tau}$ and on the Wick
contractions. Doing this will build a link with the theoretical physics
methods and with quantum field theory (QFT), in particular with the Martin--Siggia--Rose response
field formalism which has been used in the QFT analysis of the stochastic Burgers equations 
since the seminal work of Forster, Nelson, and Stephen~\cite{Forster1977}.

The explicit form of the kernels $G^{\tau}$ can be described in terms of
Feynman diagrams and the associated rules. To each kernel $G^{\tau}$ we can
associate a graph which is isomorphic to the tree $\tau$ and this graph can be
mapped with Feynman rules to the explicit functional form of $G^{\tau}$. The
algorithm goes as follows: consider $\tau$ as a graph, where each edge and
each internal vertex (i.e. not a leaf) are drawn as
\[
   \vcenter{\hbox{\resizebox{1cm}{!}{\includegraphics{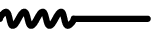}}}}
   \hspace{1em} \text{and} \hspace{1em}
   \vcenter{\hbox{\resizebox{1cm}{!}{\includegraphics{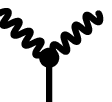}}}}.
\]
To the trees \resizebox{1em}{!}{\includegraphics{trees-1.eps}},
\resizebox{1em}{!}{\includegraphics{trees-2.eps}},
\resizebox{1em}{!}{\includegraphics{trees-3.eps}},
\resizebox{1.2em}{!}{\includegraphics{trees-4.eps}} we associate,
respectively, the diagrams
\[ \text{$\begin{array}{c|c|c|c}
     &  &  & \\
     \vcenter{\hbox{\resizebox{1em}{!}{\includegraphics{trees-1.eps}}}} &
     \vcenter{\hbox{\resizebox{1em}{!}{\includegraphics{trees-2.eps}}}} &
     \vcenter{\hbox{\resizebox{1em}{!}{\includegraphics{trees-3.eps}}}} &
     \vcenter{\hbox{\resizebox{1em}{!}{\includegraphics{trees-4.eps}}}}\\
     &  &  & \\
     \hline
     &  &  & \\
     \vcenter{\hbox{\resizebox{.5cm}{!}{\includegraphics{diagrams-4.eps}}}} &
     \vcenter{\hbox{\resizebox{!}{1.5cm}{\includegraphics{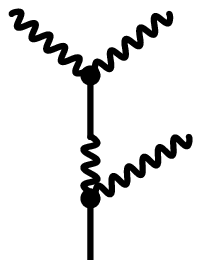}}}} &
     \vcenter{\hbox{\resizebox{!}{1.5cm}{\includegraphics{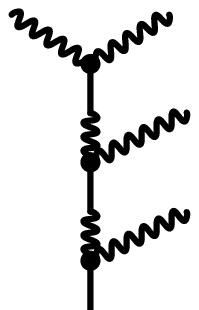}}}} &
     \vcenter{\hbox{\resizebox{!}{1.5cm}{\includegraphics{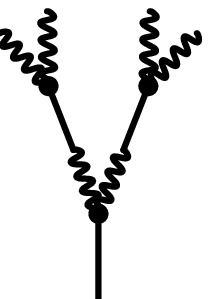}}}}\\
     &  &  & 
   \end{array}$} . \]
These diagrams correspond to kernels via the following rules: each internal
vertex comes with a time integration and a factor $(i k)$,
\[ \vcenter{\hbox{\includegraphics{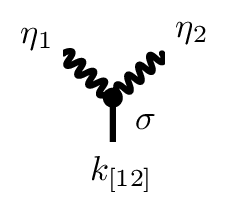}}} \hspace{2em}
   \longrightarrow \hspace{2em} (i k_{[12]}) \int_{\R_+} \mathd
   \sigma . \]
Each external wiggly line is associated to a variable $\eta_i$ and a factor
$H_{\sigma - s_i} (k_i)$, where $\sigma$ is the integration variable of the
internal vertex to which the line is attached. Each response line is
associated to a factor $h_{\sigma - \sigma'} (k)$, where $k$ is the moment
carried by the line and $\sigma, \sigma'$ are the time labels of the vertices
to which it is attached:
\[ \vcenter{\hbox{\includegraphics{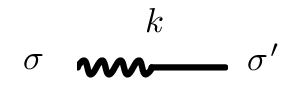}}} \hspace{2em}
   \longrightarrow \hspace{2em} h_{\sigma - \sigma'} (k) . \]
Note that these lines carry information about the casual propagation. Finally,
the outgoing line always carries a factor $h_{t - \sigma} (k)$, where $k$ is
the outgoing momentum and is $\sigma$ the time label of the vertex to which
the line is attached. For example:
\[ \vcenter{\hbox{\includegraphics{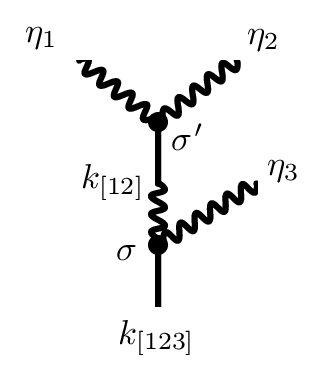}}} \hspace{1em}
   \longrightarrow \hspace{1em} \begin{array}{c}
     \int_{\R_+^2} \mathd \sigma \mathd \sigma' (i k_{[123]}) h_{t -
     \sigma} (k_{[123]}) (i k_{[12]}) h_{\sigma - \sigma'} (k_{[12]}) \times\\
     \times H_{\sigma' - s_1} (k_1) H_{\sigma' - s_2} (k_2) H_{\sigma - s_3}
     (k_3) .
   \end{array} \]
Once given a diagram, the associated Wick contractions are obtained by all
possible pairings of the wiggly lines. To each of these pairings we associate
the corresponding correlation function of the Ornstein-Uhlenbeck process and
an integration over the momentum variable carried by the line:
\[ \vcenter{\hbox{\includegraphics{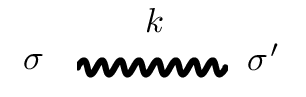}}} \hspace{2em}
   \longrightarrow \hspace{2em} \int_E \mathd k \frac{e^{- k^2 | \sigma -
   \sigma' |}}{2} . \]
So for example we have
\[
   G^{\zztwo}_1 (t, x, \eta_1) = 2 \times \vcenter{\hbox{\includegraphics{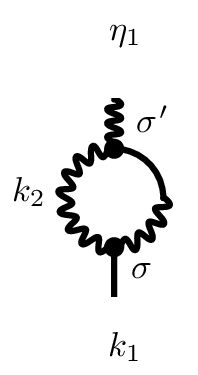}}} \small{= 2 \int_{\R_+^2} \mathd \sigma \mathd \sigma' H_{t - \sigma} (k_1) \int_E \mathd k_2  \frac{e^{-  k^2_2 | \sigma - \sigma' |}}{2} H_{\sigma - \sigma'} (k_{[1 (- 2)]}) H_{\sigma' - s_1} (k_1)},
\]
which coincides with the expression obtained previously. Note that we also
have to take into account the multiplicities of the different ways in which
each graph can be obtained. The contractions arising from $G^{\zzthree}$ and
$G^{\zzfour}$ result in the following set of diagrams:

\[ G^{\zzfour}_2 = 4 G^{\zznine} = 4 \times
   \vcenter{\hbox{\resizebox{!}{1.5cm}{\includegraphics{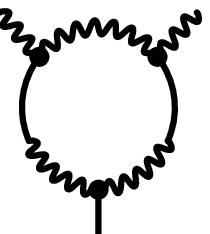}}}} \hspace{1em}
   \text{and} \hspace{1em} G^{\zzseven} =
   \vcenter{\hbox{\resizebox{!}{2cm}{\includegraphics{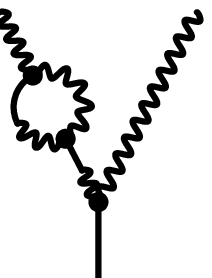}}}}, \hspace{2em}
   G^{\zzsix} =
   \vcenter{\hbox{\resizebox{!}{2.5cm}{\includegraphics{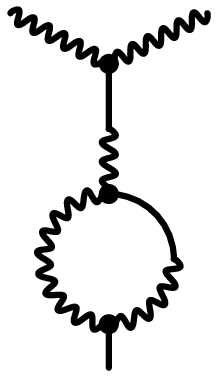}}}},
   \hspace{2em} G^{\zzeight} =
   \vcenter{\hbox{\resizebox{!}{1.5cm}{\includegraphics{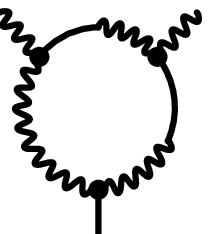}}}} . \]

The diagrammatic representation makes pictorially evident what we already have
remarked with explicit computations: $G^{\zzseven}$ and $G^{\zzsix}$ are
formed by the union of two graphs:
\[ \vcenter{\hbox{\resizebox{!}{1.0cm}{\includegraphics{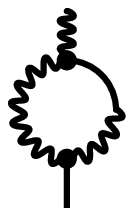}}}} \ \ and \ \
   \vcenter{\hbox{\resizebox{!}{.7cm}{\includegraphics{diagrams-4.eps}}}}, \]
while the kernel $G^{\zzeight}$ cannot be decomposed in such a way and it has
a shape very similar to that of $G^{\zznine}$.

\subsection{Bounds on the vertex functions}\label{ssec:vertex functions}

Up to now we have obtained explicit expressions for the kernels appearing in
the Wick contractions. These kernels feature vertex functions. Let us start
the analysis by bounding them. Consider for example the first non-trivial one,
\[ G^{\zzfive} (t, x, \eta_1) = \int_0^t \mathd \sigma \int_0^{\sigma} \mathd
   \sigma' H_{t - \sigma} (k_1) H_{\sigma' - s_1} (k_1) V^{\zzfive} (\sigma -
   \sigma', k_1) . \]
We have
\[ |G^{\zzfive} (t, x, \eta_1) | \leqslant \int_0^t \mathd \sigma
   \int_0^{\sigma} \mathd \sigma' M_{t - \sigma} (k_1) M_{\sigma' - s_1} (k_1)
   |V^{\zzfive} (\sigma - \sigma', k_1) |, \]
where $M_t (k) = | H_t (k) | = | k | \exp (- k^2 t) \1_{t \geqslant
0}$. Since the integrand is positive, we can extend the domain of integration
to obtain an upper bound:
\[ |G^{\zzfive} (t, x, \eta_1) | \leqslant \int_{\R}^{} \mathd \sigma
   \int_{\R}^{} \mathd \sigma' M_{t - \sigma} (k_1) M_{\sigma' - s_1}
   (k_1) |V^{\zzfive} (\sigma - \sigma', k_1) | = Z (t, x, \eta_1). \]
The quantity $Z (t, x, \eta_1)$ is given by a multiple convolution, and since
we are interested in $L^2$ bounds of $G^{\zzfive} (t, x, \eta_1)$, we can pass
to Fourier variables in time to decouple the various dependencies:
\[ \int_{\R \times E} |G^{\zzfive} (t, x, \eta_1) |^2 \mathd \eta_1
   \leqslant \int_{\R \times E} |Z (t, x, \eta_1) |^2 \mathd \eta_1 =
   (2 \pi)^{- 1} \int_{\R \times E} | \hat{Z} (t, x, \theta_1) |^2
   \mathd \theta_1, \]
where $\theta_1 = (\omega_1, k_1)$ and
\[ \hat{Z} (t, x, \theta_1) = \int_{\R} e^{- i \omega_1 s_1} Z (t, x,
   \eta_1) \mathd s_1 = \hat{M} (- \theta_1) \hat{M} (- \theta_1)
   \mathcal{V}^{\zzfive} (- \omega_1, k_1) \]
with
\[ \mathcal{V}^{\zzfive} (\omega_1, k_1) = \int_{\R} e^{- i \omega_1
   \sigma} |V^{\zzfive} (\sigma, k_1) | \mathd \sigma . \]
Then
\begin{align*}
   \int_{\R \times E} | \hat{Z} (t, x, \theta_1) |^2 \mathd \theta_1 & \leqslant \int_{\R \times E} | \hat{M} (\theta_1) \hat{M} (\theta_1) |^2 |\mathcal{V}^{\zzfive} (\omega_1, k_1) |^2 \mathd \theta_1\\
   & \leqslant \int_{\R \times E} |Q (\theta_1) Q (\theta_1) |^2 \Big( \int_{\R} |V^{\zzfive} (\sigma, k_1) | \mathd \sigma \Big)^2 \mathd \theta_1,
\end{align*}
where $Q (\theta) = \hat{H} (\theta)$ is the time Fourier transform of $H$:
\[ Q (\theta) = \frac{i k}{i \omega + k^2} . \]
This computation hints to the fact that the relevant norm on the vertex
functions is given by the supremum norm in the wave vectors of the $L^1$ norm
in the time variable. Summing up, we have in this case
\[ \int_{\R \times E} |G^{\zzfive} (t, x, \eta_1) |^2 \mathd \eta_1
   \leqslant \int_{\R \times E} |Q (\theta_1) Q (\theta_1) |^2
   \|V^{\zzfive} (\sigma, k_1)\|_{L^1_{\sigma}}^2 \mathd \theta_1 . \]
For the other contraction kernels we can proceed similarly. Consider first
\begin{align*}
   G^{\zzsixreso} (t, x, \eta_{12}) & = \int_{\R \times E} G^{\zzthreereso} (t, x, \eta_{123 (- 3)}) \mathd \eta_3 \\
   & = e^{i k_{[12]} x} \int_0^t \mathd \sigma \int_0^{\sigma} \mathd \sigma' \int_0^{\sigma'} \mathd \sigma'' H_{t - \sigma} (k_{[12]}) H_{\sigma' - \sigma''} (k_{[12]}) H_{\sigma'' - s_1} (k_1) \times \\
   &\hspace{150pt} \times H_{\sigma'' - s_2} (k_2) V^{\zzfivereso} (\sigma - \sigma', k_{[12]}),
\end{align*}
for which we have
\[ \int_{(\R \times E)^2} |G^{\zzsixreso} (t, x, \eta_{12}) |^2 \mathd
   \eta_{12} \leqslant \int_{(\R \times E)^2} |Q (\theta_{[12]}) Q
   (\theta_{[12]}) Q (\theta_1) Q (\theta_2) |^2 \|V^{\zzfivereso} (\sigma,
   k_{[12]})\|_{L^1_{\sigma}}^2 \mathd \theta_{12} . \]
The next term is
\[ G^{\zzsevenreso} (t, x, \eta_{12}) = e^{i k_{[12]} x} \psi_\circ(k_1,k_2) \int_0^t \mathd \sigma
   H_{t - \sigma} (k_{[12]}) H_{\sigma - s_2} (k_2) e^{- i k_1 x} G^{\zzfive}
   (\sigma, x, \eta_1), \]
for which
\[
   \int_{(\R \times E)^2} |G^{\zzsevenreso} (t, x, \eta_{12}) |^2 \mathd \eta_{12} \leqslant \int_{(\R \times E)^2} |\psi_\circ(k_1,k_2) Q (\theta_{[12]}) Q (\theta_1) Q (\theta_1) Q (\theta_2) |^2 \|V^{\zzfive} (\sigma, k_1)\|_{L^1_{\sigma}}^2 \mathd \theta_{12} . \]
And finally
\begin{align*}
   G^{\zzeightreso} (t, x, \eta_{12}) & = e^{i k_{[12]} x} \int_0^t \mathd \sigma \int_0^{\sigma'} \mathd \sigma' \int_0^{\sigma'} \mathd \sigma'' H_{t - \sigma} (k_{[12]}) H_{\sigma' - s_2} (k_2) H_{\sigma'' - s_1} (k_1) \times \\
   &\hspace{150pt} \times V^{\zzeightreso} (\sigma - \sigma', \sigma - \sigma'', k_{12}),
\end{align*}
for which
\[ \int_{(\R \times E)^2} |G^{\zzeightreso}_{} (t, x, \eta_{12}) |^2
   \mathd \eta_{12} \leqslant \int_{(\R \times E)^2} |Q
   (\theta_{[12]}) Q (\theta_1) Q (\theta_2) |^2 \|V_{}^{\zzeightreso} (\sigma,
   \sigma', k_{12})\|_{L^1_{\sigma} L^1_{\sigma'}}^2 \mathd \theta_{12} . \]
Similarly, we get
\[ \int_{(\R \times E)^2} |G^{\zznine}_{} (t, x, \eta_{12}) |^2
   \mathd \eta_{12} \leqslant \int_{(\R \times E)^2} |Q
   (\theta_{[12]}) Q (\theta_1) Q (\theta_2) |^2 \|V_{}^{\zznine} (\sigma,
   \sigma', k_{12})\|_{L^1_{\sigma} L^1_{\sigma'}}^2 \mathd \theta_{12} . \]
For these computations to be useful it remains to obtain explicit
bounds for the norms of the vertex functions.
   
\begin{lemma}\label{lem:Vzzfive}
   For any $\varepsilon > 0$ we have
   \[
   \int_{\R} \mathd \sigma |V^{\zzfive} (\sigma, k_1) | + \int_{\R} \mathd \sigma |V^{\zzfivereso} (\sigma, k_1) \lesssim | k_1 |^{\varepsilon}.
   \]
\end{lemma}

\begin{proof}
Some care has to be exercised since a too bold bounding would fail to give a finite result. Indeed, a direct estimation would result in
\[
   \int_{\R} \mathd \sigma \left| V^{\zzfive} (\sigma, k_1) \right| \leqslant \int_E \mathd k_2 \int_0^{\infty} \mathd \sigma | k_{[1 2]} | e^{- \sigma (k_2^2 + k_{[1 2]}^2)} \leqslant \int_E \mathd k_2 \frac{| k_{[1 2]} |}{k_2^2 + k_{[1 2]}^2} = + \infty,
\]
due to the logarithmic divergence at infinity (recall that $\int_E$ stands for $\sum_{\Z \setminus \{0\}}$). To overcome this problem, observe that
\[
   V^{\zzfive} (\sigma, 0) = 2 \int_E \mathd k_2 H_{\sigma} (k_2) \frac{e^{- | \sigma | k_2^2}}{2} = 0
\]
since the integrand is an odd function of $k_2$. So we can write instead
\[
   V^{\zzfive} (\sigma, k_1) = 2 \int_E \mathd k_2 [H_{\sigma} (k_{[1 2]}) - H_{\sigma} (k_{2})] \frac{e^{- | \sigma | k_2^2}}{2},
\]
and at this point it is easy to verify that
\[
   \int_{\R} \mathd \sigma |V^{\zzfive} (\sigma, k_1) | \lesssim | k_1 |^{\varepsilon}
\]
for arbitrarily small $\varepsilon > 0$. Indeed,
\[
   \int_{\R} \mathd \sigma |V^{\zzfive} (\sigma, k_1) | \lesssim \int_E \mathd k_2 \int_0^{\infty} \mathd \sigma | H_{\sigma} (k_{[1 2]}) - H_{\sigma} (k_2) | e^{- \sigma k_2^2}
\]
and a first order Taylor expansion gives
\[
   | H_{\sigma} (k_{[1 2]}) - H_{\sigma} (k_{2}) | \lesssim | k_1 | \int_0^1 \mathd \tau e^{- c (k_2 + \tau k_1)^2 \sigma} . \]
Therefore,
\begin{align*}
   \int_{\R} \mathd \sigma |V^{\zzfive} (\sigma, k_1) | & \lesssim | k_1 | \int_0^1 \mathd \tau \int_E \mathd k_2 \int_0^{ \infty} \mathd \sigma e^{- c (k_2 + \tau k_1)^2 \sigma - k_2^2 \sigma} \\
   & \lesssim | k_1 | \int_0^1 \mathd \tau \int_E \frac{\mathd k_2}{(k_2 + \tau k_1)^2+ k_2^2},
\end{align*}
but now
\[
   \int_E \frac{\mathd k_2}{(k_2 + \tau k_1)^2 + k_2^2} \lesssim \int_E \frac{\mathd k_2}{k_2^2} \lesssim 1.
\]
On the other side, the sum over $E$ is bounded by the corresponding integral over $\R$, so a change of variables gives
\[
   \int_E \frac{\mathd k_2}{(k_2 + \tau k_1)^2_{} + k_2^2} \lesssim \int_{\R} \frac{\mathd k_2}{(k_2 + \tau k_1)^2 + k_2^2} \lesssim \frac{1}{\tau | k_1 |} \int_{\R} \frac{\mathd k_2}{(k_2 + 1)^2 + k_2^2} \lesssim \frac{1}{\tau | k_1 |}.
\]
By interpolating these two bounds, we obtain
\[
   \int_{\R} \mathd \sigma |V^{\zzfive} (\sigma, k_1) | \lesssim | k_1 |^{\varepsilon} \int_0^1 \frac{\mathd \tau}{\tau^{1 - \varepsilon}} \lesssim | k_1 |^{\varepsilon}
\]
for arbitrarily small $\varepsilon > 0$. The same arguments also give the bound for $\int_{\R} \mathd \sigma |V^{\zzfivereso} (\sigma, k_1) |$.
\end{proof}

\begin{lemma}\label{lem:Vzzeight}
   For any $\varepsilon > 0$ we have
   \[
      \int_{\R^2} \left| V^{\zzeightreso} (\sigma, \sigma', k_{12}) \right| \mathd \sigma \mathd \sigma'  \lesssim | k_{[12]} |^{- 1 + \varepsilon}.
   \]
\end{lemma}
\begin{proof}
We get
\begin{align*}
   \int_{\R^2} \left| V^{\zzeightreso} (\sigma, \sigma', k_{12}) \right| \mathd \sigma \mathd \sigma' & \lesssim \int_{\R^2} \left| \int_E \mathd k_3 \psi_\circ(k_{[123]}, k_{-3}) H_{\sigma} (k_{[132]}) H_{\sigma'} (k_{[13]}) e^{- k_3^2 (\sigma' + \sigma)} \right| \mathd \sigma \mathd \sigma' \\
   & \lesssim \int_E \mathd k_3  \frac{| k_{[123]} |  | k_{[13]} |}{(k_3^2 + k_{[123]}^2) (k_3^2 + k_{[13]}^2)} \\
   & \lesssim \int_E \mathd k_3 \frac{1}{(k_3^2 + k_{[123]}^2)^{1 / 2} (k_3^2 + k_{[13]}^2)^{1 / 2}} \\
   & \lesssim \int_E \mathd k_3  \frac{1}{(k_3^2 + k_{[123]}^2)^{1 / 2} | k_3 |} \lesssim | k_{[12]} |^{- 1 + \varepsilon}
\end{align*}
for arbitrarily small $\varepsilon > 0$.
\end{proof}

\begin{lemma}\label{lem:Vzznine}
   For any $\varepsilon > 0$ we have
   \[
      \int_{\R^2} \left| V_{}^{\zznine} (\sigma, \sigma', k_{12}) \right| \mathd \sigma \mathd \sigma'  \lesssim | k_{[12]} |^{- 1 + \varepsilon}.
   \]
\end{lemma}
\begin{proof}
We can estimate
\begin{align*}
   &\int_{\R^2} \left| V_{}^{\zznine} (\sigma, \sigma', k_{12}) \right| \mathd \sigma \mathd \sigma' \\
   &\hspace{30pt} \lesssim \int_{\R^2} \left| \int_E \mathd k_3 H_{\sigma} (k_{[13]}) H_{\sigma'} (k_{[(- 3) 2]}) e^{- k_3^2 | \sigma' - \sigma |} \right| \mathd \sigma \mathd \sigma' \\
   &\hspace{30pt} \lesssim \int_0^{\infty} \mathd \sigma \int_0^{\infty} \mathd \sigma' \int_E \mathd k_3  | H_{\sigma} (k_{[13]}) H_{\sigma + \sigma'} (k_{[(- 3)
   2]}) | e^{- k_3^2 \sigma'} \\
   &\hspace{30pt} \quad + \int_0^{\infty} \mathd \sigma \int_0^{\infty} \mathd \sigma' \int_E \mathd k_3  | H_{\sigma + \sigma'} (k_{[13]}) H_{\sigma} (k_{[(- 3) 2]}) |   e^{- k_3^2 \sigma'} \\
   &\hspace{30pt} \lesssim \int_E \mathd k_3  \frac{| k_{[13]} |  | k_{[(- 3) 2]} |}{(k_3^2 + k_{[(- 3) 2]}^2) (k_{[13]}^2 + k_{[(- 3) 2]}^2)} + \int_E \mathd k_3    \frac{| k_{[13]} |  | k_{[(- 3) 2]} |}{(k_3^2 + k_{[13]}^2) (k_{[(- 3) 2]}^2 + k_{[13]}^2)} \\
   &\hspace{30pt} \lesssim \int_E \mathd k_3  \frac{1}{(k_3^2 + k_{[(- 3) 2]}^2)^{1 / 2} (k_{[13]}^2 + k_{[(- 3) 2]}^2)^{1 / 2}} \\
   &\hspace{30pt} \quad+ \int_E \mathd k_3 \frac{1}{(k_3^2 + k_{[13]}^2)^{1 / 2} (k_{[(- 3) 2]}^2 + k_{[13]}^2)^{1 / 2}} \\
   &\hspace{30pt} \lesssim \int_E \mathd k_3  \frac{1}{| k_3 | (k_{[13]}^2 + k_{[(- 3) 2]}^2)^{1 / 2}} \lesssim | k_{[12]} |^{- 1 + \varepsilon}
\end{align*}
whenever $\varepsilon > 0$.
\end{proof}

\subsection{Regularity of the driving terms}\label{sec:stochastic regularity}

In this section we will determine the Besov regularity of the random fields
$X^{\tau}$. Below we will derive estimates of the form
\begin{equation}\label{eq:basic Xtau estimate}
   \sup_{q \geqslant 0, x \in \T, s,t \geqslant 0} 2^{2 \gamma_1 (\tau) q} |t-s|^{-\gamma_2(\tau)} \E [(\Delta_q X^{\tau} (t, x) - \Delta_q X^\tau(s,x) )^2] \lesssim 1
\end{equation}
for any $\gamma_1(\tau), \gamma_2(\tau) \geqslant 0$ with $\gamma_1(\tau) + \gamma_2(\tau) = \gamma(\tau)$. Each $\Delta_q X^{\tau} (t, x)$ is a random variable with a finite chaos decomposition, so Gaussian hypercontractivity implies that
\[
   \sup_{q \geqslant 0, x \in \T, s,t \geqslant 0} 2^{p \gamma_1 (\tau) q} |t-s|^{-p \gamma_2(\tau)/2}\E [(\Delta_q X^{\tau} (t, x) - \Delta_q X^\tau(s,x) )^p] \lesssim 1
\]
for any $p \geqslant 2$, and then
\[
   \sup_{q \geqslant 0, s,t \geqslant 0} 2^{p \gamma_1 (\tau) q} |t-s|^{-p \gamma_2(\tau)/2} \E [\| \Delta_q X^{\tau} (t, \cdot) - \Delta_q X^\tau(s, \cdot) \|_{L^p (\T)}^p] \lesssim 1.
\]
From here we derive that for all $\varepsilon > 0$
\begin{align*}
   &\sup_{s,t \geqslant 0} |t-s|^{-p \gamma_2(\tau)/2} \E [\|X^{\tau} (t, \cdot) - X^\tau(s,\cdot) \|_{B^{\gamma_1 (\tau) - \varepsilon}_{p, p}}^p] \\
   &\hspace{50pt} = \sup_{s,t \geqslant 0}  |t-s|^{-p \gamma_2(\tau)/2} \E \Big[ \sum_{q \geqslant - 1} 2^{(\gamma_1 (\tau) - \varepsilon) q} \| \Delta_q X^{\tau} (t, \cdot) - \Delta_q X^{\tau} (s, \cdot) \|_{L^p (\T)}^p \Big] \lesssim 1,
\end{align*}
and by the Besov embedding theorem we get
\[
   \|X^{\tau} (t, \cdot) - X(s, \cdot) \|_{B^{\gamma_1 (\tau) - 2 \varepsilon}_{\infty, \infty}} \lesssim \|X^{\tau} (t, \cdot) - X^{\tau} (s, \cdot) \|_{B^{\gamma_1 (\tau) -  \varepsilon}_{p, p}},
\]
where $\varepsilon > 0$ can be chosen arbitrarily small since $p$ is arbitrarily large. Thus an application of Kolmogorov's continuity criterion gives
\[
   \E[\| X^\tau \|_{C^{\gamma_2(\tau)/2 - \varepsilon}_T \CC^{\gamma_1(\tau) - \varepsilon}}^p] \lesssim 1
\]
whenever $\varepsilon > 0$, $\gamma_1(\tau) + \gamma_2(\tau) \leqslant \gamma(\tau)$, and $p,T>0$. This argument reduces the regularity problem to second moment estimations.

\

Next, observe that every $X^\tau$ solves a parabolic equation started in 0 (except for $\tau = \bullet$ which is easy to treat by hand). Thus we get for $0 \leqslant s \leqslant t$
\begin{align}\label{eq:reducing time difference to fixed time} \nonumber
   X^\tau(t,\cdot) - X^\tau(s, \cdot) & =  \int_0^t P_{t-r} u^\tau(r) \dd r - \int_0^s P_{s-r} u^\tau(r) \dd r \\
   & = (P_{t-s} - 1) X^\tau(s) + \int_s^t P_{t-r} u^\tau(r) \dd r.
\end{align}
We estimate the first term by
\[
   |\Delta_q (P_{t-s} - 1) X^\tau(s,x)| \lesssim |e^{c (t-s) 2^{2q}} - 1| |\Delta_q X^\tau(s,x)| \lesssim |t-s|^{\gamma_2(\tau)/2} 2^{q\gamma_2(\tau)} |\Delta_q X^\tau(s,x)|
\]
whenever $\gamma_2(\tau) \in [0,2]$. The first estimate may appear rather formal, but it is not difficult to prove it rigorously, see for example Lemma~2.4 in~\cite{Bahouri2011}. So if we can show that
\[
   \E[ |\Delta_q X^\tau(s,x)|^2] \lesssim 2^{-q\gamma(\tau)},
\]
for $\gamma(\tau) \leqslant 2$, then the estimate~\eqref{eq:basic Xtau estimate} follows for the first term on the right hand side of~\eqref{eq:reducing time difference to fixed time}. For the second term, we will see below that when estimating $\E[ |\Delta_q X^\tau(t,x)|^2]$ we can extend the domain of integration of $\int_0^t P_{t-r} \Delta_q u^\tau(r) \dd r$ until $-\infty$, and this gives us a factor $2^{-2q}$. If instead we integrate only over the time interval $[s,t]$, then we get an additional factor $1 - e^{c (t-s) 2^{2q}}$, which we can then treat as above. We will therefore only prove bounds of the form $\E[ |\Delta_q X^\tau(t,x)|^2] \lesssim 2^{-2q\gamma(\tau)}$.

\

Let us start by analyzing $Q\reso X$, whose kernel is given by
\[
   G^{Q \reso X} (t, x, \eta_{12}) = e^{ik_{[12]} x} \psi_{\circ} (k_1, k_2) H_{t - s_1} (k_1) \int_0^t \mathd \sigma H_{t - \sigma} (k_2) H_{\sigma -  s_2} (k_2).
\]
Now note that, by symmetry under the change of variables $k_1 \rightarrow -
k_1$ we have
\[ G^{Q \reso X}_0 (t, x) = \int_{(\R \times E)} G^{Q \reso X} (t, x,
   \eta_{1 (- 1)}) \mathd \eta_1 = 0 \]
since
\[ G^{Q \reso X} (t, x, \eta_{1 (- 1)}) = \psi_{\circ} (k_1, - k_1) H_{t -
   s_1} (k_1) \int_0^t \mathd \sigma H_{t - \sigma} (- k_1) H_{\sigma - s_2}
   (- k_1) \]
and $H_{t - \sigma} (- k_1) = - H_{t - \sigma} (k_1)$. So only the second chaos is involved in the following computation:
\begin{align*}
   \E [(\Delta_q (Q \reso X)(t,x))^2] & \lesssim \int_{E^2} \mathd k_{12} \rho_q (k_{[12]})^2 \psi_{\circ} (k_1, k_2)^2 \int_0^t \int_0^t \mathd \sigma \mathd \sigma' H_{t - \sigma} (k_2) H_{t - \sigma'} (k_2) \frac{e^{- k_2^2 | \sigma - \sigma' |}}{2} \\
  & \lesssim \int_{E^2} \mathd k_{12} \rho_q (k_{[12]})^2 \psi_{\circ} (k_1, k_2)^2 | k_2 |^2 \int_0^t \int_0^t \mathd \sigma \mathd \sigma' e^{- k_2^2 [(t - \sigma) + (t - \sigma') + | \sigma - \sigma' |]} \\
  & \lesssim \int_{E^2} \mathd k_{12} \rho_q (k_{[12]})^2 \psi_{\circ} (k_1, k_2)^2 | k_2 |^{- 2} \lesssim \int_{E^2} \mathd k \mathd k' \rho_q (k)^2  \psi_{\circ} (k - k', k')^2 | k' |^{- 2} \\
  & \lesssim 2^q \sum_{i \gtrsim q} 2^{i - 2 i} \lesssim 1.
\end{align*}
Similarly we see that $\E [(\Delta_q (Q \reso X)(t,x) - \Delta_q (Q \reso X)(s,x))^2] \lesssim 2^{q \kappa} |t-s|^{\kappa/2}$ whenever $\kappa \in [0,2]$, from where we get the required temporal regularity.

\

Next we treat the $X^\tau$. As we have seen in the case of the vertex functions it will be convenient to
pass to Fourier variables. In doing so we will establish uniform bounds for
the kernel functions $(G^{\tau})_{\tau}$ in terms of their stationary versions
$(\Gamma^{\tau})_{\tau}$: that is the kernels which govern the statistics of
the random fields $X^{\tau} (t, \cdot)$ when $t \rightarrow + \infty$.

Let recursively $\Gamma^{\bullet} = G^{\bullet}$ and
\[ \Gamma^{(\tau_1 \tau_2)} (t, x, \eta_{(\tau_1 \tau_2)}) = \int_{- \infty}^t
   \mathd s \partial_x P_{t - s} (\Gamma^{\tau_1} (s, \cdot, \eta_{\tau_1}),
   \Gamma^{\tau_2} (s, \cdot, \eta_{\tau_2})) (x) . \]
Like the $G$ kernels they have the factorized form $\Gamma^{\tau} (t, x,
\eta_{\tau}) = e^{i k_{[\tau]} x} \gamma_t^{\tau} (\eta_{\tau})$. The
advantage of the kernels $\Gamma$ is that their Fourier transform $Q$ in the
time variables $(s_1, \ldots, s_n)$ is very simple. Letting $\theta_i =
(\omega_i, k_i)$ we have
\[
   Q^{\tau} (t, x, \theta_{\tau}) = \int_{(\R \times E)^n} \mathd s_{\tau} e^{i \omega_{\tau} \cdot s_{\tau}} \Gamma^{\tau} (t, x, \eta_{\tau}) = e^{i k_{[\tau]} x + i \omega_{[\tau]} t} q^{\tau}  (\theta_{\tau}),
\]
where
\[ q^{\bullet} (\theta) = \frac{i k}{i \omega + k^2}, \hspace{2em}
   q^{(\tau_1 \tau_2)} (\theta_{(\tau_1 \tau_2)}) = q^{\bullet}
   (\theta_{[(\tau_1 \tau_2)]}) q^{\tau_1} (\theta_{\tau_1}) q^{\tau_2}
   (\theta_{\tau_2}) . \]
The kernels for the $\Gamma$ terms bound the corresponding kernels for the $G$ terms:
\[
   | G^{(\tau_1 \tau_2)} (t, x, \eta_{(\tau_1 \tau_2)}) | = | H^{\tau} (t,\eta_{\tau}) | \leqslant | \gamma_t^{\tau} (\eta_{\tau}) | = | \Gamma^{(\tau_1 \tau_2)} (t, x, \eta_{(\tau_1 \tau_2)}) |.
\]
This is true for $\tau = \bullet$ and by induction
\begin{align*}
   | G^{(\tau_1 \tau_2)} (t, x, \eta_{(\tau_1 \tau_2)}) | & = |H^{(\tau_1 \tau_2)} (t, \eta_{(\tau_1 \tau_2)}) | = | k_{[(\tau_1 \tau_2)]} |  \int_0^t h_{t - s} (k_{[(\tau_1 \tau_2)]}) | H^{\tau_1} (s, \eta_{\tau_1}) | | H^{\tau_2} (s,\eta_{\tau_2}) | \mathd s \\
   & \leqslant | k_{[(\tau_1 \tau_2)]} |  \int_{- \infty}^t h_{t - s} (k_{[(\tau_1 \tau_2)]}) | g^{\tau_1}_s (\eta_{\tau_1}) | | g^{\tau_2}_s (\eta_{\tau_2}) | \mathd s \\
   & \leqslant | k_{[(\tau_1 \tau_2)]} |  \int_{- \infty}^t h_{t - s} (k_{[(\tau_1 \tau_2)]}) | \gamma^{\tau_1}_s (\eta_{\tau_1}) | | \gamma^{\tau_2}_s (\eta_{\tau_2}) | \mathd s \\
   & = | \gamma_t^{(\tau_1 \tau_2)}(\eta_{(\tau_1 \tau_2)}) | = | \Gamma^{(\tau_1 \tau_2)} (t, x, \eta_{(\tau_1 \tau_2)}) |.
\end{align*}
These bounds and the computation of the Fourier transform imply the following
estimation for the $L^2$ norm of the kernels $G$ uniformly in $t \geqslant 0$:
\[
   \int_{\R^n} | G^{\tau} (t, x, \eta_{\tau}) |^2 \mathd s_{\tau} \leqslant \int_{\R^n} | \Gamma^{\tau} (t, x, \eta_{\tau}) |^2 \mathd s_{\tau} \leqslant \int_{\R^n} | Q^{\tau} (t, x, \theta_{\tau}) |^2 \mathd \omega_{\tau},
\]
that is
\begin{equation}\label{eq:G-unif-bound}
  \int_{\R^n} | G^{\tau} (t, x, \eta_{\tau}) |^2 \mathd s_{\tau} \leqslant \int_{\R^n} | q^{\tau} (\theta_{\tau}) |^2 \mathd \omega_{\tau}.
\end{equation}
This observation simplifies many computations of the moments of the $X^{\tau}$'s and gives estimates that are uniform in $t \geqslant 0$. Actually it shows that the statistics of the $X^{\tau}$'s are bounded by the statistics in the stationary state.

Now, note first that due to the factorization~(\ref{eq:G-kernel-factoriz}) we have that
\[
   \Delta_q G^{\tau} (t, x, \eta_{\tau}) = \rho_q (k_{[\tau]}) G^{\tau} (t, x, \eta_{\tau}),
\]
so the Littlewood-Paley blocks of $X^{\tau}$ have the expression
\[ \Delta_q X^{\tau} (t, x) = \int_{(\R \times E)^n} \rho_q
   (k_{[\tau]}) G^{\tau} (t, x, \eta_{\tau}) \prod_{i = 1}^n W (\mathd
   \eta_i) \]
which we rewrite in terms of the chaos expansion as
\[ \Delta_q X^{\tau} (t, x) = \sum_{\ell = 0}^{d (\tau)} \int_{(\R
   \times E)^n} \rho_q (k_{[\tau]}) G^{\tau}_{\ell} (t, x, \eta_{1 \cdots
   \ell}) W (\mathd \eta_{1 \cdots \ell}) . \]
By the orthogonality of the different chaoses and because of the bound~(\ref{eq:G-unif-bound}) we have
\begin{align*}
   \E [(\Delta_q X^{\tau} (t, x))^2] & \lesssim \sum_{\ell = 0}^{d (\tau)} \int_{(\R \times E)^n} \rho_q (k_{[\tau]}) | G^{\tau}_{\ell} (t, x, \eta_{1 \cdots \ell}) |^2 \mathd \eta_{1 \cdots \ell} \\
   & \lesssim \sum_{\ell = 0}^{d (\tau)} \int_{(\R \times E)^n} \rho_q (k_{[\tau]}) | q^{\tau}_{\ell} (\eta_{1 \cdots \ell}) |^2 \mathd \eta_{1 \cdots \ell}.
\end{align*}
By proceeding recursively from the leaves to the root and using the bounds on the vertex functions that we already proved, the problem of the estimation of the above integrals is reduced to estimate at each step an integral of the form
\[
   \int_{\R \times E} | \theta_{} |^{- \alpha} | \theta' - \theta |^{- \beta} \mathd \theta_{},
\]
where we have a joining of two leaves into a vertex, each leave carrying a
factor proportional either to $| \theta |^{- \alpha}$ or $| \theta |^{-
\beta}$ with $\alpha, \beta \geqslant 2$ and where the length $| \theta |$ of
the $\theta$ variables is conveniently defined as 
\[
   | \theta | = | \omega |^{1/ 2} + | k |,
\]
so that the estimate $| q (\theta) | \sim | \theta |^{- 1}$ holds for $q = q^\bullet$.

\begin{lemma}\label{lemma:basic-estimate-integrals}
   For this basic integral we have the estimate
   \[
     \int_{\R \times E} | \theta_{} |^{- \alpha} | \theta' - \theta |^{- \beta} \mathd \theta_{} \lesssim | \theta' |^{- \rho},
  \]
  where $\rho = \alpha + \beta - 3$ if $\alpha, \beta < 3$ and $\alpha+\beta>3$, and where $\rho = \alpha - \varepsilon$ for an arbitrarily small $\varepsilon > 0$ if $\beta \geqslant 3$ and  $\alpha \in (0,\beta]$. Similarly, if $\beta \geqslant 3$ and $\alpha \in (0,\beta]$, then
   \[
     \int_{\R \times E}  \psi_\circ(k, k' - k)^2 | \theta_{} |^{- \alpha} | \theta' - \theta |^{- \beta} \mathd \theta_{} \lesssim | \theta' |^{- \alpha + \varepsilon} |k'|^{3-\beta}.
  \]
\end{lemma}

\begin{proof}
  Let $\ell \in \N$ such that $| \theta' | \sim 2^\ell$:
  \[
     \int_{\R \times E} | \theta_{} |^{- \alpha} | \theta' - \theta |^{- \beta} \mathd \theta \lesssim \sum_{i, j \geqslant 0} 2^{- \alpha i - \beta j} \int_{\R \times E} \1_{| \theta' | \sim 2^\ell, | \theta_{} | \sim 2^i, | \theta' - \theta | \sim 2^j} \mathd \theta.
  \]
  Then there are three possibilities, either $\ell \lesssim i \sim j$ or $i
  \lesssim j \sim \ell$ or $j \lesssim i \sim \ell$. In the first case we bound
  \[ \int_{\R \times E} \1_{| \theta' | \sim 2^\ell, |
     \theta_{} | \sim 2^i, | \theta' - \theta | \sim 2^j} \mathd \theta
     \lesssim \int_{\R \times E} \1_{| \theta | \sim 2^i}
     \mathd \theta \lesssim 2^{3 i},
  \]
  and in the second one
  \[ \int_{\R \times E} \1_{| \theta' | \sim 2^\ell, |
     \theta_{} | \sim 2^i, | \theta' - \theta | \sim 2^j} \mathd \theta
     \lesssim \int_{\R \times E} \1_{| \theta | \sim 2^i}
     \mathd \theta \lesssim 2^{3 i}, \]
  and similarly in the third case
  \[ \int_{\R \times E} \1_{| \theta' | \sim 2^\ell, |
     \theta_{} | \sim 2^i, | \theta' - \theta | \sim 2^j} \mathd \theta
     \lesssim \int_{\R \times E} \1_{| \theta_{} | \sim 2^j}
     \mathd \theta \lesssim 2^{3 j}. \]
  So if $\alpha + \beta > 3$ we have
  \begin{align*}
     \int_{\R \times E} | \theta_{} |^{- \alpha} | \theta' - \theta  |^{- \beta} \mathd \theta_{} & \lesssim \sum_{\ell \lesssim i \sim j} 2^{-  \alpha i - \beta j + 3 i} + \sum_{i \lesssim j \sim \ell} 2^{- \alpha i -  \beta j + 3 i} + \sum_{j \lesssim i \sim \ell} 2^{- \alpha i - \beta j + 3 j} \\
     & \lesssim 2^{- \ell (\alpha + \beta - 3)} + 2^{- \beta \ell + (3 - \alpha)_+ \ell} + 2^{- \alpha \ell + (3 - \beta)_+ \ell} \lesssim 2^{- \rho \ell},
  \end{align*}
  where $\rho$ can be chosen as announced and where we understand that $(\delta)_+ = \varepsilon$ if $\delta = 0$.
  
  Let us get to the estimate for the integral with $\psi_\circ$. Let $\ell' \leqslant \ell$ be such that $|k'| \sim 2^{\ell'}$ and write
  \begin{align*}
     &\int_{\R \times E} |\psi_\circ(k,k'-k)| | \theta_{} |^{- \alpha} | \theta' - \theta |^{- \beta} \mathd \theta \\
     &\hspace{50pt} \lesssim \sum_{i, j \geqslant 0} \sum_{i' \leqslant i, j' \leqslant j} 2^{- \alpha i - \beta j } \int_{\R \times E} \1_{|k| \sim |k'-k|} \1_{ | \theta_{} | \sim 2^i, | \theta' - \theta | \sim 2^j} \1_{|k| \sim 2^{i'}, |k-k'| \sim 2^{j'} }  \mathd \theta \\
     &\hspace{50pt} \lesssim \sum_{i, j \geqslant 0} \sum_{i' \lesssim i \wedge j} 2^{- \alpha i - \beta j } \int_{\R \times E}  \1_{i'\gtrsim \ell} \1_{ | \theta_{} | \sim 2^i, | \theta' - \theta | \sim 2^j} \1_{|k| \sim 2^{i'}}  \mathd \theta.
  \end{align*}
  Now we have to consider again the three possibilities $\ell \lesssim i \sim j$ or $i \lesssim j \sim \ell$ or $j \lesssim i \sim \ell$. In the first and second case we bound the integral on the right hand side by $2^{2i + i'}$, and in the third case by $2^{2 j + i'}$. Then we end up with
   \begin{align*}
     &\int_{\R \times E} |\psi_\circ(k,k'-k)| | \theta_{} |^{- \alpha} | \theta' - \theta |^{- \beta} \mathd \theta \\
     &\hspace{60pt} \lesssim \sum_{\substack{ \ell \lesssim i \sim j \\ \ell' \lesssim i' \lesssim i}}  2^{- \alpha i - \beta j + 2i + i'} + \sum_{\substack{i \lesssim j \sim \ell \\ \ell' \lesssim i' \lesssim i}} 2^{- \alpha i - \beta j + 2i + i'}  + \sum_{\substack{j \lesssim i \sim \ell \\ \ell' \lesssim i' \lesssim j}} 2^{- \alpha i - \beta j + 2j + i'} \\
     &\hspace{60pt} \lesssim 2^{-\ell(\alpha+\beta-3) } + 2^{-\ell \beta + (3 - \alpha)_+ \ell + (3 - \alpha)_- \ell' } + 2^{-\ell \alpha + (3 - \beta)_+ \ell + (3 - \beta)_- \ell' },
  \end{align*}  
  with the same convention for $(\delta)_+$ as above.
\end{proof}

We will also need the following simple observation:
\begin{lemma}\label{lem:simple integral}
   Let $\alpha > 2$. Then
   \[
      \int_{\R} \dd \omega (|\omega|^{1/2} + |k|)^{-\alpha} \lesssim |k|^{2-\alpha}.
   \]
\end{lemma}
\begin{proof}
   We have
   \[
      \int_{\R} \dd \omega (|\omega|^{1/2} + |k|)^{-\alpha} = |k|^{-\alpha} \int_{\R} \dd \omega (|\omega |k|^{-2} |^{1/2} + 1)^{-\alpha} = |k|^{2-\alpha} \int_{\R} \dd \omega (|\omega|^{1/2} + 1)^{-\alpha},
   \]
   and if $\alpha > 2$ the integral on the right hand side is finite.
\end{proof}

Consider now $X^{\zzone}$. Combining the bound~(\ref{eq:G-unif-bound}) with Lemma~\ref{lemma:basic-estimate-integrals} and Lemma~\ref{lem:simple integral}, we get
\begin{align*}
   \E [ (\Delta_q X^{\zzone} (t, x))^2 ] & \lesssim \int_{(\R \times E)^2} \rho_q (k_{[12]})^2 | q (\theta_{[12]}) q (\theta_1) q (\theta_2) |^2 \mathd \theta_{12} \\
   & = \int_{(\R \times E)^2} \dd \theta_{[12]} \dd \theta_2 \rho_q (k_{[12]})^2 | q (\theta_{[12]}) q (\theta_{[12]}-\theta_2) q (\theta_2) |^2 \\
   & \lesssim \int_{(\R \times E)^2} \dd \theta_{[12]} \dd \theta_2 \rho_q (k_{[12]})^2 | q (\theta_{[12]})|^2 |\theta_{[12]}-\theta_2|^{-2} |\theta_2|^{-2} \\
   & \lesssim \int_{\R \times E} \dd \theta_{[12]} \rho_q (k_{[12]})^2 | q (\theta_{[12]})|^2 |\theta_{[12]}|^{-1} \lesssim  \int_{E} \dd k_{[12]} \frac{\rho_q (k_{[12]})^2}{|k_{[12]}|^{-1}} \lesssim 1. 
\end{align*}

As far as $X^{\zztwo}$ is concerned, we have
\begin{align*}
   \E [ ( \Delta_q X^{\zztwo} (t, x) )^2 ] & \lesssim \int_{(\R \times E)^3} \rho_q (k_{[123]})^2 \left| G^{\zztwo} (t, x, \eta_{123}) \right|^2 \mathd \eta_{123} \\
  &\quad + \int_{\R \times E} \rho_q (k_1)^2 \left| G^{\zzfive} (t, x, \eta_1) \right|^2 \mathd \eta_1 \\
  & \lesssim \int_{(\R \times E)^3} \rho_q (k_{[123]})^2 | q (\theta_{[123]}) q (\theta_3) q (\theta_{[12]}) q (\theta_1) q (\theta_2) |^2 \mathd \theta_{123} \\
  &\quad + \int_{\R \times E} \rho_q (k_1)^2 |q (\theta_1) q (\theta_1) |^2 \left( \int_{\R} |V^{\zzfive} (\sigma, k_1) | \mathd \sigma \right)^2 \mathd \theta_1.
\end{align*}
For the contraction term we already know that $\int_{\R} |V^{\zzfive} (\sigma, k_1) | \mathd \sigma \lesssim | k_1 |^{\varepsilon} \lesssim | \theta_1 |^{\varepsilon}$, so
\begin{align*}
   \int_{\R \times E} \rho_q (k_1)^2 |q (\theta_1) q (\theta_1) |^2 \Big( \int_{\R} |V^{\zzfive} (\sigma, k_1) | \mathd \sigma \Big)^2 \mathd \theta_1 & \lesssim \int_{\R \times E} \rho_q (k_1)^2 | \theta_1 |^{- 4 + 2 \varepsilon} \mathd \theta_1 \\
   & \lesssim \int_E \rho_q (k_1)^2 | k_1 |^{- 2 + 2 \varepsilon} \dd k_1 \lesssim 2^{(2 \varepsilon - 1) q} .
\end{align*}
The contribution of the third chaos can be estimated by
\begin{align*}
   &\int_{(\R \times E)^3} \rho_q (k_{[123]})^2 | q (\theta_{[123]}) q (\theta_3) q (\theta_{[12]}) q (\theta_1) q (\theta_2) |^2 \mathd \theta_{123} \\
   &\hspace{100pt} \lesssim \int_{(\R \times E)^2} \rho_q (k_{[12]})^2 | q (\theta_{[12]}) q (\theta_1) q (\theta_2) |^2 | \theta_2^{} |^{- 1} \mathd \theta_{12} \\
   &\hspace{100pt} \lesssim \int_{\R \times E} \rho_q (k_1)^2 | q (\theta_1) |^2 | \theta_1^{} |^{- 2 + \varepsilon} \mathd \theta_1 \lesssim 2^{(\varepsilon - 1) q},
\end{align*}
which is enough to conclude that
\[
   \E [ (\Delta_q X^{\zztwo} (t, x))^2 ] \lesssim 2^{-q(1 - \varepsilon)}
\]
for arbitrarily small $\varepsilon > 0$.

The next term is
\begin{align*}
   &\E [ (\Delta_q X^{\zzthreereso} (t, x))^2 ] \\
   &\hspace{20pt} \lesssim \int_{(\R \times E)^4} \rho_q (k_{[1234]})^2 | \psi_\circ(k_{[123]}, k_4) q (\theta_{[1234]}) q (\theta_4) q (\theta_{[123]}) q (\theta_3) q (\theta_{[12]}) q (\theta_1) q (\theta_2) |^2 \mathd \theta_{1234} \\
   &\hspace{20pt} \quad + \int_{(\R \times E)^2} \rho_q (k_{[12]})^2 |G^{\zzthreereso}_2 (t, x, \eta_{12}) |^2 \mathd \eta_{12}.
\end{align*}
By proceeding inductively we bound
\begin{align*}
   &\int_{(\R \times E)^4} \rho_q (k_{[1234]})^2 | \psi_\circ(k_{[234]}, k_1) q (\theta_{[1234]}) q (\theta_1) q (\theta_{[234]}) q (\theta_2) q (\theta_{[34]}) q (\theta_3) q (\theta_4) |^2 \mathd \theta_{1234} \\
   &\hspace{40pt} \lesssim \int_{(\R \times E)^3} \rho_q (k_{[123]})^2 | \psi_\circ(k_{[23]}, k_1) q (\theta_{[123]}) q (\theta_1) q (\theta_{[23]}) q (\theta_2) q (\theta_3) |^2 | \theta_3 |^{- 1} \mathd \theta_{123} \\
   &\hspace{40pt} \lesssim \int_{(\R \times E)^2} \rho_q (k_{[12]})^2 | \psi_\circ(k_2,k_1) q(\theta_{[12]}) q (\theta_1) q (\theta_2) |^2 | \theta_2 |^{- 2 + \varepsilon} \mathd \theta_{12} \\
   &\hspace{40pt} \lesssim \int_{\R \times E} \rho_q (k_{1})^2 |q(\theta_{1})|^2 | \theta_1 |^{- 2 + \varepsilon} |k_1|^{-1+\varepsilon} \mathd \theta_{1} \lesssim  \int_{E} \rho_q (k_{1})^2 | k_1 |^{- 2 + \varepsilon} |k_1|^{-1+\varepsilon} \mathd k_{1} \lesssim 2^{q(-2+2\varepsilon)}.
\end{align*}
For the contractions we have
\begin{align*}
   &\int_{(\R \times E)^2} \rho_q(k_{[12]})^2 |G^{\zzsixreso} (t, x, \eta_{12}) |^2 \mathd \eta_{12} \\
   &\hspace{40pt} \leqslant \int_{(\R \times E)^2} \rho_q(k_{[12]})^2 |q (\theta_{[12]}) q (\theta_{[12]}) q (\theta_1) q (\theta_2) |^2 \|V^{\zzfivereso} (\sigma, k_{[12]})\|_{L^1_{\sigma}}^2 \mathd \theta_{12} \\
   &\hspace{40pt}  \leqslant \int_{(\R \times E)^2} \rho_q(k_{1})^2 |q (\theta_{1})|^4 |\theta_1 - \theta_2|^{-2} |\theta_2|^{-2} |k_1|^{2\varepsilon} \mathd \theta_{12} \\
   &\hspace{40pt} \lesssim \int_{\R \times E} \rho_q(k_{1})^2 |\theta_{1}|^{-4} |\theta_1|^{-1} |k_1|^{2\varepsilon} \mathd \theta_{1} \lesssim \int_{E} \rho_q(k_{1})^2 |k_1|^{-3+2\varepsilon} \mathd k_{1} \lesssim 2^{q(-2+2\varepsilon)}.
\end{align*}
The following term is
\begin{align*}
   &\int_{(\R \times E)^2} \rho_q(k_{[12]})^2 |G^{\zzsevenreso} (t, x, \eta_{12}) |^2 \mathd \eta_{12} \\
   &\hspace{40pt} \leqslant \int_{(\R \times E)^2} \rho_q(k_{[12]})^2 \psi_\circ(k_1,k_2)^2 |q (\theta_{[12]}) q (\theta_{1}) q (\theta_1) q (\theta_2) |^2 \|V^{\zzfivereso} (\sigma, k_{[12]})\|_{L^1_{\sigma}}^2 \mathd \theta_{12} \\
   &\hspace{40pt} \leqslant \int_{(\R \times E)^2} \rho_q(k_{1})^2 \psi_\circ(k_1-k_2, k_2)^2 |q (\theta_{1})|^2 |\theta_1 - \theta_2|^{-4} |\theta_2|^{-2} |k_1|^{2\varepsilon} \mathd \theta_{12} \\
   &\hspace{40pt} \lesssim \int_{\R \times E} \rho_q(k_{1})^2 |q (\theta_{1})|^2 |\theta_1|^{-2+\varepsilon} |k_1|^{-1+2\varepsilon} \mathd \theta_{1} \lesssim \int_{\R \times E} \rho_q(k_{1})^2  |k_1|^{-3+3\varepsilon} \mathd k_{1} \lesssim 2^{q(-2+3\varepsilon)}.
\end{align*}
Next, we have
\begin{align*}
   &\int_{(\R \times E)^2} \rho_q(k_{[12]})^2 |G^{\zzeightreso} (t, x, \eta_{12}) |^2 \mathd \eta_{12} \\
   &\hspace{40pt} \leqslant \int_{(\R \times E)^2} \rho_q(k_{[12]})^2 |q(\theta_{[12]}) q (\theta_1) q (\theta_2) |^2 \|V_{}^{\zzeightreso} (\sigma, \sigma', k_{12})\|_{L^1_{\sigma} L^1_{\sigma'}}^2 \mathd \theta_{12} \\
   &\hspace{40pt} \lesssim \int_{(\R \times E)^2} \rho_q(k_{1})^2  |q(\theta_{1})|^2 |\theta_1 - \theta_2|^{-2} |\theta_2|^{-2} |k_{1}|^{-2+2\varepsilon} \mathd \theta_{12} \\
   &\hspace{40pt} \lesssim  \int_{\R \times E} \rho_q(k_{1})^2  |\theta_1|^{-3} |k_{1}|^{-2+2\varepsilon} \mathd \theta_{1} \lesssim \int_E \rho_q(k_{1})^2 |k_{1}|^{-3+2\varepsilon} \mathd k_{1} \lesssim 2^{q(-2 + 2\varepsilon)},
\end{align*}
and therefore $\E [ (\Delta_q X^{\zzthreereso} (t, x))^2 ] \lesssim 2^{-q(2 - 3\varepsilon)}$.

The last term is then
\begin{align*}
   &\E [ (\Delta_q X^{\zzfour} (t, x))^2 ] \\
   &\hspace{20pt} \lesssim \int_{(\R \times E)^4} \rho_q (k_{[1234]})^2 | q(\theta_{[1234]}) q (\theta_{[12]}) q (\theta_1) q (\theta_2) q(\theta_{[34]}) q (\theta_3) q (\theta_4) |^2 \mathd \theta_{1234}  \\
   &\hspace{20pt} \quad + \int_{(\R \times E)^2} \rho_q (k_{[12]})^2 |G^{\zzfour}_2 (t, x, \eta_{12}) |^2 \mathd \eta_{12}.
\end{align*}
The first term on the right hand side can be bounded as follows:
\begin{align*}
   &\int_{(\R \times E)^4} \rho_q (k_{[1234]})^2 | q(\theta_{[1234]}) q (\theta_{[12]}) q (\theta_1) q (\theta_2) q(\theta_{[34]}) q (\theta_3) q (\theta_4) |^2 \mathd \theta_{1234} \\
   &\hspace{80pt} \lesssim \int_{(\R \times E)^4} \rho_q (k_{[12]})^2 | q(\theta_{[12]}) q (\theta_1) q (\theta_2) |^2 | \theta_1 |^{- 1} | \theta_2 |^{- 1} \mathd \theta_{12} \\
   &\hspace{80pt} \lesssim \int_{(\R \times E)^4} \rho_q (k_1)^2 | q (\theta_1) |^2 | \theta_1 |^{- 3 + \varepsilon} \mathd \theta_1 \lesssim 2^{(\varepsilon - 2) q}.
\end{align*}
On the other side, the contraction term is given by
\begin{align*}
    &\int_{(\R \times E)^2} \rho_q (k_{[12]})^2 |G^{\zzfour}_2 (t, x, \eta_{12}) |^2 \mathd \eta_{12} \\
    &\hspace{40pt} \lesssim \int_{(\R \times E)^2} \rho_q (k_{[12]})^2 |q (\theta_{[12]}) q (\theta_1) q (\theta_2) |^2 \|V_{}^{\zznine} (\sigma, \sigma', k_{12})\|_{L^1_{\sigma} L^1_{\sigma'}}^2 \mathd \theta_{12} \\
    &\hspace{40pt} \lesssim \int_{(\R \times E)^2} \rho_q (k_{1})^2 |q (\theta_{1})|^2 |\theta_1 - \theta_2|^{-2} |\theta_2|^{-2} |k_1|^{-2+2\varepsilon} \mathd \theta_{12} \\
    &\hspace{40pt}  \lesssim \int_{\R \times E} \rho_q (k_{1})^2 |\theta_1|^{-3} |k_1|^{-2+2\varepsilon} \mathd \theta_{1} \lesssim \int_{E} \rho_q (k_{1})^2 |k_1|^{-3+2\varepsilon} \mathd k_{1} \lesssim 2^{q(-2+2\varepsilon)},
\end{align*}
so we can conclude that  $\E [ (\Delta_q X^{\zzfour} (t, x))^2 ] \lesssim 2^{-q(2 - 2\varepsilon)}$.

\subsection{Divergences in the KPZ equation}

The data we still need to control for the KPZ equation is
\[
   Y, Y^{\zzone}, Y^{\zztwo}, Y^{\zzthreereso}, Y^{\zzfour}
\]
since $q\circ X$ was already dealt with. The kernels for the chaos
decomposition of these random fields are given by
\[
   \tilde{G}^{\bullet}_t(k) = \1_{t \geqslant 0} e^{- t k^2}, \hspace{2em} \tilde{G}^{(\tau_1, \tau_2)} (t, x, \eta_{\tau}) = \int_0^t \dd \sigma P_{t-\sigma} (G^{\tau_1} (\sigma, \cdot , \eta_{\tau_1}) G^{\tau_2} (\sigma, \cdot, \eta_{\tau_2})) (x),
\]
so they enjoy similar estimates as the kernels $G^\tau$ and therefore all the chaos components different from the $0$-th are under control. The only difference is the missing derivative which in the case of the $X^{\tau}$ is responsible for the fact that the constant component in the chaos expansion vanishes. The $0$-th
component it given by
\[
   t c^{\tau} = \E [Y^{\tau}(t,x)].
\]
Some of these expectations happen to be infinite which will force us to renormalize $Y^\tau$ by subtracting its mean.

\

For $Y^\zzone_\varepsilon$ we have
\begin{align*}
   \E[ Y^\zzone_\varepsilon(t,x)] & = \int_0^t \E[P_{t-\sigma} (X_\varepsilon(\sigma,\cdot)^2)(x)] \dd \sigma \\
   & = \int_0^t \dd \sigma \int_{\R \times E} \dd \eta_1 e^{i k_{1(-1)}x} e^{-(t-\sigma) k_{1(-1)}^2} \varphi(\varepsilon k_1) H_{\sigma - s_1} (k_1) \varphi(\varepsilon k_{-1}) H_{\sigma - s_1} (k_{-1}) \\
   & = \int_0^t \dd \sigma \int_{\R \times E} \dd \eta_1 \varphi^2(\varepsilon k_1) H_{\sigma - s_1} (k_1) H_{\sigma - s_1} (k_{-1}) = \frac{t}{2} \int_E \dd k_1 \varphi^2(\varepsilon k_1),
\end{align*}
and since $\varphi \in C^1$ we get
\begin{align*}
   \frac{1}{\varepsilon} \Big( \int_\R \varphi^2(x) \dd x - \varepsilon \int_E \dd k_1 \varphi^2(\varepsilon k_1) \Big) & = \frac{1}{\varepsilon} \Big(\sum_k \int_{\varepsilon k}^{\varepsilon(k+1)} (\varphi^2(x) - \varphi^2(\varepsilon k)) \dd x \Big) \\
   & = \frac{1}{\varepsilon} \sum_k \int_0^1 \int_{0}^{\varepsilon} \mathD (\varphi^2)(\varepsilon k + \lambda x) x \dd x \dd \lambda,
\end{align*}
which converges to $\int_\R \mathD \varphi^2(x) \dd x = 0$ as $\varepsilon \to 0$. Now recall that here we are dealing with $(2\pi)^{1/2}$ times the white noise, so that the constant $c_\varepsilon^\zzone$ of Theorem~\ref{thm:kpz data} can be chosen as
\[
   c_\varepsilon^\zzone = \frac{1}{4\pi \varepsilon} \int_\R \varphi^2(x) \dd x.
\]

\

The next term is $Y^{\zztwo}$, which has mean zero since it belongs to the odd chaoses and thus $c^{\zztwo} = 0$.

\

What we want to show now is that a special symmetry of the equation induces
cancellations which are responsible for the fact that while $c^{\zzfour}$ and
$c^{\zzthree}$ are separately divergent, the particular combination
\[ c^{\zzfour} + 4 c^{\zzthree} \]
is actually finite. If this is true, then we can renormalize $Y^\zzthreereso$ and $Y^\zzfour$ as announced in Theorem~\ref{thm:kpz data}.
In terms of Feynman diagrams (which have the same translation into kernels as for Burgers equation, except that for the outgoing line the factor $i k_{[\tau]}$ is suppressed), this quantity is given by
\[
   I = c^{\zzfour} + 4 c^{\zzthree} = 2 \times \vcenter{\hbox{\resizebox{!}{1.5cm}{\includegraphics{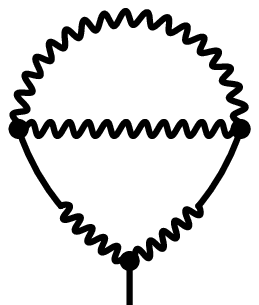}}}} + 8 \times \vcenter{\hbox{\resizebox{!}{1.5cm}{\includegraphics{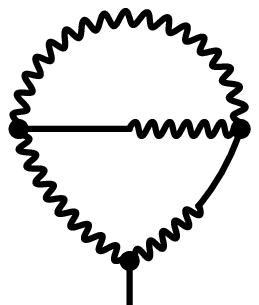}}}},
\]
because all other contractions vanish since they involve contractions of the topmost leaves (e.g. $\tilde{G}^{\zzthree} (t, x, \eta_{1 (- 1) 2 (- 2)}) = 0$). Moreover, writing explicitly the remaining two contributions and fixing the
integration variables according to the following picture
\[ I = 2 \times \vcenter{\hbox{\includegraphics{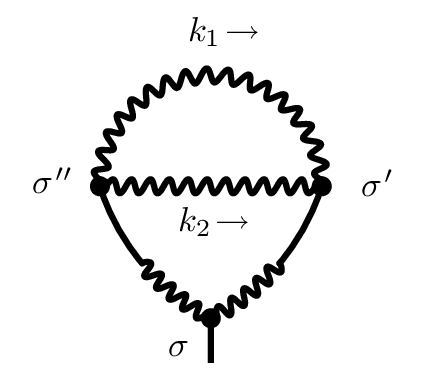}}} + 8
   \times \vcenter{\hbox{\includegraphics{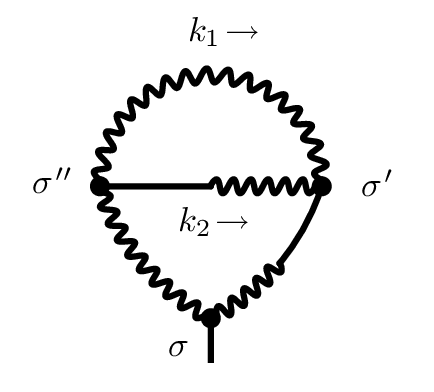}}},
 \]
we get
\begin{align*}
   I & = 2 \int \mathd k_1 \mathd k_2 \int_0^t \mathd \sigma \int_0^{\sigma} \mathd \sigma' \int_0^{\sigma} \mathd \sigma''  (i k_{[12]}) (- i k_{[12]}) e^{- | \sigma' - \sigma'' | (k_1^2 + k_2^2) - (\sigma - \sigma') k_{[12]}^2 - (\sigma - \sigma'') k_{[12]}^2} \\
   &\quad + 8 \int \mathd k_1 \mathd k_2 \int_0^t \mathd \sigma \int_0^{\sigma} \mathd \sigma' \int_0^{\sigma'} \mathd \sigma''  (i k_{[12]}) (i k_2)  e^{- (\sigma' - \sigma'') (k_1^2 + k_2^2) - (\sigma - \sigma') k_{[12]}^2 - (\sigma - \sigma'') k_{[12]}^2}.
\end{align*}
Now note that the second term can be symmetrized over $\sigma', \sigma''$ to get
\begin{align*}
   I &= 2 \int \mathd k_1 \mathd k_2 \int_0^t \mathd \sigma \int_0^{\sigma} \mathd \sigma' \int_0^{\sigma} \mathd \sigma''  (i k_{[12]}) (- i k_{[12]}) e^{- | \sigma' - \sigma'' | (k_1^2 + k_2^2) - (\sigma - \sigma') k_{[12]}^2 - (\sigma - \sigma'') k_{[12]}^2} \\
   &\quad + 4 \int \mathd k_1 \mathd k_2 \int_0^t \mathd \sigma \int_0^{\sigma} \mathd \sigma' \int_0^{\sigma} \mathd \sigma''  (i k_{[12]}) (i k_2) e^{- | \sigma' - \sigma'' | (k_1^2 + k_2^2) - (\sigma - \sigma') k_{[12]}^2 - (\sigma - \sigma'') k_{[12]}^2}.
\end{align*}
At this point the second term can still be symmetrized over $k_2, k_1$ to finally get
\begin{align*}
   I &= 2 \int \mathd k_1 \mathd k_2 \int_0^t \mathd \sigma \int_0^{\sigma} \mathd \sigma' \int_0^{\sigma} \mathd \sigma''  (i k_{[12]}) (- i k_{[12]}) e^{- | \sigma' - \sigma'' | (k_1^2 + k_2^2) - (\sigma - \sigma') k_{[12]}^2 - (\sigma - \sigma'') k_{[12]}^2} \\
   &\quad + 2 \int \mathd k_1 \mathd k_2 \int_0^t \mathd \sigma \int_0^{\sigma} \mathd \sigma' \int_0^{\sigma} \mathd \sigma''  (i k_{[12]}) (i k_{[12]}) e^{- | \sigma' - \sigma'' | (k_1^2 + k_2^2) - (\sigma - \sigma') k_{[12]}^2 - (\sigma - \sigma'') k_{[12]}^2}
\end{align*}
and conclude that $I = 0$. For the moment this computation is only
formal since we did not take properly into account the regularization.
Introducing the regularization $\varphi$ on the noise, the integral to be
considered is
\begin{align*}
   I_{\varepsilon} & = 2 \int \mathd k_1 \mathd k_2 \varphi^2 (\varepsilon k_1) \varphi^2 (\varepsilon k_2) \int_0^t \mathd \sigma \int_0^{\sigma} \mathd \sigma' \int_0^{\sigma} \mathd \sigma'' \times \\
   &\hspace{150pt} \times (i k_{[12]}) (- i k_{[12]}) e^{- | \sigma' - \sigma'' | (k_1^2 + k_2^2) - (\sigma - \sigma') k_{[12]}^2 - (\sigma - \sigma'') k_{[12]}^2} \\
   &\quad + 8 \int \mathd k_1 \mathd k_2 \varphi^2 (\varepsilon k_1) \varphi^2(\varepsilon k_{[12]}) \int_0^t \mathd \sigma \int_0^{\sigma} \mathd \sigma' \int_0^{\sigma'} \mathd \sigma''  (i k_{[12]}) (i k_2) \times \\
   &\hspace{150pt} \times e^{- (\sigma' - \sigma'') (k_1^2 + k_2^2) - (\sigma - \sigma') k_{[12]}^2 - (\sigma - \sigma'') k_{[12]}^2}.
\end{align*}
Here the symmetrization of $\sigma'$ and $\sigma''$ can still be performed, giving
\begin{align*}
   I_{\varepsilon} & = 2 \int_{E^2} \mathd k_1 \mathd k_2 \int_0^t \mathd \sigma \int_0^{\sigma} \mathd \sigma' \int_0^{\sigma} \mathd \sigma''  (i k_{[12]}) [\varphi^2 (\varepsilon k_1) \varphi^2 (\varepsilon k_2) (- i k_{[12]}) + 2 \varphi^2 (\varepsilon k_1) \varphi^2(\varepsilon k_{[12]}) (i k_2)] \times \\
   &\hspace{150pt} \times e^{- | \sigma' - \sigma'' | (k_1^2 + k_2^2) - (\sigma - \sigma') k_{[12]}^2 - (\sigma - \sigma'') k_{[12]}^2},
\end{align*}
which is equivalent to
\begin{align*}
   I_{\varepsilon} & = 4 \int_{E^2} \mathd k_1 \mathd k_2 \int_0^t \mathd \sigma \int_0^{\sigma} \mathd \sigma' \int_0^{\sigma} \mathd \sigma'' (i k_{[12]}) \varphi^2 (\varepsilon k_1) (i k_2) [\varphi^2(\varepsilon k_{[12]}) - \varphi^2 (\varepsilon k_2)] \times \\
   &\hspace{150pt}  \times  e^{- | \sigma' - \sigma'' | (k_1^2 + k_2^2) - \sigma' k_{[12]}^2 - \sigma'' k_{[12]}^2}.
\end{align*}
Now perform the change of variables $\sigma' \rightarrow \varepsilon^2 \sigma'$, $\sigma'' \rightarrow \varepsilon^2 \sigma''$ to obtain
\begin{align*}
   I_{\varepsilon} & = 4 \varepsilon^2 \int_{E^2} \mathd k_1 \mathd k_2 \int_0^t \mathd \sigma \int_0^{\sigma / \varepsilon^2} \mathd \sigma' \int_0^{\sigma / \varepsilon^2} \mathd \sigma'' (i \varepsilon k_{[12]}) \varphi^2 (\varepsilon k_1) (i \varepsilon k_2) [\varphi^2(\varepsilon k_{[12]}) - \varphi^2 (\varepsilon k_2)] \times \\
   &\hspace{170pt} \times  e^{- | \sigma' - \sigma'' | (\varepsilon^2 k_1^2 + \varepsilon^2 k_2^2) - (\sigma' + \sigma'') \varepsilon^2 k_{[12]}^2}.
\end{align*}
By taking the limit $\varepsilon \rightarrow 0$, the two sums over $k_1$ and $k_2$ become integrals:
\begin{align*}
   \lim_{\varepsilon \rightarrow 0} I_{\varepsilon} & = 4 t \int_0^{\infty} \mathd \sigma' \int_0^{\infty} \mathd \sigma'' \int_{\R^2} \mathd k_1 \mathd k_2  (i k_{[12]}) \varphi^2 (k_1) (i k_2) [\varphi^2(k_{[12]}) - \varphi^2 (k_2)] \times \\
   &\hspace{150pt} \times e^{- | \sigma' - \sigma'' | (k_1^2 + k_2^2) - (\sigma' + \sigma'') k_{[12]}^2} \\
   & = 4 t \int_{\R^2} \mathd k_1 \mathd k_2  (i k_{[12]}) \varphi^2 (k_1) (i k_2) [\varphi^2(k_{[12]}) - \varphi^2 (k_2)] \frac{1}{k_{[12]^{}}^2} \frac{1}{k_1^2 + k^2_2 + k_{[12]}^2} \\
   & = 2 t \int_{\R^2} \mathd k_1 \mathd k_2  (i k_{[12]}) [2 \varphi^2 (k_1) \varphi^2(k_{[12]}) (i k_2) - \varphi^2 (k_1) (i k_{[12]}) \varphi^2 (k_2)] \frac{1}{k_{[12]^{}}^2}  \frac{1}{k_1^2 + k^2_2 + k_{[12]}^2} \\
   & = 2 t \int_{\R^2} \mathd k_1 \mathd k_2  (i k_{[12]}) \frac{\varphi^2 (k_1) (i k_{[12]}) (\varphi^2(k_{[12]}) - \varphi^2 (k_2)) + \varphi^2 (k_1) \varphi^2(k_{[12]}) i (k_2 - k_1) }{k_{[12]^{}}^2(k_1^2 + k^2_2 +  k_{[12]}^2)} \\
   & = - 2 t \int_{\R^2} \mathd k_1 \mathd k_2  \left[ \varphi^2 (k_1) (\varphi^2(k_{[12]}) - \varphi^2 (k_2)) + \varphi^2 (k_1) \varphi^2(k_{[12]}) \frac{(k_2 - k_1)}{k_{[12]}} \right] \frac{1}{k_1^2 + k^2_2 + k_{[12]}^2},
\end{align*}
which is indeed finite.

\section{Stochastic data for the Sasamoto-Spohn model}\label{sec:discrete stochastics}

\subsection{Convergence of the RBE-enhancement}

Here we study the convergence of the data
\[
   X_N, X^{\zzone}_N, X^{\zztwo}_N, X^{\zzthreereso}_N, X^{\zzfour}_N, Q_N, B_N(Q_N \reso X_N)
\]
for the discrete Burgers equation~\eqref{eq:discrete sbe}. We will pick up some correction terms as we pass to the limit, which is due to the fact that in the continuous setting not all kernels were absolutely integrable and at some points we used certain symmetries that are violated now. We work in the following setting:

\begin{assumption}[f,g,h] Let $f, g, h \in C^1_b(\R,\C)$ satisfy $f (0) = g (0) = h (0, 0) = 0$ and assume that $f$ is a real valued even function with $f(x) \geqslant c_f > 0$ for all $| x | \leqslant \pi$, that $g (- x) = g(x)^{\ast}$, and that $h(x,y)=h(y,x)$ and $h(-x,-y) = h(x,y)^\ast$; we write $\bar h(x) = h(x,0)$. Define for $N \in \N$
  \begin{gather*}
     \CF \Delta_N \varphi (k) = - | k |^2 f_\varepsilon (k) \CF \varphi (k), \hspace{2em} \CF \mathD_N \varphi (k) = i k g_\varepsilon (k) \CF \varphi (k), \\
     \CF B_N (\varphi, \psi) (k) = (2 \pi)^{- 1} \sum_{\ell} \CF \varphi (\ell) \CF \psi (k - \ell) h_\varepsilon (\ell, (k - \ell)),
  \end{gather*}
  where $\varepsilon = 2\pi/N$ and $f_\varepsilon(x) = f(\varepsilon x)$ and similarly for $g$ and $h$. Let $\xi$ be a space-time white noise and set
  \begin{equation*}
     \X_N(\xi) = (X_N (\xi), X^{\zzone}_N(\xi), X^{\zztwo}_N(\xi), X^{\zzthreereso}_N(\xi), X^{\zzfour}_N (\xi), B_N(Q_N \reso X_N) (\xi)), 
  \end{equation*}
  where
  \begin{equation}\label{eq:discrete data}
    \begin{array}{rll} 
      \LL X_N (\xi) & = & \mathD_N \PC_N \xi,\\
      \LL X^{\zzone}_N (\xi) & = & \mathD_N \Pi_N B_N (X_N(\xi),X_N(\xi)),\\
      \LL X^{\zztwo}_N (\xi) & = & \mathD_N \Pi_N B_N (X_N (\xi), X^{\zzone}_N (\xi)),\\
      \LL X^{\zzthreereso}_N (\xi) & = & \mathD_N \Pi_N B_N(X^{\zztwo}_N(\xi) \reso X_N(\xi)), \\
      \LL X^{\zzfour}_N (\xi) & = & \mathD_N \Pi_N B_N (X^{\zzone}_N (\xi), X^{\zzone}_N (\xi)),\\
      \LL Q_N (\xi) & = & \mathD_N B_N (X_N (\xi), 1),
    \end{array}
  \end{equation}
    all with zero initial conditions except $X_N(\xi)$ for which we choose the ``stationary'' initial condition
    \[
       X_N(\xi)(0) = \int_{-\infty}^0 e^{-s f_\varepsilon | \cdot|^2}(\mathrm{D}) \mathD_N \PC_N \xi(s) \dd s.
    \]
\end{assumption}

We are ready to state the main result of this section:

\begin{theorem}\label{thm:discrete stochastic}
   Let $\xi$ be a space-time white noise and make Assumption (f,g,h). Then for any $0 < \delta < T$ and $p \ge q$ the sequence $\X_N = \X_N(\xi)$ converges to
   \[
      \tilde \X = (X, X^{\zzone}, X^{\zztwo} + 2 c Q,  X^{\zzfour}, X^\zzthreereso + c Q^\zzone + 2 c Q^{Q\reso X}, Q\reso X + c),
  \]
  in $L^p[\Omega, C_T \CC^{\alpha-1} \times C_T \CC^{2\alpha-1} \times \LL_T^{\alpha} \times \LL_T^{2\alpha} \times \LL^{2\alpha}_T \times C([\delta,T], \CC^{2\alpha-1})]$ and is uniformly bounded in $L^p[\Omega, C_T \CC^{\alpha-1} \times C_T \CC^{2\alpha-1} \times \LL_T^{\alpha} \times \LL_T^{2\alpha} \times \LL^{2\alpha}_T \times C_T \CC^{2\alpha-1}]$. Here we wrote
  \begin{equation}\label{eq:correction constant}
     c = - \frac{1}{4\pi} \int_0^\pi \frac{\mathrm{Im} (g (x) \bar{h} (x))}{x} \frac{h (x, - x) | g (x) |^2}{| f (x) |^2} \dd x \in \R.
  \end{equation}
  and
  \[
     \LL Q^{\zzone} = \mathD X^\zzone, \qquad \LL Q^{Q\reso X} = \mathD (Q\reso X),
  \]
  both with $0$ initial condition.
\end{theorem}

The proof of this theorem will occupy us for the remainder of the section. Define the kernel
\[
   H^N_t (k) =  \1_{t \geqslant 0} \1_{|k|<N/2} i k g_\varepsilon (k) e^{- t k^2 f_\varepsilon ( k)}
\]
which satisfies
\begin{equation}\label{eq:discrete ou covariance}
   \int_{\R} H^N_{s - \sigma} (k) H^N_{t - \sigma} (- k) \mathd \sigma = \1_{|k| < N/2} \frac{| g_\varepsilon (k) |^2}{f_\varepsilon (k)} \frac{e^{- k^2 f_\varepsilon (k) | t - s |}}{2}
\end{equation}
for all $k \in E$. Let us start then by analyzing the resonant product $B_N(Q_N \reso X_N)$.

\begin{lemma}\label{lem:discrete resonant}
  For all $0 < \delta < T$, the resonant products $(B_N (Q_N \reso X_N))$ are bounded in $L^p ( \Omega, C_T\CC^{2 \alpha - 1})$ and converge to $Q \reso X + c$ in $L^p ( \Omega,C ( [\delta, T], \CC^{2 \alpha - 1} ))$.
\end{lemma}

\begin{proof}
  We have the chaos decomposition
  \[
     B_N (Q_N \reso X_N) (t,x) = \int_{(\R \times E)^2} K_N (t, x, \eta_{12}) W (\mathd \eta_{12}) + (2\pi)^{-1} \int_{\R \times E} K_N (t, x, \eta_{1 (- 1)}) \mathd \eta_1,
  \]
  where
  \[
     K_N (t, x, \eta_{12}) = \int_0^t \dd \sigma e^{i k_{[12]} x} \psi_\circ(k_1, k_2) h_\varepsilon (k_1, k_2) H^N_{t - \sigma} (k_1) \bar{h}_\varepsilon (k_1) H^N_{\sigma - s_1} (k_1) H^N_{t - s_2} (k_2). 
  \]
  Now $K_N$ satisfies the same bounds as the kernel in the definition of $Q\reso X$ and converges pointwise to it, so that the convergence of the second order Wiener-It\^o integral over $K_N$ to $Q\reso X$ in $L^p ( \Omega, C_T\CC^{2 \alpha - 1})$ follows from the dominated convergence theorem. For the term in the chaos of order 0 we use~\eqref{eq:discrete ou covariance} to obtain
  \begin{align}\label{eq:discrete resonant pr1} \nonumber
     &\int_{\R \times E} K^N (t, x, \eta_{1 (- 1)}) \mathd \eta_1 \\ \nonumber
     &\hspace{30pt} = \int_{E} \dd k_1 \int_0^t \dd \sigma h_\varepsilon (k_1, -k_1) H^N_{t - \sigma} (k_1) \bar{h}_\varepsilon (k_1) \frac{| g_\varepsilon ( k_1) |^2}{f_\varepsilon ( k_1)} \frac{e^{- k_1^2 f_\varepsilon ( k_1) ( t - \sigma )}}{2} \\ \nonumber
     &\hspace{30pt} = \int_{E} \dd k_1 \int_0^t \dd \sigma h_\varepsilon (k_1, -k_1) \1_{|k_1| < N/2} i k_1 g_\varepsilon (k_1) \bar{h}_\varepsilon (k_1) \frac{| g_\varepsilon ( k_1) |^2}{f_\varepsilon ( k_1)} \frac{e^{- 2k_1^2 f_\varepsilon ( k_1) ( t - \sigma )}}{2} \\
     &\hspace{30pt} = \varepsilon \sum_{|k_1|<N/2} \1_{k_1\neq 0} i \frac{g (\varepsilon k_1) \bar{h} (\varepsilon k_1)}{\varepsilon k_1} \frac{h (\varepsilon k_1, -\varepsilon k_1) | g (\varepsilon k_1) |^2}{4 f^2 (\varepsilon k_1)} (1 - e^{- 2k_1^2 f_\varepsilon ( k_1) t}).
  \end{align}
  It is not hard to see that the sum involving the exponential correction term converges to zero uniformly in $t \in [\delta,T]$ whenever $\delta>0$. For the remaining sum, note that
  \[
    i \frac{g (-\varepsilon k_1) \bar{h} (-\varepsilon k_1)}{-\varepsilon k_1} = - i\Big( \frac{g (\varepsilon k_1) \bar{h} (\varepsilon k_1)}{\varepsilon k_1} \Big)^\ast,
  \]
  while the second fraction on the right hand side of~\eqref{eq:discrete resonant pr1} is an even function of $k_1$, and thus
  \begin{align*}
     &\varepsilon \sum_{|k_1|<N/2} \1_{k_1 \neq 0} i \frac{g (\varepsilon k_1) \bar{h} (\varepsilon k_1)}{\varepsilon k_1} \frac{h (\varepsilon k_1, -\varepsilon k_1) | g (\varepsilon k_1) |^2}{4 f^2 (\varepsilon k_1)} \\
     &\hspace{50pt} = - \varepsilon \sum_{0 < k_1 < N/2} \frac{\mathrm{Im} (g (\varepsilon k_1) \bar{h} (\varepsilon k_1))}{\varepsilon k_1} \frac{h (\varepsilon k_1, -\varepsilon k_1) | g (\varepsilon k_1) |^2}{2 f^2 (\varepsilon k_1)}.
  \end{align*}
  Now $\mathrm{Im} (g\bar{h}) \in C^1_b$ with $\mathrm{Im} (g\bar{h})(0) = 0$, and therefore $\mathrm{Im} (g(x)\bar{h}(x))/x$ is a bounded continuous function. Since furthermore $f(\varepsilon k_1) \geqslant c_f > 0$ for all $|k_1| < N/2$, the right hand side is a Riemann sum and converges to $2\pi c$ which is what we wanted to show.
\end{proof}

Next, define
\[
   E_N = E \cap (-N/2, N/2)
\]
and the vertex function
\begin{equation}\label{eq:discrete vertex}
   V^{\zzfive}_N (\sigma, k_1) = \1_{|k_1| \leqslant N / 2} \int_{E_N} h_\varepsilon (k_{[12]}^N, k_{-2}) h_\varepsilon (k_1, k_2) H^N_{\sigma} (k_{[12]}^N) \frac{| g_\varepsilon ( k_2) |^2}{f_\varepsilon ( k_2)} \frac{e^{- k_2^2 f_\varepsilon ( k_2) | \sigma |}}{2} \mathd k_2.
\end{equation}
We also set
\begin{align*}
   V^{\zzfivereso}_N (\sigma, k_1) &= \1_{|k_1| \leqslant N / 2} \int_{E_N} \mathd k_2 \psi_\circ(k_{[12]}^N,k_{-2}) h_\varepsilon (k_{[12]}^N, k_{-2}) \\
   &\hspace{70pt} \times h_\varepsilon (k_1, k_2) H^N_{\sigma} (k_{[12]}^N) \frac{| g_\varepsilon ( k_2) |^2}{f_\varepsilon ( k_2)} \frac{e^{- k_2^2 f_\varepsilon ( k_2) | \sigma |}}{2}.
\end{align*}

\begin{lemma}\label{lem:discrete vertex}
  Let $c \in \R$ be the constant defined in~\eqref{eq:correction constant}. Then $V_N^\zzfive(\cdot, k)$ converges weakly to $V^\zzfive(\sigma, k) + 2\pi c \delta(\sigma)$ for all $k \in E$, in the sense that
  \[
     \lim_{N \rightarrow \infty} \int_0^{\infty} \varphi (\sigma) V^{\zzfive}_N(\sigma, k) \dd \sigma = \int_0^{\infty} \varphi (\sigma) V^{\zzfive} (\sigma, k) \dd \sigma + 2\pi c \varphi (0)
  \]
  for all measurable $\varphi\colon [0,\infty) \to \R$ with $|\varphi(\sigma) - \varphi(0)| \lesssim |\sigma|^\kappa$ for some $\kappa>0$. Similarly, $V_N^{\zzfivereso}(\cdot, k)$ converges weakly to $V^{\zzfivereso}(\sigma, k) + 2\pi c \delta(\sigma)$ for all $k \in E$. Moreover, for all $\delta> 0$ we have
  \[
     \sup_N \int_0^{\infty} |V_N^{\zzfive} (\sigma, k) | \dd \sigma + \sup_N \int_0^{\infty} |V_N^{\zzfivereso} (\sigma, k) | \dd \sigma  \lesssim | k |^{\delta} .
  \]
\end{lemma}

\begin{proof}
  We write $V^{\zzfive}_N (\sigma, k) = (V^{\zzfive}_N (\sigma, k) - \1_{|k| \leqslant N/2} V^{\zzfive}_N (\sigma, 0)) + \1_{|k| \leqslant N/2} V^{\zzfive}_N (\sigma, 0)$. Let us first concentrate on the second term:
  \begin{align}\label{eq:discrete vertex pr1}
     &\int_0^\infty \mathd \sigma \varphi (\sigma) V^{\zzfive}_N (\sigma, 0)\\ \nonumber
    &\hspace{25pt} = \int_{E_N} \dd k_2 h_\varepsilon (k_2, - k_2) \bar{h}_\varepsilon (k_2) (i k_2) g_\varepsilon (k_2) \frac{| g_\varepsilon (k_2) |^2}{2 f_\varepsilon (k_2)} \int_0^\infty \mathd \sigma \varphi (\sigma) e^{- 2 k_2^2 f_\varepsilon (k_2) \sigma} \\ \nonumber
    &\hspace{25pt} = - \sum_{0 < |k_2| \leqslant N/2} h_\varepsilon (k_2, - k_2) k_2 \mathrm{Im}( \bar{h}_\varepsilon (k_2) g_\varepsilon (k_2)) \frac{| g_\varepsilon (k_2) |^2}{f_\varepsilon (k_2)} \int_0^\infty \mathd \sigma \varphi (\sigma) e^{- 2 k_2^2 f_\varepsilon (k_2) \sigma}.
  \end{align}
  Now add and subtract $\varphi(0)$ in the integral over $\sigma$ and observe that
  \begin{dmath*}
     \Big| \int_0^\infty \mathd \sigma (\varphi (\sigma) - \varphi(0)) e^{- 2 k_2^2 f_\varepsilon (k_2) \sigma} \Big| \lesssim  \int_0^\infty \mathd \sigma \sigma^\kappa e^{- 2 k_2^2 c_f \sigma} \lesssim |k_2|^{-2\kappa} \int_0^\infty \dd \sigma e^{-c k_2^2 \sigma} \lesssim |k_2|^{-2-2\kappa},
  \end{dmath*}
  and then
  \begin{dmath*}
     \Big| \sum_{0 < |k_2| \leqslant N/2} h_\varepsilon (k_2, - k_2) k_2 \mathrm{Im}( \bar{h}_\varepsilon (k_2) g_\varepsilon (k_2)) \frac{| g_\varepsilon (k_2) |^2}{f_\varepsilon (k_2)} |k_2|^{-2-2\kappa} \Big| \\
     \lesssim \varepsilon^{\kappa'} \Big(\varepsilon \sum_{0 < |k_2| \leqslant N/2} | h (\varepsilon k_2, - \varepsilon k_2) \mathrm{Im}( \bar{h} (\varepsilon k_2) g (\varepsilon k_2))| \frac{| g (\varepsilon k_2) |^2}{f (\varepsilon k_2)} |\varepsilon k_2|^{-1-\kappa'}\Big)
  \end{dmath*}
  whenever $\kappa' \in [0,1 \wedge 2\kappa)$. The term in the brackets is a Riemann sum and since $|Im(\bar{h} g)(x)| \lesssim |x|$ and $\kappa' < 1$, it converges to a finite limit. Thus we may replace $\varphi(\sigma)$ by $\varphi(0)$ and end up with
  \begin{dmath*}
     - \varphi(0) \sum_{0 < |k_2| \leqslant N/2} h_\varepsilon (k_2, - k_2) \mathrm{Im}( \bar{h}_\varepsilon (k_2) g_\varepsilon (k_2)) \frac{| g_\varepsilon (k_2) |^2}{2 f_\varepsilon^2 (k_2) k_2},
  \end{dmath*}
  which converges to $2\pi c \varphi(0)$.
  
  It remains to treat the term
  \[
     V^{\zzfive}_N (\sigma, k) - \1_{|k| \leqslant N/2} V^{\zzfive}_N (\sigma, 0) = \1_{|k| \leqslant N / 2} \int_{E_N} \dd k_2 (W_N^{\zzfive} (\sigma, k, k_2) - W_N^{\zzfive} (\sigma, 0, k_2)),
  \]
  where the right hand side is to be understood as the (indirect) definition of $W_N^{\zzfive}$.  Since for fixed $(k, k_2,\sigma)$ the integrand converges to $(H_{\sigma} (k + k_2) - H_\sigma(k_2)) e^{- k_2^2 | \sigma |}/2$, it suffices to bound $|W_N^{\zzfive} (\sigma, k, k_2) - W_N^{\zzfive} (\sigma, 0, k_2)|$ uniformly in $N$ by an expression which is integrable over $(\sigma,k_2) \in [0,\infty) \times E$. For the remainder of the proof let us write $\psi_\varepsilon (\ell, m) = h_\varepsilon ( (\ell + m)^N, - m) h_\varepsilon (\ell, m) g_\varepsilon ((\ell + m)^N)$, which satisfies uniformly over $|\ell|, |m| < N/2$
  \[
     | \psi_\varepsilon (\ell, m) - \psi_\varepsilon (0, m) | \lesssim \varepsilon |\ell | \lesssim (\varepsilon |\ell |)^\delta
  \]
  for all $\delta \in [0,1]$ (to bound $|(\ell + m)^N - m|$ note that $(\ell + m)^N = \ell + m + j(m,\ell)N$ for some $j(m,\ell) \in \Z$ and that if $|\ell| < N/2$, then $| \ell + j N| \geqslant |\ell |$ for all $j \in \Z$). We now have to estimate the term
  \begin{align*}
     &| \psi_\varepsilon (k, k_2) (k + k_2)^N e^{- \sigma f_\varepsilon ((k + k_2)^N) ((k+k_2)^N)^2} - \psi_\varepsilon (0, k_2) k_2 e^{- \sigma f_\varepsilon (k_2) k_2^2} | \1_{|k|, | k_2 | < N / 2} e^{- k_2^2 f_\varepsilon (k_2) \sigma} \\
     &\hspace{20pt} \leqslant | \psi_\varepsilon (k, k_2) - \psi_\varepsilon (0, k_2) | | k_2 | e^{- \sigma f_\varepsilon (k_2) k_2^2} \1_{|k|, | k_2 | < N / 2} e^{- k_2^2 c_f \sigma} \\
     &\hspace{20pt}\quad + | \psi_\varepsilon (k, k_2) |  | (k + k_2)^N e^{- \sigma f_\varepsilon ((k + k_2)^N) ((k+k_2)^N)^2} - k_2 e^{- \sigma f_\varepsilon (k_2) k_2^2} | \1_{|k|, | k_2 | < N / 2} e^{- k_2^2 c_f \sigma} \\
     &\hspace{20pt} \lesssim \big( (\varepsilon |k|)^\delta | k_2 | + | (k + k_2)^N e^{- \sigma f_\varepsilon ((k + k_2)^N) ((k+k_2)^N)^2} - k_2 e^{- \sigma f_\varepsilon (k_2) k_2^2} | \big) \1_{|k|, | k_2 | < N / 2} e^{- k_2^2 c_f \sigma}.
  \end{align*}
  The first addend on the right hand side is bounded by $| k |^{\delta} | k_2 |^{1 - \delta} e^{- \sigma c_f k_2^2}$, and for $\delta>0$ this is integrable in $(\sigma,k_2)$ and the integral is $\lesssim |k|^\delta$. To bound the second addend, let us define $\varphi_{\varepsilon,\sigma} (x) = x e^{- \sigma f_\varepsilon (x) x^2}$ and note that for $| x | < N / 2$
  \[
     | \varphi_{\varepsilon,\sigma}' (x) | = | 1 - x \sigma f' (\varepsilon x) \varepsilon x^2 - x \sigma f_\varepsilon (x) 2 x | e^{- \sigma f_\varepsilon (x) x^2} \lesssim (1+\sigma |x|^2) e^{-c_f \sigma x^2} \lesssim e^{- \sigma c x^2}
  \]
  for some $c > 0$. From here we can use the same arguments as in the proof of Lemma~\ref{lem:Vzzfive} to show that
  \[
     | (k + k_2)^N e^{- \sigma f_\varepsilon ((k + k_2)^N) ((k+k_2)^N)^2} - k_2 e^{- \sigma f_\varepsilon (k_2) k_2^2} |) \1_{|k|, | k_2 | < N / 2} e^{- k_2^2 c_f \sigma} \leqslant F(k,k_2,\sigma)
  \]
  for a function $F$ with $\int_E \dd k_2 \int_0^\infty \dd \sigma F(k,k_2,\sigma) \lesssim |k|^\delta$, and thus we get the convergence and the bound for $V^\zzfive_N$. The term $V^{\zzfivereso}_N$ can be treated using the same arguments.
\end{proof}

\begin{proof}[Proof of Theorem~\ref{thm:discrete stochastic}]
   We introduce analogous kernels as in the continuous setting: Define $G_N^{\bullet}(t,x,\eta) = e^{i k x} H^N_{t-s}(k)$ and and then inductively for $\tau = (\tau_1 \tau_2)$
\begin{align*}
   G^{\tau}_N (t, x, \eta_{\tau}) &= e^{i k^N_{[\tau]} x} H^{\tau}_N (t, \eta_{\tau}) \\
   & = e^{i k^N_{[\tau]} x} \int_0^t \mathd \sigma H^N_{t - \sigma} (k^N_{[\tau]}) h_\varepsilon (k^N_{[\tau_1]}, k^N_{[\tau_2]}) H^{\tau_1}_N (\sigma, \eta_{\tau_1}) H^{\tau_2}_N (\sigma,  \eta_{\tau_2}).
\end{align*}
A first consequence of this recursive description is that the contribution to the 0-th chaos always vanishes: just as for the kernels $G^{\tau}$ of Section~\ref{sec:stochastics} we have
\[
   G^{\tau}_N (t, x, \sigma (\eta_{1 \ldots n (- 1) \ldots (- n)})) = 0
\]
whenever $\sigma \in \mathcal{S}_{2n}$, because $G^{\tau}_N (t, x, \eta_{1 \ldots 2 n}) \propto [k_{1 \ldots 2 n}^N]$.

We decompose every $G^\tau_N$ into two parts:
\begin{align}\label{eq:GN decomposition} \nonumber
   G^{\tau}_N (t, x, \eta_{\tau}) & = G^{\tau}_N (t, x, \eta_{\tau}) \1_{|k_1|,\dots, |k_{d(\tau)}| \leqslant N/(2d(\tau))} + G^{\tau}_N (t, x, \eta_{\tau}) (1 -  \1_{|k_1|,\dots, |k_{d(\tau)}| \leqslant N/(2d(\tau))}) \\
   & = K^\tau_N(t, x, \eta_{\tau}) + F^\tau_N(t, x, \eta_{\tau}).
\end{align}
Let us first indicate how to deal with $F^\tau_N$, which gives a vanishing contribution in the limit. Define
\[
   q_N(\theta) = \1_{|k|\leqslant N/2} \frac{i k g_\varepsilon(k)}{i \omega + f_\varepsilon(k) k^2},
\]
which satisfies $|q_N(\theta)| \lesssim |\theta|^{-1}$, uniformly in $N$ (recall that we defined $|\theta| = |\omega|^{1/2} + |k|$). Similarly as in Section~\ref{sec:stochastic regularity} we can bound every integrand that we need to control by a product of terms of the form $q_N(\theta_{[\tau']}^N)$, discrete vertex functions, and factors like $\psi_\circ(k_{[\tau']}^N, k_{[\tau'']}^N)$, where $\theta^N = (\omega, k^N)$. Moreover, every integrand contains a factor $|q_N(\theta_i)|$ for each its integration variables $\theta_i$, and we can decompose $1 - \1_{|k_1|,\dots, |k_{d(\tau)}| \leqslant N/(2d(\tau))}$ into a finite sum of terms that each contain a factor $\1_{|\theta_i| > N/(2d(\tau))}$ for some $i$. We can therefore estimate $|q_N(\theta_i)| \1_{|\theta_i| > N/(2d(\tau))} \lesssim N^{-\delta} |\theta_i|^{1-\delta}$ for an arbitrarily small $\delta > 0$, which gives us a small factor. We now only have to show that the remaining integral stays bounded. To do so, we need to control the basic integral
\[
   \int_{\R \times E_N} |\theta|^{-\alpha} |(\theta' - \theta)^N|^{-\beta} \dd \theta \1_{|k'|\leqslant N/2}.
\]
But $(k' - k)^N = k' - k + jN$ for some $|j| \leqslant 1$ and if $|k'|\leqslant N/2$, then $|k' + jN| \geqslant |k'|$ for all $j \in \Z$. So by Lemma~\ref{lemma:basic-estimate-integrals} we have
\begin{align*}
   \int_{\R \times E_N} |\theta|^{-\alpha} |(\theta' - \theta)^N|^{-\beta} \dd \theta \1_{|k'|\leqslant N/2} & \leqslant \sum_{|j| \leqslant 1} \int_{\R \times E_N} |\theta|^{-\alpha} |\theta' + (0,jN) - \theta|^{-\beta} \dd \theta \1_{|k'|\leqslant N/2} \\
   & \lesssim \sum_{|j| \leqslant 1} |\theta' + (0,jN)|^{-\rho} \1_{|k'|\leqslant N/2} \lesssim |\theta'|^{-\rho} \1_{|k'|\leqslant N/2},
\end{align*}
where $\alpha,\beta,\rho$ are as in the announcement of the lemma. The second integral in Lemma~9.8 is estimated using the same argument, which also allows us to show that all discrete vertex functions apart from $V^\zzfive_N$ and $V^\zzfivereso_N$ satisfy the same bounds as their continuous counterparts. Since we estimated $V^\zzfive_N$ and $V^\zzfivereso_N$ in Lemma~\ref{lem:discrete vertex}, we now simply need to repeat the calculations of Section~\ref{sec:stochastic regularity} to show that the Wiener-It\^o integral over $F_N^\tau$ converges to 0 in $L^p$ in the appropriate Besov space.

We still have to treat the term $K_N^\tau$ in the decomposition~\eqref{eq:GN decomposition} of $G_N^\tau$. In the description of $K_N^\tau$ we can now replace every $k_{[\tau']}^N$ by the usual sum $k_{[\tau']}$ (where $\tau'$ is an arbitrary subtree of $\tau$). Then $K_N^\tau$ satisfies the same bounds as $G^\tau$, uniformly in $N$, and converges pointwise to it. So whenever $G^\tau$ is absolutely integrable we can apply the dominated convergence theorem to conclude. However, in the continuous case we used certain symmetries to derive the bounds for $Q \reso X$, $V^{\zzfive}$, and $V^{\zzfivereso}$, and these symmetries turn out to be violated in the discrete case. This is why we separately studied the convergence of $B_N(Q_N \reso X_N)$ in Lemma~\ref{lem:discrete resonant}, and in Lemma~\ref{lem:discrete vertex} we showed that $V^\zzfive_N$ and $V^{\zzfivereso}_N$ satisfy the same bounds as $V^\zzfive$ and $V^{\zzfivereso}$, uniformly in $N$, and converge to $V^\zzfive + 2\pi c \delta(\sigma)$ and $V^{\zzfivereso} + 2\pi c \delta(\sigma)$ respectively. It thus remains to see which correction terms we pick up from the additional Dirac deltas.

Let us start with the contraction of $G_N^{\zztwo}$. Here we have
\[
   X^{\zztwo}_N (t, x) = \int_{(\R \times E)^3} G^{\zztwo}_N (t, x, \eta_{123}) W (\mathd \eta_{123}) + \int_{\R \times E} G^{\zztwo}_{N,1} (t, x, \eta_1) W (\mathd \eta_1)
\]
with
\[
   G^{\zztwo}_{N,1} (t, x, \eta_1) = \pi^{-1} G^{\zzfive}_N (t, x, \eta_1) = \pi^{-1} e^{i k_1 x} \int_0^t \mathd \sigma \int_0^{\sigma} \mathd \sigma' H^N_{t - \sigma} (k_1) H^N_{\sigma' - s_1} (k_1) V^{\zzfive}_N (\sigma - \sigma', k_1)
\]
for $V^\zzfive_N$ as defined in~\eqref{eq:discrete vertex}. Now by Lemma~\ref{lem:discrete vertex}, the right hand side converges to
\[
   \pi^{-1} G^{\zzfive} (t, x, \eta_1) + 2 c e^{i k_1 x} \int_0^t \mathd \sigma H_{t - \sigma} (k_1) H_{\sigma - s_1} (k_1),
\]
and we have
\[
   \int_{\R \times E} \Big( 2 c e^{i k_1 x} \int_0^t \mathd \sigma H_{t - \sigma} (k_1) H_{\sigma - s_1} (k_1) \Big) W(\dd \eta_1) = 2 c Q(t,x).
\]
In conclusion, $X^\zztwo_N$ converges to $X^\zztwo + 2 c Q$ in $L^p(\Omega, \LL^\alpha_T)$.

The only other place where $V_N^\zzfive$ appears is in the contractions of $X_N^\zzthreereso$. We have
\[
   X_N^\zzthreereso (t, x) = \int_{(\R \times E)^4} G^{\zzthreereso}_N (t, x, \eta_{1234}) W (\mathd \eta_{1234}) + \int_{(\R \times E)^2} G^{\zzthreereso}_{N,2} (t, x, \eta_{12}) W (\mathd \eta_{12})
\]
with
\begin{align*}
   G^{\zzthreereso}_{N,2} (t, x, \eta_{12}) & = (2\pi)^{-1} (G^{\zzsixreso}_N (t, x, \eta_{12}) + 2 G^{\zzsevenreso}_N (t, x, \eta_{12}) + 2 G^{\zzeightreso}_N (t, x, \eta_{12})).
\end{align*}
Now $G^{\zzsixreso}_N (t, x, \eta_{12})$ can be factorized as
\begin{align*}
   G^{\zzsixreso}_N (t, x, \eta_{12}) & = e^{i k_{[12]} x} h_\varepsilon(k_1, k_2) \int_0^t\dd \sigma \int_0^{\sigma} \mathd \sigma' \int_0^{\sigma'} \mathd \sigma'' H^N_{t-\sigma}(k_{[12]}) H^N_{\sigma' - \sigma''} (k_{[12]})  \\
   &\hspace{130pt} \times H^N_{\sigma'' - s_1} (k_1) H^N_{\sigma'' - s_2} (k_2) V^{\zzfivereso}_N (\sigma - \sigma', k_{[12]}),
\end{align*}
and Lemma~\ref{lem:discrete vertex} shows that the right hand side converges to
\[
   G^{\zzsixreso} (t, x, \eta_{12}) + 2\pi c e^{i k_{[12]} x} \int_0^t \dd \sigma \int_0^\sigma \mathd \sigma'' H_{t - \sigma} (k_{[12]}) H_{\sigma - \sigma''} (k_{[12]}) H_{\sigma'' - s_1} (k_1) H_{\sigma'' - s_2} (k_2),
\]
with
\begin{align*}
   &\int_{(\R\times E)^2} \Big(e^{i k_{[12]} x} \int_0^t \dd \sigma \int_0^\sigma \mathd \sigma'' H_{t - \sigma} (k_{[12]}) H_{\sigma - \sigma''} (k_{[12]}) H_{\sigma'' - s_1} (k_1) H_{\sigma'' - s_2} (k_2)\Big) W(\dd \eta_{12}) \\
   &\hspace{70pt} = Q^\zzone(t,x).
\end{align*}
Similarly, we factorize $G^{\zzsevenreso}_N (t, x, \eta_{12})$ as
\begin{align*}
   G^{\zzsevenreso}_N (t, x, \eta_{12}) & = e^{i k_{[12]} x} h_\varepsilon(k_1, k_2) \psi_\circ(k_1,k_2) \int_0^t \dd \sigma \int_0^{\sigma} \mathd \sigma' \int_0^{\sigma'} \mathd \sigma'' H^N_{t-\sigma}(k_{[12]}) H^N_{\sigma - s_2} (k_2) \\
   &\hspace{130pt} \times H^N_{\sigma - \sigma'} (k_1) H^N_{\sigma'' - s_1} (k_1) V^{\zzfive}_N (\sigma' - \sigma'', k_1),
\end{align*}
and Lemma~\ref{lem:discrete vertex} shows that the right hand side converges to
\[
   G^{\zzsevenreso} (t, x, \eta_{12}) + 2\pi c e^{i k_{[12]} x} \psi_\circ(k_1,k_2) \int_0^t \dd \sigma \int_0^\sigma \mathd \sigma' H_{t-\sigma}(k_{[12]}) H_{\sigma - s_2} (k_2) H_{\sigma - \sigma'} (k_1) H_{\sigma' - s_1} (k_1)
\]
with
\begin{align*}
   & \int_{(\R\times E)^2} \Big( e^{i k_{[12]} x} \psi_\circ(k_1,k_2) \int_0^t \dd \sigma \int_0^\sigma \mathd \sigma' H_{t-\sigma}(k_{[12]}) H_{\sigma - s_2} (k_2) H_{\sigma - \sigma'} (k_1) H_{\sigma' - s_1} (k_1) \Big) W(\dd \eta_{12}) \\
   &\hspace{70pt}= Q^{Q\reso X}(t,x).
\end{align*}
In conclusion, $X_N^\zzthreereso$ converges to $ X^\zzthreereso + c Q^\zzone + 2 cQ^{Q\reso X}$ in $L^p(\Omega, \LL_T^{2\alpha})$.
\end{proof}

\subsection{Convergence of the random operator}

The purpose of this subsection is to prove that the random operator
\begin{align}\label{eq:random operator def sec10} \nonumber
   A_N (f) & \assign \Pi_N \Big(\int (\Pi_{N} (f \para \tau_{- \varepsilon y} Q_N) \reso (\tau_{- \varepsilon z} X_N) \mu (\mathd y, \mathd z) \Big) \\
   &\quad - \PC_N \Big(\int ( (f \para \tau_{- \varepsilon y} Q_N)) \reso (\tau_{- \varepsilon z} X_N ) \mu (\mathd y, \mathd z) \Big)
\end{align}
converges to zero in probability.

\begin{theorem}\label{thm:random operator}
   Let $\alpha \in (1/4, 1/2)$. Then we have for all $\delta, T > 0$ and $r \geqslant 1$
   \begin{equation}\label{eq:random operator bound}
      \E[ \| A_N \|^r_{C_T L ( \CC^{\alpha}, \CC^{2\alpha-1} )} ]^{1/r} \lesssim N^{\alpha - 1/2 + \delta} .
   \end{equation}
\end{theorem}

We will prove the theorem by building on three auxiliary lemmas.

\begin{lemma}\label{lem:random operator decomposition}
   The operator $A_N$ is given by
   \[
      A_N (f) (t, x) = \sum_{q, p} \Delta_q A_N (\Delta_p f) (t, x) = \sum_{q, p} \int_{\T} g^N_{p, q} (t, x, y) \Delta_p f (y) \mathd y
   \]
   with
\begin{equation}\label{eq:gpq fourier transform} \CF g^N_{p, q} (t, x, \cdummy) (k) =
    \sum_{k_1, k_2} \Gamma_{p, q}^N (x ; k, k_1, k_2) \CF Q_N (t, k_1) \CF X_N
    (t, k_2),
  \end{equation}
  where
  \begin{align*}
     \Gamma_{p, q}^N (x ; k, k_1, k_2) & = (2 \pi)^{- 2} \tilde{\rho}_p (k) \psi_{\prec} (k, k_1) h_{\varepsilon} (k_1, k_2)   \\
     &\quad \times \big[ e^{i (k_{[12]} - k)^N x} \rho_q ((k_{[12]} - k)^N) \psi_{\circ}((k_1 - k)^N, k_2)  \\
     &\hspace{30pt} -  e^{i (k_{[12]} - k) x} \rho_q (k_{[12]} - k) \psi_{\circ}(k_1 - k, k_2) \1_{|k_{[12]}-k| \leqslant N/2}\big]
  \end{align*}
  and where $\tilde{\rho}_p$ is a smooth function supported in an annulus $2^p
  \CA$ such that $\tilde{\rho}_p \rho_p = \rho_p$.
\end{lemma}

\begin{proof}
  Parseval's formula gives
  \begin{align*}
     \Delta_q A_N (\Delta_p f) (t, x) & = \int_{\T} g^N_{p, q} (t, x,y) \Delta_p f (y) \mathd y \\
     & = (2 \pi)^{- 1} \sum_k \CF g^N_{p, q} (t, x,\cdummy) (- k) \rho_p (k) \CF f (k).
  \end{align*}
  It suffices to verify that the same identity holds if we replace $\CF g^N_{p,q} (t, x, \cdummy) (k)$ with the right hand side of~(\ref{eq:gpq fourier transform}) and if we fix some $k$ and consider $f (x) = (2 \pi)^{- 1} e^{i k x}$, i.e. $\CF f (\ell) = \delta_{\ell,k}$. Let us look at the first term on the right hand side of~\eqref{eq:random operator def sec10}
  \begin{align*}
     & \int \Delta_q \Pi_N [(\Pi_{N} (\Delta_p ((2 \pi)^{- 1} e^{i k \cdummy}) \para \tau_{- \varepsilon z_1} ((2 \pi)^{- 1} e^{i k_1 \cdummy}))) \reso (\tau_{- \varepsilon z_2} ((2 \pi)^{- 1} e^{i k_2 \cdummy}))] \mu (\mathd z_1, \mathd z_2) \\
     &\hspace{15pt} = (2 \pi)^{- 3} \Delta_q \Pi_N [(\Pi_{N} (\rho_p (k) \psi_{\prec} (k, k_1) e^{i (k + k_1) \cdummy})) \reso (e^{i k_2 \cdummy})]  \int e^{i k_1 \varepsilon z_1 + i k_2 \varepsilon z_2} \mu (\mathd z_1, \mathd     z_2) \\
     &\hspace{15pt} = (2 \pi)^{- 3} \Delta_q \Pi_N [\psi_{\circ} ((k + k_1)^N, k_2) \rho_p (k) \psi_{\prec} (k, k_1) e^{i ((k + k_1)^N + k_2)  \cdummy} ] h_{\varepsilon} (k_1, k_2) \\
     &\hspace{15pt} = (2 \pi)^{- 3} e^{i (k_{[12]} + k)^N \cdummy}  \rho_q ((k_{[12]} + k)^N) \psi_{\circ} ((k + k_1)^N, k_2)  \tilde{\rho}_p (k) \rho_p (k)  \psi_{\prec} (k, k_1) h_{\varepsilon} (k_1, k_2),
  \end{align*}
  where in the last step we used that $((k+k_1)^N + k_2)^N = (k+k_1+k_2)^N$. The second term in~\eqref{eq:random operator def sec10} is treated with the same arguments, and writing $Q_N$ and $X_N$ as inverse Fourier transforms equation~\eqref{eq:gpq fourier transform} follows.
\end{proof}

\begin{lemma}\label{lem:random operator kolmogorov}
   For all $r\geqslant 1$ and all $\alpha,\beta \in \R$ we can estimate
  \begin{align*}
     &\E[\| A_N(t) - A_N(s) \|^r_{L ( \CC^{\alpha}, B^{\beta}_{r, r})}] \\
     &\hspace{50pt} \lesssim \sum_{q, p} 2^{q r \beta} 2^{- p r \alpha} \Big( \sup_{x \in \T} \sum_k \E \big[ |( \CF g^N_{p, q} (t, x, \cdummy) - \CF g^N_{p, q} (s, x, \cdummy) ) (k) |^2\big] \Big)^{r / 2}.
  \end{align*}
\end{lemma}

\begin{proof}
  From the decomposition in Lemma~\ref{lem:random operator decomposition} we get
  \begin{align*}
     \| A_N (f) (t) - A_N (f) (s) \|^r_{B^{\beta}_{r, r}} & = \sum_q 2^{q r \beta} \Big\| \sum_p \int_{\T} (g^N_{p, q} (t, x,y) - g^N_{p, q} (s, x, y)) \Delta_p f (y) \mathd y \Big\|_{L^r_x(\T)}^r \\
     & \leqslant \sum_{q, p} 2^{q r \beta} 2^{- p r \alpha} \left\| \int_{\T} | g^N_{p, q} (t, x, y) - g^N_{p, q} (s, x, y) | \mathd y \right\|_{L^r_x (\T)}^r \| f \|_{\alpha}^r,
  \end{align*}
  so that
  \begin{align*}
     &\E[\| A_N (t) - A_N (s) \|^r_{L ( \CC^{\alpha}, B^{\beta}_{r, r})}] \\
     &\hspace{50pt} \lesssim \sum_{q, p} 2^{q r \beta} 2^{- p r \alpha} \E\Big[ \Big\| \int_{\T} | g^N_{p, q} (t, x, y) - g^N_{p, q} (s, x, y) | \mathd y \Big\|_{L^r_x (\T)}^r \Big] \\
     &\hspace{50pt} \lesssim  \sum_{q, p} 2^{q r \beta} 2^{- p r \alpha} \E\Big[ \Big\| \Big(\int_{\T} | g^N_{p, q} (t, x, y) - g^N_{p, q} (s, x, y) |^2 \mathd y \Big)^{1 / 2} \Big\|_{L^r_x (\T)}^r \Big] \\
     &\hspace{50pt} \lesssim \sum_{q, p} 2^{q r \beta} 2^{- p r \alpha} \int_\T \E\Big[ \Big( \sum_k | ( \CF g^N_{p, q} (t, x, \cdummy) - \CF g^N_{p, q} (s, x, \cdummy) ) (k) |^2 \Big)^{r / 2} \Big] \dd x,
  \end{align*}
  where we applied Parseval's formula. Now $\sum_k | ( \CF g^N_{p, q} (t, x, \cdummy) - \CF g^N_{p, q} (s, x, \cdummy) ) (k) |^2$ is a random variable in the second inhomogeneous chaos generated by $\xi$, and therefore our claim follows from Nelson's estimate.
\end{proof}

\begin{lemma}\label{lem:random operator fourier mode}
  For all $p, q \geqslant - 1$, all $0 \leqslant t_1 < t_2$, and all $\lambda \in [0,1]$ we have
  \begin{equation}\label{eq:random operator fourier mode}
     \sum_{k} \E [ | \CF g_{p, q} (t_2, x, \cdummy) (k) - \CF g_{p, q} (t_1, x, \cdummy) (k) |^2 ] \lesssim \1_{2^p, 2^q \lesssim N} 2^{p(1-\lambda) + q} N^{- 1 + 3 \lambda} | t_2 - t_1 |^{\lambda}.
  \end{equation}
\end{lemma}

\begin{proof}
  We write
  \begin{align*}
     &\CF Q_N (t_2, k_1) \CF X_N (t_2, k_2) - \CF Q_N (t_1, k_1) \CF X_N (t_1, k_2) \\
     &\hspace{50pt} = ( \CF Q_N (t_2, k_1) - \CF Q_N (t_1, k_1)) \CF X_N (t_2, k_2) \\
     &\hspace{50pt}\quad + \CF Q_N (t_1, k_1) ( \CF X_N (t_2, k_2) - \CF X_N (t_1, k_2) )
  \end{align*}
  and estimate both terms on the right hand side separately. We only concentrate on the second one, which is slightly more difficult to treat, the first one being accessible to essentially the same arguments. We have the following chaos decomposition:
  \begin{align*}
     &\CF Q_N (t_1, k_1) ( \CF X_N (t_2, k_2) - \CF X_N (t_1, k_2) ) \\
     &\hspace{10pt} = \int_{(\R \times E_N)^2} \left( (2 \pi)^2 \delta_{\ell_1, k_1} \delta_{\ell_2, k_2} \int_0^{t_1} H^N_{t_1 - r} (\ell_1) H^N_{r - s_1} (\ell_1) \mathd r (H^N_{t_2 - s_2} (\ell_2) - H^N_{t_1 - s_2} (\ell_2)) \right) W (\mathd \eta_{12}) \\
     &\hspace{10pt}\quad + 2 \pi \int_{\R \times E_N} \left( \delta_{\ell_1, k_1} \delta_{- \ell_1, k_2} \int_0^{t_1} H^N_{t_1 - r} (\ell_1) H^N_{r - s_1} (\ell_1) \mathd r (H^N_{t_2 - s_1} (- \ell_1) - H^N_{t_1 - s_1} (- \ell_1)) \right) \mathd \eta_1.
  \end{align*}
  Using the same arguments as in Section~\ref{sec:stochastic regularity}, we can estimate the second term on the right hand side by
  \begin{align*}
     &\Big| \int_{\R \times E_N} \left( \delta_{\ell_1, k_1} \delta_{- \ell_1, k_2} \int_0^{t_1} H^N_{t_1 - r} (\ell_1) H^N_{r - s_1} (\ell_1) \mathd r (H^N_{t_2 - s_1} (- \ell_1) - H^N_{t_1 - s_1} (- \ell_1)) \right) \mathd \eta_1\Big| \\
     &\hspace{70pt} \lesssim \frac{\delta_{k_{-1},k_2}}{ | k_1 |} | t_2 - t_1 |^{\lambda /     2} | k_1 |^{\lambda}
  \end{align*}
  for any $\lambda \in [0,2]$. Based on the decomposition of Lemma~\ref{lem:random operator decomposition}, the contribution from the chaos of order 0 is thus bounded by
  \begin{align}\label{eq:random operator fourier mode pr1}
     &\sum_k \Big|\sum_{k_1, k_2 \in E_N} \Gamma_{p, q}^N (x ; k, k_1, k_2) \E[\CF Q_N (t_1, k_1) ( \CF X_N (t_2, k_2) - \CF X_N (t_1, k_2) )] \Big|^2 \\ \nonumber
     &\hspace{20pt} \lesssim \sum_k \Big| \sum_{k_1 \in E_N} \Gamma_{p, q}^N (x ; k, k_1, k_{-1}) \frac{| t_2 - t_1 |^{\lambda /     2} }{ | k_1 |^{1-\lambda}} \Big|^2 \\ \nonumber
     &\hspace{20pt} \lesssim \sum_k  \Big|  \rho_q (k) \tilde{\rho}_p (k) \sum_{ k_1 \in E_N}  \psi_{\prec} (k, k_1) (\psi_{\circ} ((k_1 -  k)^N, k_1) - \psi_{\circ} (k_1 -  k, k_1)) \frac{| t_2 - t_1 |^{\lambda /     2} }{| k_1 |^{1-\lambda}} \Big|^2,
  \end{align}
  where we used that $|k| < N/2$ on the support of $\psi_{\prec} (\cdot, k_1)$. Now $\psi_{\circ} ((k_1 -  k)^N, k_1) = \psi_{\circ} (k_1 -  k, k_1)$ unless $|k_1 - k| > N/2$ and $|k_1| \simeq N$, and there are at most $| k|$ values of $k_1$ with $|k|<|k_1|< N / 2$ and $| k_1 - k | > N / 2$. Therefore, the right hand side of~\eqref{eq:random operator fourier mode pr1} is bounded by
  \begin{align*}
     \lesssim | t_2 - t_1 |^{\lambda} \sum_k  \1_{| k | < N / 2} | \rho_q (k) \tilde{\rho}_p (k) |^2 | k |^2  N^{- 2 + 2 \lambda} & \lesssim \1_{p \sim q} \1_{2^q     \lesssim N} | t_2 - t_1 |^{\lambda} 2^{3 q} N^{- 2 + 2 \lambda} \\
     &\lesssim \1_{2^p, 2^q \lesssim N} 2^{p(1-\lambda)  + q} N^{- 1 + 3 \lambda} | t_2 - t_1 |^{\lambda} .
  \end{align*}
  Next consider the contribution from the second chaos, which is given by
  \begin{align*}
     &(2 \pi)^2 \int_{(\R \times E)^2} \Big( \Gamma_{p, q}^N (x ; k, k_1, k_2) \int_0^{t_1} H^N_{t_1 - r} (k_1) H^N_{r - s_1} (k_1) \mathd r \\
     &\hspace{150pt} \times (H^N_{t_2 - s_2} (k_2) - H^N_{t_1 - s_2} (k_2)) \Big) W(\mathd \eta_{12}).
  \end{align*}
  Taking the expectation of the norm squared, we obtain (up to a multiple of $2\pi$)
  \begin{align}\label{eq:random operator second integral}
     \int_{(\R \times E_N)^2} \left| \Gamma_{p, q}^N (x ; k, k_1, k_2) \int_0^{t_1} H^N_{t_1 - r} (k_1) H^N_{r - s_1} (k_1) \mathd r    (H^N_{t_2 - s_2} (k_2) - H^N_{t_1 - s_2} (k_2)) \right|^2 \mathd \eta_{12}.
  \end{align}
   As in Section~\ref{sec:stochastic regularity} we can show that
   \[
      \int_{\R^2} \Big| \int_0^{t_1} H^N_{t_1 - r} (k_1) H^N_{r -  s_1} (k_1) \mathd r  (H^N_{t_2 - s_2} (k_2) - H^N_{t_1 - s_2} (k_2)) \Big|^2 \mathd s_{12} \lesssim \frac{| k_2 |^{2  \lambda} | t_2 - t_1 |^{\lambda}}{| k_1 |^2}
   \]   
   whenever $\lambda \in [0,2]$, and plugging this into~(\ref{eq:random operator second integral}) we get
  \begin{align*}
     \eqref{eq:random operator second integral} & \lesssim \int_{E_N^2} | \Gamma_{p, q}^N (x ; k, k_1, k_2) |^2 \frac{| k_2 |^{2  \lambda} | t_2 - t_1 |^{\lambda}}{| k_1 |^2} \mathd k_{12} \\
     & \lesssim N^{2 \lambda} | t_2 - t_1 |^{\lambda} \int_{E^2_N}  \mathd k_{12} \tilde{\rho}^2_p (k) \psi^2_{\prec} (k, k_1) |k_1|^{-2} \\
     &\hspace{90pt} \times \Big| e^{i(k_{[12]} - k)^N x} \rho_q ((k_{[12]} - k)^N) \psi_{\circ} ((k_1 -  k)^N, k_2) \\
     &\hspace{130pt} - e^{i(k_{[12]} - k) x} \rho_q (k_{[12]} - k) \psi_{\circ} (k_1 -  k, k_2) \1_{|k_{[12]} - k|\leqslant N/2} \Big|^2 .
  \end{align*}
  First observe that the difference on the right hand side is zero unless $|k_1| \simeq N$, so that we can estimate $|k_1|^{-2} \lesssim N^{-1+\lambda} |k|^{-1-\lambda}$, and also we only have to sum over $|k_1| > |k|$. Moreover, the summation over $k_2$ gives $O(2^q)$ terms which leads to
  \[
     \eqref{eq:random operator second integral} \lesssim \1_{2^p,2^q \lesssim N} 2^q  N^{-1+3 \lambda} | t_2 - t_1 |^{\lambda} \tilde{\rho}^2_p (k) |k|^{-\lambda},
  \]
  and the sum over $k$ of the right hand side is bounded by
  \[
     \lesssim \1_{2^p,2^q \lesssim N} 2^{p(1-\lambda) + q} N^{-1+3 \lambda} | t_2 - t_1 |^{\lambda},
  \]
  which completes the proof of~\eqref{eq:random operator fourier mode}.
\end{proof}

\begin{proof}[Proof of Theorem~\ref{thm:random operator}]
   Let $\alpha \in (1/4,1/2)$, $r\geqslant 1$, and write $\gamma = 2\alpha-1+1/r$. Combining Lemma~\ref{lem:random operator kolmogorov} and Lemma~\ref{lem:random operator fourier mode}, we get for all $\lambda \geqslant 0$ with $(1-\lambda)/2 > \alpha$
   \begin{align*}
      \E[\| A_N (t) - A_N(s) \|^r_{L ( \CC^{\alpha}, \CC^{2\alpha-1})}]^{1/r} & \lesssim \E[\| A_N (f) (t) - A_N (f) (s) \|^r_{L ( \CC^{\beta}, B^{\gamma}_{r, r})}]^{1/r} \\
      & \lesssim \Big(\sum_{q, p \lesssim \log_2 N} 2^{q r \gamma} 2^{- p r \alpha} \big( 2^{p(1-\lambda) + q} N^{- 1 + 3 \lambda} | t - s |^{\lambda} \big)^{r / 2}\Big)^{1/r} \\
      & \simeq | t - s |^{\lambda / 2} N^{ \alpha - 1/2 + 1/r + 3\lambda/2}
   \end{align*}
   Since $A_N(0) = 0$, Kolmogorov's continuity criterion gives
   \[
      \E[\| A_N (f) \|^r_{C_T L ( \CC^{\alpha}, \CC^{2\alpha-1})}]^{1/r} \lesssim N^{\alpha - 1/2 + 1/r + 3\lambda/2}
   \]
   whenever $\lambda/2>1/r$. Choosing $1/r + 3\lambda/2 \leqslant \delta$, the estimate~\eqref{eq:random operator bound} follows.
\end{proof}

\appendix\section{Proofs of some auxiliary results}\label{app:some proofs}

\begin{proof}[Proof of Lemma~\ref{lem:modified paraproduct exp}]
  We prove the second claim, the first one follows from the same arguments. We also assume that $f\in \mathcal{M}^\gamma_t L^p$, the case $f \in \CM^\gamma_t L^\infty$ is again completely analogous. Let $t \in [0, T]$ and $j \geqslant 0$ and write $M = \|f\|_{\mathcal{M}^{\gamma}_t L^{p}} \|g (t) \|_{\beta}$. Then
  \[
     \Big\| \int_0^t 2^{2 j} \varphi (2^{2 j} (t - s)) S_{j - 1} f (s \vee 0) \mathd s \Delta_j g (t) \Big\|_{L^{p}} \leqslant 2^{- j \beta} M \int_0^t 2^{2 j} | \varphi (2^{2 j} (t - s)) | s^{- \gamma} \mathd s.
  \]
  We now split the integral at $t / 2$. The Schwartz function $\varphi$ satisfies $|
  \varphi (r) | \lesssim r^{- 1}$, so using the fact that $\gamma \in [0,1)$ we get
  \[
     \int_0^{t / 2} 2^{2 j} | \varphi (2^{2 j} (t - s)) | s^{- \gamma}  \mathd s \lesssim \int_0^{t / 2} (t - s )^{- 1} s^{- \gamma} \mathd  s \lesssim t^{- 1} \int_0^{t / 2} s^{- \gamma} \mathd s \lesssim t^{- \gamma} .
  \]
  The remaining part of the integral can be simply estimated by
  \[
     \int^t_{t / 2} 2^{2 j} | \varphi (2^{2 j} (t - s)) |  s^{- \gamma} \mathd s \leqslant (t / 2)^{- \gamma} \int_{\R} 2^{2 j} | \varphi (2^{2 j} (t - s)) | \mathd s \lesssim t^{- \gamma},
  \]
  which concludes the proof.
\end{proof}

\begin{proof}[Proof of Lemma~\ref{lem:modified paraproduct commutators exp}]
  Let us begin with the first claim. By spectral support properties it suffices to control
  \begin{align} \label{eq:modified paraproduct commutators pr1}
     & \Big\| \Big( \int_0^t 2^{2 j} \varphi (2^{2 j} (t - s)) S_{j - 1} f (s) \mathd s - S_{j - 1} f(t) \Big) \Delta_j g (t) \Big\|_{L^{p}} \\ \nonumber
     &\hspace{100pt} \leqslant \Big\| \int_{- \infty}^0 2^{2 j} \varphi (2^{2 j} (t - s))  S_{j - 1} f (t) \mathd s \Delta_j g (t) \Big\|_{L^{p}} \\ \nonumber
     &\hspace{100pt} \qquad + \Big\|\int_0^t 2^{2 j} \varphi (2^{2 j} (t - s)) S_{j - 1} f_{t, s} \mathd s  \Delta_j g (t) \Big\|_{L^{p}},
  \end{align}
  where we used that $\varphi$ has total mass 1. Using that $| \varphi(r) | \lesssim | r |^{- 1 - \alpha / 2}$, we can estimate the first term on the right hand side by
  \[
     \lesssim 2^{- j \beta} t^{- \gamma + \alpha / 2} \| f \|_{\LL^{\gamma, \alpha}_p(t)} \| g (t) \|_{\beta} \int_{2^{2 j} t}^{\infty} | r |^{- 1 - \alpha / 2} \mathd r \lesssim 2^{- j (\alpha + \beta)} t^{- \gamma} \| f \|_{\LL^{\gamma, \alpha}_p(t)} \| g (t) \|_{\beta}.
  \]
  As for the second term in~\eqref{eq:modified paraproduct commutators pr1}, we split the domain of integration into the intervals $[0,t / 2]$ and $[t / 2, t]$. On the first interval we use again that $| \varphi(r) | \lesssim | r |^{- 1 - \alpha / 2}$ and then simply estimate $\| f_{t, s} \|_{L^{p}} \leqslant (t^{\alpha / 2 - \gamma} + s^{\alpha / 2 -
  \gamma}) \| f \|_{\LL^{\gamma, \alpha}_p(t)}$ and $(t - s)^{- 1 - \alpha / 2} \lesssim t^{- 1 - \alpha / 2}$, and the required bound follows. On the second interval, an application of the triangle inequality shows that
  \[
     \| f_{t, s} \|_{L^{p}} = \| s^{- \gamma} (s^{\gamma} f (s)) - t^{- \gamma} (t^{\gamma} f (t)) \|_{L^{p}} \lesssim t^{- \gamma} | t - s |^{\alpha / 2} \| f \|_{\LL^{\gamma, \alpha}_p(t)},
  \]
  from where we easily deduce the claimed estimate.
  
  Let us now get to the bound $t^{\gamma} \left\| \left( \LL (f \mpara g) - f \mpara \left( \LL g \right) \right) (t) \right\|_{\CC_p^{\alpha + \beta - 2}} \lesssim \| f \|_{\LL^{\gamma, \alpha}_p(t)} \| g (t) \|_{\beta}$. Given Lemma~\ref{lem:modified paraproduct exp}, only the estimate for
  \[
     \sum_j \partial_t \left( \int_0^t 2^{2 j} \varphi (2^{2 j} (t - s)) S_{j- 1} u (s) \mathd s \right) \Delta_j v (t).
  \]
  is non-trivial. For fixed $j$ we obtain
  \[
     \int_0^t 2^{4 j} \varphi' (2^{2 j} (t - s)) S_{j - 1} u (s) \mathd s = 2^{2 j} \int_{\R} 2^{2 j} \varphi' (2^{2 j} (t - s)) S_{j - 1} u(s) \1_{s \geqslant 0} \mathd s.
  \]
  Since $\varphi'$ integrates to zero we can subtract $2^{2 j} \int_{\R} 2^{2 j} \varphi' (2^{2 j} (t - s)) S_{j - 1} u (t) \mathd s$. The result then follows as in the first part of the proof.
\end{proof}

\begin{proof}[Proof of Lemma~\ref{lemma:schauder exp}]
   We start by observing that exactly the same arguments as in Lemma~\ref{lemma:schauder} (replacing $L^\infty$ by $L^p$ at the appropriate places) yield
   \[
      \| I f\|_{\LL^{\infty,\alpha}_p(T)} \lesssim \| f \|_{C_T \CC^{\alpha-2}_p}, \qquad \| s \mapsto P_s u_0 \|_{\LL^{\infty,\alpha}_p(T)} \lesssim \| u_0 \|_{\CC^\alpha_p}.
   \]
   It remains to include the possible blow up at 0. The estimate~\eqref{eq:schauder without time hoelder exp},
   \begin{equation}\label{eq:schauder exp pr1}
      \| I f\|_{\CM^\gamma_T \CC^\alpha_p} \lesssim \| f \|_{\CM^\gamma_T \CC^{\alpha-2}_p}
   \end{equation}
   is shown in Lemma~A.9 of \cite{gubinelli_paraproducts_2012} (for $p = \infty$, but the extension to general $p$ is again trivial). For the temporal H\"older regularity of $I f$ that we need in~\eqref{eq:schauder initial contribution exp}, we note that the estimate
   \[
      \| P_t u \|_{\CC^{\alpha}_p} \lesssim t^{-(\alpha+\beta)/2} \|u\|_{\CC^{-\beta}_p}
   \]
   for $\beta \geqslant - \alpha$ is again a simple extension from $p = \infty$ to general $p$, this time of Lemma~A.7 in \cite{gubinelli_paraproducts_2012}. An interpolation argument then yields
   \begin{equation}\label{eq:schauder exp pr2}
      \| P_t u \|_{L^p} \lesssim t^{-\beta/2} \| u\|_{\CC^{-\beta}_p}
   \end{equation}
   whenever $\beta < 0$, and from here we obtain
   \begin{align}\label{eq:schauder exp pr3} \nonumber
      \|(P_t - \mathrm{id}) u\|_{L^p} & = \Big\| \int_0^t \partial_s P_s u \mathd s\Big\|_{L^p} = \Big\| \int_0^t P_s \Delta u \mathd s\Big\|_{L^p} \\
      & \lesssim \int_0^t s^{-1+\alpha/2} \| u \|_{\CC^\alpha_p} \mathd s \lesssim t^{\alpha/2} \| u \|_{\CC^\alpha_p}
   \end{align}
   for $\alpha \in (0,2)$. To estimate the temporal regularity, we now have to control
   \begin{align*}
      &\Big\| t^{\gamma} \int_0^t P_{t - r} f_r \mathd r - s^{\gamma} \int_0^s P_{t - r} f_r \mathd r \Big\|_{L^{p}} \\
      &\hspace{40pt} \leqslant (t^{\gamma} - s^{\gamma}) \int_0^t \| P_{t - r} f_r \|_{L^{p}} \mathd r  + s^{\gamma} \int_s^t \| P_{t - r} f_r \|_{L^{p}} \mathd r +  s^{\gamma} \Big\| (P_{t - s} - \tmop{id}) If(s) \Big\|_{L^{p}} \\
      &\hspace{40pt} \lesssim \Big( (t^{\gamma} - s^{\gamma}) \int_0^t (t - r)^{\alpha / 2 - 1} r^{- \gamma} \mathd r + s^{\gamma} \int_s^t (t - r)^{\alpha / 2 - 1} r^{-\gamma} \mathd r \Big) \| f \|_{\mathcal{M}^{\gamma}_T \CC^{\alpha - 2}_p} \\
      &\hspace{40pt} \qquad + s^{\gamma} | t - s |^{\alpha / 2} \| I f(s) \|_{\CC^\alpha_p},
  \end{align*}
  where we used~\eqref{eq:schauder exp pr2} for the first two terms and~\eqref{eq:schauder exp pr3} for the last term. Now~\eqref{eq:schauder exp pr1} allows us to further simplify this to
  \[
     \lesssim \Big( (t^{\gamma} - s^{\gamma})  t^{\alpha / 2 - \gamma} \int_0^1 (1 - r)^{\alpha / 2 - 1} r^{- \gamma} \mathd r + \int_s^t (t - r)^{\alpha / 2 - 1} \mathd r + | t - s |^{\alpha / 2} \Big) \| f \|_{\mathcal{M}^{\gamma}_T \CC^{\alpha - 2}_p}.
  \]
  Since $\alpha > 0$ and $\gamma < 1$, the first time integral is finite. Moreover it is not hard to see that $(t^{\gamma} - s^{\gamma})  t^{\alpha / 2 - \gamma} \lesssim | t - s |^{\alpha / 2}$ (distinguish for example the cases $s\leqslant t/2$ and $s \in (t/2,t]$), and eventually we obtain
  \[
     \| t^{\gamma} If(t) - s^{\gamma} If(s) \|_{L^{p}} \lesssim | t - s |^{\alpha / 2} \| f \|_{\mathcal{M}^{\gamma}_T \CC^{\alpha - 2}_p} .
  \]
  It remains to control the temporal regularity of $s \mapsto s^\gamma P_s u_0$, but this can be done using similar (but simpler) arguments as above, so that the proof is complete.
\end{proof}

\begin{proof}[Proof of Lemma~\ref{lem:lower explosive regularity}]
  For the spatial regularity, observe that
  \[
     t^{\gamma} \| \Delta_j u (t) \|_{L^{p}} \leqslant \min \{ 2^{- j \alpha}, t^{\alpha / 2} \} \| u \|_{\LL^{\gamma, \alpha}_p(T)}.
  \]
  Interpolating, we obtain
  \[
     t^{\gamma} \| \Delta_j u (t) \|_{L^{p}} \leqslant 2^{- j (\alpha - \varepsilon)} t^{\varepsilon / 2} \| u \|_{\LL^{\gamma, \alpha}_p(T)},
  \]
  or in other words $t^{\gamma - \varepsilon / 2} \| u (t) \|_{\alpha -
  \varepsilon} \leqslant \| u \|_{\LL^{\gamma, \alpha}_p(T)}$. The statement about
  the temporal regularity is a special case of the following lemma.
\end{proof}

\begin{lemma}
  Let $\alpha \in (0, 1)$, $\varepsilon \in [0, \alpha)$, and let $f \colon [0,\infty) \rightarrow X$ be an $\alpha$--H{\"o}lder continuous function with values in a normed vector space $X$, such that $f (0) = 0$. Then
  \[
     \| t \mapsto t^{- \varepsilon} f (t) \|_{C^{\alpha - \varepsilon}_T X} \lesssim \| f \|_{C^{\alpha}_T X} .
  \]
\end{lemma}

\begin{proof}
  Let $0 \leqslant s < t$. If $s = 0$, the required estimate easily follows
  from the fact that $f (0) = 0$, so assume $s > 0$. If $t > 2 s$, then
  we use again that $f (0) = 0$ to obtain
  \[ \frac{\| t^{- \varepsilon} f (t) - s^{- \varepsilon} f (s) \|_X}{| t - s
     |^{\alpha - \varepsilon}} \leqslant (t^{\alpha - \varepsilon} + s^{\alpha
     - \varepsilon}) | t - s |^{\varepsilon - \alpha} \| f \|_{C^{\alpha}_T X}
     . \]
  Now $t^{\alpha - \varepsilon} \leqslant (| t - s | + s)^{\alpha -
  \varepsilon}$, and $s / | t - s | < 1$ by assumption, thus the result
  follows. If $s < t \leqslant 2 s$, we apply a first order Taylor expansion
  to $t^{- \varepsilon} - s^{- \varepsilon}$ and obtain
  \[ \frac{\| t^{- \varepsilon} f (t) - s^{- \varepsilon} f (s) \|_X}{| t - s
     |^{\alpha - \varepsilon}} \leqslant \frac{t^{- \varepsilon}}{| t - s |^{-
     \varepsilon}} \frac{\| f (t) - f (s) \|_X}{| t - s |^{\alpha}} +
     \frac{\varepsilon r^{- \varepsilon - 1} (t - s) \| f (s) \|_X}{| t - s
     |^{\alpha - \varepsilon}} \]
  for some $r \in (s, t)$. Now clearly the first term on the right hand side
  is bounded by $\| f \|_{C^{\alpha}_T X}$. For the second term we use $\| f
  (s) \|_X \leqslant s^{\alpha} \| f \|_{C^{\alpha}_T X}$ and get
  \[ \frac{\varepsilon r^{- \varepsilon - 1} (t - s) \| f (s) \|_X}{| t - s
     |^{\alpha - \varepsilon}} \leqslant \varepsilon s^{- \varepsilon - 1 +
     \alpha} | t - s |^{\varepsilon + 1 - \alpha} \| f \|_{C^{\alpha}_T X}
     \leqslant \varepsilon \| f \|_{C^{\alpha}_T X}, \]
  using $t - s \leqslant s$ in the last step.
\end{proof}

%\bibliography{kpz-reloaded.bib}
%\bibliographystyle{amsalpha}

% \bib, bibdiv, biblist are defined by the amsrefs package.

\end{document}